\documentclass[12pt,twoside,reqno]{amsart}

\usepackage{amsmath, amssymb, amsthm, enumitem, esint}
\usepackage[hyperindex]{hyperref}

\usepackage{hyperref}
\hypersetup{bookmarksdepth=3}

\newcommand{\HKPropVarTriangle}{Lemma~3.2}

\newcommand{\HKThmGradientPhiEstimate}{Theorem~4.1}

\newcommand{\HKCenterConstantRmBound}{Proposition~9.5}

\newcommand{\HKVarMonotonicityCHFHnConcentration}{Corollaries~3.7, 3.8}
\newcommand{\HKHCenterDef}{Definition~3.10}

\newcommand{\HKHCenterPropExistBoundBallHCenter}{Propositions~3.12, 3.13}
\newcommand{\HKSecNashEntropy}{Section~5}
\newcommand{\HKSecPstar}{Section~9}
\newcommand{\HKPropPropertiesNash}{Proposition~5.2}
\newcommand{\HKVariationboundNN}{Corollary~5.11}
\newcommand{\HKLinftyHKbound}{Theorem~7.1}
\newcommand{\HKDefNN}{Definition~5.1}
\newcommand{\HKThmLinftyKbound}{Theorem~7.1}

\newcommand{\HKThmNabKboundThmNabKbound}{Theorems~7.1, 7.5}
\newcommand{\HKThmNabNNSquareNN}{Theorem~5.9}
\newcommand{\HKThmUpperVolBound}{Theorem~8.1}
\newcommand{\HKSecGradientEstimate}{Section~4}
\newcommand{\HKBasicCovering}{Theorem~9.11}
\newcommand{\HKThmHypercontractivity}{Theorem~12.1}
\newcommand{\HKThmEpsRegularity}{Theorem~10.2}

\newcommand{\HKThmNLC}{Theorem~6.1}
\newcommand{\HKPstarVolBound}{Theorem~9.8}
\newcommand{\HKThmHKboundNNbound}{Theorems~7.1, 5.9}
\newcommand{\HKThmHKboundGauss}{Theorem~7.2}
\newcommand{\HKThmsHKboundGaussPstarVolBound}{Theorems~7.2, 9.8}

\newcommand{\HKCorPnbhdinPstar}{Corollary~9.6}
\newcommand{\HKCorsPnbhdPstarboth}{Corollary~9.6}
\newcommand{\HKDefsPstarPNBHD}{Definitions~9.2, 9.3}
\newcommand{\HKPropPstarContainments}{Proposition~9.4}
\newcommand{\HKWoneMonotone}{Lemma~2.7}
\newcommand{\SYNCorCompactness}{Corollary~7.5} \newcommand{\SYNThmCompactnessFutCont}{Theorem~7.6} \newcommand{\SYNThmsIntrinsic}{Theorems~8.2, 8.4} \newcommand{\SYNSubsecRegPart}{Subsection~9.2} \newcommand{\SYNSubsecRFST}{Subsection~9.1} \newcommand{\SYNSubsecSmoothConv}{Subsection~9.4} \newcommand{\SYNThmRegPartProperties}{Theorem~9.12} \newcommand{\SYNThmSmoothConv}{Theorem~9.31} \newcommand{\SYNDefConvPtsWithin}{Definition~6.18} \newcommand{\SYNDefSmoothConv}{Definition~9.30} \newcommand{\SYNThmExistencConvSeq}{Theorem~6.45} \newcommand{\SYNConvParabNbhd}{Theorem~9.58}  \newcommand{\SYNChangeBaseConv}{Theorem~6.40} \newcommand{\SYNPropLocConc}{Proposition~9.17} \newcommand{\SYNVarIdentities}{Lemma~2.8}\newcommand{\SYNDefHK}{Definition~9.14} \newcommand{\SYNLemSubseqTimewiseAE}{Lemma~6.7} \newcommand{\SYNThmConvImpliesStrict}{Theorem~6.25} \newcommand{\SYNPropFutContTwoPts}{Proposition~4.40} \newcommand{\SYNLemStdParEst}{Lemma~9.15} \newcommand{\SYNThmConvSubsequwithinCF}{Theorem~6.49} \newcommand{\SYNThmTangenFlowasLimit}{Theorem~6.58}
\newcommand{\SYNFconvimplieswithin}{Theorem~6.9}

\newcommand{\IR}{\mathbb{R}}
\newcommand{\lb}{\linebreak[1]}

\newcommand{\IN}{\mathbb{N}}

\newcommand{\IZ}{\mathbb{Z}}

\newcommand{\LL}{\mathcal{L}}
\newcommand{\IF}{\mathbb{F}}
\newcommand{\NN}{\mathcal{N}}
\newcommand{\MM}{\mathcal{M}}
\newcommand{\RR}{\mathcal{R}}
\renewcommand{\SS}{\mathcal{S}}
\newcommand{\XX}{\mathcal{X}}
\newcommand{\WW}{\mathcal{W}}
\newcommand{\eps}{\varepsilon}
\newcommand{\la}{\lambda}
\newcommand{\tf}{\mathfrak{t}}
\newcommand{\CF}{\mathfrak{C}}
\newcommand{\ov}[1]{\overline{#1}}
\newcommand{\td}[1]{\widetilde{#1}}

\DeclareMathOperator*{\osc}{osc}
\DeclareMathOperator{\loc}{loc}
\DeclareMathOperator{\sign}{sign}

\DeclareMathOperator{\Var}{Var}
\DeclareMathOperator{\Int}{Int}
\DeclareMathOperator{\Ric}{Ric}
\DeclareMathOperator{\Rm}{Rm}

\DeclareMathOperator{\eucl}{eucl}
\DeclareMathOperator{\vol}{vol}

\DeclareMathOperator{\supp}{supp}

\DeclareMathOperator{\DIV}{div}
\DeclareMathOperator{\proj}{proj}

\DeclareMathOperator{\thick}{thick}
\DeclareMathOperator{\thin}{thin}
\newcommand{\tdrrm}{\td{r}_{\Rm}}
\newcommand{\rrm}{r_{\Rm}}

\newcommand{\dotcup}{\ensuremath{\mathaccent\cdot\cup}}
\newcommand{\EMPTY}[1]{}

\newtheorem{Theorem}[equation]{Theorem}
\newtheorem{Lemma}[equation]{Lemma}
\newtheorem{Corollary}[equation]{Corollary}
\newtheorem{Proposition}[equation]{Proposition}
\newtheorem{Claim}[equation]{Claim}

\newtheorem{Conjecture}[equation]{Conjecture}
\newtheorem{Assumption}[equation]{Assumption}
\newtheorem{Addendum}[equation]{Addendum}

\theoremstyle{definition}
\newtheorem{Definition}[equation]{Definition}
\theoremstyle{remark}
\newtheorem{Remark}[equation]{Remark}

\numberwithin{equation}{section}

\title{Structure theory of non-collapsed limits of Ricci flows}
\author{Richard H  Bamler}
\address{Department of Mathematics, UC Berkeley, CA 94720, USA}
\email{rbamler@berkeley.edu}
\thanks{This work was supported by NSF grant DMS-1906500.}
\date{\today}

\begin{document}
\begin{abstract}
In this paper we characterize non-collapsed limits  of Ricci flows.
We show that such limits are smooth away from a set of codimension $\geq 4$ in the parabolic sense and that the tangent flows at every point are given by gradient shrinking solitons, possibly with a singular set of codimension $\geq 4$.
We furthermore obtain a stratification result of the singular set with optimal dimensional bounds, which depend on the symmetries of the tangent flows.
Our methods also imply the corresponding quantitative stratification result and  the expected $L^p$-curvature bounds.

As an application of our theory, we obtain a description of the singularity formation of a Ricci flow at its first singular time and a thick-thin decomposition characterizing the long-time behavior of immortal flows.
These results generalize Perelman's results in dimension 3 to higher dimensions.
We also obtain a Backwards Pseudolocality Theorem and discuss several other applications.
\end{abstract}

\maketitle

\tableofcontents

\part{Introduction and preliminary discussion}

\section{Introduction}
\subsection{Opening}
The Ricci flow, which was introduced by Hamilton \cite{Hamilton_3_manifolds}, has become the subject of intensive research.
Its importance has become evident due to several  applications in  topology, Riemannian and K\"ahler geometry.
Perhaps one of the most spectacular applications were the resolutions of the Poincar\'e and Geometrization Conjectures by Perelman~\cite{Perelman1, Perelman2, Perelman3} using 3-dimensional Ricci flow.
In Perelman's work, and in most other applications of Ricci flow, a proper understanding of the finite-time singularity formation and asymptotic long-time behavior of the flow is key.
In dimension 3, Perelman obtained a satisfactory characterization of the singularity formation via blow-up techniques, which ultimately enabled him to extend the flow past the first singular time using a surgery process.
The same tools also allowed him to understand the long-time asymptotics of the flow to a certain extent; this was later improved by the author \cite{Bamler_longtime_0,Bamler_longtime_A,Bamler_longtime_B,Bamler_longtime_C,Bamler_longtime_D}.

Unfortunately, Perelman's techniques break down in dimensions $n > 3$ and a general theory characterizing finite-time singularities and the long-time behavior has not been available to date.
The goal of this paper is to fill this void.
More specifically, we will develop a compactness, partial regularity and structure theory of Ricci flows in all dimensions.
This theory will allow us to analyze the short-time singularity formation of the flow and bound the size of the singular set.
In addition, we will also use the theory to characterize the long-time behavior of immortal flows.
In dimension 3, our theory will essentially recover Perelman's results; so it can be seen as a generalization of his results to higher dimensions.

The study of blow-ups and singularity models of the Ricci flow goes back to Hamilton in the 90s \cite{Hamilton_RF_compactness, Hamilton-formation-sing}.
The most basic example of a singularity model would be the round sphere $S^n$, which evolves by rescaling and shrinks to a point under the flow in finite time.
Further generalizations of this example are round cylinders $S^k \times \IR^{n-k}$ or products $M^k \times \IR^{n-k}$, where $M^k$ is an Einstein manifold with positive Einstein constant.
Beyond that, Hamilton observed that due to the diffeomorphism invariance of the Ricci flow equation, there is a more general class of examples, called \emph{(gradient) shrinking solitons}.
These are selfsimilar solutions to the Ricci flow that shrink modulo diffeomorphisms.
Note that any gradient soliton trivially occurs as a singularity model of its own flow.

The following conjecture arose from Hamilton's work and has guided research in Ricci flow for almost two decades:

\begin{Conjecture} \label{Conj_main}
Finite-time singularities of the Ricci flow are ``mostly'' (or at least somewhere) modeled on a gradient shrinking soliton.
\end{Conjecture}

In dimension 3, this conjecture was verified by Perelman to a certain extent and led to the construction of Ricci flow with surgery.
In higher dimensions, Conjecture~\ref{Conj_main} remained, however, largely unresolved.
This was, in part, due to a lack of a suitable compactness theory for Ricci flows.
Conjecture~\ref{Conj_main} was verified in the case of Type I singularities (i.e. assuming a curvature bound of the form $|{\Rm}| \leq C(T-t)^{-1}$) by Sesum, Naber, Enders, Buzano, Topping, Cao, Zhang \cite{Sesum_conv_to_soliton, Naber_soliton, Enders_Mueller_Topping, Cao_Zhang_2011}, using Perelman's monotonicity formulas.
However, in the absence of the Type I condition, examples constructed by Appleton, Stolarski and Li, Tian, Zhu \cite{Appleton:2019ws, Stolarski:2019uk, Li_Tian_Zhu_Fano_singular} have shown that Conjecture~\ref{Conj_main} needs to be revised to allow singularity models that have a singular set themselves.

In this paper we will resolve Conjecture~\ref{Conj_main} in its revised form.
We will show that finite-time singularities in any Ricci flow can essentially be described by two different types of singularity models: gradient shrinking solitons and static, Ricci flat cones, each with a possible singular set of codimension $\geq 4$.
This implies that any finite-time singularity of the Ricci flow has at least one blowup of this form.
We will moreover construct a ``singular time-slice'' at the singular time.
Every sequence of spacetime points 
converging to the singular time 
subsequentially converges to a point in this ``singular time-slice'' and any such limit point is either regular or the Ricci flow near it is asymptotic to a gradient shrinking soliton.
We will also characterize the long-time asymptotics on an immortal solution to the Ricci flow.
More specifically, we will show that at large times the manifold can be decomposed into a ``thick'' and a ``thin'' part; the metric on the thick part essentially approaches a singular Einstein metric with negative Einstein constant and the thin part is locally collapsed in a certain sense.

Our results are based on the compactness and partial regularity theory of Ricci flows from \cite{Bamler_RF_compactness}, concerning $\IF$-limits.
This theory shows that \emph{any} sequence of pointed Ricci flows (of the same dimension) contains a subsequence that has an $\IF$-limit.
We will study this limit in the non-collapsed case, i.e. under a uniform local entropy bound; such a bound is always available in the study of finite-time singularities.
It will turn out that in this non-collapsed case, the $\IF$-limit can be described by a smooth Ricci flow spacetime away from a possible singular set of codimension $\geq 4$ (here we count the time direction twice).
The singular points of this limit, in turn, admit blowups (a.k.a. tangent flows) that are again possibly singular static Ricci flat cones or gradient shrinking solitons.
We will also show that the singular set is stratified according to the symmetries of tangent flows, and obtain the expected dimensional bounds.
In addition, we will obtain the corresponding quantitative stratification result for Ricci flows.
Lastly, we will discuss several applications of our compactness theory, including a Backwards Pseudolocality Theorem.

\subsection{Historical context} \label{subsec_history}
%
%
%
Our theory parallels the compactness and partial regularity theory --- and the corresponding quantitative stratification --- of spaces with lower Ricci curvature bounds and Einstein metrics, which was developed by Cheeger, Colding, Tian and Naber  \cite{Colding-vol-conv, Cheeger-Colding-Cone, Cheeger_Colding_struc_Ric_below_I,Cheeger_Colding_struc_Ric_below_III,Cheeger-Colding-structure-II, Cheeger_Colding_Tian_2002, Cheeger-Naber-quantitative, Cheeger-Naber-Codim4}.
In fact, in the special case of (trivial) Ricci flows evolving through Einstein metrics, our theory implies the results of \cite{Cheeger-Naber-Codim4}.

Similar compactness and partial regularity theories exist for many other geometric equations such as:  minimal surfaces \cite{Almgren_2000, Simon_1995, Federer_book,Cheeger-Naber-13, Naber_Valtorta_2015}, mean curvature flow \cite{White_94, White_97, White_2000, White_2003, White_2011, Cheeger_Haslhofer_Naber_2013, Cheeger_Haslhofer_Naber_2015}, harmonic maps \cite{Schoen_Uhlenbeck_82, Simon_96,Cheeger-Naber-13} and the harmonic map heat flow \cite{Chen_Struwe_1989,Lin_Wang_1999, Lin_1999, Cheeger_Haslhofer_Naber_2015}.
%
On the other hand, a comparable theory for Ricci flows, in its full generality, has been unavailable so far, however, partial progress has been obtained.

In dimension~3, Perelman \cite{Perelman1} obtained  curvature estimates that allowed the extraction of {\it smooth} sublimits in many cases.
So a compactness theory, as derived here, with possibly singular limits, is not necessary in this dimension.
Note that Perelman's compactness theory was specific to dimension 3, as it relied heavily on the existence of lower sectional curvature bounds (via the Hamilton-Ivey pinching), which are not available in higher dimensions \cite{Maximo_Ric_sing,ZhangZ_2014}.
The theory developed in this paper implies Perelman's theory in dimension 3 and does not rely on any lower sectional curvature bounds (see Subsection~\ref{subsec_dimensions23_intro} for more details).

In higher dimensions, partial progress has been obtained only under relatively strong additional assumptions.
In the case of bounded scalar curvature, which is natural in the setting of the Hamilton-Tian Conjecture, a compactness theory for Ricci flows was developed in \cite{Bamler-Annals-2018, Chen_Wang_2017_Part_A, Chen-Wang-II, Tian_Zhang_2016}.
This theory implies subsequential convergence away from a singular set of codimension $\geq 4$, assuming a uniform pointwise scalar curvature bound.
The theory developed in this paper implies these results, so it is a strict generalization.
The methods employed in this paper are, however, vastly different from those of previous work, because the general case demands a large array of new techniques to deal with various new geometric and analytic phenomena of the Ricci flow equation, as well as a new approach towards proving partial regularity.
Let us offer some more background.
In the previous work, the bounded scalar curvature assumption was crucial at several steps of the proof.
Amongst other things, it made it possible to reduce large parts of the problem to the analysis of \emph{time-slices} that satisfied certain Ricci curvature bounds;
these time-slices could, in turn, be analyzed using familiar techniques from Riemannian geometry.
The scalar curvature bound also guaranteed helpful \emph{distance distortion estimates} and justified the analysis of the spacetime geometry via \emph{worldlines.}
Unfortunately, none of these short-cuts work for general Ricci flows \cite{Bamler_HK_entropy_estimates, Bamler_RF_compactness}, so completely new ideas are needed in this setting.
For example, we will be forced to consider the set of \emph{all} time-slices at every step of the proof and we won't be able to apply any previous partial regularity theorems to time-slices.
In addition, the lack of distance distortion estimates forces us to do away with the concept of worldlines and study the \emph{spacetime geometry} of Ricci flows as a whole in more detail (see Subsection~\ref{subsec_key_diff_proof_strategy} and Section~\ref{sec_outline} for more details). 

In another direction, past progress on partial regularity of Ricci flows has been obtained under a Type-I curvature bounds ($|{\Rm}| \leq C (T- t)^{-1}$).
In this setting, Gianniotis  \cite{Gianniotis_Type_I} derived quantitative stratification results, which are similar to those of Einstein metrics.
These are also implied by our theory (see Subsection~\ref{subsec_quant_strat_intro}).

In the absence of curvature bounds, analysis on the singularity formation in dimensions $\geq 4$ was furthermore conducted in certain cohomogeneity-one settings, by Appleton \cite{Appleton:2019ws} and Stolarski \cite{Stolarski:2019uk}.
Their work proved the existence of Ricci flat, asymptotically conical blow-up limits, therefore providing valuable examples for our theory.
At an intermediate scale, these examples prove the existence of singular blow-up limits that are isometric to static solutions on a Ricci flat cone.
In the case of Appleton's work, the Ricci flat metric was the Eguchi-Hanson metric and the intermediate blow-up limit was isometric to $\IR^4/\IZ_2$.
This shows that our dimensional bounds are sharp, and that  we have to expect singular sets of codimension 4 in certain settings.
See also Subsection~\ref{subsec_dim4_intro} where some of the findings of Appleton's work are confirmed.

\subsection{Key difficulties and proof strategy} \label{subsec_key_diff_proof_strategy}
While compactness and partial regularity theories exist for a plethora of geometric equations, their proofs all share several core elements.
Among these are:
\begin{itemize}
\item a notion of geometric closeness combined with a compactness theory, which allows the extraction of synthetic limits (such as Gromov-Hausdorff limits),
\item a monotone quantity (such as the normalized volume),
\item an almost rigidity theorem, which guarantees some form of local closeness to a symmetric model (such as a cone) whenever this monotone quantity is almost constant.
\item a cone splitting theorem, guaranteeing geometric splitting whenever a space is conical with respect to  more than one point and
\item an $\eps$-regularity theorem, which establishes local regularity bounds near points where the space is close to a trivial model space (such as Euclidean space).
\end{itemize}

One of the main difficulties in obtaining a general compactness and partial regularity theory for Ricci flows was the absence of most of these ingredients in the Ricci flow setting.
Up until recently there was no sense of geometric convergence for Ricci flows or a general synthetic theory of a possible limit.
Next, a canonical choice for a monotone quantity would be the $\mathcal{W}$-functional, introduced by Perelman \cite{Perelman1}, or the pointed Nash-entropy \cite{Hein-Naber-14}, which are only constant on shrinking solitons.
However, an almost rigidity theorem implying some sort of closeness to a shrinking soliton was unknown up until this point.
Moreover, while a cone splitting theorem involving (smooth) solitons is not hard to derive, the necessary synthetic version has been unavailable.
Lastly, up until recently, the best $\eps$-regularity theorem for Ricci flows is due to Hein and Naber \cite{Hein-Naber-14}; however this theorem still relied on extra assumptions, such as an upper scalar curvature bound.


One of the main innovations of our proof is a new perspective on the spacetime geometry of Ricci flows, which was initiated in \cite{Bamler_HK_entropy_estimates, Bamler_RF_compactness} and is continued in the present paper.
Broadly speaking, this perspective unveils certain surprising geometric analogies between Ricci flows and Einstein metrics.
The first paper \cite{Bamler_HK_entropy_estimates} contains various new and improved geometric and analytic estimates involving the heat kernel on a Ricci flow, which relied on this new perspective.
The second paper \cite{Bamler_RF_compactness} introduces a new theory of $\mathbb{F}$-convergence, which allows the extraction of (synthetic) metric flows as sublimits from {\it any} sequence of pointed Ricci flows --- so it is comparable to the theory of Gromov-Hausdorff convergence for spaces with lower Ricci curvature bounds.
This work establishes \emph{some} of the ingredients listed above.
However, since the Ricci flow equation is far more complex and the analogy to the Einstein equation breaks down soon enough, several important ingredients, such as an almost cone rigidity or cone splitting theorem are left to be desired.\footnote{Note that the theory of metric flows is too weak to even deduce the relatively elementary cone splitting result.}
We are therefore forced to take a different route in proving partial regularity.

Our approach is further complicated by the fact that various elementary estimates, which are normally key in studying Ricci flows in specific settings, are not available in the general setting.
For example, there are heuristic reasons \cite{Bamler_HK_entropy_estimates, Bamler_RF_compactness} to assume that basic distance distortion and lower heat kernel bounds fail in the general setting.
To illustrate the severity of these complications, we note that as a consequence, we will be unable to localize most of our estimates (i.e., work in a ball or parabolic neighborhood); instead, most of our estimates \emph{must} be carried out \emph{globally.}
It will also be hard to relate points in different time-slices and compare integrals against different conjugate heat kernel densities.

Let us now briefly digress and use our hindsight to reevaluate the past progress on this topic (as detailed in Subsection~\ref{subsec_history}).
One could conclude that most of these previous results were possible, because the specific settings that were analyzed \emph{did} guarantee the elementary estimates mentioned above.
As a result, it was largely possible to reduce the study of Ricci flows to single time-slices (e.g., via worldlines) and apply available tools from Riemannian or comparison geometry.
Since this is not possible anymore in our case, we are forced to consider the \emph{entire} flow as a geometric object and conduct our estimates \emph{globally} at almost every step of the proof.

We now explain in some more detail how we will address a few of the issues presented so far.

First, we will follow a new overall strategy, which avoids almost rigidity theorems at the expense of a more quantitative cone splitting result.
Key to this novel strategy is to work with certain weak \emph{global} and \emph{analytic} splitting and selfsimilarity conditions for as long as possible and draw the desired geometric conclusions only in the very end via a limit argument.
In addition, we will use smooth convergence on the regular set to deduce various geometric properties of the \emph{entire} synthetic limit flow.
This can be seen in contrast to the common tactic of \emph{first} deriving softer synthetic conditions for the smooth flows and \emph{then} passing these to the limit.


Second, we derive a number of new estimates for Ricci flows, which allow us to circumvent the more elementary issues mentioned earlier in very specific settings.
These include:
\begin{itemize}
\item A  technique allowing us to exchange conjugate heat kernel densities in integral bounds. 
Associated to this, we will also derive new improved $L^p$-bounds on geometric quantities.
\item A distance-expansion bound for pairs of conjugate heat kernels near almost selfsimilar points.
\item An almost monotonicity result for the scalar curvature near almost selfsimilar points, allowing us  to show monotonicity of integrals of the form $\int_M \tau R \, d\nu_t$.
\item Several integral bounds characterizing almost soliton equations near almost selfsimilar points.
\item An improvement result for globally defined, weak splitting maps.
\item An $\eps$-regularity theorem involving weak or strong splitting maps.
\end{itemize}
We refer to Section~\ref{sec_outline} for a more detailed description of the proof.

\subsection{Acknowledgements}
I thank Bennett Chow, Gang Tian and Guofang Wei for some inspiring conversations.
Particular thanks also go to Bennett Chow and Yi Lai for pointing out some typos in an earlier version of the manuscript, Wangjian Jian for a very helpful and extensive list of further typos and Yuxing Deng, Max Hallgren, Zilu Ma and Yongjia Zhang for pointing out an issue in the proof of Proposition~\ref{Prop_almost_static}, which is fixed in this version.
Furthermore, I thank Qi Zhang and Christina Sormani for pointing out some incorrect references in an earlier version.

\section{Main results}
We will now describe the main  results of this paper in detail.
Most of these results involve notions introduced in \cite{Bamler_HK_entropy_estimates,Bamler_RF_compactness}, which we will recall or summarize whenever necessary.
Towards the end of each subsection we may provide some further details on each topic; the first-time reader may want to skip these details and continue with the following subsection.
The proofs of all main theorems can be found in Part~\ref{part_proofs} of this paper.

\subsection{Partial regularity theory and tangent flows of limits of Ricci flows} \label{subsec_part_reg_limit_intro}
The results of this paper rely on a partial regularity and structure theory of non-collapsed limits of Ricci flows, which we will describe now.

Let $(M_i, (g_{i, t})_{t \in (-T_i, 0]}, (x_i, 0))$ be a sequence of pointed Ricci flows on compact manifolds of the same dimension $n$ and $T_\infty := \lim_{i \to \infty} T_i \in (0, \infty]$.
Note that the basepoints $(x_i, 0) \in M_i \times (-T_i, 0]$ are chosen from the final time-slice.
By the results of \cite{Bamler_RF_compactness} we may pass to a subsequence and obtain $\IF$-convergence on compact time-intervals
\begin{equation} \label{eq_F_conv_intro}
 (M_i, (g_{i,t})_{t \in (-T_i,0]}, (\nu_{x_i,0;t})_{t \in (-T_i, 0]}) \xrightarrow[i \to \infty]{\quad \IF, \CF \quad} (\mathcal{X}, (\nu_{x_\infty;t})_{t \in (-T_\infty,0]}), 
\end{equation}
within some correspondence $\CF$, where we may assume that $\XX$ is a future continuous and $H_n$-concentrated metric flow of full support over $(-T_\infty,0]$ (for more details see  \cite{Bamler_RF_compactness} or Section~\ref{Sec_basic_limits}).

In general --- similar to the case of spaces with lower Ricci curvature bounds --- the time-slices of the limiting flow $\XX$ may be collapsed, i.e., have dimension $< n$.
In order to exclude this case, we will impose the following non-collapseding assumption throughout the remainder of this subsection for some uniform $0 < \tau_0 < T_\infty$, $Y_0 < \infty$: 
\begin{equation} \label{eq_intro_non_collapsing}
 \NN_{x_i, 0} (\tau_0) \geq - Y_0. 
\end{equation}
Here $\NN_{x_i, 0}$ denotes the pointed Nash-entropy \cite{Hein-Naber-14,Bamler_HK_entropy_estimates}, which is a monotone quantity that is related to Perelman's $\mathcal{W}$-functional.
In fact, a lower bound on Perelman's $\mathcal{W}$-functional, or more precisely the $\mu$-entropy, implies a condition of the form the form (\ref{eq_intro_non_collapsing}); so a bound of this form is ubiquitous in Ricci flow.
Vice versa, a lower bound on the pointed Nash-entropy at every point implies a lower bound on the $\mu$-entropy.
The advantage of a lower bound on the pointed Nash-entropy over a lower bound on Perelman's $\mu$-entropy is that it is more local.
It can be compared to a local volume bound of the form $|B(x,\tau_0^{1/2})| \geq e^{-Y_0} \tau_0^{n/2}$, while a lower bound on Perelman's $\mu$-entropy could be compared to a statement that such a volume bound holds for all $x \in M$.
We remark that it can actually be shown that a bound of the form $|B(x_i,\tau_0^{1/2})| \geq c \tau_0^{n/2}$ implies (\ref{eq_intro_non_collapsing}) for some $Y_0(c) < \infty$.
For more details see \cite{Bamler_HK_entropy_estimates}.

Our first main result states that if the non-collapsing condition (\ref{eq_intro_non_collapsing}) holds, then $\XX$ is regular away from a subset of parabolic Minkowski dimension $n-2$, i.e. of parabolic Minkowski \emph{co-dimension} $4$.
To state our result, we will consider the restriction $\XX_{<0}$ of our flow to negative times (recall that $\XX_{0} = \{ x_\infty \}$), and we consider the decomposition
\[ \XX_{<0} = \RR {\; \dot\cup \;} \SS. \]
into its regular ($\RR$) and singular ($\SS$) part.
We recall that the flow on $\RR$ is given by a smooth Ricci flow spacetime, which is the same notion as used in the analysis of 3-dimensional singular Ricci flows \cite{Kleiner_Lott_singular,bamler_kleiner_uniqueness_stability}. 
Note that, a priori, we may have $\RR = \emptyset$.

We can now state our first main result:

\begin{Theorem} \label{Thm_XX_reg_sing_dec_main}
$ \dim_{\mathcal{M}^*} \SS \leq n-2$.
\end{Theorem}

Here $\dim_{\mathcal{M}^*}$ denotes the $*$-Minkowski dimension from \cite{Bamler_RF_compactness}, which is similar to the standard definition, but uses $P^*$-parabolic balls instead of distance balls and counts the time-direction twice.

Note that this bound, which is the analog of \cite{Cheeger-Naber-Codim4} in the Einstein case, is optimal.
Consider for example a blow-down sequence of static flows corresponding to the (Ricci flat) Eguchi-Hanson metric.
Its limit is a static flow on the orbifold $\IR^4/\IZ_2$.

An important consequence of Theorem~\ref{Thm_XX_reg_sing_dec_main} is that the regular part $\RR$ is dense in $\XX$.
This and further characterizations of the limit are giving in the next theorem.

\begin{Theorem} \label{Thm_XX_reg_determines_XX_intro}
The following is true for any $t \in (-T_\infty,0)$:
\begin{enumerate}[label=(\alph*)]
\item $\RR \subset \XX_{<0}$ is dense with respect to the natural topology on $\XX$.
\item $\SS_t := \SS \cap \XX_t$ is a set of measure zero.
\item The time-slice $(\XX_t, d_t)$ arises as a metric completion of the length metric on $(\RR_t, g_t)$.
In other words, $\RR_t := \RR \cap \XX_t \subset \XX_t$ is open and dense and the restriction of $d_t$ to $\RR_t$ agrees with the length metric of $g_t$.
\item Every uniformly bounded $u \in C^0 ( \RR_{ [t', 0)} ) \cap C^\infty ( \RR_{(t', 0)})$, $t' \in (-T_\infty,0)$, that satisfies the heat equation $\square u = 0$ on $ \RR_{ (t', 0)}$ is a restriction of the heat flow with initial condition $u  \chi_{\RR_{t'}} : \XX_{t'} \to \IR$ (defined using the metric flow structure on $\XX$) restricted to $\RR_{ [t', 0)}$.
\end{enumerate}
\end{Theorem}

We remark that Theorem~\ref{Thm_XX_reg_determines_XX_intro} implies that the Ricci flow spacetime structure on $\RR$ fully characterizes the metric flow $\XX_{<0}$.
To see this, note that the time-slices of $\XX_{<0}$ arise from the time-slices of $\RR$ by passing to the metric completion and the heat equation on $\RR$ fully characterizes heat flows on $\XX_{<0}$.
We refer to Theorem~\ref{Thm_XX_uniquely_det_RR} for a more formal statement that $\RR$ determines $\XX_{<0}$.

The next theorem addresses the convergence on the regular part $\RR$.
Recall that $\RR \subset \XX$ is defined solely based on the geometry of $\XX$ and does not depend on the sequence of Ricci flows $(M_i, (g_{i,t})_{t \in (-T_i,0]})$ or the convergence (\ref{eq_F_conv_intro}).
So it could a priori happen that $\XX$ is regular in some region even though the flows $(M_i, (g_{i,t})_{t \in (-T_i,0]})$ degenerate there.
The next theorem states that this cannot be the case.

\begin{Theorem} \label{Thm_R_RR_star_intro}
The $\IF$-convergence (\ref{eq_F_conv_intro}) is smooth on all of $\RR$.
So, using the terminology from \cite{Bamler_RF_compactness}, we have $\RR^* = \RR$.
\end{Theorem}

Next, we characterize the tangent flows of $\XX$.
We recall from \cite{Bamler_RF_compactness} that a tangent flow at some point $x \in \XX$ is an $\IF$-limit of a sequence of parabolic rescalings of $\XX$, equipped with the conjugate heat kernel measure $(\nu_{x;t})_{t \leq \tf(x)}$.
A priori, there may be many tangent flows at $x$, corresponding to different sequences of rescaling factors.
If $x \in \RR$, then all tangent flows at $x$ are static flows on $\IR^n$.
In the case of limits of Einstein metrics, or metrics with lower Ricci curvature bounds, it was shown by Cheeger and Colding \cite{Cheeger-Colding-Cone} that every tangent cone is a metric cone.
The analog of a metric cone for metric flows is a metric soliton \cite{Bamler_RF_compactness}, which is a metric flow equipped with a conjugate heat flow that is self-similar under a family of time-dilations.
In fact, we have:

\begin{Theorem} \label{Thm_tangent_cone_metric_cone_main}
Every tangent flow $(\XX', (\nu_{x'_\infty;t})_{t \leq 0})$ at every point $x \in \XX$ is a metric soliton.
Moreover, $\XX'$ corresponds to the constant flow on $\IR^n$ if and only if $x \in \RR$.
\end{Theorem}

For further characterizations of tangent flows, or more generally $\IF$-limits that are shrinking solitons, see Subsection~\ref{subsec_further_struc_static_soltion}.
The following addendum foreshadows some of these results.

\begin{Addendum} \label{Add_tangent_flow}
$(\XX', (\nu_{x'_\infty;t})_{t \leq 0})$ in Theorem~\ref{Thm_tangent_cone_metric_cone_main} is again an $\IF$-limit of smooth Ricci flows, so all the previous theorems also hold for $\XX'$.
In addition, if we write $\tau = - t$ and $\nu_{x'_\infty} = (4\pi \tau)^{-n/2} e^{-f'} dg'$ on $\RR'$, then
\[ \Ric + \nabla^2 f' - \frac1{2\tau} g' = 0\]
and the time-slices $(\XX'_t, |t|^{-1/2} d'_t, \nu_{x'_\infty;t})$ (resp. $(\RR'_t, |t|^{-1} g'_t, f'_t)$) for all $t < 0$ are all isometric as metric measure spaces (resp. Riemannian manifolds equipped with a potential function).
\end{Addendum}

The next result is a stratification result for $\XX$, which is analogous to the stratification result of Cheeger and Naber \cite{Cheeger-Naber-quantitative,Cheeger-Naber-Codim4}.
Roughly speaking, the following theorem states that $\SS$ possesses a filtration of subsets $\SS^k \subset \SS$ of dimension $\leq k$ such that every point $x \in \XX_{<0} \setminus \SS^{k-1}$ has a tangent flow $\XX'$ that is of the form $\XX'' \times \IR^k$ for some metric soliton $\XX''$ or $\XX'' \times \IR^{k-2}$ for some metric soliton $\XX''$ that is even a static cone \cite{Bamler_RF_compactness}.
Here a static cone is essentially a metric flow that is constant in time and consists of time-slices that are isometric to metric cones.

\begin{Theorem} \label{Thm_stratification_limit_main}
There is a filtration of the singular part
 \[ \SS^0 \subset \SS^1 \subset \ldots \subset \SS^{n-2} = \SS \]
 such that for every $k =0, \ldots, n-2$:
 \begin{enumerate}[label=(\alph*)]
\item \label{Thm_stratification_limit_main_a} 
 $\dim_{\mathcal{H}^*} \SS^k \leq k$.
 \item \label{Thm_stratification_limit_main_b} Every point $x \in \XX_{<0} \setminus \SS^{k-1}$ has a tangent flow $(\XX', (\nu_{x':t})_{t \leq 0})$ that is a metric soliton and satisfies one of the following conditions:
 \begin{enumerate}[label=(b\arabic*)]
 \item \label{Thm_stratification_limit_main_b1} $\XX'_{<0} \cong \XX''_{<0} \times \IR^{k}$ and $(\nu_{x';t})_{t <0} \cong (\mu''_t \otimes \mu^{\IR^k}_t)_{t<0}$ for some metric soliton $(\XX'', (\mu''_t)_{t < 0})$. 
 \item \label{Thm_stratification_limit_main_b2} $\XX'_{<0} = \XX''_{<0} \times \IR^{k-2}$ for some static cone $\XX''$ with vertex $x'' \in \XX''_0$ and $(\nu_{x';t})_{t <0} = (\nu_{x'';t} \otimes \mu^{\IR^{k-2}}_t)_{t<0}$.
 \end{enumerate}
 Here $(\mu^{\IR^k}_t)_{t < 0}$ denotes the standard conjugate heat kernel on $\IR^k$ based at the origin.
 \end{enumerate}
\end{Theorem}

Note that the Cartesian splittings and the terms ``metric soliton'' and ``static cone'' are defined purely in terms of the synthetic structure of $\XX$.
The following addendum provides some extra details on the regular ($\RR', \RR''$) and singular parts ($\SS', \SS''$) of $\XX', \XX''$.
For more details, see again Subsection~\ref{subsec_further_struc_static_soltion}.

\begin{Addendum} \label{Add_limit_stratification_addendum_main}
In Theorem~\ref{Thm_stratification_limit_main} the flow $\XX'$ can be required to be an $\IF$-limit of smooth Ricci flows, so it also satisfies Theorems~\ref{Thm_XX_reg_sing_dec_main}, \ref{Thm_XX_reg_determines_XX_intro}, \ref{Thm_tangent_cone_metric_cone_main}, \ref{Thm_stratification_limit_main}.
In particular, we can require that the pair $(\XX', (\nu_{x';t})_{t \leq 0})$ satisfies the properties of Addendum~\ref{Add_tangent_flow}.
Moreover, we can also require that $\XX''$ satisfies Theorems~\ref{Thm_XX_reg_sing_dec_main}, \ref{Thm_XX_reg_determines_XX_intro}, \ref{Thm_tangent_cone_metric_cone_main}, \ref{Thm_stratification_limit_main} and if we denote by $\RR''$ its regular part, then in the corresponding two cases of Theorem~\ref{Thm_stratification_limit_main}\ref{Thm_stratification_limit_main_b} we can require that:
\begin{enumerate}[label=(b\arabic*)]
\item \label{Add_limit_stratification_addendum_main_b1} $\RR' \cong \RR'' \times \IR^{k}$ as Ricci flow spacetimes.
\item \label{Add_limit_stratification_addendum_main_b2} $\RR' \cong \RR'' \times \IR^{k-2}$ as Ricci flow spacetimes. 
$\Ric \equiv 0$ on $\RR'$ and $\RR''$. 
All time-slices of $\RR''$ are isometric to each other as Riemannian manifolds and isometric to the same Riemannian cone, possibly minus its vertex.
\end{enumerate}
In both cases the potential function from Addendum~\ref{Add_tangent_flow} can be expressed as $f' = f'' + \frac1{4\tau} |\vec x|^2$, where $f''$ is a gradient shrinking soliton potential on $\RR''$ and $\vec x$ denotes the coordinate function on the last factor of the Cartesian product of $\RR'$.
\end{Addendum}
\bigskip

\subsection{The pointed Nash entropy in the limit} \label{subsec_Nash_in_limit_intro}
Suppose that we are still in the same setting as in Subsection~\ref{subsec_part_reg_limit_intro} and consider the limiting metric flow $\XX$.
Recall that for any $x \in \XX_t$ and $-T_\infty < s < t$ we have $d\nu_{x;s} = K(x, \cdot) dg_s$ on $\RR_s$ and that $\nu_{x;s} (\SS_s) = 0$.
So we can define the pointed Nash entropy at $x$ as usual:
Write $K(x,\cdot) = (4\pi\tau)^{-n/2} e^{-f}$ on $\RR_s$ for $\tau = t-s$ and $f \in C^\infty (\RR_s)$ and set
\[ \NN_x (\tau) := \int_{\RR_s} f \, d\nu_{x;s} - \frac{n}2. \]

The following result shows that the pointed Nash entropy on the Ricci flows $(M_i, (g_{i,t})_{t \in (-T_i, 0]})$ converges to the pointed Nash entropy on $\XX$.

\begin{Theorem} \label{Thm_NN_pass_to_limit}
Consider a sequence $(x'_i, t'_i) \in M_i \times (-T_i,0]$ and a point $x'_\infty \in \XX_{t'_\infty}$ such that we have the following convergence within $\CF$:
\[ (x'_i, t'_i) \xrightarrow[i \to \infty]{\quad \CF \quad} x'_\infty. \]
Then for any $\tau > 0$ with $t'_\infty - \tau > -T_\infty$ we have $\lim_{i \to \infty} \NN_{x'_i,t'_i} (\tau) = \NN_{x'_\infty} (\tau)$.
\end{Theorem}

In particular, this implies that $\NN_{x'_\infty} (\tau)$ is still non-increasing in $\tau$ and that the pointed Nash entropy on $\XX$ satisfies the same variation bounds as in \cite{Bamler_HK_entropy_estimates}.
Set in the following for any $x \in \XX$
\[ \NN_{x} (0) := \lim_{\tau \to 0} \NN_x (\tau). \]
The following theorem extends Theorem~\ref{Thm_tangent_cone_metric_cone_main} and Addendum~\ref{Add_tangent_flow}.

\begin{Theorem} \label{Thm_Nash_on_tangent_cones}
If $(\XX', (\nu_{x'_\infty;t})_{t \leq 0})$ is a tangent flow of $\XX$ at some point $x \in \XX$, then for all $\tau \geq 0$ we have
\[ \NN_{x'_\infty;t}(\tau) = \NN_{x} (0). \]
Moreover, there is a dimensional constant $\eps_n > 0$ such that if $\NN_x(0) \geq - \eps_n$, then $\NN_x(0) = 0$ and $x \in \RR$.
\end{Theorem}
\bigskip

\subsection{Limits of Ricci flows that are not defined at time 0} \label{subsec_more_general_F_limit}
Before continuing with the presentation of further results, we first generalize the results described so far to $\IF$-limits of Ricci flows $(M_i, (g_{i, t})_{t \in (-T_i, 0]})$ that are not required to be defined at time $0$.
The first-time reader may want to skip this subsection and, for simplicity, consider the setting from Subsection~\ref{subsec_part_reg_limit_intro} instead of the following construction.

Note that in Subsection~\ref{subsec_part_reg_limit_intro} we have chosen basepoints $(x_i,0)$, which were located at the final time $0$.
Since the given Ricci flows may not exist at this time, we will instead consider conjugate heat flows on the entire flows whose variance vanishes asymptotically near time $0$.
More specifically, let $(M_i, (g_{i, t})_{t \in (-T_i, 0)})$ be a sequence of Ricci flows on compact, $n$-dimensional manifolds, defined over the \emph{open time-intervals $(-T_i, 0)$} and $T_\infty := \lim_{i \to \infty} T_i \in (0, \infty]$.
Consider conjugate heat flows $(\mu_{i,t})_{t \in (-T_i, 0)}$ on these flows, i.e. $d\mu_{i,t} = v_{i,t} dg_{i,t}$ where $\square^* v_i =0$, with the property that $\lim_{t \nearrow 0} \Var (\mu_{i,t}) = 0$.
Note that in the setting from Subsection~\ref{subsec_part_reg_limit_intro}, we may take $(\mu_{i,t})_{t \in (-T_i, 0)} := (\nu_{x_i;t})_{t \in (-T_i,0)}$.
By the compactness theory in \cite{Bamler_RF_compactness} we may again pass to a subsequence such that we have $\IF$-convergence on compact time-intervals
\begin{equation} \label{eq_F_conv_intro_open_t_i}
 (M_i, (g_{i,t})_{t \in (-T_i,0)}, (\mu_{i,t})_{t \in (-T_i, 0)}) \xrightarrow[i \to \infty]{\quad \IF, \CF \quad} (\mathcal{X}^*, (\mu_{\infty;t})_{t \in (-T_\infty,0)}), 
\end{equation}
within some correspondence $\CF$, where we may assume that $\XX^*$ is a future continuous and $H_n$-concentrated metric flow of full support over $(-T_\infty,0)$.
We may now extend $\XX^*$ to a metric flow $\XX$ over $(-T_\infty,0]$ by setting $\XX_{<0} := \XX^*$, $\XX_0 := \{ x_{\max} \}$ and $\nu_{x_{\max};t} := \mu_{\infty;t}$ ($t < 0$), $\nu_{x_{\max};0} := \delta_{x_{\max}}$.
Note that if the Ricci flows  $(g_{i,t})_{t \in (-T_i,0)}$ can be extended to time $0$ and if $(\mu_{i,t})_{t \in (-T_i, 0)} := (\nu_{x_i;t})_{t \in (-T_i,0)}$, then we obtain the same limit as in Subsection~\ref{subsec_part_reg_limit_intro}.

Next, we need to express the non-collapsing condition (\ref{eq_intro_non_collapsing}) using the conjugate heat flows $(\mu_{i,t})_{t \in (-T_i, 0)}$.
To do this, we write $d\mu_{i,t} = (4\pi |t|)^{-n/2} e^{-f_{i,t}} dg_{i,t}$ and define for $\tau \in (0, T_i)$
\[ \NN_{(\mu_{i,t})} (\tau) := \int_{M_i} f_{i,-\tau} \, d\mu_{i,-\tau} - \frac{n}2. \]
We can now replace (\ref{eq_intro_non_collapsing}) by
\begin{equation} \label{eq_intro_non_collapsing_more_general}
 \NN_{(\mu_{i,t})} (\tau_0) \geq - Y_0. 
\end{equation}
Note that if $(g_{i,t})_{t \in (-T_i,0)}$ can be extended to time $0$, and if $(\mu_{i,t})_{t \in (-T_i, 0)} := (\nu_{x_i;t})_{t \in (-T_i,0)}$, then $\NN_{(\mu_{i,t})} (\tau) = \NN_{x_i,0} (\tau)$.
So in this case (\ref{eq_intro_non_collapsing_more_general}) is equivalent to (\ref{eq_intro_non_collapsing}).
We also note that for any sequence $(y_j, t'_j) \in M_i \times (0,T_i)$ with $t'_j \nearrow 0$ and $d_{W_1}^{g_{t'_j}} (\delta_{y_j}, \mu_{i,t'_j}) \to 0$ we have $\NN_{y_j, t'_j} (\tau + t'_j) \to \NN_{(\mu_{i,t})} (\tau)$ as $j \to \infty$.
Such a sequence $(y_j, t'_j)$ always exists.

The following theorem states that our theory also applies to $\IF$-limits $\XX$ obtained from (\ref{eq_F_conv_intro_open_t_i}).

\begin{Theorem}  \label{Thm_more_general_F_limit}
All theorems stated so far, including the Addendums~\ref{Add_tangent_flow}, \ref{Add_limit_stratification_addendum_main}, also hold in this more general setting.
In particular, the tangent flows at $x_\infty \in \XX_0$ may be obtained as an $\IF$-limit of smooth Ricci flows as in (\ref{eq_F_conv_intro_open_t_i}).
Moreover, Theorem~\ref{Thm_NN_pass_to_limit} continues to hold if $x'_\infty = x_\infty$ and if we set $d\mu_{i} := (4\pi|t|)^{-n/2} e^{-f_i}dg_i$, then we also have for $\tau \in (0, T_\infty)$
\[ \NN_{x_\infty} (\tau) = \lim_{i \to \infty} \NN_{(\mu_{i,t})} (\tau). \]
\end{Theorem}
\bigskip

\subsection{Further structural results of tangent flows and an almost cone rigidity theorem for flows} \label{subsec_further_struc_static_soltion}
The goal of this subsection will be to characterize the geometric properties of the tangent flows from Theorems~\ref{Thm_tangent_cone_metric_cone_main}, \ref{Thm_stratification_limit_main} in more detail.
We will see that these tangent flows can be understood by their so-called geometric models.
We will also obtain further geometric characterizations of these models.

Throughout this subsection we suppose that $\XX$ is a metric flow over $(-T_\infty,0)$ that is obtained as a non-collapsed $\IF$-limit of a sequence of smooth Ricci flows as described in Subsection~\ref{subsec_more_general_F_limit}.
Note that by Addendum~\ref{Add_limit_stratification_addendum_main} and Theorem~\ref{Thm_more_general_F_limit} every tangent flow has this property.

We first recall the notion of a singular metric space from \cite{Bamler_struc_theory_sing_spaces}.
Consider a metric space $(X,d)$ and fix a dimension $n \geq 0$.
A point $x \in X$ is called {\bf regular} if there is an $n$-dimensional Riemannian manifold $(M, g)$ and a local isometry $\phi : (M, d_g) \to (X,d)$ with $x \in \phi (M)$.
The set of all regular points, called the {\bf regular part} $\RR_X \subset X$, is open and can be equipped with a natural structure of a Riemannian manifold $(\RR_X, g_X)$, such that the inclusion map $(\RR_X, d_{g_X}) \to (X,d)$ is a local isometry.
We also define the {\bf singular part} by $\SS_X := X \setminus \RR_X$.
We now define:

\begin{Definition}[Singular space]
A metric space $(X,d)$ is called a {\bf singular space (of dimension $n$)} if the following holds:
\begin{enumerate}
\item $(X,d)$ is a complete metric length space.
\item The regular part $\RR_X \subset X$ is dense in $X$.
\item The restriction of $d$ to $\RR_X$ is equal to the length metric $d_{g_X}$.
In other words, $(X,d)$ is isometric to the completion of $(\RR_X, d_{g_X})$.
\item For any compact subset $K \subset X$ and any $D < \infty$ there are constants $0 < \kappa_1 (K,D) < \kappa_2 (K,D) < \infty$ such that for all $x \in K$ and $0 < r < D$ we have
\[ \kappa_1 r^n < |B(x,r) \cap \RR_X |_{g_X} < \kappa_2 r^n. \]
\end{enumerate}
\end{Definition}

We can now state our main results concerning the $\IF$-limit $\XX$.
Our first result addresses the case in which this limit is Ricci flat on its regular part.
This case includes the case of a static cone as in Theorem~\ref{Thm_stratification_limit_main}\ref{Thm_stratification_limit_main_b2}.

\begin{Theorem} \label{Thm_static_limit_main}
Suppose that $\Ric \equiv 0$ on $\RR \subset \XX$.
Then $\XX_{<0}$ is static as a metric flow over $(-T_\infty,0)$ and there is a singular space $(X,d)$ of dimension $n$ and an identification $\XX_{<0} = X \times (-T_\infty,0)$ such that the following is true for all $t \in (-T_\infty,0)$:
\begin{enumerate}[label=(\alph*)]
\item \label{Thm_static_limit_main_a} $(\XX_t, d_t) = (X \times \{ t \}, d)$.
\item \label{Thm_static_limit_main_b} $\RR = \RR_X \times (-T_\infty,0)$ and $\partial_{\tf}$ corresponds to the standard vector field on the second factor.
\item \label{Thm_static_limit_main_c} $(\RR_t, g_t) = (\RR_X \times \{ t \}, g_X)$.
\end{enumerate}
Moreover, there is a unique family of probability measures $(\nu'_{x;t})_{x \in X; (-T_\infty,0) \cap (t + (-T_\infty,0)) \neq \emptyset}$ such that the triple
\[ (X,d, (\nu'_{x;t})_{x \in X; (-T_\infty,0) \cap (t + (-T_\infty,0)) \neq \emptyset} ) \]
is a static model for $\XX_{<0}$ corresponding to the same identification.
Therefore, $(X,d)$ characterizes $\XX_{<0}$ up to flow-isometry.
Lastly, the space $(X,d)$ has the following properties:
\begin{enumerate}[label=(\alph*), start=4]
\item \label{Thm_static_limit_main_d} We have the Minkowski dimension bound
\[ \dim_{\mathcal{M}} \SS_X \leq n-4. \]
\item \label{Thm_static_limit_main_e} Every tangent cone $(X',d',x')$ at any point $x$ of $(X,d)$ is a metric cone and a singular space that arises from a non-collapsed $\IF$-limit of a sequence of smooth Ricci flows in the same way as described in this theorem.
More specifically, if $(X',d',x')$ is obtained using a sequence of blow-up factors $\lambda_j \to \infty$, then the same sequence produces a tangent flow of $\XX$ at any point in $\{ x \} \times (-T_\infty,0)$ that is a static cone and whose time-slices are isometric to $(X',d')$.

Lastly, when considering the Gromov-Hausdorff convergence of blow-ups of $(X,d,x)$ to $(X',d',x')$, the convergence is locally smooth on $\RR_{X'}$.
\item \label{Thm_static_limit_main_f} There is a filtration $\SS^0_X \subset \SS^1_X \subset \ldots \subset \SS^{n-4}_X = \SS_X$ with the property that $\dim_{\mathcal{H}} \SS^k \leq k$ and every point $x \in \SS_X \setminus \SS^{k-1}_X$ has a tangent cone that metrically splits off an $\IR^k$-factor.
\end{enumerate}
\end{Theorem}

We note that if $\XX_{<0}$ is even a static cone as in Theorem~\ref{Thm_stratification_limit_main}\ref{Thm_stratification_limit_main_b2}, then $(X,d)$ is a metric cone.

\begin{Remark}
In forthcoming work we will show that the regular part $\RR_X \subset X$ has the following weak convexity property, which is also called ``mild singularities'' in \cite{Bamler_struc_theory_sing_spaces}.
For any $x \in \RR_X$ there is a closed subset $Q_x \subset \RR_X$ of measure zero such that for any $y \in \RR_X \setminus Q_x$ there is a minimizing geodesic between $x,y$ whose image lies in $\RR_X$.
This property allows us to deduce standard volume comparison estimates.
\end{Remark}

Our next result concerns the case in which the $\IF$-limit $\XX$ has constant Nash entropy at $x_\infty$.
In this case the flow is a metric soliton, which can again be described by a singular space with the same properties as before.

\begin{Theorem} \label{Thm_metric_soliton_limit_main}
Suppose that for some $W \in \IR$ we have $\NN_{x_\infty} (\tau) = W$ for all $\tau \in (0, T_\infty)$.
Then $(\XX_{<0}, (\nu_{x_\infty;t})_{t \in (-T_\infty,0)})$ is a metric soliton.
Moreover, if we write $\tau = -t$ and $d\nu_{x_\infty} = K(x_\infty; \cdot) dg = (4\pi \tau)^{-n/2} e^{-f} dg$ on $\RR$, then
\begin{equation} \label{eq_sol_id_met_sol_limit_thm}
\Ric + \nabla^2 f - \frac1{2\tau} g = 0, \qquad 
- \tau (R + |\nabla f|^2) + f = W, \qquad
\partial_{\tf} f = |\nabla f|^2.
\end{equation}
In addition, there is a singular space $(X,d)$ of dimension $n$, a probability measure $\mu$ on $X$ and an identification $\XX_{<0} = X \times (-T_\infty, 0)$ such that the following is true for all $t \in (-T_\infty,0)$ 
\begin{enumerate}[label=(\alph*)]
\item \label{Thm_metric_soliton_limit_main_a}  $(\XX_t, d_t) = (X \times \{ t \}, |t|^{1/2} d)$.
\item \label{Thm_metric_soliton_limit_main_b} $\RR = \RR_X \times (-T_\infty,0)$ and $\partial_{\tf} - \nabla f$ corresponds to the standard vector field on the second factor.
\item \label{Thm_metric_soliton_limit_main_c}  $(\RR_t, g_t) = (\RR_X \times \{ t \}, |t| g_X)$.
\item \label{Thm_metric_soliton_limit_main_d} $\nu_{x_\infty; t} = \mu$.
\end{enumerate}
Moreover, there is a unique family of probability measures $(\nu'_{x;t})_{x \in X, t \geq 0}$ such that the tuple
\[ (X,d, \mu, (\nu'_{x;t})_{x \in X, t \geq 0} ) \]
is a model for $(\XX_{<0}, (\nu_{x_\infty;t})_{t \in (-T_\infty,0)})$ corresponding to the same identification.
Therefore, $(X,d, \mu)$ characterizes $(\XX_{<0}, (\nu_{x_\infty;t})_{t \in (-T_\infty,0)})$ up to flow-isometry.

The metric measure space $(X,d, \mu)$ has the following properties.
We have $\mu(\SS_X) = 0$ and $d\mu = (4\pi)^{-n/2} e^{-f'} dg$ on $\RR_X$ for $f' := f_{-1} \in C^\infty(\RR_X)$ such that the first two soliton identities in (\ref{eq_sol_id_met_sol_limit_thm}) hold on $\RR_X$ for $f = f'$ and $\tau = 1$.
Moreover, the metric space $(X,d)$ satisfies Assertions~\ref{Thm_static_limit_main_d}--\ref{Thm_static_limit_main_f} of Theorem~\ref{Thm_static_limit_main} (the tangent cones are still metric cones corresponding to $\IF$-limits that are Ricci flat on their regular part).
Lastly, $(X,d)$ is either a metric cone or $R > 0$ on $\RR_X$.
\end{Theorem}

We remark that in dimension~4, the singular space $(X,d)$ in Theorems~\ref{Thm_static_limit_main}, \ref{Thm_metric_soliton_limit_main} is given by a smooth orbifold with isolated singularities; see Subsection~\ref{subsec_dim4_intro} for a precise statement.

An immediate application of Theorem~\ref{Thm_metric_soliton_limit_main} is the following Almost Cone Splitting Theorem, which is similar to \cite{Cheeger-Colding-Cone} in the case of spaces with lower Ricci curvature bounds.
We remark that while the analog of this theorem was a key ingredient in the proof of partial regularity for spaces with lower Ricci curvature bounds \cite{Cheeger-Naber-quantitative}, it is \emph{a priori} not available in the Ricci flow setting.
However, it can be deduced from the partial regularity theory of Ricci flows \emph{a posteriori.}

\begin{Theorem}[Almost Cone Rigidity] \label{Thm_almost_cone_rigidity_intro}
For any $n \in \IN$, $Y < \infty$, $\eps, T > 0$ there is a $\delta (n,Y, \eps, T) > 0$ such that the following holds.
Let $(M, (g_t)_{t \in [-T, 0]})$ be a Ricci flow on a compact, $n$-dimensional manifold and $x \in M$ a point.
Suppose that for some $\tau \in (0, (1-\eps) T]$ we have
\[ \NN_{x,0} (T) + \delta \geq \NN_{x,0} (\tau) \geq - Y. \]
Then there is a metric soliton $(\XX, (\mu_t)_{t \in (-T,0)})$ over $(-T, 0)$ such that
\[ d_{\IF} \big( (M, (g_t)_{t \in (- T, 0)}, (\nu_{x, 0;t})_{t \in (-T,0)} ), (\XX, (\mu_t)_{t \in (-T,0)}) \big) \leq \eps. \]
\end{Theorem}
\bigskip

\subsection{Local regularity and quantitative stratification} \label{subsec_quant_strat_intro}
In this subsection we state a local regularity theorem and a quantitative stratification result, which can be seen as  the Ricci flow analogs of \cite{Cheeger-Naber-quantitative,Cheeger-Naber-Codim4}.
The latter can be seen as an effective version of Theorem~\ref{Thm_stratification_limit_main} for a smooth Ricci flow.
More specifically, we will replace the singular strata $\SS^k$ by \emph{quantitative} strata $\SS^{\eps,k}_{r_1,r_2}$,  along which the flow lacks certain regularity or symmetry properties.
Instead of bounding their dimension, we will bound the volume of a small parabolic tubular neighborhood of $\SS^{\eps,k}_{r_1,r_2}$.

In order to define these quantitative strata, we first have to introduce the notion of $(k, \eps, r)$-symmetric points.
Roughly speaking, we will call a point $x_0 \in \XX$ in a metric flow $(k, \eps, r)$-symmetric if, after parabolic rescaling by $r^{-1}$, the flow near $x_0$ is $\eps$-close in the $\IF$-sense to a metric soliton that splits off an $\IR^k$-factor, or to a static cone that splits off an $\IR^{k-2}$-factor.
This is analogous to Theorem~\ref{Thm_stratification_limit_main}\ref{Thm_stratification_limit_main_b}.
The following is the precise definition:

\begin{Definition}[$(k,\eps,r)$-symmetric points] \label{Def_symm_points}
Let $\XX$ be a metric flow over some interval $I$, $x_0 \in \XX_{t_0}$ a point and $r, \eps > 0$, $k \geq 0$.
We say that $x_0$ is {\bf $(k, \eps, r)$-symmetric} if $[t_0 - \eps^{-1} r^2, t_0] \subset I$ and if there is a metric flow pair $(\XX', (\nu_{x'; t})_{t \leq 0})$ over $(-\infty,0]$ that arises as a non-collapsed $\IF$-limit of smooth Ricci flows as described in Subsection~\ref{subsec_part_reg_limit_intro} and satisfies the Property~\ref{Thm_stratification_limit_main_b1} or \ref{Thm_stratification_limit_main_b2} of Theorem~\ref{Thm_stratification_limit_main} and the corresponding property of Addendum~\ref{Add_limit_stratification_addendum_main} such that the following is true.
Consider the metric flow pair 
\[ \big(\XX_{(t_0 - \eps^{-1} r^2, t_0]}, (\nu_{x_0;t})_{t \in (t_0 - \eps^{-1} r^2, t_0]} \big), \] 
which is defined over $(t_0 - \eps^{-1} r^2, t_0]$, apply a time-shift by $-t_0$ and parabolic rescaling by $r^{-1}$, to obtain a flow pair over $(-\eps^{-1}, 0]$.
Then this flow pair has $d_{\IF}$-distance $< \eps$ to the restricted metric flow pair $(\XX'_{(-\eps^{-1},0]}, (\nu_{x'; t})_{t \in (-\eps^{-1},0]})$.
\end{Definition}

We can now define the quantitative strata:

\begin{Definition}[Quantitative strata] \label{Def_SS_quantitative}
Let $\XX$ be a metric flow over some interval $I$.
For $\eps > 0$ and $0 \leq r_1 < r_2 \leq \infty$ the {\bf quantitative strata},
\[ \SS^{\eps, 0}_{r_1, r_2} \subset \SS^{\eps, 1}_{r_1, r_2} \subset \SS^{\eps,2}_{r_1, r_2} \subset \ldots \subset \SS^{\eps,n-2}_{r_1, r_2} \subset  \XX \]
are defined as follows:
$x \in \SS^{\eps, k}_{r_1, r_2}$ if and only if $[t'- \eps^{-1} r_2^{2}, t'] \subset I$ and $x$ is \emph{not} $(k+1, \eps, r')$-symmetric for any $r' \in (r_1, r_2)$.
\end{Definition}

Note if $(M, (g_t)_{t \in I})$ is a Ricci flow on a compact manifold, then there is a metric flow associated to $(M, (g_t)_{t \in I})$.
So Definition~\ref{Def_SS_quantitative} also applies to Ricci flows and in this case $\SS^{\eps, k}_{r_1,r_2} \subset M \times I$.

Let us first state the local regularity theorem, which is similar to the corresponding theorem for Einstein metrics.
The following theorem, which applies to Ricci flows, asserts local curvature bounds near points that are $(n-1, \eps, r)$-symmetric for sufficiently small $\eps$.
Recall that this means that the flow locally either almost splits off an $\IR^{n-3}$-factor or is almost static and splits off an $\IR^{n-1}$-factor.

\begin{Theorem}[Local regularity] \label{Thm_loc_reg_RF_intro}
If $n \in \IN$, $Y, A < \infty$ and $\eps \leq \ov\eps (n,Y, A) > 0$, then the following holds.
Consider a Ricci flow $(M, (g_t)_{t \in I})$ on a compact $n$-dimensional manifold and let $(x_0, t_0) \in M \times I$ be a point and $r > 0$ a scale with $[t_0 - \eps^{-1} r^2, t_0] \subset I$.
Assume that $\NN_{x_0,t_0} ( r^2) \geq - Y$ and that $(x_0,t_0)$ is $(n-1, \eps, r)$-symmetric.
Then
\begin{equation} \label{eq_loc_reg_rrm_sig_r}
\rrm (x_0,t_0) \geq A  r.
\end{equation}
\end{Theorem}

Recall from \cite{Bamler_HK_entropy_estimates} that $\rrm (x_0,t_0)$ is defined to be the supremum over all $r' > 0$ for which $|{\Rm}| \leq r^{\prime -2}$ on the two-sided parabolic ball $P(x_0, t_0; r')$ around $(x_0,t_0)$ of radius $r'$.
We remark that for most applications it is sufficient to choose $A = 1$.

Next, we state the quantitative stratification result for Ricci flows.
For this purpose we briefly recall the definition of $P^*$-parabolic neighborhoods from \cite{Bamler_HK_entropy_estimates}.
It has been common to define a parabolic ball in a Ricci flow as $P(x_0,t_0;r) := B(x_0,t_0,r) \times [t_0 - r^2, t_0+ r^2]$, where $B(x_0,t_0,r)$ denotes the distance ball with respect to $g_{t_0}$.
However, as explained in \cite{Bamler_HK_entropy_estimates}, this definition is somewhat unnatural, because it does not take into account certain heat kernel drifts and distance expansion effects.
Instead, we define the $P^*$-parabolic ball $P^* (x_0, t_0; r)$ as the set of spacetime points $(x,t) \in M \times [t_0 - r^2, t_0+ r^2]$ with the property that the conjugate heat kernels based at $(x,t)$ and $(x_0, t_0)$ have $W_1$-Wasserstein distance $< r$ at time $t_0 - r^2$.

\begin{Theorem}[Quantitative Stratification] \label{Thm_Quantitative_Strat}
If
\begin{equation*} 
n \in \IN, \qquad
Y < \infty, \qquad
 \eps > 0, \qquad
C \geq \underline{C}(n, Y, \eps)
\end{equation*}
then the following holds.
Let $(M, (g_t)_{t \in I})$ be a Ricci flow on a compact manifold, $(x_0, t_0) \in M \times I$ a point and $r > 0$ a scale with $[t_0 - 2 r^2, t_0] \subset I$.
Assume that $\NN_{x_0,t_0} (r^2) \geq - Y$ and consider the quantitative strata $\SS^{\eps,k}_{r_1 ,r_2} \subset M \times I$ from Definition~\ref{Def_SS_quantitative}.
Then for any $\sigma \in (0, \eps)$ there are points $(x_1, t_1), \ldots, (x_N, t_N) \in \SS^{\eps,k}_{\sigma r, \eps r} \cap P^* (x_0, t_0; r)$ with  $N \leq C \sigma^{-k-\eps}$ and
\begin{equation} \label{eq_quant_strat_covering_intro}
 \SS^{\eps, k}_{\sigma r, \eps r} \cap P^* (x_0, t_0; r) \subset \bigcup_{i=1}^N P^* (x_i, t_i; \sigma r). 
\end{equation}
Moreover, if $\eps \leq \ov\eps (Y)$, then
\[ \rrm \geq  \sigma r \qquad \text{on} \quad P^*(x_0, t_0; r) \setminus \SS^{\eps, n-2}_{\sigma r, \eps r}. \]
\end{Theorem}

\begin{Remark}
For a measurable subset $Y \subset M \times I$ we may define the parabolic volume as
\[ |Y| :=\int_I \int_{Y \cap M \times \{ I \}} dg_t \, dt. \]
It follows from \cite{Bamler_HK_entropy_estimates} that $|P^*(x_i, t_i; \sigma r)| \leq C (\sigma r)^{n+2}$.
So the parabolic volume of the right-hand side of (\ref{eq_quant_strat_covering_intro}) is bounded by $C(Y, \eps) \sigma^{n+2-k-\eps} r^{n+2}$.
By standard containment relationships of $P^*$-parabolic balls \cite{Bamler_HK_entropy_estimates} we also obtain a bound of the same form on a parabolic tubular neighborhood of $S := \SS^{\eps, k}_{\sigma r, \eps r} \cap P^* (x_0, t_0; r)$, which is defined as 
\[ B^* (S, \sigma r) := \bigcup_{(x,t) \in S} P^* (x,t;\sigma r). \]
Note that every time-slice of $B^* (S, \sigma r)$ contains the $\sigma r$-tubular neighborhood of the corresponding time-slice of $S$.
\end{Remark}

Theorem~\ref{Thm_Quantitative_Strat} implies the following integral estimate.

\begin{Theorem} \label{Thm_int_curv_bounds_RF_intro}
Let $(M, (g_t)_{t \in I})$ be a Ricci flow on a compact, $n$-dimensional manifold, $(x_0, t_0) \in M \times I$ a point and $r > 0$ a scale with $[t_0 - 2 r^2, t_0] \subset I$.
Assume that $\NN_{x_0,t_0} (r^2) \geq - Y > -\infty$.
Then for any $\eps > 0$
\begin{multline*}
 \int_{[t_0 - r^2, t_0 + r^2] \cap I} \int_{P^* (x_0, t_0; r) \cap M \times \{ t \}} |{\Rm}|^{2-\eps} dg_t dt \\
\leq  \int_{[t_0 - r^2, t_0 + r^2] \cap I} \int_{P^* (x_0, t_0; r) \cap M \times \{ t \}} \rrm^{-4+2\eps} dg_t dt
\leq C(Y,  \eps) r^{n-2+2\eps}. 
\end{multline*}
Moreover, if $r_{\eps, k}(x,t)$ denotes the supremum over all $r' > 0$ for which $(x,t)$ is $(k, \eps, r')$-symmetric, then we have
\begin{equation*}
  \int_{[t_0 - r^2, t_0 + r^2] \cap I} \int_{P^* (x_0, t_0; r) \cap M \times \{ t \}} r_{\eps, k}^{-k - n- 3 + \eps} dg_t dt 
\leq C(Y,  \eps) r^{n-2+2\eps}. 
\end{equation*}
\end{Theorem}

Lastly, we discuss in what sense the corresponding results hold on non-collapsed $\IF$-limits of Ricci flows.
In the following denote by $\XX$ a metric flow over $I := (-T_\infty, 0]$ that is obtained from a non-collapsing sequence of $n$-dimensional Ricci flows via an $\IF$-limit as described in Subsection~\ref{subsec_part_reg_limit_intro} or \ref{subsec_more_general_F_limit}.

The following theorem, which is the analog of Theorem~\ref{Thm_loc_reg_RF_intro}, asserts local curvature bounds near $(n-1, \eps,r)$-symmetric points on $\XX$.
However, the (conventional) parabolic neighborhood, on which this bound holds, may not be unscathed at its forward end.
If this happens, then it is due to the fact that the flow ``ceases to exist at a certain time'' over the entire ball.

\begin{Theorem}[Local regularity of $\IF$-limits] \label{Thm_loc_reg_XX_intro}
If $Y, A < \infty$ and $\eps \leq \ov\eps (n,Y, A) > 0$, then the following holds.
Let $x_0 \in\XX_{t_0}$, $t_0 < 0$, be a point and $r > 0$ a scale with $[t_0 - \eps^{-1} r^2, t_0] \subset I$.
Assume that $\NN_{x_0} ( r^2) \geq - Y$ and that $(x_0)$ is $(n-1, \eps, r)$-symmetric.
Then for $r' := A r$ the following holds.
There is some $T^* \in (0, r^{\prime 2}]$ such that:
\begin{enumerate}[label=(\alph*)]
\item \label{Thm_loc_reg_XX_intro_a} For all $T' \in (0, T^*)$ the (conventional) parabolic neighborhood $P(x_0;r', - r^{\prime 2}, T') \subset \RR$ is unscathed and we have the curvature bound $|{\Rm}| \leq r^{\prime -2}$ on this parabolic neighborhood.
\item \label{Thm_loc_reg_XX_intro_b} If $T^* < r^{\prime 2}$ and $t_0 + T^* < 0$, then no point in $B(x_0,r )$ survives past time $t_0 + T^*$ and for any conjugate heat flow $(\mu'_t = v'_t \, dg_t)_{t \in I, t< t^*}$, $t^*> t_0 + T^*$, and $r'' \in (0, r')$ we have
\[ \lim_{t \nearrow t_0 + T^*} \sup_{(B(x_0, r''))(t)} v_t = 0. \]
\end{enumerate}
\end{Theorem}

Denote by $\rrm' (x_0, t_0)$ the supremum over all $r' > 0$ such that Assertions~\ref{Thm_loc_reg_XX_intro_a}, \ref{Thm_loc_reg_XX_intro_b} of Theorem~\ref{Thm_loc_reg_XX_intro} hold.
This turns out to be a useful generalization of $\rrm$ to singular Ricci flows.
Using this notion, we can generalize Theorem~\ref{Thm_Quantitative_Strat} to limits $\XX$.

\begin{Theorem}[Quantitative Stratification of $\IF$-limits] \label{Thm_Quantitative_Strat_XX}
If
\begin{equation*} 
n \in \IN, \qquad
Y < \infty, \qquad
 \eps > 0, \qquad
C \geq \underline{C}(n, Y, \eps)
\end{equation*}
then the following holds.
Let $x_0 \in\XX_{t_0}$, be a point and $r > 0$ a scale with $[t_0 - \eps^{-1} r^2, t_0] \subset I$.
Assume that $\NN_{x_0} (r^2) \geq - Y$ and consider the effective strata $\SS^{\eps,k}_{r_1 ,r_2} \subset \XX$ from Definition~\ref{Def_SS_quantitative}.
Then for any $\sigma \in (0, \eps)$ there are points $x_1, \ldots, x_N \in \SS^{\eps,k}_{\sigma r, \eps r} \cap P^* (x_0; r) \cap \XX_{<0}$ with  $N \leq C \sigma^{-k-\eps}$ and
\begin{equation*} 
 \SS^{\eps, k}_{\sigma r, \eps r} \cap P^* (x_0; r) \cap \XX_{<0} \subset \bigcup_{i=1}^N P^* (x_i; \sigma r). 
\end{equation*}
Moreover, if $\eps \leq \ov\eps (Y)$, then
\[ \rrm' \geq  \sigma r \qquad \text{on} \quad (P^*(x_0; r) \cap \XX_{<0}) \setminus \SS^{\eps, n-2}_{\sigma r, \eps r}. \]
\end{Theorem}

Similar to Theorem~\ref{Thm_int_curv_bounds_RF_intro}, we also obtain:

\begin{Theorem}[Integral curvature bounds on $\IF$-limits] \label{Thm_integral_bounds_F_limit}
Let $x_0 \in\XX_{t_0}$, be a point and $r > 0$ a scale with $[t_0 - \eps^{-1} r^2, t_0] \subset I$.
Assume that $\NN_{x_0} (r^2) \geq - Y > - \infty$.
Then for any $\eps > 0$
\begin{multline*}
 \int_{[t_0 - r^2, t_0 + r^2] \cap I} \int_{P^* (x_0; r) \cap \RR_t} |{\Rm}|^{2-\eps} dg_t dt \\
\leq  \int_{[t_0 - r^2, t_0 + r^2] \cap I} \int_{P^* (x_0; r) \cap \RR_t} \rrm^{\prime -4+2\eps} dg_t dt
\leq C(Y,  \eps) r^{n-2+2\eps}. 
\end{multline*}
Moreover, if $r_{\eps, k}(x,t)$ denotes the supremum over all $r' > 0$ for which $(x,t)$ is $(k, \eps, r')$-symmetric, then we have
\begin{equation*}
  \int_{[t_0 - r^2, t_0 + r^2] \cap I} \int_{P^* (x_0; r) \cap \RR_t} r_{\eps, k}^{-k - n- 3 + \eps} dg_t dt 
\leq C(Y,  \eps) r^{n-2+2\eps}. 
\end{equation*}
\end{Theorem}
\bigskip

\subsection{The picture at the first singular time} \label{subsec_first_sing_intro}
In this subsection we explain how our compactness theory can be used to characterize the singularity formation at the first singular time of a Ricci flow.
Let in the following $(M, (g_t)_{t\in [0,T)})$ be a Ricci flow on a compact, $n$-dimensional manifold that (possibly) develops a singularity at time $T < \infty$.

We will first construct a limiting space at time $T$.

\begin{Definition} \label{Def_MT}
Let $M_T$ be the set of all conjugate heat flows $(\mu_t)_{t \in [0,T)}$ on $(M, (g_t)_{t\in [0,T)})$ with the property that
\begin{equation} \label{eq_lim_Var_mu_0_at_T}
 \lim_{t \nearrow T} \Var (\mu_t) = 0. 
\end{equation}
For any two conjugate heat flows $(\mu_{i,t})_{t \in [0,T)}$, $i = 1,2$, we define
\begin{equation} \label{eq_d_M_T_def}
 d_{M_T} \big( (\mu_{1,t})_{t \in [0,T)}, (\mu_{2,t})_{t \in [0,T)} \big) := \lim_{t \nearrow T} d_{W_1}^{g_t} \big( \mu_{1,t}, \mu_{2,t} \big) \in [0,\infty]. 
\end{equation}
\end{Definition}

Note that (\ref{eq_lim_Var_mu_0_at_T}) implies that $\Var (\mu_t) \leq H_n (T - t)$ and by the monotonicity of the $W_1$-Wasserstein distance \cite{Bamler_HK_entropy_estimates} the limit in (\ref{eq_d_M_T_def}) exists.
Moreover, for any $(\mu_t)_{t \in [0,T)}$ we can find a sequence of points $(y_i,t_i) \in M \times [0,T)$ with $t_i \nearrow T$ such that $\nu_{y_i,t_i; t} \to \mu_t$ in the $W_1$-sense on compact time-intervals.
So $M_T$ can be regarded as the ``asymptotic boundary'' of $M \times [0,T)$.
The following lemma shows that $(M_T, d_{M_T})$ can moreover be viewed as the singular time-$T$-slice of the flow $(M, (g_t)_{t\in [0,T)})$.

\begin{Lemma} \label{Lem_MT}
$(M_T, d_{M_T})$ is a complete metric space if we allow the distance function to attain the value $\infty$.
If there is some open subset $U \subset M$ such that the restricted family of metrics $(g_t|_U)_{t\in [0,T)}$ can be extended smoothly to a family of the form $(g_t|_U)_{t\in [0,T]}$, then the map
\begin{equation} \label{eq_map_U_MT}
 U \longrightarrow M_T, \qquad x \longmapsto (\lim_{t' \nearrow T} \nu_{x,t'; t})_{t \in [0, T)} 
\end{equation}
is a local isometry of the form $(U, d_{g_T}) \to (M_T ,d_{M_T})$.
Here the limit in (\ref{eq_map_U_MT}) is taken in the $W_1$-distance.
Moreover, if $U = M$, then this map is an isometry.
\end{Lemma}

From now on we will denote elements in $M_T$ as points $x = (\mu_t)_{t \in [0,T)}$ and we will write, by abuse of notation, $d\nu_{x,T;t} := d\mu_t = K(x,T; \cdot, t) dg_t$.
If $x \in M_T$ is some fixed point, then we will also write $K(x,T; \cdot, t) = (4\pi \tau)^{-n/2} e^{-f}$ for $\tau = T-t$ and define the pointed Nash entropy $\NN_{x,T} (\tau_0)$, $\tau_0 \in (0, T]$ at $(x, T)$ in the usual way.
The following is the main result of this subsection:

\begin{Theorem} \label{Thm_first_sing_time_intro}
For any $x \in M_T$ the pointed Nash entropy $\NN_{x,T} (\tau)$ is non-increasing in $\tau$ and
\[ \NN_{x,T} (0) := \lim_{\tau \searrow 0} \NN_{x,T} (\tau) \in (-\infty, 0] \]
exists.
Moreover, for any sequence $\tau_i \to 0$ there is a subsequence such that we have uniform $\IF$-convergence within some correspondence $\CF$ on compact time-intervals
\begin{equation} \label{eq_IF_conv_first_sing_time}
 \big( M, (\tau_i^{-1} g_{\tau_i t+ T} )_{t \in [-\tau_i^{-1}T, 0)}, (\nu_{x,T;\tau_i t+ T})_{t \in [-\tau_i^{-1}T, 0)} \big) \xrightarrow[i \to \infty]{\quad \IF, \CF \quad} \big( \XX, (\mu^\infty_t)_{t < 0} \big) 
\end{equation}
to a metric soliton with Nash-entropy 
\[ \NN_{(\mu^\infty_t)} (\tau) = \NN_{x,T}(0). \]
Lastly, there is a dimensional constant $\eps_n > 0$ such that if $\NN_{x,T} (0) \geq - \eps_n$, then $\NN_{x,T}(0) =0$ and there is a neighborhood of $x$ in $M_T$ on which $(M_T, d_{M_T})$ is regular.
More specifically, there is an open subset $U \subset M$ on which the restricted family of metrics $(g_t|_U)_{t\in [0,T)}$ can be extended smoothly to a family of the form $(g_t|_U)_{t\in [0,T]}$ and the image of the map (\ref{eq_map_U_MT}) contains $x$.
\end{Theorem}

We note that since $(\XX, (\mu^\infty_t)_{t < 0} )$ is an $\IF$-limit of smooth Ricci flows, Theorem~\ref{Thm_metric_soliton_limit_main} can be applied, which yields further characterizations of the limit.

\subsection{The tangent flow at infinity} \label{subsec_asympt_soliton}
In this subsection, we analyze $\IF$-limits of Ricci flows that are non-collapsed at all times and exist for all negative times.
We show that such limits have a blow-down limit as $t \to -\infty$, which is a gradient shrinking soliton and which we will call the \emph{tangent flow at infinity.}
In the case of $\kappa$-solutions this tangent flow at infinity agrees with the asymptotic soliton constructed by Perelman \cite{Perelman1,Chan_Ma_Zhang_2021}.

In the following consider a sequence of pointed Ricci flows $(M_i, (g_{i,t})_{t \in (-T_i, 0]}, (x_i,0))$ on compact, $n$-dimensional manifolds such that $\lim_{i \to \infty} T_i = \infty$.
We assume that there is a uniform $Y_0 < \infty$ such that the following non-collapsedness condition holds for any $\tau_0 > 0$
\begin{equation} \label{eq_non_collapse_large_scales}
 \NN_{x_i, 0} (\tau_0) \geq - Y_0 \qquad \text{for large} \quad i. 
\end{equation}
Such a sequence arises, for example, when we take the blow-ups near singularities.

As discussed in Subsection~\ref{subsec_part_reg_limit_intro}, we may pass to a subsequence such that we have $\IF$-convergence on compact time-intervals within some correspondence $\CF$:
\[  (M_i, (g_{i,t})_{t \in (-T_i, 0]}, (\nu_{x_i,0;t})_{t \in (-T_i, 0]}) \xrightarrow[i \to \infty]{\quad \IF, \CF \quad} (\mathcal{X}, (\nu_{x_\infty;t})_{t \in (-\infty,0]}). \]
Here $\XX$ is an $H_n$-concentrated, future continuous metric flow of full support over $(-\infty,0]$.
As explained in Subsection~\ref{subsec_Nash_in_limit_intro}, the bound (\ref{eq_non_collapse_large_scales}) passes to the limit and due to \cite{Bamler_HK_entropy_estimates} we obtain that
\[ \NN_{x} (\tau) \geq - Y_0 \qquad \text{for all} \quad x \in \XX, \tau > 0 \]
and the quantity
\[ \NN_{\XX} (\infty) := \lim_{\tau \to \infty} \NN_{x_\infty} (\tau) > -\infty \]
is independent of $x$.
We note that the case in which $\XX$ is represented by a smooth, ancient Ricci flow that is uniformly non-collapsed at all scales is also interesting.

We recall from \cite{Bamler_RF_compactness} that a tangent flow at infinity of $\XX$ at $x_0 \in \XX_{t_0}$ is defined as the $\IF$-limit of parabolic rescalings of $(\XX_{\leq t_0}, (\nu_{x_0;t})_{t \leq t_0})$ by factors $\lambda_j \to 0$.
The following is the main result of this subsection.

\begin{Theorem} \label{Thm_tangent_flow_at_infty_intro}
Every tangent flow at infinity $(\XX', (\nu_{x'_\infty;t})_{t \leq 0})$ of $\XX$ is a metric soliton with $\NN_{x'_\infty} (\tau) = \NN_{\XX} (\infty)$.
Moreover, there is a dimensional constant $\eps_n > 0$ such that if $\NN_{\XX} (\infty) \geq - \eps_n$, then $\NN_{\XX}(\infty) = 0$ and $\XX$ is isometric to the constant flow on Euclidean space.
\end{Theorem}

Note that since $(\XX^\infty, (\nu_{x'_\infty;t})_{t \leq 0})$ is also an $\IF$-limit of suitable parabolic rescalings of the flows $(M_i, (g_{i,t})_{t \in (-T_i, 0]})$ (see \cite{Bamler_RF_compactness}) the tangent flow at infinity has the same regularity properties as any non-collapsed $\IF$-limit of Ricci flows.
So Theorem~\ref{Thm_metric_soliton_limit_main} can be applied to the limit, which implies further characterizations.

\subsection{Thick-thin decomposition for large times} \label{subsec_thick_thin_intro}
Consider an immortal Ricci flow $(M, (g_t)_{t \geq 0})$ on a compact, $n$-dimensional manifold.
The main result of this subsection states that any non-collapsed blow-down $\IF$-limit of this Ricci flow is a flow that evolves through Einstein metrics away from a possible singular set of codimension at least 4.
This result can be viewed as a generalization of the thick-thin decomposition in the bounded curvature case due to Hamilton \cite{Hamilton_1999} or the 3-dimensional case due to Perelman \cite{Perelman2}.

\begin{Theorem} \label{Thm_longtime_limit_intro}
Consider a sequence $(x_i, t_i) \in M \times [0, \infty)$ such that $t_i \to \infty$ and $\NN_{x_i,t_i} (\frac12 t_i) \geq -Y > -\infty$.
Then, after passing to a subsequence, we have the following $\IF$-convergence within some correspondence $\CF$:
\[  (M, (g_{t_i \, t})_{t \in (0,1]}, (\nu_{x_i,t_i; t_i \, t})_{t \in (0,1]}) \xrightarrow[i \to \infty]{\quad \IF, \CF \quad} (\mathcal{X}, (\nu_{x_\infty;t})_{t \in (0,1]}), \]
where $\XX$ is an $H_n$-concentrated, future continuous metric flow of full support over $(0,1]$.
The restriction $\XX_{<1}$, which is a metric flow over $(0,1)$, has the following properties.
There is a singular space $(X,d)$ of dimension $n$ and an identification $\XX_{<1} = X \times (0,1)$ such that the following is true for all $t \in (0,1)$:
\begin{enumerate}[label=(\alph*)]
\item $\Ric = - \frac12 g_X$ on $\RR_X$.
\item $(\XX_t, d_t) = (X \times \{ t \}, t^{1/2} d)$.
\item $\RR = \RR_X \times (0,1)$ and $\partial_{\tf}$ corresponds to the standard vector field on the second factor.
\item $(\RR_g, g_t) = (\RR_X \times \{ t \}, t g_X)$.
\item For any $\lambda \in (0,1)$ and $0 < s \leq t < 1$, $x \in X$ we have $\nu_{(x,t);s} = \nu_{(x, \lambda^2 t); \lambda^2 s}$.
\end{enumerate}
Moreover, the metric space $(X,d)$ satisfies Assertions~\ref{Thm_static_limit_main_d}--\ref{Thm_static_limit_main_f} of Theorem~\ref{Thm_static_limit_main} (the tangent cones are still metric cones corresponding to $\IF$-limits that are Ricci flat on their regular part).
\end{Theorem}

As a corollary we obtain:
\begin{Corollary}
For any $Y < \infty$ and $t > 0$ we can find a decomposition
\[ M = M_{\thick}(t) \,\, \dot\cup \,\, M_{\thin}(t) \]
such that $\NN_{x,t}(\frac12 t) < - Y$ for all $x \in M_{\thin}(t)$ and such that for any $\eps > 0$ there is a $T_\eps < \infty$, which may depend on the given flow, such that for any $t \geq T_\eps$ and $x \in M_{\thick}(t)$ the parabolically rescaled pair
\[ (M, (g_{i,t \, t'})_{t' \in (0,1]}, (\nu_{x,t; t \, t'})_{t' \in (0,1]}) \]
has $\IF$-distance $< \eps$ to a pair $(\mathcal{X}, (\nu_{x_\infty;t})_{t \in (0,1]})$ that satisfies the assertions of Theorem~\ref{Thm_longtime_limit_intro}.
\end{Corollary}

\bigskip

\subsection{Limits in dimensions 2 and 3} \label{subsec_dimensions23_intro}
In this subsection, we briefly discuss the implications of our work in dimensions $n = 2,3$ and show how these can be used to recover some of Perelman's results from \cite{Perelman1}.

Assume in the following that $\XX$ is a metric flow over $(-\infty,0)$ that is obtained from a non-collapsing sequence of Ricci flows $(M_i, (g_{i,t})_{t \in (-T_i, 0)})$ on compact, $n \leq 3$-dimensional manifolds, as described in Subsection~\ref{subsec_more_general_F_limit}.
Assume that $\XX$ is not isometric to the constant flow on Euclidean space.

Due to the Hamilton-Ivey pinching \cite{Hamilton_non_singular} (or the standard lower bound on the scalar curvature in dimension 2), the metric $g$ on the regular part $\RR$ of $\XX$ must have non-negative sectional curvature.
Theorem~\ref{Thm_XX_reg_sing_dec_main} implies that $\SS_t = \emptyset$ for almost every $t < 0$.
By the same reason, Theorems~\ref{Thm_tangent_cone_metric_cone_main}, \ref{Thm_tangent_flow_at_infty_intro} imply that all tangent flows at singular points  of $\XX$ are given by non-trivial, smooth gradient shrinking solitons with non-negative sectional curvature.
These are isometric to quotients of round shrinking spheres, the round shrinking cylinder or its $\IZ_2$-quotient \cite{Hamilton_3_manifolds,Munteanu_Wang_2017}.
Since we have assumed that the Ricci flows $(M_i, (g_{i,t})_{t \in (-T_i, 0)})$ exist until time $0$, we can conclude, using \cite{Hamilton_3_manifolds}, that the tangent flows at singular points of $\XX$ are non-compact, so they are isometric to the round shrinking cylinder or its $\IZ_2$-quotient.
Using this observation and a basic volume comparison argument, it is possible to conclude (see Section~\ref{subsec_dimensions23_proof} for a proof):

\begin{Theorem} \label{Thm_SS_empty}
$\SS_{<0} = \emptyset$ and the trajectories of $\partial_{\tf}$ are defined on maximal time-intervals of the form $(-T, 0)$ for $T \in (0, \infty]$.
\end{Theorem}

It is not clear to the author whether, under the current assumptions, the trajectories of $\partial_{\tf}$ are even defined on all of $(-\infty,0)$, which would imply that the flow on $\RR$ is given by a conventional Ricci flow;
see also \cite{Topping_cusp, Lai_2020} for related examples.
However, this is the case under a global non-collapsing condition, for example as discussed in Subsection~\ref{subsec_asympt_soliton}.
More specifically we have (see again Section~\ref{subsec_dimensions23_proof} for a proof):

\begin{Theorem} \label{Thm_lim_dim23_global_NC}
If there are uniform $Y_0 < \infty$, $\tau_0 > 0$ such that $\NN_{x} (\tau_0) \geq - Y_0$ for all $x \in \XX$, then the flow on $\RR$ is given by a smooth Ricci flow $(M_\infty, (g_{\infty,t})_{t< 0})$ with bounded curvature on compact time-intervals.
If $\NN_{x} (\tau) \geq - Y_0$ for all $x \in \XX$ and $\tau > 0$, then $(M_\infty, (g_{\infty,t})_{t< 0})$ is a $\kappa$-solution.
\end{Theorem}

Next assume that in addition $\XX$ is a limit of Ricci flows $(M_i, (g_{i,t})_{t \in (-T_i, 0]}, x_i)$ that are defined up to time $0$, as in Subsection~\ref{subsec_part_reg_limit_intro}.
Then the following is true:

\begin{Theorem} \label{Thm_limit_3d_to_time_0_intro}
If there is a uniform $Y_0 < \infty$ such that $\NN_{x} (\tau) \geq - Y_0$ for all $x \in \XX_{<0}$ and $\tau > 0$ and if $\limsup_{i \to \infty} R(x_i, 0) < \infty$, then the flow on $\RR = \XX_{<0}$ is given by a smooth $\kappa$-solution $(M_\infty, (g_{\infty,t})_{t< 0})$ that can be smoothly extended to time $0$.
\end{Theorem}

Theorem~\ref{Thm_limit_3d_to_time_0_intro} implies Perelman's Canonical Neighborhood Theorem \cite[12.1]{Perelman1} for smooth Ricci flows.

\subsection{Limits in dimension 4} \label{subsec_dim4_intro}
The following result provides an improved structural characterization of tangent flows (at infinity) and long-time blow downs in dimension 4.

\begin{Theorem} \label{Thm_dim4_intro}
In dimension $n =4$ the singular space $(X,d)$ in Theorems~\ref{Thm_static_limit_main}, \ref{Thm_metric_soliton_limit_main}, \ref{Thm_longtime_limit_intro} is the length space of a complete, smooth 4-dimensional Riemannian orbifold $(M_\infty, g_\infty)$ with isolated singularities.

In Theorem~\ref{Thm_metric_soliton_limit_main} we either have $R > 0$ on $M_\infty$ or the orbifold is isometric to $\IR^4 / \Gamma$ for some finite $\Gamma \subset O(4)$.
\end{Theorem}
\bigskip

\subsection{Backwards Pseudolocality Theorem} \label{subsec_bckw_pseudo}
In the following we discuss a purely elementary consequence of our theory.
The following Backwards Pseudolocality Theorem is similar to Perelman's (Forward) Pseudolocality Theorem \cite[Sec 10]{Perelman1}, with the difference that it asserts curvature bounds on a \emph{backwards} parabolic neighborhood if the geometry near a point is sufficiently regular in one time-slice.
Similar Backwards Pseudolocality Theorems were established in dimension 3 by Perelman \cite{Perelman1} and under an additional scalar curvature bound by Zhang and the author \cite{Bamler-Zhang-1}.
The following theorem generalizes these results and it shows that our theory has applications that are not limited to convergence results.

\begin{Theorem} \label{Thm_bckwds_pseudo}
For any $n \in \IN$ and $\alpha > 0$ there is an $\eps (n, \alpha) > 0$ such that the following holds.

Let $(M, (g_t)_{t \in I})$ be a Ricci flow on a compact, $n$-dimensional manifold and let $(x_0, t_0) \in M \times I$ a point and $r > 0$ a scale such that $[t_0 -   r^2, t_0] \subset I$ and
\begin{equation} \label{eq_bckwds_pseudo_conditions}
 |B(x_0,t_0,r)| \geq \alpha r^n, \qquad |{\Rm}| \leq (\alpha r)^{-2} \quad \text{on} \quad B(x_0, t_0,r). 
\end{equation}
Then
\[ |{\Rm}| \leq (\eps r)^{-2} \quad \text{on} \quad P(x_0,t_0; (1-\alpha) r, -(\eps r)^2). \]
\end{Theorem}

We note that combined with Perelman's (Forward) Pseudolocality Theorem, Theorem~\ref{Thm_bckwds_pseudo}, implies the same curvature bounds on a \emph{two-sided} parabolic neighborhood of the form $P(x_0, \lb t_0; \lb (1-\alpha) r, \lb -(\eps r)^2,\lb (\eps r)^2)$.
In addition, we obtain the following corollary which allows us to upgrade the condition in Perelman's (Forward) Pseudolocality Theorem \cite[10.1]{Perelman1} to a condition similar to (\ref{eq_bckwds_pseudo_conditions}) if $[t_0 - \alpha  r^2, t_0 + \alpha r^2] \subset I$.

\begin{Corollary} \label{Cor_combination_fwd_bckwd_pseudo}
For any $n \in \IN$ and $\alpha > 0$ there is an $\eps (n, \alpha) > 0$ such that the following holds.

Let $(M, (g_t)_{t \in I})$ be a Ricci flow on a compact, $n$-dimensional manifold and let $(x_0, t_0) \in M \times I$ a point and $r > 0$ a scale such that $[t_0 - \alpha  r^2, t_0 + \alpha r^2] \subset I$ and
\begin{equation} \label{eq_pseudo_original_condition}
R \geq - (\alpha r)^{-2} \quad \text{on} \quad B(x_0, t_0,r),  \qquad
 |\partial\Omega|^n_{t_0} \geq (1-\eps) c_n|\Omega|^{n-1}_{t_0} \quad \text{for all} \quad \Omega \subset B(x_0, t_0, r),
\end{equation}
where $c_n$ denotes the Euclidean isoperimetric constant.
Then
\[ |{\Rm}| \leq (\eps r)^{-2} \quad \text{on} \quad B(x_0,t_0; (1-\alpha)r). \]
\end{Corollary}

So given (\ref{eq_pseudo_original_condition}), we can apply Theorem~\ref{Thm_bckwds_pseudo} and Perelman's Forward Pseudolocality Theorem to obtain curvature estimates on a two-sided parabolic neighborhood of the form $P(x_0, \lb t_0; \lb (1-\alpha) r, \lb -(\eps r)^2,\lb (\eps r)^2)$ if $\eps \leq \ov\eps (n, \alpha)$.

The following theorem is slightly stronger as it asserts \emph{long-time} backwards estimates.
The corresponding forward facing curvature estimates seem to be unknown.

\begin{Theorem} \label{Thm_bckwds_pseudo_optimal}
For any $n \in \IN$, $\alpha > 0$ and $1 \leq A < \infty$ there is an $\eps (n, \alpha, A) > 0$ such that the following holds.

Let $(M, (g_t)_{t \in I})$ be a Ricci flow on a compact, $n$-dimensional manifold and let $(x_0, t_0) \in M \times I$ a point and $r > 0$ a scale such that:
\begin{enumerate}[label=(\roman*)]
\item $|B(x_0,t_0,r)| \geq \alpha r^n$.
\item $|{\Rm}| \leq (\alpha r)^{-2}$ on $B(x_0, t_0,r)$.
\item $[t_0 -  (A + \alpha) r^2, t_0] \subset I$
\item $R \geq - \eps r^{-2}$ on $M \times [t_0 - A r^2, t_0]$.
\item $R(x_0, t_0) \leq \eps r^{-2}$.
\end{enumerate}
Then the following bounds hold on $P(x_0,t_0; (1-\alpha) r, -A r^2)$:
\[ 
|{\Ric}| \leq \alpha r^{-2} 
\qquad \text{and} \qquad 
|{\Rm}| \leq \alpha r^{-2} + \sup_{B(x_0, t_0, r)} |{\Rm}|. 
 \]
\end{Theorem}
\bigskip

\section{Outline of the proof and structure of the paper} \label{sec_outline}
Let us now give an overview of the methods that are used in the paper.
As mentioned in Subsection~\ref{subsec_key_diff_proof_strategy}, we have to overcome two types of complications.
First, there are a number of elementary geometric or analytic estimates that are or have not been known to hold in the Ricci flow setting.
Second, many of the standard ingredients that are normally used to prove partial regularity are or have not been available in our setting.
It is important to note that, while some of these desired estimates or ingredients are established in this or previous work \cite{Bamler_HK_entropy_estimates,Bamler_RF_compactness}, many others seem to either fail in the Ricci flow setting or may only be deduced \emph{a posteriori} from our theory; so either way they won't be at our disposal.
These remaining complications can roughly be summarized as follows:
%
\begin{itemize}
\item There seem to be no distance distortion bounds in the general setting. In particular, we do not expect bounds on the \emph{expansion} of distances or points or conjugate heat kernel measures under the flow.
\item Lower bounds on the (conjugate) heat kernel on a general Ricci flow seem to fail.
\item It seems impossible to construct cutoff functions on general Ricci flows.
This and the issues mentioned above cause complications in localizing estimates.
In other words, it will be hard to derive or formulate estimates that only hold on a ball or parabolic neighborhood.
\item An almost cone rigidity theorem, which implies closeness to a soliton if a monotone quantity is almost constant, seems to be unavailable for Ricci flows, \emph{a priori.}
\item The theory of synthetic, metric flows, which concerns geometric limits of Ricci flows \cite{Bamler_RF_compactness}, is unfortunately too weak to derive an analog of the relatively elementary cone splitting result.
\end{itemize}

Before discussing our methods in more detail and explaining how these difficulties are overcome, it is helpful to review the main strategy for proving partial regularity theories for other geometric equations.
Let us for a moment consider the partial regularity theory of limits of Einstein metrics or spaces with lower Ricci curvature bounds in the following, as it is closest to our setting.
Let $(M,g)$ be a complete, $n$-dimensional Riemannian manifold with $\Ric \geq - (n-1)g$.
By Bishop-Gromov volume comparison, we know that the volume ratio
\[ \mathcal{V}(x,r) := \frac{|B(x,r)|}{v_{-1}(r)}, \]
which is defined for any $x \in M$ and $r > 0$ is non-increasing in $r$ (here $v_{-1}(r)$ denotes the volume of a ball of radius $r$ in the model space of constant curvature $-1$).
The key result in proving partial regularity is an \emph{almost cone rigidity theorem} involving this monotone quantity \cite{Cheeger-Colding-Cone}.
It states that if $\delta \leq \ov\delta(\eps)$ and
\begin{equation} \label{eq_VV_alm_const}
 0 < r \leq \delta, \qquad \mathcal{V}(x,r) \geq \eps, \qquad \mathcal{V}(x,\delta^{-1} r) \geq \mathcal{V}(x,\delta r) - \delta, 
\end{equation}
then the rescaled, pointed manifold $(M, r^{-2} g, x)$ is $\eps$-Gromov-Hausdorff-close to a metric cone $(C, c)$; here $c \in C$ denotes the vertex.
We will say that $x$ is \emph{$(\eps, r)$-conical} in this case.
As a consequence of the almost cone rigidity theorem and the monotonicity of $\mathcal{V}(x,r)$, we obtain a bound on the number of scales $r = \lambda^j$, $0 < \lambda < 1$, for which $x$ is \emph{not} $(\eps, r)$-conical.
Let us now consider a non-collapsed limit space $(X, d, x_\infty)$ of spaces with lower Ricci curvature bounds; by this we mean a Gromov-Hausdorff limit of a sequence of pointed  Riemannian manifolds $(M_i, g_i, x_i)$ with $\Ric \geq - (n-1)g_i$ that satisfies a non-collapsing condition of the form $\mathcal{V} (x_i, 1) \geq c > 0$.
Passing the consequence of the almost cone rigidity theorem to the limit, implies that every tangent cone $(C_0,c_0)$ of $X$ at any point $x_0 \in X$ must indeed be a metric cone.
This tangent cone $(C_0,c_0)$ is again a non-collapsed limit of spaces with lower Ricci curvature bounds, so by the same reason as before, every tangent cone $(C_1, c_1)$ at any point $x_1 \in C_0$ must be a metric cone.
An important observation is now the following \emph{cone splitting result:} If $x_1 \neq c_0$ is not the vertex of $C_0$, then we must have a Cartesian splitting of the form $C_1 = C'_1 \times \IR$.
By iterating this construction, we can find successive tangent cones of the form $C_k = C_k' \times \IR^k$ for $k = 1, \ldots, n$.
A refinement of this argument allows the conclusion that ``many'' points in the original limit space $X$ have tangent cone isometric to $\IR^n$ \cite{Cheeger_Colding_struc_Ric_below_I,Cheeger_Colding_struc_Ric_below_III,Cheeger-Colding-structure-II, Cheeger-Naber-quantitative, Cheeger-Naber-Codim4}.
In the case in which $(X,d,x_\infty)$ is obtained as a limit of Einstein metrics, an \emph{$\eps$-regularity theorem} implies that points that have tangent cones isometric to $\IR^n$ must, in fact, be regular \cite{Colding-vol-conv,Cheeger_Colding_Tian_2002}.
So the regular part of $X$ must be dense and a refinement of the arguments mentioned above implies further dimensional bounds on the singular set.\footnote{We note that for the sake of brevity we have omitted many details here. The study of the singular set is, of course, highly non-trivial.}
So, in summary, the building blocks of a partial regularity theory for limits of Einstein metrics consist of: a monotone quantity, an almost cone rigidity theorem, a cone splitting result for metric spaces and an $\eps$-regularity theorem.

Let us now focus on the partial regularity and structure theory of $\IF$-limits of Ricci flows.
Unfortunately, in the Ricci flow setting there are, \emph{a priori,} no counterparts for most of the building blocks mentioned in the last paragraph.
This is why our theory has to take a different approach, which focuses more on the geometry of \emph{smooth} Ricci flows and less on the closeness to some idealized synthetic space.
In addition, we have to encode our desired geometric characterizations in analytic bounds that don't interfere with the lack of distance expansion and lower heat kernel bounds.
This roughly entails that our bounds will be formally global and the local nature will only be expressed through the use of (fast decaying) conjugate heat kernels as weight functions.
This will lead to further complications.

Our theory will rely on the monotonicity of the pointed Nash entropy $\NN_{x,t}(\tau)$ in place of the volume ratio $\mathcal{V}(x,r)$.
We will also rely on the estimates from \cite{Bamler_HK_entropy_estimates}, which relate bounds on the pointed Nash entropy at a given point to similar bounds at nearby points and to further geometric and analytic bounds.
Our expectation is that if $\NN_{x,t}(\tau)$ is almost constant for $\tau \in [\delta r^2, \delta^{-1}r^2]$ (corresponding to (\ref{eq_VV_alm_const})), then the flow near $(x,t)$ and at scale $r$ has similar properties as a gradient shrinking soliton (corresponding to metric cones).
However, unlike in the regularity theory of spaces with lower Ricci curvature bounds, an almost cone rigidity theorem seems a priori out of reach; such a result only follows \emph{a posteriori} from our theory (see Theorem~\ref{Thm_almost_cone_rigidity_intro}).
Instead, we will only obtain relatively weak, integral and global bounds of the form
\begin{align}
  \int_{t_0 - \eps^{-1} r^2}^{t_0 - \eps r^2} \int_M \tau \Big| \Ric + \nabla^2 f - \frac1{2\tau} g \Big|^2 d\nu_{x_0, t_0; t} dt &\leq \eps, \label{eq_selfs_intro_1} \\
 \int_M \big| \tau (2\triangle f - |\nabla f|^2 + R) + f - n - W \big| d\nu_{x_0, t_0; t
} &\leq \eps , \qquad t \in [\eps r^2, \eps^{-1} r^2],\label{eq_selfs_intro_2} 
\end{align}
which express similarity to a gradient shrinking soliton.
The left-hand side of the first bound (\ref{eq_selfs_intro_1}) is essentially the excess terms in the monotonicity formula of the pointed Nash-entropy, or to be precise the $\mathcal{W}$-functional, and the second bound will be derived in Section~\ref{sec_geom_bounds_almost_ss}.
We recall that $d\nu_{x_0,t_0;t} = (4\pi \tau)^{-n/2} e^{-f} dg$ denotes the background measure corresponding to the conjugate heat kernel based at $(x_0, t_0)$.
Whenever $(x_0, t_0)$ satisfies the bounds (\ref{eq_selfs_intro_1}), (\ref{eq_selfs_intro_2}), then we will say that the point $(x_0, t_0)$ is \emph{$(\eps, r)$-selfsimilar} (this corresponds to the $(\eps, r)$-conical condition).

Since we will not be able to establish closeness  of the flow to a gradient shrinking soliton at this point, we will need to deduce a weak version of an almost cone splitting theorem  for Ricci flows.
This is carried out in the framework of a quantitative stratification result, borrowing from \cite{Cheeger-Naber-quantitative}.
This will allow us to deduce an \emph{almost static} condition or the existence of \emph{weak splitting maps} near ``many'' points.
The almost static condition is an integral condition of the form
\begin{equation} \label{eq_static_outline}
 r^2 \int_{t_0 - \eps^{-1} r^2}^{t_0 - \eps r^2} \int_M   |{\Ric}|^2  d\nu_{x_0, t_0; t} dt \leq \eps 
\end{equation}
and weak splitting maps are of the form $\vec y = (y_1, \ldots, y_k) : M \times [t_0 - \eps^{-1} r^2, t_0 - \eps r^2] \to \IR^k$, where
\begin{align}
  r^{-2} \int_{t_0 - \eps^{-1} r^2}^{t_0 - \eps r^2} \int_M |\nabla y_i \cdot \nabla y_j - \delta_{ij} | d\nu_{x_0, t_0;t} dt &\leq \eps, \label{eq_weak_splitting_outline_1} \\
 r^{-1} \int_{t_0 - \eps^{-1} r^2}^{t_0 - \eps r^2} \int_M |\square y_i  | d\nu_{x_0, t_0;t} dt &\leq \eps.\label{eq_weak_splitting_outline_2} 
\end{align}
Note that, in contrast to the setting of spaces with lower Ricci curvature bounds (see for example \cite{Cheeger-Naber-Codim4}), this static condition is required to hold \emph{globally} and the splitting maps are defined \emph{globally} as well; their local properties only follow from the fast spatial decay of the measures $d\nu_{x_0,t_0;t}$.

The proof of the static condition and the construction of the splitting maps require a new theory of the geometry of Ricci flows --- in particular in regard to the $(\eps,r)$-selfsimilarity condition.
This theory, which is developed in Part~\ref{part_prel_quant_strat} of this paper, will require us to address most of the issues mentioned in the beginning of this outline.
We will now outline \emph{some} of the new ideas that enter this theory in Part~\ref{part_prel_quant_strat} of this paper.

First, we develop a calculus that allows us to change the basepoint of the background measure $d\nu_{x_0,t_0;t}$ in integral bounds in certain settings.
By this we mean the following.
We will often encounter integral bounds of the form
\begin{equation} \label{eq_have_dnu_x0_bound}
 \int_M |X| \, d\nu_{x_0,t_0;t} \leq C, 
\end{equation}
 where $X$ denotes a geometric quantity and $d\nu_{x_0,t_0} = (4\pi \tau_0)^{-n/2} e^{-f_0}dg$ denotes the measure of the conjugate heat kernel based at $(x_0,t_0)$.
It will often be necessary to change the basepoint $(x_0,t_0)$, i.e. we have to translate this bound into a bound of the form 
\begin{equation} \label{eq_want_dnu_x1_bound}
 \int_M |X| \, d\nu_{x_1,t_1;t} \leq C',
\end{equation}
where $d\nu_{x_1,t_1}$ denotes the measure of the conjugate heat kernel based at some different point $(x_1,t_1)$.
Unfortunately, due to the lack of lower heat kernel bounds and the non-local nature of our integral bounds, the measures $\nu_{x_0,t_0;t}, \nu_{x_1,t_1;t}$ may not be comparable.
However, we will show in Section~\ref{sec_comparing_CHK} that we have a bound of the form
\begin{equation} \label{eq_compare_nu_outline}
 d\nu_{x_1,t_1;t} \leq C e^{\alpha f_0} d\nu_{x_0,t_0;t}, 
\end{equation}
for arbitrarily small $\alpha > 0$, where $C$ depends on $\alpha$ and various geometric quantities involving the pointed Nash entropy of the flow and the relative location of the points $(x_0, t_0)$, $(x_1,t_1)$.
So we can only obtain a bound of the form (\ref{eq_want_dnu_x1_bound}) if instead of (\ref{eq_have_dnu_x0_bound}) the slightly stronger bound
\begin{equation} \label{eq_have_dnu_x0_bound_improved}
  \int_M |X| e^{\alpha f_0} d\nu_{x_0,t_0;t} \leq C. 
\end{equation}
holds. 
It will turn out that a bound of the form (\ref{eq_have_dnu_x0_bound}) can often be upgraded to a bound of the form (\ref{eq_have_dnu_x0_bound_improved}).
This is achieved via a new integral bound on geometric quantities, which is established in Section~\ref{sec_improved_Lp} and roughly states that for sufficiently small $\alpha > 0$
\begin{align}
 \int_{t_0-\eps^{-1} r^2}^{t_0-\eps r^2} \int_M \big( \tau |{\Ric}|^2 + \tau |\nabla^2 f_0|^2  + \tau |\nabla f_0|^4 +   \tau^{-1} \big) e^{\alpha f_0} d\nu_{x_0,t_0;t} dt &\leq C , \label{eq_improved_Lp_outline_1} \\
  \int_M \big( \tau |R| + \tau |\triangle f_0| + \tau |\nabla f_0|^2   + 1 \big) e^{\alpha f_0} d\nu_{x_0,t_0;t_0 -r^2}  &\leq C .  \label{eq_improved_Lp_outline_2}  
\end{align}
The important detail here is the extra $e^{\alpha f_0}$-term.
These bounds allow us to estimate additional terms that occur in the derivation of a bound of the form (\ref{eq_have_dnu_x0_bound_improved}) from a bound of the form (\ref{eq_have_dnu_x0_bound}), for a certain class of geometric quantities $X$.
It turns out that this class is general enough for our purposes if estimates are structured appropriately.

Second, we address the issue regarding a lack of distance expansion bounds.
Let us describe this issue more precisely.
Previous work \cite{Bamler_HK_entropy_estimates,Bamler_RF_compactness} of the author has shown that the concept of worldlines only plays a secondary role in the study of higher-dimensional Ricci flows if no curvature bounds are present.
Instead, given a point $(x_0, t_0) \in M \times I$ and time $t < t_0$, it is more useful to consider the probability measure $\nu_{x_0,t_0;t}$ associated to the conjugate heat kernel based at $(x_0,t_0)$.
It is a known fact that the 1-Wasserstein distance $d^{g_t}_{W_1}(\nu_{x_0,t_0;t}, \nu_{x_1, t_0;t})$,  between two conjugate heat kernel measures, based at two points $(x_0, t_0)$, $(x_1, t_0)$, is non-decreasing in $t$ and that
\[ d^{g_t}_{W_1}(\nu_{x_0,t_0;t}, \nu_{x_1, t_0;t}) \leq d_{t_0} (x_0, x_1). \]
This bound can be viewed as a form of ``non-contraction of distances''; most arguments in this paper rely on this point of view.
The issue, however, is that a reverse bound on the growth of $d^{g_t}_{W_1}(\nu_{x_0,t_0;t}, \nu_{x_1, t_0;t})$ (i.e. on the \emph{expansion} of the $1$-Wasserstein distance) seems to be false in general.
That is for $t' < t'' < t_0$ we cannot expect a bound of the form
\begin{equation} \label{eq_expansion_bound_outline}
 d^{g_{t''}}_{W_1}(\nu_{x_0,t_0;t''}, \nu_{x_1, t_0;t''}) \leq C ( d^{g_{t'}}_{W_1}(\nu_{x_0,t_0;t'}, \nu_{x_1, t_0;t'})  ) 
\end{equation}
in a general setting.
We will overcome this issue as follows.
In Section~\ref{sec_dist_expansion}, we will prove that  an expansion bound of the form (\ref{eq_expansion_bound_outline}) is indeed true to a certain extent if $(x_0,t_0)$ is sufficiently selfsimilar.
Moreover, we will employ a covering lemma from \cite[\HKBasicCovering]{Bamler_HK_entropy_estimates}, which allows us to sidestep this issue in certain cases, at the expense of a purely numerical estimate.
We have structured our arguments in such a way that an expansion bound is only necessary in these two settings.

Third, we prove a new almost monotonicity result for the integral of the scalar curvature near sufficiently selfsimilar points $(x_0,t_0)$.
This bound will be roughly of the form 
\begin{equation} \label{eq_almost_monotone_int_R_outline}
 \int_M \tau R\, d\nu_{t_1} \leq \int_M \tau R \, d\nu_{t_2} + \eps, \qquad t_1 \leq t_2 < t_0. 
\end{equation}
The bound will allow us to derive the static condition near certain points of the flow.
Interestingly, a reverse bound, seems unavailable \emph{a priori,} however, it follows from our theory \emph{a posteriori.}

Let us now focus on Part~\ref{part_loc_reg} of this paper.
Recall that the quantitative stratification result from Part~\ref{part_prel_quant_strat} only implied relatively weak integral local regularity properties --- involving the static condition (\ref{eq_static_outline}) and weak splitting maps as in (\ref{eq_weak_splitting_outline_1}), (\ref{eq_weak_splitting_outline_2}).
In order to deduce meaningful geometric characterizations from these regularity properties, we first have to upgrade the \emph{weak} splitting maps to \emph{strong} splitting maps.
These are essentially maps of the form $\vec y = (y_1, \ldots, y_k) : M \times [t_0 - \eps^{-1} r^2, t_0 - \eps r^2] \to \IR^k$ that still satisfy a bound of the form (\ref{eq_weak_splitting_outline_1}), but in place of (\ref{eq_weak_splitting_outline_2}), its component functions satisfy the precise heat equation
\[ \square y_i = 0. \]
We will obtain such strong splitting maps via a cutoff process, using a generalization of a Gaussian estimate from \cite{Hein-Naber-14}.
The resulting strong stability maps have similar properties as splitting maps in the setting of spaces with lower Ricci curvature bounds (see for example \cite{Cheeger-Naber-Codim4}).
However, due to the global nature of the integral bound (\ref{eq_weak_splitting_outline_1}), they lack a pointwise gradient bound (even on bounded domains or near the $H_n$-center of $\nu_{x_0,t_0;t}$).
Instead, we will derive various \emph{integral} bounds on the gradients $\nabla y_i$ and Hessians $\nabla^2 y_i$, which will turn out to be sufficient for our purposes.

Next, we analyze the geometry near points that possess sufficiently precise strong stability maps and are  static or selfsimilar to a sufficient degree.
We will obtain an \emph{integral} characterization of the geometry, which essentially states that such points are close to a constant flow on a space of the form $C \times \IR^k$, where $C$ is a Ricci flat cone with vertex $c \in C$.
Amongst other things, we will also obtain another integral cone-splitting result, which asserts the existence of a $(k+1)$-dimensional strong splitting map sufficiently far away from $\{ c \} \times \IR^k$.

Lastly, we derive an $\eps$-regularity result for points that are sufficiently static and possess a sufficiently precise strong splitting map of dimension $n$.
The proof of this follows the general strategy of \cite{Colding-vol-conv}, but is different at its core due to the lack of lower Ricci curvature bounds.

Eventually, by combining the quantitative stratification result from Part~\ref{part_prel_quant_strat} with the local regularity theory of Part~\ref{part_loc_reg}, we first derive a preliminary partial regularity theory in Part~\ref{part_partial_reg}.
This theory essentially states that non-collapsed $\IF$-limits of Ricci flows are regular on the complement of a singular subset of codimension $\geq 1$.
We also derive certain symmetry properties of the limit if the given sequence of Ricci flows are sufficiently static, selfsimilar or admit strong splitting maps of sufficiently high accuracy.
Some of these results depend on a working assumption, which governs the codimension of the singular set. 
Next, we use the preliminary partial regularity theory to rule out certain tangent flows in the limit.
This will allow us to conclude first that the singular set in the limit has codimension $\geq 2$ and then codimension $\geq 4$.
For the latter, we adapt the arguments from \cite{Cheeger-Naber-Codim4} to the parabolic setting.

Finally, in Part~\ref{part_proofs} we combine all the results obtained so far and prove the main theorems mentioned in the introduction.

\section{Conventions and preliminary results}
In this section we will first discuss some basic conventions that will be used throughout this paper.
Next, we will recall some basic terminology and geometric and analytic bounds from \cite{Bamler_HK_entropy_estimates}, which can be viewed as a preparatory paper.

\subsection{Basic conventions}
In the following we will fix a dimension $n$ and assume that all manifolds are $n$-dimensional, unless noted otherwise.
For the remainder of this paper, we will also omit the dependence on the dimension $n$ in all bounds. 
Throughout this paper, we will often consider Ricci flows $(M, (g_t)_{t \in I})$,
\[ \partial_t g_t = -2 \Ric_{g_t}, \]
and we will assume that $M$ is $n$-dimensional and that $I \subset \IR$ is an interval with non-empty interior, which we will refer to as the {\bf time-interval}.
We will often switch between the conventional picture (in which $x \in M$ are points and $t \in I$ are times) and the space-time picture (in which $(x,t) \in M \times I$ are points), as long as this does not create any confusion.

In the following, unless stated otherwise, capital roman or greek letters will denote large constants (larger than $1$), while small roman or greek letters will denote small constants (smaller than $1$).
The letter $C$ (respectively $c$) will mainly denote a large (respectively small) generic constant.
We will express the dependence of constants in parentheses; i.e. $A(B, \beta)$ means that $A$ depends only on $B$ and $\beta$ in a continuous fashion.
A condition of the form ``if $\eps \leq \ov\eps$'' or ``if $A \geq \underline{A} (\eps)$'' will mean ``there is a universal constant $\ov\eps > 0$ such that if $\eps \leq \ov\eps$, then \ldots'' or ``there is a universal continuous function $\underline{A} : (0,1) \to (1, \infty)$ such that if $A \geq \underline{A} (\eps)$, then \ldots''.

Expressions of the form $\Psi (\eps, A | B, \beta)$ denote generic functions that depend on their parameters in a continuous fashion and that converge to $0$ as their first set of parameters degenerates --- in this example this means that $\eps \to 0$, $A \to \infty$.
This convergence is assumed to be locally uniform the second set of parameters --- i.e. $B, \beta$ in this example.
We allow $\Psi$ to take values in $(0, \infty]$, equipped with the standard topology, where a bound of the form $a \leq \infty$ is vacuous.
This convention will allow us to simplify statements of the form ``if $\eps \leq \ov\eps$, then $a \leq \Psi (\eps)$'' to ``$a \leq \Psi (\eps)$''.

\subsection{Heat kernels, conjugate heat kernel measures and important identities} \label{subsec_H_operator_identities}
Consider a Ricci flow $(M, (g_t)_{t \in I})$.
We assume that $M$ is compact, however, under suitable assumptions some of the following discussion will also make sense on non-compact manifolds. 
We will denote by $dg_t$ the Riemannian volume measure corresponding to the metric $g_t$ and we will sometimes write for any $u : M \times I \to \IR$ or $u : M \times \{ t \} \to \IR$ of suitable regularity
\[ \int_M u(x,t) \, dg_t(x) = \int_M u(\cdot,t) \, dg_t = \int_M u \, dg_t. \]

In this paper we will frequently consider the {\bf heat operator}
\[ \square := \partial_t - \triangle_{g_t}, \]
which is coupled to the Ricci flow and can be applied to any function of the form $u \in C^2 (M \times I')$ for any non-trivial subinterval $I' \subset I$.
A solution to the equation $\square u = 0$ is called a solution to the {\bf heat equation (coupled with the Ricci flow).}
Correspondingly, the operator 
\[ \square^* := -\partial_t - \triangle_{g_t} + R \]
is called the {\bf conjugate heat operator}.
Here $R$ denotes the scalar curvature at time $t$.
Note that for any $u, v \in C^2 (M \times I')$ we have
\[    \int_M ( \square u) v \, dg_t -  \int_M u (\square^* v) \, dg_t = \frac{d}{dt} \int_M uv \, dg_t . \]
So if $I'$ is open and $u, v$ are compactly supported, then this gives
\[ \int_{I'} \int_M ( \square u) v \, dg_t dt = \int_{I'} \int_M u (\square^* v) \, dg_t dt.  \]
We also recall that for any solution to the {\bf conjugate heat equation} $\square^* v = 0$ we have
\begin{equation} \label{eq_ddt_intv_0}
 \frac{d}{dt} \int_M v \, dg_t = 0. 
\end{equation}

For any $x, y \in M$ and $s,t \in I$, $s < t$ we denote by $K(x,t;y,s)$ the heat kernel of $\square$.
That is, for fixed $(y,s)$, the function $K(\cdot, \cdot; y,s)$ is a {\bf heat kernel based at $(y,s)$}, i.e.
\[ \square K(\cdot, \cdot; y,s) = 0, \qquad \lim_{t \searrow s} K(\cdot, t; y,s) = \delta_{y}. \]
By duality, for fixed $(x,t)$, the function $K(x,t; \cdot, \cdot)$ is a {\bf conjugate heat kernel based at $(x,t)$}, i.e.
\[ \square^* K(x,t; \cdot, \cdot) = 0, \qquad \lim_{s \nearrow t} K(x,t; \cdot,s) = \delta_{x}. \]
Due to (\ref{eq_ddt_intv_0}) we have
\[ \int_M K(x,t; \cdot, s) \, dg_s = 1. \]
Hence we will frequently use the following abbreviated notion.

\begin{Definition}
For $(x,t) \in M \times I$ and $s \in I$, $s \leq t$, we denote by $\nu_{x,t;s} = \nu_{x,t}(s)$ the {\bf conjugate heat kernel measure,} i.e. the probability measure on $M$ defined by
\[ d\nu_{x,t;s} := K(x,t; \cdot, s) \, dg_s, \qquad d\nu_{x,t;t} := \delta_x. \]
\end{Definition}

We will often omit the index $s$ and view $\nu_{x,t} = (\nu_{x,t;s})_{s \in I, s \leq t}$ as a family of probability measures.
Moreover, we will often write
\[ d\nu_{x,t;s} = (4\pi \tau)^{-n/2} e^{-f} dg_s, \]
where $\tau(s) = t - s$ and $f \in C^\infty ( M \times ( I \cap (-\infty,t)))$ is called the {\bf potential}, which satisfies the evolution equation
\begin{equation} \label{eq_potential_evolution_equation}
 - \partial_t f = \triangle f - |\nabla f|^2 + R - \frac{n}{2\tau}. 
\end{equation}
Note that on Euclidean space we have $f = \frac1{4\tau} d^2(x, \cdot)$.
If $u \in C^0 (M \times I')$ is a function defined on a time-slab of $M \times I$, $(x_0, t_0) \in M \times I$ and $s \in I'$, $s \leq t_0$, then we will often use the abbreviations
\[ \int_M u \, d\nu_{x_0, t_0; s} = \int_M u \, d\nu_{x_0, t_0} (s) = \int_M u \, d\nu_{x_0,t_0} \bigg|_{t = s} = \int_M u(\cdot, s) d\nu_{x_0, t_0; s}. \]

The following identities, which can be checked easily, will be used frequently throughout this paper without further reference.

\begin{Lemma}
For any vector field $X$ on $M$ and any $u \in C^2 (M \times I)$ we have
\begin{align}
 \int_M \DIV (X) \, d\nu_{x,t;s} &= \int_M X \cdot \nabla f \, d\nu_{x,t;s}, \notag \\
 \int_M \triangle u \, d\nu_{x,t;s} &= \int_M \nabla u \cdot \nabla f \, d\nu_{x,t;s}, \notag \\
 \frac{d}{ds} \int_M u \, d\nu_{x,t;s} &= \int_M \square u \, d\nu_{x,t;s}  \label{eq_dds_int_u_int_square_u}.
 \end{align}
\end{Lemma}

Lastly, we discuss a weak interpretation of the heat operator, applied to certain non-smooth functions, which will mostly arise by taking absolute value of smooth, possibly vector valued functions.
We remark that the following discussion can be generalized to functions of even lower regularity.
However, this will not be needed here.

\begin{Definition} \label{Def_weakly_2_diff}
Let $N$ be a manifold, possibly with boundary, and $u \in C^0(N)$.
We say that $u$ is {\bf weakly twice differentiable} if it is locally Lipschitz and for every $x \in N$ there is an open neighborhood $x \in U \subset N$ and a family $(u_\eps \in C^2(U))_{\eps \in (0,1)}$ such that for some Riemannian metric $g$ on $U$ and some constant $C < \infty$ we have:
\begin{enumerate}[label=(\arabic*)]
\item \label{Def_weakly_2_diff_1} $u_\eps \to u|_{U}$ uniformly as $\eps \to 0$.
\item \label{Def_weakly_2_diff_2} $\nabla u_\eps \to \nabla u|_U$ a.e. as $\eps \to 0$.
\item \label{Def_weakly_2_diff_3} $|\nabla u_\eps| \leq C$.
\item \label{Def_weakly_2_diff_4} $\int_{U} |\nabla^2 u_\eps| dg \leq C$.
\end{enumerate}
Note that if \ref{Def_weakly_2_diff_1}--\ref{Def_weakly_2_diff_4} hold for some Riemannian metric $g$, then the same conditions hold for any other Riemannian metric, possibly after restriction to a smaller subset and for a larger constant $C$.
\end{Definition}

The following proposition provides a criterion for a function to be weakly twice differentiable.

\begin{Proposition} \label{Prop_what_is_weakly_2_diff}
Let $N$ be a manifold, possibly with boundary.
Any $u \in C^2(N)$ is weakly twice differentiable.
Moreover, if $u_1, u_2 \in C^0 (N)$ are weakly twice differentiable, then so are $u_1 \pm u_2$, $u_1 u_2$ and $|u_1|$.
Lastly, if $E$ is a Euclidean vector bundle over $N$ and $X \in C^2(N; E)$ is a section of regularity $C^2$, then $u := |X|$ is weakly twice differentiable.
\end{Proposition}

\begin{proof}
The fact that $u_1 \pm u_2$, $u_1 u_2$ are weakly twice differentiable is clear.
By passing to a coordinate system, the remaining statements can be reduced to the following claim.

\begin{Claim}
Assume that $U \subset \IR^n$ is either the open Euclidean unit ball $B_1$ or the intersection $B_1 \cap \IR^{n-1} \times [0, \infty)$ and consider a vector-valued function $\vec u : U \to \IR^k$ whose component functions are weakly twice differentiable.
Suppose there is a constant $C < \infty$ and a family $(\vec u_\eps)_{\eps \in (0,1)}$ whose component functions satisfy Properties~\ref{Def_weakly_2_diff_1}--\ref{Def_weakly_2_diff_4} of Definition~\ref{Def_weakly_2_diff}.

Then there is a neighborhood $U' \subset U$ of $\vec 0$, a constant $C' < \infty$ and a family $(u'_\eps)_{\eps \in (0,1)}$ such that Properties~\ref{Def_weakly_2_diff_1}--\ref{Def_weakly_2_diff_4} of Definition~\ref{Def_weakly_2_diff} hold for $u' := |\vec u|$.
\end{Claim}

\begin{proof}
Note that without loss of generality, we may choose the Euclidean metric $g = g_{\eucl}$ on $U$ to verify Properties~\ref{Def_weakly_2_diff_1}--\ref{Def_weakly_2_diff_4}.
Set
\[ u'_\eps := \sqrt{ |\vec u_\eps|^2 + \eps } . \]
Then $u'_\eps \to u'$ uniformly and $u'$ is locally Lipschitz.
Moreover, for any $v \in \IR^n$
\begin{multline} \label{eq_up_eps_der}
 |\partial_v u'_\eps| = \bigg| \frac{\vec u_\eps \cdot  \partial_v \vec u_\eps}{\sqrt{|\vec u_\eps|^2 + \eps}} \bigg| \leq  |\partial_v \vec u_\eps|, \qquad \\
\partial^2_{v,v} u'_\eps = \frac{\vec u_\eps \cdot  \partial^2_{v,v} \vec u_\eps}{\sqrt{| \vec u_\eps |^2 + \eps}} + \frac{|\partial_v \vec u_\eps|^2}{\sqrt{|\vec u_\eps|^2 + \eps}} - \frac{| \vec u_\eps |^2 \, |\partial_v \vec u_\eps|^2}{(|\vec u_\eps|^2 + \eps)^{3/2}} 
\geq \frac{\vec u_\eps \cdot \partial^2_{v,v}  \vec u_\eps }{\sqrt{|\vec u_\eps|^2 + \eps}}. 
\end{multline}
So $|\partial_v u'_\eps| \leq |\partial_v \vec u_\eps| \leq C$.
It remains to verify Properties~\ref{Def_weakly_2_diff_2}, \ref{Def_weakly_2_diff_4} of Definition~\ref{Def_weakly_2_diff}.

For Property~\ref{Def_weakly_2_diff_2} observe that due to the local Lipschitz property, $\nabla \vec u, \nabla u'$ both exist almost everywhere.
Consider a point $ y \in U$ at which both derivatives exist.
If $\vec u ( y) \neq 0$, then $\lim_{\eps \to 0} (\nabla u'_\eps)( y) =  (|\vec u|^{-1} \vec u \cdot \nabla \vec u)( y) =( \nabla u')( y)$.
If $\vec u( y) = 0$, then we must have $(\nabla u' )( y)= 0$, which implies that $\nabla \vec u( y) = 0$ and therefore $\lim_{\eps \to 0} (\nabla u'_\eps)( y) = (\nabla u') ( y)$, as well.

It remains to verify Property~\ref{Def_weakly_2_diff_4}.
Fix some $v \in \IR^n$ with $|v| = 1$.
By (\ref{eq_up_eps_der}), the negative part of $\partial^2_{v,v} u'_\eps$ satisfies
\[ \int_U (\partial^2_{v,v} u'_\eps)_-  dg_{\eucl}  \leq  \int_U  \frac{|\vec u_\eps| \, |\partial^2_{v,v} \vec u_\eps|}{\sqrt{| \vec u_\eps|^2 + \eps}} dg_{\eucl}
\leq \int_U |\partial^2_{v,v} \vec u_\eps| dg_{\eucl} \leq C. \]
Let $\phi \in C^2_c (\Int U)$ be a function that is compactly supported inside the interior of $U$ and satisfies $0 \leq \phi \leq 1$.
Then due to integration by parts we have
\[  \int_U (\partial^2_{v,v} u'_\eps) \phi \,  dg_{\eucl}
= - \int_U (\partial_{v} u'_\eps)(\partial_v \phi) \, dg_{\eucl}
\leq  C \int_U |\partial_v \phi| dg_{\eucl}. \]
It follows that
\begin{equation} \label{eq_abs_partial2_vv_up_eps}
 \int_U |\partial^2_{v,v} u'_\eps| \phi \,  dg_{\eucl}
\leq 2\int_U (\partial^2_{v,v} u'_\eps)_- \phi \,  dg_{\eucl} + \int_U (\partial^2_{v,v} u'_\eps) \phi \,  dg_{\eucl} \leq 2C + C \int_U |\partial_v \phi| dg_{\eucl}. 
\end{equation}
If $U = B_1$, then we are done by choosing $\phi$ such that $\phi \equiv 1$ on $B_{1/2}$.
If $U = B_1 \cap \IR^{n-1} \times [0, \infty)$, then we choose a sequence $\phi_i \in C^2_c (\Int U)$, $0 \leq \phi_i \leq 1$ with $\phi_i \to 1$ pointwise on $U' := B_{1/2} \cap U$ such that the last integral in (\ref{eq_abs_partial2_vv_up_eps}) remains uniformly bounded, for example $\phi_i = \phi' \eta(i y_n)$, where $\phi' \in C^2_c( U)$ with $\phi' \equiv 1$ on $U'$ and $\eta$ is a suitable cutoff function.
After passing to a limit we obtain that there is a uniform constant $C' < \infty$ such that
\[  \int_{U'} |\partial^2_{v,v} u'_\eps|  \,  dg_{\eucl} \leq C'. \]
This proves Property~\ref{Def_weakly_2_diff_4}.
\end{proof}

This finishes the proof of the proposition.
\end{proof}

The next proposition states that we can still make sense of the heat operator applied to a weakly twice differentiable function $u$ if we view $d\mu_{\square u} := \square u \, dg$ as a signed measure.

\begin{Proposition} \label{Prop_weak_square}
Let $(M, (g_t)_{t \in I})$ be a Ricci flow on a compact manifold and $u \in C^0 (M \times I)$ a weakly twice differentiable function.
Then the spatial gradient $\nabla u$ exists almost everywhere on $M \times I$ and for every compactly supported vector field $X \in C^1_c (M ; TM)$ and any time $t \in I$ we have
\begin{equation} \label{eq_weak_nabla_duality}
 \int_M X \cdot \nabla u \, dg_t = - \int_M (\DIV X) u \, dg_t. 
\end{equation}
Moreover, there is a unique signed measure $\mu_{\square u}$ on $M \times I$ of locally finite total variation such that for any compact sub-interval $[t_1, t_2] \subset I$ and any compactly supported $\phi \in C^2_c(  M \times [t_1, t_2])$ we have
\begin{equation} \label{eq_weak_square_duality}
  \int_{M \times [t_1, t_2]} \phi \, d\mu_{\square u} 
= \int_{M} u \, \phi \, dg_t \bigg|_{t=t_1}^{t=t_2} + \int_{t_1}^{t_2} \int_M u \square^* \phi \, dg_t dt
= \int_{t_1}^{t_2} \int_M (\partial_t u \,  \phi  + \nabla u \cdot \nabla \phi ) dg_t dt . 
\end{equation}
We have $d\mu_{\square u} =( \square u )dg_t dt$ wherever $u$ is $C^2$. 
\end{Proposition}

\begin{proof}
We will show the local version of this proposition, i.e. for every $(x_0, t_0) \in M \times I$ there is an open neighborhood $(x_0, t_0) \in U \subset M \times I$ such that there is a unique signed measure $\mu_{\square u}$ of finite total variation satisfying (\ref{eq_weak_nabla_duality}), (\ref{eq_weak_square_duality}) for all $\phi \in C^2_c (U) \subset C^2_c (M \times I)$.
The proposition in its full generality follows from this statement using a partition of unity.

So assume that $(x_0, t_0) \in M \times I$ is given and choose $(x_0, t_0) \in U \subset M \times I$, $C < \infty$ and $(u_\eps)_{\eps \in (0,1)}$ as in Definition~\ref{Def_weakly_2_diff}.
Consider the functional $F : C^2_c ( U ) \to \IR$ defined by
\[ F(\phi) := \int_{M} u \, \phi \, dg_t \bigg|_{t=t_1}^{t=t_2} + \int_{t_1}^{t_2} \int_M u \square^* \phi \, dg_t dt, \]
where $[t_1, t_2] \subset I$ is chosen such that $\supp u \subset M \times [t_1, t_2]$.
We have
\[ |F(\phi)| = \lim_{\eps \to 0}  \bigg| \int_{M} u_\eps \phi \, dg_t \bigg|_{t=t_1}^{t=t_2} + \int_{t_1}^{t_2} \int_M u_\eps \square^* \phi \, dg_t dt \bigg|
=   \lim_{\eps \to 0} \bigg| \int_{t_1}^{t_2} \int_M (\square u_\eps) \, \phi \, dg_t dt \bigg|
\leq C' \max_{U} |\phi| , \]
where $C' < \infty$ depends only on $C$ and the Ricci flow.
So by the Riesz-Markov Theorem there is a unique signed measure $\mu_{\square u}$ of finite total variation on $U$ such that
\[ F(\phi) = \int_{U} \phi \, d\mu_{\square u}. \]
This proves the first equality in (\ref{eq_weak_square_duality}).
The other equalities in (\ref{eq_weak_nabla_duality}) and (\ref{eq_weak_square_duality}) follow using Stokes' Theorem; note that $u$ is locally Lipschitz.
\end{proof}

Note that any measure of the form $\mu_{\square u}$ can be expressed as the difference of its positive and negative part
\[ \mu_{\square u} = (\mu_{\square u})_+ - (\mu_{\square u})_-,\]
where $(\mu_{\square u})_\pm$ are non-negative measures.

We have the following computational rules:

\begin{Proposition} \label{Prop_mu_square_rules}
In the setting of Proposition~\ref{Prop_weak_square} the following holds.
If $u_1, u_2 \in C^0 (M \times I)$ are weakly twice differentiable and $a \in \IR$, then
\[ \mu_{\square (u_1 \pm u_2)} = \mu_{\square u_1} \pm \mu_{\square u_2}, 
\qquad \mu_{\square (au_1)} = a \mu_{\square u_1}, 
\qquad \mu_{\square |u_1| } \leq  |\mu_{\square u_1}| =  (\mu_{\square u_1})_+ + (\mu_{\square u_1})_-. \]
Moreover, we have $\mu_{\square |u_1| } = (\sign u_1) \mu_{\square u_1}$ on $M \times I \setminus u_1^{-1} (0)$.
Lastly, if $E$ is a Euclidean vector bundle over $M \times I$  equipped with a metric connection and $X \in C^2(M \times I; E)$ is a section of regularity $C^2$, then $d\mu_{\square |X| } \leq  |\square X| dg_t dt$. 
\end{Proposition}

\begin{proof}
All equalities follow from the duality (\ref{eq_weak_square_duality}) in Proposition~\ref{Prop_weak_square}.
The remaining statements are a consequence of the following claim.

Let us show that $\mu_{\square |u| } \leq  |\mu_{\square u}|$ for any weakly twice differentiable $u \in C^0 (M \times I)$.
Let $\eta : \IR \to [0,1]$ be a smooth and non-decreasing function with $\eta \equiv -1$ on $(-\infty, -1]$ and $\eta \equiv 1$ on $[1, \infty)$.
Set $\eta_\eps := \eta (u / \eps)$. 
Observe that $\lim_{\eps \to 0} u \eta_\eps = |u|$ pointwise and that almost everywhere
\[ \nabla \eta_\eps = \eps^{-1} \eta' (u/\eps) \nabla u, \qquad \partial_t \eta_\eps = \eps^{-1} \eta' (u/\eps) \partial_t u. \] 
Then for any compact sub-interval $I' := [t_1, t_2] \subset I$ and any compactly supported $\phi \in C^2_c (M \times I')$ with $\phi \geq 0$ we have, using Proposition~\ref{Prop_weak_square},
\begin{align*}
 \int_{M \times I'} \phi \, d\mu_{\square |u|} 
&= \int_M |u| \phi \, dg_t \bigg|_{t=t_1}^{t=t_2} +  \int_{t_1}^{t_2} \int_M |u| \square^* \phi \,  dg_t dt \\
&= \lim_{\eps \to 0} \bigg(  \int_M u \eta_\eps \phi \, dg_t \bigg|_{t=t_1}^{t=t_2} + \int_{t_1}^{t_2} \int_M u \eta_\eps  \square^* \phi \, dg_t dt \bigg) \displaybreak[1] \\
&=    \lim_{\eps \to 0} \int_{t_1}^{t_2} \int_M \big(  \partial_t ( u \eta_\eps )  \phi + \nabla (u \eta_\eps) \cdot \nabla \phi \big) \, dg_t dt \displaybreak[1]  \\
&=    \lim_{\eps \to 0} \int_{t_1}^{t_2} \int_M\big( \partial_t  u \, ( \eta_\eps  \phi ) + \eps^{-1} u \, \eta' ( u/\eps )\, \partial_t u \, \phi + \nabla u \cdot  \nabla ( \eta_\eps \phi  ) \\
&\qquad \qquad \qquad \qquad - \eps^{-1} \eta'(u/ \eps) |\nabla u|^2 \phi + \eps^{-1} u \,  \eta'( u / \eps ) \,  \nabla u \cdot \nabla \phi \big) \, dg_t dt \displaybreak[1]  \\
&\leq   \lim_{\eps \to 0} \bigg( \int_{M \times I'}   \eta_\eps  \phi \, d\mu_{\square u}
 +  \int_{t_1}^{t_2} \int_M ( \eps^{-1} u \, \eta' ( u / \eps)\, \partial_t u \, \phi + \eps^{-1} u \, \eta'(u / \eps) \,  \nabla u \cdot \nabla \phi \big) \, dg_t dt \bigg) \displaybreak[1]  \\
&\leq    \int_{M \times I'}    \phi \, d|\mu_{\square u}|
 + \limsup_{\eps \to 0} C  \int_{t_1}^{t_2} \int_M   | \eps^{-1} u \, \eta' (u / \eps)|  \, dg_t dt \displaybreak[1]  \\
&\leq    \int_{M \times I'}    \phi \, d|\mu_{\square u}|
 + \limsup_{\eps \to 0} C \int_{\{ 0 < |u| \leq \eps \} \cap M \times I'}    dg_t dt  \\
 &=   \int_{M \times I'}    \phi \, d|\mu_{\square u}|.
 \end{align*}
 
 For last statement, we argue similarly.
If $\phi \in C^2_c (M \times I')$ with $\phi \geq 0$, then we have, using Proposition~\ref{Prop_weak_square} and due to the same inequality as in (\ref{eq_up_eps_der}),
\begin{align*}
 \int_{M \times I'} \phi \, d\mu_{\square |X|} 
& = \int_M |X| \phi \, dg_t \bigg|_{t=t_1}^{t=t_2} +  \int_{t_1}^{t_2} \int_M |X| \square^* \phi \,  dg_t dt \\
 &= \lim_{\eps \to 0} \int_M \sqrt{|X|^2+\eps} \, \phi \, dg_t \bigg|_{t=t_1}^{t=t_2} +  \int_{t_1}^{t_2} \int_M \sqrt{|X|^2+\eps} \, \square^* \phi \,  dg_t dt \displaybreak[1] \\
 &=    \lim_{\eps \to 0} \int_{t_1}^{t_2} \int_M \big( \square \sqrt{|X|^2+\eps} \big)   \,\phi  \, dg_t dt  \\
& \leq \lim_{\eps \to 0} \int_{t_1}^{t_2} \int_M  \frac{X \cdot \square X}{ \sqrt{|X|^2+\eps}}  \,\phi  \, dg_t dt 
 =  \int_{t_1}^{t_2} \int_M |\square X|  \,\phi  \, dg_t dt.
\end{align*}
This finishes the proof.
\end{proof}

By abuse of notation, we will often write for any weakly twice differentiable function $u \in C^0 (M \times [t_1, t_2])$.
\[ \int_{t_1}^{t_2} \int_M ( \square u ) \, \phi \, dg_t dt := \int_{M \times [t_1,t_2]}  \phi \, d\mu_{\square u}. \]
Moreover, if $d\nu = d\nu_{x_0, t_0} = (4\pi \tau)^{-n/2} e^{-f} dg$ denotes a conjugate heat kernel for some $t_0 > t_1$, then we will also write
\[ \int_{t_1}^{t_2} \int_M ( \square u ) \, \phi \, d\nu_t dt 
= \int_{t_1}^{t_2} \int_M ( \square u ) \, \phi \, (4\pi \tau)^{-n/2} e^{-f} \, dg_t dt
= \int_{M \times [t_1,t_2]}  \phi \, (4\pi \tau)^{-n/2} e^{-f} \, d\mu_{\square u}. \]
Note that (\ref{eq_weak_square_duality}) from Proposition~\ref{Prop_weak_square} implies that, as usual,
\begin{equation} \label{eq_weak_heat_op}
 \int_{t_1}^{t_2} \int_M ( \square u ) \, d\nu_t dt = \int_M u \, d\nu_t \bigg|_{t=t_1}^{t=t_2}. 
\end{equation}
We remark that $d\mu_{\square u}$ is a measure on the spacetime $M \times [t_1, t_2]$, so expressions of the form $\int_M (\square u) \, \phi \,  dg_t$ are only defined for almost every $t \in [t_1, t_2]$.

\subsection{\texorpdfstring{Monotonicity of the $W_1$-distance}{Monotonicity of the W\_1-distance}}
We briefly recall the definition of the $W_1$-Wasserstein distance between two probability measures and its monotonicity property on a Ricci flow, which was already used extensively in \cite{Bamler_HK_entropy_estimates,Bamler_RF_compactness}.
We also refer to these papers for a more in-depth discussion.

If $\mu_1, \mu_2$ denote two probability measures on a complete and separable metric space $(X,d)$, then we define the $W_1$-distance between $\mu_1, \mu_2$ by
\begin{equation} \label{eq_d_W1_def}
 d_{W_1} (\mu_1, \mu_2) := \sup_f  \bigg( \int_X f \, d\mu_1  - \int_X f \, d\mu_2 \bigg), 
\end{equation}
where the supremum is taken over all bounded, $1$-Lipschitz functions $f : X \to \IR$.
If $(X,d)$ is the length space of a Riemannian manifold $(M,g)$, then we will often express the $W-1$-distance as $d^g_{W_1}$.

We remark that $d_{W_1}$ defines a complete metric on the space of probability measures on $X$ if we allow infinite distances \cite[Theorem~7.3]{Villani_topics_in_OT}.
For an alternative definition of $d_{W_1}$ using couplings see \cite{Bamler_RF_compactness} or \cite[Definition~7.1.1]{Villani_topics_in_OT}; the equivalence of both definitions holds due to the Kantorovich-Rubinstein Theorem \cite[Theorem~1.14]{Villani_topics_in_OT}.

We will frequently use the following monotonicity result (see \cite[\HKWoneMonotone]{Bamler_HK_entropy_estimates}):

\begin{Proposition} \label{Prop_monotonicity_W1}
Let $(M, (g_t)_{t \in I})$ be a super Ricci flow on a compact manifold and denote by $v_1, v_2 \in C^\infty(M \times I')$, $I' \subset I$, two non-negative solutions to the conjugate heat equation $\square^* v_1 = \square^* v_2 = 0$ such that $\int_M v_i (\cdot, t) dg_t = 1$ for all $t \in I'$, $i=1,2$.
Denote by $\mu_{1,t}, \mu_{2,t}$ the associated probability measures with $d\mu_{i,t} = v_i(\cdot, t) dg_t$, $i=1,2$.
Then
\[ I' \longrightarrow [0, \infty], \qquad t \longmapsto d^{g_t}_{W_1} ( \mu_{1,t}, \mu_{2,t} ) \]
is non-decreasing.
Moreover, for any two points $x_1, x_2 \in M$ and $t_0 \in I$ we have for all $t \leq t_0$, $t \in I$
\[ d^{g_t}_{W_1} (\nu_{x_1,t_0; t}, \nu_{x_2, t_0; t} ) \leq d_{t_0} (x_1, x_2). \]
\end{Proposition}
\bigskip

\subsection{Variance and concentration bounds}
The concept of the variance between two measures, and its monotonicity on a Ricci flow, has been the cornerstone for the heat kernel and entropy bounds in \cite{Bamler_HK_entropy_estimates}, as well as the compactness theory and definition of a metric flow in \cite{Bamler_RF_compactness}.
It also plays an important role in this paper.
We recall the following definition from \cite{Bamler_HK_entropy_estimates}:

\begin{Definition}[Variance]
If $\mu_1, \mu_2$ are two probability measures on a complete and separable metric space $(X,d)$, then the {\bf variance between $\mu_1$ (and $\mu_2$)} is denoted by
\[ \Var (\mu_1, \mu_2) = \int_X \int_X d^2(x_1, x_2) \, d\mu_1(x_1) d\mu_2 (x_2), 
\qquad \Var (\mu_1) = \Var (\mu_1, \mu_1). \]
If $(M, (g_t)_{t \in I})$ is a Ricci flow, then we will sometimes write $\Var_t$ for the variance at time $t$, in case there is a chance of confusion.
\end{Definition}

The following lemma is a restatement of \cite[\HKPropVarTriangle]{Bamler_HK_entropy_estimates} or \cite[\SYNVarIdentities]{Bamler_RF_compactness}.

\begin{Lemma}
If $\mu_1, \mu_2, \mu_3$ denote three probability measures on a complete and separable metric space $(X, d)$, then
\[ \sqrt{\Var(\mu_1, \mu_3)} \leq \sqrt{\Var(\mu_1, \mu_2)} + \sqrt{\Var(\mu_2, \mu_3)}, \]
\[ d_{W_1} (\mu_1, \mu_2) \leq \sqrt{\Var(\mu_1, \mu_2)} \leq d_{W_1} (\mu_1, \mu_2) + \sqrt{\Var(\mu_1)} + \sqrt{\Var(\mu_2)}. \]
\end{Lemma}
\medskip

For the remainder of this subsection consider a Ricci flow $(M, (g_t)_{t \in I})$ on a compact, $n$-dimensional manifold.
The following proposition summarizes  \cite[\HKVarMonotonicityCHFHnConcentration]{Bamler_HK_entropy_estimates}:

\begin{Proposition} \label{Prop_monotonicity_Var}
If $(M, (g_t)_{t \in I})$ is a Ricci flow on a compact manifold, $v_1, v_2 \in C^\infty(M \times I')$ are two solutions to the conjugate heat $\square^* v_i = 0$ equation over a subinterval $I' \subset I$ with $\int_M v_i(\cdot, t) \, dg_t = 1$, $i=1,2$, and $d\mu_{i,t} := v_i(\cdot, t) dg_t$ denote the corresponding families of probability measures, then in the barrier sense
\[ - \frac{d}{dt} \Var_t ( \mu_{1,t}, \mu_{2,t}) \leq H_n := \frac{(n-1)\pi^2}{2} + 4 . \]
Moreover, for any two points $x_1, x_2 \in M$ and any time $s,t \in I$, $s \leq t$, the conjugate heat kernels $\nu_{x_1, t; s}, \nu_{x_2, t; s}$ based at $(x_1, t)$, $(x_2, t)$ satisfy the following bound
\begin{equation} \label{eq_Var_difference}
 \Var_s (\nu_{x_1, t;s}, \nu_{x_2, t;s}) \leq d^2_{t} (x_1, x_2) + H_n (t-s), \qquad \Var_s (\nu_{x_1, t;s} ) \leq H_n (t-s). 
\end{equation}
\end{Proposition}

We will fix the constant $H_n =  \frac{(n-1)\pi^2}{2} + 4$ for the remainder of this paper; its exact value will never be important.

We also recall the notion of $H_n$-centers from \cite[\HKHCenterDef]{Bamler_HK_entropy_estimates}.

\begin{Definition}[$H_n$-center]
A point $(z,t) \in M \times I$ is called an {\bf $H_n$-center} of a point $(x_0,t_0) \in M \times I$ if $t \leq t_0$ and
\[ \Var_t (\delta_z, \nu_{x_0,t_0;t}) \leq H_n (t_0-t). \]
\end{Definition}

We recall the following two facts from \cite[\HKHCenterPropExistBoundBallHCenter]{Bamler_HK_entropy_estimates}, which will be used frequently throughout this paper.
Let in the following 

\begin{Lemma}
Given $(x_0, t_0) \in M$ and $t \leq t_0$ there is (at least) one point $z \in M$ such that $(z,t)$ is an $H_n$-center of $(x_0, t_0)$ and for any two such points $z_1, z_2 \in M$ we have $d_t (z_1, z_2) \leq 2 \sqrt{H_n (t_0 -t)}$.
\end{Lemma}
\medskip

\begin{Lemma} \label{Lem_mass_ball_Var}
If $(z,t)$ is an $H_n$-center of $(x_0, t_0)$, then for $A > 0$
\[ \nu_{x_0,t_0;t} \big(  B(z, t, \sqrt{A H_n (t_0 - t)} ) \big) \geq 1- \frac1{A}. \]
\end{Lemma}
\medskip

\subsection{Parabolic neighborhoods} \label{subsec_parabo_nbhd}
In the following, we recall the definition of conventional and $P^*$-parabolic neighborhoods from \cite{Bamler_HK_entropy_estimates,Bamler_RF_compactness}.
For simplicity, we will state the definitions for Ricci flows.
These definitions generalize to Ricci flow spacetimes and metric flows in a straight-forward way; see \cite{Bamler_RF_compactness} for more details.

Let in the following $(M, (g_t)_{t \in I})$ be a Ricci flow on a compact manifold.
Let $(x_0, t_0) \in M \times I$ and $A, T^-, T^+ \geq 0$.
Then the {\bf (conventional) parabolic neighborhood} is defined as
\[ P(x_0, t_0; A, -T^-, T^+) := B(x_0, t_0, A) \times \big( [t_0 - T^-, t_0 + T^+] \cap I \big), \]
where we may omit $-T^-$ or $T^+$ if it is zero.
Note that we have introduced the separator ``$;$'' in order to avoid confusion if we consider parabolic neighborhoods on a Ricci flow spacetime. 
We also define the {\bf (conventional) parabolic ball} around $(x_0, t_0)$ with radius $r > 0$ by
\[ P(x_0, t_0; r) := P(x_0, t_0; r, -r^2, r^2) \]
and the {\bf backward ($-$)} and {\bf forward ($+$) parabolic balls} by
\[ P^-(x_0, t_0; r) := P(x_0, t_0; r, -r^2), \qquad  P^+(x_0, t_0; r) := P(x_0, t_0; r, r^2).  \]

On the other hand, for any $(x_0, t_0) \in M \times I$ and $A, T^-, T^+ \geq 0$ we define the {$P^*$-parabolic neighborhood} (Compare with \cite[\HKDefsPstarPNBHD]{Bamler_HK_entropy_estimates})
\[ P^* (x_0, t_0; A, -T^-, T^+) \subset M \times I \]
as the set of points $(x, t) \in M \times I$ with $t \in [t_0-T^-, t_0+T^+]$ and 
\[ d^{g_{t_0 - T^-}}_{W_1} (\nu_{x_0, t_0; t_0 - T^-}, \nu_{x,t;  t_0 - T^-}) < A. \]
We also define the {\bf $P^*$-parabolic $r$-ball} by
\[ P^* (x_0, t_0; r) := P^* (x_0, t_0;  r, -r^2, r^2 ) \]
and the {\bf forward ($+$) and backward ($-$) $P^*$-parabolic $r$-balls} by
\[ P^{*+} (x_0, t_0; r) := P^* (x_0, t_0;  r, 0, r^2 ),
\qquad P^{*-} (x_0, t_0; r) := P^* (x_0, t_0; r, -r^2, 0 )
. \]

The following proposition, lists useful properties of $P^*$-parabolic neighborhoods.
These show that $P^*$-parabolic neighborhoods satisfy similar containment relationships as standard parabolic balls.
See \cite[\HKPropPstarContainments]{Bamler_HK_entropy_estimates} for more details and some further properties.

\begin{Proposition} \label{Prop_basic_parab_nbhd}
The following holds for any $(x_1, t_1), (x_2, t_2) \in M \times I$ as long as the corresponding $P^*$-parabolic neighborhoods or balls are defined:
\begin{enumerate}[label=(\alph*)]
\item \label{Prop_basic_parab_nbhd_a} For any $A \geq 0$ we have
\[ P^* (x_1, t_1; A,0,0) = B(x_1, t_1, A) \times \{ t_1 \}. \]
\item \label{Prop_basic_parab_nbhd_b} If $0 \leq A_1 \leq A_2$, $0 \leq T^\pm_1 \leq T^\pm_2$, then 
\[ P^{*} (x_1, t_1; A_1, -T_1^-, T_1^+) \subset P^{*} (x_1, t_1; A_2, -T_2^-, T_2^+). \]
\item \label{Prop_basic_parab_nbhd_d} If $r_1, r_2 > 0$ and $P^* (x_1, t_1; r_1) \cap P^* (x_2, t_2; r_2) \neq \emptyset$, then $P^* (x_1, t_1; r_1) \subset P^* (x_2, t_2; 2r_1+r_2)$.
\end{enumerate}
\end{Proposition}

Lastly, we refer to \cite[\HKSecPstar]{Bamler_HK_entropy_estimates} for further properties of $P^*$-parabolic neighborhoods, on which we will rely in this paper.
Among these are volume bounds of time-slices of $P^*$-parabolic neighborhoods, a covering lemma and results relating $P^*$-parabolic neighborhoods to conventional parabolic neighborhoods in the presence of curvature bounds.

\subsection{\texorpdfstring{Pointed Nash entropy and $\mu$-entropy}{Pointed Nash entropy and {\textbackslash}mu-entropy}} \label{subsec_NN_WW}
Next, we will recall the definition of the $\mu$-entropy, as defined by Perelman \cite{Perelman1}, and the pointed Nash entropy, as defined in \cite{Hein-Naber-14}.
Both notions are closely related and lower bounds on these quantities are generally used to express a certain kind of non-collapsedness of the flow.
In this paper, we will mainly work with the pointed Nash-entropy, as it is more local than the $\mu$-entropy and therefore allows us to state more general, local results.

As a guiding rule, a lower bound on the pointed Nash entropy of the form $\NN_{x,t} (r^2) \geq - Y$ can be compared to a lower volume bound on a distance ball of the form 
\begin{equation} \label{eq_rnBxrceY}
r^{-n} |B(x,r)| \geq c e^{-Y}
\end{equation}
 on a manifold $M$ ---  similarly for an upper bound.
Therefore, it can be viewed as a non-collapsedness condition near $(x,t)$ and at scale $r$.
By contrast, a lower bound on the $\mu$-entropy implies a lower bound on the pointed Nash entropy at any point, so it is comparable to a bound of the form (\ref{eq_rnBxrceY}) at \emph{every} point $x \in M$.
Upon first reading of this paper, it may be helpful to work under the more restrictive global non-collapsing condition in which a lower bound on the Nash entropy holds at \emph{any} point.

A lower bound on the pointed Nash entropy implies various analytic and geometric bounds, such as volume or heat kernel estimates.
These bounds were discussed in \cite{Bamler_HK_entropy_estimates} and will be used frequently throughout this paper.
They are comparable to similar bounds for spaces with lower Ricci curvature bounds if the pointed Nash entropy bound is replaced by the corresponding volume bound.

Let us now be more precise.
We first recall the necessary definitions.
See also \cite[\HKSecNashEntropy]{Bamler_HK_entropy_estimates} or \cite{Perelman1,Topping-book} for more details.
If $(M,g)$ is a Riemannian manifold and $\mu$ is a probability measure on $M$ with density  $d\nu = (4\pi \tau)^{-n/2} e^{-f} dg$, $f \in C^1(M)$, then we define
\[ \NN [g,f,\tau] = \int_M f d\nu - \frac{n}2, \qquad
\WW[g,f,\tau] = \int_M \big( \tau (|\nabla f|^2 + R) + f - n \big) d\nu. \]
The {$\mu$-entropy} is defined as
\[ \mu[g, \tau] := \inf_{\int_M (4\pi \tau)^{-n/2} e^{-f} dg = 1} \WW[g,f,\tau]. \]

Let now $(M, (g_t)_{t \in I})$ be a Ricci flow on compact manifold for the remainder of this subsection.

\begin{Definition}
Let $(x_0, t_0) \in M \times I$ and $\tau > 0$ such that $[t_0 - \tau, t_0] \subset I$.
Denote by $d\nu = (4\pi \tau)^{-n/2} e^{-f} dg$ the conjugate heat kernel based at $(x_0,t_0)$.
The {\bf pointed Nash-entropy at $(x_0, t_0)$} is defined as
\[ \NN_{x_0, t_0} (\tau) := \NN [g_{t_0 - \tau}, f_{t_0-\tau},\tau]. \]
We set $\NN_{x_0, t_0} (0) := 0$.
For $s < t_0$, $s \in I$, we also write
\[ \NN^*_s (x_0, t_0) := \NN_{x_0, t_0} (t_0 - s). \]
\end{Definition}

The following proposition, which is taken from \cite[\HKPropPropertiesNash]{Bamler_HK_entropy_estimates}, summarizes the basic (mostly well known) properties of the pointed Nash entropy.

\begin{Proposition} \label{Prop_NN_basic_properties}
Let $(x_0, t_0) \in M \times I$.
The expression $\NN_{x_0,t_0}(\tau)$ is continuous for $\tau \geq 0$.
Moreover, if $\tau > 0$ such that $[t_0 - \tau, t_0] \subset I$, $R(\cdot, t_0 - \tau) \geq R_{\min}$ and if $d\nu = (4\pi \tau)^{-n/2} e^{-f} dg$ denotes the conjugate heat kernel based at $(x_0,t_0)$, then
\begin{equation} \label{eq_NN_basic_1}
  \NN_{x_0, t_0} (0)  = 0, 
\end{equation}
\begin{equation} \label{eq_NN_basic_2}
 \frac{d}{d\tau} \big(\tau \NN_{x_0, t_0} (\tau)\big) = \WW[g_{t_0-\tau}, f_{t_0 - \tau}, \tau] \leq 0,
\end{equation}
\begin{equation} \label{eq_ddt_2_NN}
  \frac{d^2}{d\tau^2} \big( \tau \NN_{x_0, t_0} (\tau) \big) =  - 2\tau \int_M \bigg| \Ric + \nabla^2 f - \frac1{2\tau} g \bigg|^2 d\nu_{t_0 - \tau} \leq 0, 
\end{equation}
\begin{equation} \label{eq_NN_basic_3}
 - \frac{n}{2\tau} + R_{\min} \leq  \frac{d}{d\tau} \NN_{x_0, t_0} (\tau) \leq 0. 
\end{equation}
In particular, if $\tau_1 \leq \tau_2$ and $R \geq R_{\min}$ on $M \times [t_0 - \tau_2, t_0 - \tau_1]$, then
\begin{equation} \label{eq_NN_doubling}
 \NN_{x_0, t_0} (\tau_1) - \frac{n}2 \log \Big( \frac{\tau_2}{\tau_1} \Big( 1 - \frac2n R_{\min} (\tau_2 - \tau_1) \Big) \Big) \leq \NN_{x_0, t_0} (\tau_2) \leq \NN_{x_0, t_0} (\tau_1). 
\end{equation}
\end{Proposition}

Note that (\ref{eq_NN_basic_2}) implies that
\begin{equation} \label{eq_NN_geq_WW}
\NN_{x_0, t_0} (\tau) \geq \frac1{\tau} \int_0^{\tau} \WW [ g_{t_0 - \td\tau}, f_{t_0 - \td\tau}, \td\tau ] d\td\tau \geq \WW [ g_{t_0 - \tau}, f_{t_0 - \tau}, \tau ] \geq  \mu[g_{t- \tau}, \tau] 
\end{equation}
So by the monotonicity of the $\mu$-functional we can always deduce a lower bound on the pointed Nash-entropy that depends only on the initial data and the time:

\begin{Proposition}
Let $(M, (g_t)_{t \in [0,T)})$ be a Ricci flow on a compact manifold and assume that $\mu [ g_0 , 0 ] \geq  -Y$.
Then for any $(x_0,t_0) \in M \times [0, T)$ and $0 < \tau \leq t_0$ we have
\[ \NN_{x_0,t_0} (\tau) \geq - Y'(Y, T). \]
\end{Proposition}

So the main theorems of this paper apply if we consider blow-up sequences of finite-time singularities.

In connection with Proposition~\ref{Prop_NN_basic_properties}, we also recall the following lower bound on the scalar curvature:

\begin{Lemma} \label{Lem_lower_scal}
If $R (\cdot, t_0) \geq R_{\min}$, then for all $t \geq t_0$, $t \in I$, we have
\[ R (\cdot, t) \geq \frac{n}2 \frac{R_{\min}}{\frac{n}2 - R_{\min} (t - t_0)}. \]
Moreover, if $t_0 := \inf I \in [-\infty, \infty)$, then
\[ R(\cdot, t) \geq \begin{cases} - \frac{n}{2(t - t_0)} &\text{if $t_0 > -\infty$} \\ 0 & \text{if  $t_0 = -\infty$} \end{cases} \]
\end{Lemma}

Lastly, we recall a result that allows us to compare the pointed Nash-entropies between two different points, and possibly for different scales.
Comparing this with the geometry of spaces with lower Ricci curvature bounds, such a result would be comparable to the following bound:
\[ c(r_1, r_2, d(x_1,x_1))  r_1^{-n} |B(x_1,r_1)| \leq r_2^{-n} |B(x_2,r_2)| \leq C(r_1, r_2, d(x_1,x_1))  r_1^{-n} |B(x_1,r_1)|. \]

\begin{Proposition}\label{Prop_NN_variation_bound}
If $R (\cdot, t^*) \geq R_{\min}$ and $s < t^* \leq t_1, t_2$, then for $x_1, x_2 \in M$
\begin{equation} \label{eq_NNx1t1_NNx2t2}
 \NN^*_s(x_1, t_1) - \NN^*_s(x_2, t_2)  \leq \Big( \frac{n}{2(t^*-s)} -  R_{\min} \Big)^{1/2}   d_{W_1}^{g_{t^*}} (\nu_{x_1, t_1}(t^*), \nu_{x_2, t_2}(t^*))    + \frac{n}2 \log \Big( \frac{t_2-s}{t^*-s} \Big) . 
\end{equation}
Moreover if the $P^*$-parabolic neighborhood $P:=P^*(x_0,t_0;A, -T^-, T^+) \subset M \times I$ is defined, then we have the following bound for any $(x,t) \in P$ and $\tau_0, \tau > 0$, whenever the corresponding quantities are defined and assuming that we have $R( \cdot, t_0 - \tau_0), R(\cdot, t - \tau) \geq R_{\min}$,
\[ |\NN_{x,t}(\tau) - \NN_{x_0,t_0}(\tau_0)| \leq C \big( R_{\min}, A, T^-, T^+,  \tau_1, \tau_2 \big). \]
\end{Proposition}

\begin{proof}
See \cite[\HKVariationboundNN]{Bamler_HK_entropy_estimates}.
\end{proof}

\subsection{Soliton identities} \label{subsec_soliton_identities}
We will briefly recall the most common identities that hold on gradient shrinking solitons.
In the course of this paper, we will often encounter almost-versions of the following identities.

A gradient shrinking soliton is an ancient Ricci flow $(M, (g_t)_{t < 0})$ on a possibly non-compact manifold $M$ together with a potential function $f \in C^\infty ( M \times (-\infty, 0))$ such that if we set $\tau := - t$, then
\begin{equation} \label{eq_soliton_equation}
 \Ric + \nabla^2 f - \frac1{2\tau} g  =0, \qquad \partial_t f =  |\nabla f|^2.
\end{equation}
If the flow has bounded curvature and complete time-slices, then all time-slices $(M,g_t)$ and the potential function $f$ are homothetic, so the flow is given up to isometry by a single time-slice.
More specifically, if $g_{-1}, f_{-1}$ on $M$ are given such that the first identity in (\ref{eq_soliton_equation}) holds for $\tau = 1$, then $g_t = |t|\phi_t^* g_{-1}$ and $f_t = f_{-1} \circ \phi_t$, where the family of diffeomorphisms $(\phi_t : M \to M)_{t < 0}$ is generated by the flow of the vector field $\tau \nabla f$.
Note also that if we restrict (\ref{eq_soliton_equation}) to a single time-slice, then $f$ is determined up to a function with vanishing Hessian.
This necessitates the second equation in (\ref{eq_soliton_equation}).
See \cite[Chp~4]{Chow_book_Hamiltons}, \cite{Topping-book} for more details.

By tracing and applying the divergence to (\ref{eq_soliton_equation}), see \cite[Chp~4]{Chow_book_Hamiltons} for more details, we obtain the following useful identities
\begin{align} \label{eq_soliton_equation_traced}
 R + \triangle f - \frac{n}{2\tau} &= 0 , \\
- \tau (R + |\nabla f|^2) + f &\equiv const. \label{eq_soliton_equation_div}
\end{align}
Combining (\ref{eq_soliton_equation_traced}) with the second equation of (\ref{eq_soliton_equation}), we obtain
\[ -\partial_t f = \triangle f - |\nabla f|^2 + R - \frac{n}{2\tau}. \]
This is equivalent to the following conjugate heat equation
\[ \square^* (4\pi \tau)^{-n/2} e^{-f} = 0. \]
Since we may adjust $f$ by an additive constant, we will frequently use the convention that
\begin{equation} \label{eq_potential_normalization}
 \int_M (4\pi \tau)^{-n/2} e^{-f} dg_t = 1. 
\end{equation}
If $(M,g_t)$ does not have any non-constant functions of vanishing Hessian, then this normalization uniquely characterizes $f$ and we can omit the second equation in (\ref{eq_soliton_equation}).
Assuming the normalization (\ref{eq_potential_normalization}) and recalling (\ref{eq_soliton_equation_div}), we set
\[ W := - \tau ( R + |\nabla f|^2 ) + f , \]
which is constant and independent of time, because time-slices are homothetic.
Combining this with (\ref{eq_soliton_equation_traced}), (\ref{eq_soliton_equation_div}) implies the following identities:
\begin{align}
\tau (-|\nabla f|^2 + \triangle f ) + f - \frac{n}2 - W  &= 0, \label{eq_sol_id_1} \\
\square (\tau f) + \frac{n}2 + W &= 0,  \label{eq_sol_id_2} \\
\tau (2\triangle f - |\nabla f|^2 + R) + f - n - W &= 0.  \label{eq_sol_id_3}
\end{align}
Integrating (\ref{eq_sol_id_1}), (\ref{eq_sol_id_2}) and integration by parts, whenever this is permitted, implies that for all $\tau > 0$.
\[ \NN [g_{-\tau}, f_{-\tau}, \tau ] = \WW[g_{-\tau}, f_{-\tau}, \tau ] = W. \]

\subsection{Poincar\'e inequalities}
We will frequently use the following Poincar\'e inequalities from \cite{Hein-Naber-14,Bamler_HK_entropy_estimates}.

\begin{Proposition} \label{Prop_Poincare}
Let $(M, (g_t)_{t \in I})$ be a Ricci flow on a compact manifold and consider the conjugate heat kernel $d\nu = (4\pi \tau)^{-n/2} e^{-f} dg$ based at some point $(x_0,t_0) \in M \times I$.
If $t_0 - \tau \in I$ for some $\tau > 0$, then we have for any $h \in C^1(M)$ and $p \geq 1$
\[ \int_M h \, d\nu_{t_0-\tau} = 0 \qquad \Longrightarrow \qquad \int_M |h|^p d\nu_{t_0-\tau} \leq C(p) \tau^{p/2} \int_M |\nabla h|^p d\nu_{t_0-\tau} \]
and
\[ \int_M |h|^p d\nu_{t_0-\tau} \leq C(p) \tau^{p/2} \int_M |\nabla h|^p d\nu_{t_0-\tau} + C(p) \bigg(  \int_M h \, d\nu_{t_0-\tau} \bigg)^p. \]
We may choose $C(1) = \sqrt{\pi}$ and $C(2) = 2$.
\end{Proposition}

The case $p=2$ is due to Hein, Naber \cite{Hein-Naber-14}.

\subsection{Heat kernel bounds}
We will frequently use bounds on the heat kernel on a Ricci flow background from \cite{Bamler_HK_entropy_estimates}.
The following $L^\infty$-bound will be particularly important, see \cite[\HKLinftyHKbound]{Bamler_HK_entropy_estimates}:

\begin{Proposition}\label{Prop_L_infty_HK_bound}
Let $(M, (g_t)_{t \in I})$ be a Ricci flow on a compact manifold, $[s,t] \subset I$ and suppose that $R \geq R_{\min}$ on $M \times [s,t]$.
Then
\[ K(x,t;y,s) \leq \frac{C(R_{\min} (t-s))}{(t-s)^{n/2}} \exp ( - \NN_{x,t}(t-s) ). \]
\end{Proposition}

This bound will often take the following form.
Denote by $d\nu = (4\pi \tau)^{-n/2} e^{-f}$ the conjugate heat kernel based at some point $(x_0, t_0) \in M \times I$.
Let $\tau_0 > 0$ be such that $[t_0 - \tau_0, t_0] \subset I$ and $R(\cdot, t_0 - \tau_0) \geq R_{\min}$.
Then we have on $M \times [t_0 - \tau_0,t_0)$
\[ f \geq - C(R_{\min} \tau_0) + \NN_{x_0,t_0}(\tau_0). \]
\bigskip

\section{Integral almost properties} \label{sec_int_almost_prop}
In the course of this paper we will be using certain integral properties that characterize the local geometry of a Ricci flow.
These are related to the property that a flow is locally $\IF$-close to a metric flow $\XX$ that is a metric soliton, static or a Cartesian product of the form $\XX = \XX' \times \IR^k$.
In fact, these properties will eventually imply such an $\IF$-closeness statement.
The reason why we will be using these integral properties is that they will initially be more manageable from an analytical perspective.

In the following let $(M, (g_t)_{t \in I})$ be a Ricci flow on a compact manifold and $(x_0, t_0) \in M \times I$, $r > 0$, $0 < \eps < 1$.
We denote by $ d\nu_{x_0, t_0} = (4\pi \tau)^{-n/2} e^{-f} dg$ the conjugate heat kernel based at $(x_0, t_0)$.

The first definition expresses a form of closeness to a shrinking soliton near $(x_0, t_0)$.
It is helpful to recall the soliton identities from Subsection~\ref{subsec_soliton_identities}.

\begin{Definition}[Almost self-similarity] \label{Def_almost_self_similar}
The point $(x_0,t_0)$ is called {\bf $(\eps, r)$-selfsimilar} if $[t_0 - \eps^{-1} r^2, t_0 ] \subset I$ and if the following holds for $W := \NN_{x,t}(r^2)$:
\begin{alignat}{2}
 \int_{t_0 - \eps^{-1} r^2}^{t_0 - \eps r^2} \int_M \tau \Big| \Ric + \nabla^2 f - \frac1{2\tau} g \Big|^2 d\nu_{x_0, t_0; t} dt &\leq \eps, && \label{eq_almost_self_similar_1} \\
 \int_M \big| \tau (2\triangle f - |\nabla f|^2 + R) + f - n - W \big| d\nu_{x_0, t_0; t
} &\leq \eps \quad && \text{for all} \quad t \in [t_0 - \eps^{-1} r^2, t_0 - \eps r^2], \label{eq_almost_self_similar_2} \\
 R &\geq -\eps r^{-2} \quad && \text{on} \quad M \times [t_0 - \eps^{-1} r^2, t_0 - \eps r^2]. \label{eq_almost_self_similar_3}
\end{alignat}
\end{Definition}

The next definition expresses a form of closeness to a static, Ricci flat flow.

\begin{Definition}[Almost static] \label{Def_static}
The point $(x_0,t_0)$ is called {\bf $(\eps, r)$-static} if $[t_0 - \eps^{-1} r^2, t_0 ] \subset I$ and if the following holds:
\begin{alignat*}{2}
 r^2 \int_{t_0 - \eps^{-1} r^2}^{t_0 - \eps r^2} \int_M   |{\Ric}|^2  d\nu_{x_0, t_0; t} dt &\leq \eps, && \\
 \int_M R  \, d\nu_{x_0, t_0;t} &\leq \eps \qquad &&\text{for all} \quad t \in [t_0 - \eps^{-1} r^2, t_0 - \eps r^2], \\
 R &\geq -\eps r^{-2} \qquad &&\text{on} \quad M \times [t_0 - \eps^{-1} r^2, t_0 - \eps r^2]. 
 \end{alignat*}
\end{Definition}

Next, we introduce a form of closeness to a metric flow of the form $\XX = \XX' \times \IR^k$.
More specifically, we will characterize vector-valued functions on the flow that resemble coordinate functions for the last factor.
These will be called splitting maps, as they are comparable to splitting maps in the analysis of spaces with lower Ricci curvature bounds, see for example \cite{Cheeger-Naber-Codim4}.
We will need two versions of this notion, called \emph{weak} and \emph{strong} splitting maps.

\begin{Definition}[Weak splitting map] \label{Def_weak_splitting_map}
Suppose that $[t_0 - \eps^{-1} r^2, t_0 ] \subset I$.
An {\bf weak $(k,  \eps, r)$-splitting map at $(x_0, t_0)$} is a smooth vector-valued function $\vec y = (y_1, \ldots, y_k) : M \times [t_0 - \eps^{-1} r^2, t_0 - \eps r^2] \to \IR^k$ with the following properties for all $i,j = 1, \ldots, k$:
\begin{enumerate}[label=(\arabic*)]
\item \label{Def_weak_splitting_map_1} We have
\[  r^{-1} \int_{t_0 - \eps^{-1} r^2}^{t_0 - \eps r^2} \int_M |\square y_i  | d\nu_{x_0, t_0;t} dt \leq \eps. \]
\item \label{Def_weak_splitting_map_2} We have
\[  r^{-2} \int_{t_0 - \eps^{-1} r^2}^{t_0 - \eps r^2} \int_M |\nabla y_i \cdot \nabla y_j - \delta_{ij} | d\nu_{x_0, t_0;t} dt \leq \eps. \]
\end{enumerate} 
If there exists a weak $(k,  \eps, r)$-splitting map at $(x_0, t_0)$, then we say that $(x_0, t_0)$ is {\bf weakly $(k,  \eps, r)$-split.}
If $k \leq 0$, then we define the property of being weakly $(k,  \eps, r)$-split to always hold.
\end{Definition}
\medskip

\begin{Definition}[Strong splitting map] \label{Def_strong_splitting_map}
Suppose that $[t_0 - \eps^{-1} r^2, t_0 ] \subset I$.
A {\bf strong $(k,\eps, r)$-splitting map at $(x_0, t_0)$} is a vector-valued function $\vec y = (y_1, \ldots, y_k) : M \times [t_0 - \eps^{-1} r^2, t_0 - \eps r^2] \to \IR^k$ with the following properties for all $i,j = 1, \ldots, k$:
\begin{enumerate}[label=(\arabic*)]
\item \label{Def_strong_splitting_map_1} $y_i$ solves the heat equation $\square y_i = 0$ on $M \times [t_0 - \eps^{-1} r^2, t_0 - \eps r^2]$.
\item \label{Def_strong_splitting_map_2} We have
\[  r^{-2} \int_{t_0 - \eps^{-1} r^2}^{t_0 - \eps r^2} \int_M |\nabla y_i \cdot \nabla y_j - \delta_{ij} | d\nu_{x_0, t_0;t} dt \leq \eps. \]
\item \label{Def_strong_splitting_map_3} For all $t \in [t_0 - \eps^{-1} r^2, t_0 - \eps r^2]$ we have
\begin{equation} \label{eq_def_splitting_average_0}
 \int_M y_i \, d\nu_{x_0 t_0; t} = 0. 
\end{equation}
\end{enumerate} 
If there exists a strong $(k,  \eps, r)$-splitting map at $(x_0, t_0)$, then we say that $(x_0, t_0)$ is {\bf strongly $(k,  \eps, r)$-split.}
If $k \leq 0$, then we define the property of being strongly $(k,  \eps, r)$-split to always hold.
\end{Definition}

Note that, at first sight, Definition~\ref{Def_strong_splitting_map} seems to be weaker than the corresponding definition for spaces with lower Ricci curvature bounds, because it lacks a pointwise bound of the form $|\nabla y_i| \leq 1 + \eps$ and an $L^2$-bound on the Hessians $\nabla^2 y_i$.
The reason is that Definition~\ref{Def_strong_splitting_map} only lists a minimal set of properties of a splitting map.
As we will see in Section~\ref{sec_improving_splitting_maps}, these properties imply --- after adjusting $\eps$ --- stronger useful analytical bounds, such as higher $L^p$-bounds on $y_i$ and an $L^2$-bound on $\nabla^2 y_i$.
However, we will not obtain a global \emph{pointwise} bound on the gradient, as splitting maps are defined on the entire manifold and may behave poorly far away from $(x_0, t_0)$.
The lack of such a bound will turn out to be irrelevant, however.
We also refer to Proposition~\ref{Prop_weak_splitting_map_to_splitting_map}, which provides a somewhat unconventional pointwise gradient bound.

The difference between Definitions~\ref{Def_weak_splitting_map}, \ref{Def_strong_splitting_map} lies essentially in Property~(1).
Property~\ref{Def_strong_splitting_map_3} has been introduced for convenience.
Observe that we have the following result:

\begin{Lemma} \label{Lem_prop_3_strong_splitting}
Let $(M, (g_t)_{t \in I})$ be a Ricci flow on a compact manifold and suppose that $\vec y : M \times [t_0 - \eps^{-1} r^2, t_0 - \eps r^2] \to \IR^k$ is a smooth vector-valued function satisfying Properties~\ref{Def_strong_splitting_map_1}, \ref{Def_strong_splitting_map_2} of Definition~\ref{Def_strong_splitting_map}.

Then there is a vector $\vec a := (a_1, \ldots, a_k) \in \IR^k$ such that $\vec y - \vec a$ is a strong splitting map.
Moreover, if $\vec y$ is the restriction of a map $\vec y^{\, \prime} : M \times [t_0 - \eps^{-1} r^2, t_0] \to \IR^k$ whose components satisfy the heat equation $\square y'_i = 0$, then $\vec a = \vec y^{\,\prime} (x_0, t_0)$.

Lastly, if (\ref{eq_def_splitting_average_0}) in Property~\ref{Def_strong_splitting_map_3} holds for some $t \in [t_0 - \eps^{-1} r^2, t_0 - \eps r^2]$, then it holds for all such $t$.
\end{Lemma}

\begin{proof}
The lemma follows from the fact that if $\square y_i = 0$, then
\[ \frac{d}{dt}  \int_M y_i \, d\nu_{x_0 t_0; t} =  \int_M \square y_i \, d\nu_{x_0 t_0; t}  = 0. \qedhere \]
\end{proof}
\medskip

We will frequently work with splitting maps $\vec y$ that arise as the restriction of some map $\vec y^{\, \prime} : M \times I' \to \IR^k$ to $M \times [t_0 - \eps^{-1} r^2, t_0 - \eps r^2]$.
In this case, we will often also say that $\vec y^{\, \prime}$ a splitting map at $(x_0, t_0)$, omitting the restriction to $M \times [t_0 - \eps^{-1} r^2, t_0 - \eps r^2]$.

Let us comment on the different uses of weak and strong splitting maps throughout this paper.
We will first derive splitting theorems that allow us to prove a quantitative stratification result involving \emph{weak} splitting maps.
These maps will arise from linear combinations of potentials of the conjugate heat equation.
Subsequently, in Section~\ref{sec_improving_splitting_maps} we will see that the existence of a weak splitting map in fact implies the existence of a \emph{strong} splitting map, after adjusting $\eps$.
The advantage of strong splitting maps is, among other things, that they enable us to derive useful local regularity theorems.

Lastly, we discuss the behavior of the notions introduced in this section under parabolic rescaling.
If a point $(x_0, t_0)$ is $(\eps, r)$-selfsimilar, $(\eps, r)$-static, or weakly/strongy $(k, \eps, r)$-split, then after parabolically rescaling the flow by some constant $\lambda$, the corresponding point is $(\eps, \lambda r)$-selfsimilar, $(\eps, \lambda r)$-static, or weakly/strongy $(k, \eps, \lambda r)$-split, respectively.
Moreover, if $\vec y  : M \times [t_0 - \eps^{-1} r^2, t_0 - \eps r^2] \to \IR^k$ is a weak/strong $(k, \eps, r)$-splitting map on the original flow, then the map 
\[ M \times [\lambda^2 t_0 - \eps^{-1} (\lambda r)^2, \lambda^2 t_0 - \eps(\lambda r)^2] \longrightarrow \IR^k, \qquad (x, t) \longmapsto \lambda \vec y (x, \lambda^{-2} t)  \]
is a weak/strong $(k, \eps, \lambda r)$-splitting map on the parabolically rescaled flow.
In the beginning of most proofs in this paper, we will apply parabolic rescaling, and possibly a time-shift, in order to normalize a certain scale.
When doing so, will always assume that $\vec y$ transforms as described here.

\bigskip

\part{Preliminary quantitative stratification} \label{part_prel_quant_strat}
The goal of the part will be to obtain a preliminary quantitative stratification result for Ricci flows.
This result will state that a Ricci flow that satisfies a non-collapsedness bound involving the pointe Nash entropy is $(k, \eps, r)$-split on the complement of a bounded number of $P^*$-parabolic balls of radius $r$.
In the subsequent part, this bound will allow us to conclude a partial regularity theory for $\IF$-limits.

In contrast to other quantitative stratification results, we will not be able to use any geometric limit arguments in this part, because the theory of metric flows and $\IF$-convergence is still too weak at this point.
Instead, we will carry out all our argument in the setting of a single (smooth) Ricci flow.
Further complications arise from the nature of Ricci flows and the lack of distance distortion and lower heat kernel bounds.

Let us now provide a brief overview over the following sections.
In Section~\ref{sec_improved_Lp} we derive improved integral bounds on geometric quantities on a Ricci flow.
These are similar to known bounds, but contain a crucial $e^{\alpha f}$ term; compare with (\ref{eq_improved_Lp_outline_1}), (\ref{eq_improved_Lp_outline_2}) in the outline.
In Section~\ref{sec_geom_bounds_almost_ss} we characterize the geometry of Ricci flows near almost selfsimilar points.
We will see that the almost selfsimilarity condition is essentially equivalent to almost constancy of the pointed Nash entropy and we will show that several geometric quantities, which vanish on a gradient shrinking soliton, are small in an integral sense.
We will also derive an almost monotonicity result for the integral of the scalar curvature; compare with (\ref{eq_almost_monotone_int_R_outline}) in the outline.
In Section~\ref{sec_comparing_CHK} we compare two conjugate heat kernel measures, based at different points, and derive a bound of the form (\ref{eq_compare_nu_outline}) from the outline.
In Section~\ref{sec_dist_expansion} we establish a distance expansion bound near sufficiently selfsimilar points.
In Section~\ref{sec_alm_splitting} we prove almost splitting theorems.
These are similar to almost cone splitting theorems in the case of spaces with lower Ricci curvature bounds.
The main difficulty here is, however, that these bounds have to be obtained without resorting to a limit argument.
Lastly, in Section~\ref{sec_prelim_q_strat} we combine our results obtained so far and derive the desired preliminary quantitative stratification result.

\section{\texorpdfstring{Improved $L^p$-bounds on geometric quantities}{Improved L\^{}p-bounds on geometric quantities}} \label{sec_improved_Lp}
In this section we derive new $L^p$-bounds on geometric quantities that only depend on a non-collapsing bound of the form $\NN_{x_0, t_0} (r^2) \geq - Y$.
The most important and novel aspect of these bounds is an additional weight of the form $e^{2\alpha f}$.
This weight will later enable us to compare integral quantities involving  conjugate heat kernels based at different points.

Let us first discuss some of the more elementary $L^p$-bounds, which form the basis of the main result, but lack the $e^{2\alpha f}$ weight.
The proofs of the main results of this section will borrow from the following ideas. 
Let $(M, (g_t)_{t \in (-T, 0]})$ be a Ricci flow on a compact manifold and let $x \in M$ be a point.
Denote by $d\nu := d\nu_{x,0} = (4\pi \tau)^{-n/2} e^{-f} dg$ the conjugate heat kernel based at $(x, 0)$.
Assume that we have a non-collapsing bound $\NN_{x, 0} (\tau) \geq - Y$ and a lower scalar curvature bound of the form $R \geq R_{\min}$ on $M \times (-T, 0]$ for some $R_{\min} \leq 0$.
Note that by Lemma~\ref{Lem_lower_scal} the second bound can always be ensured for a suitable $R_{\min }$ if we restrict the flow to a smaller sub-interval of the form $(-T', 0] \subset (-T, 0]$.

By the discussion in Subsection~\ref{subsec_NN_WW} we have for any $\tau_1 \in (0, T)$
\[\tau_1 \int_M (|\nabla f|^2 + R) d\nu_{- \tau_1} 
= \WW [ g_{-\tau_1}, f(\cdot, -\tau_1), \tau_1] - \NN_{x_0, 0}(\tau_1) + \frac{n}2 \leq Y + \frac{n}2. \]
So
\[  \tau_1 \int_M R \, d\nu_{- \tau_1} \leq  C(Y ), \qquad 
\tau_1 \int_M |\nabla f|^2 \, d\nu_{- \tau_1} \leq C(Y , R_{\min}). \]
Due to the evolution equation for the scalar curvature, $\square R = 2 |{\Ric}|^2$, we obtain the following $L^2$-bound on the Ricci tensor for any $0 < \tau_0 \leq \tau_1 < T$; see (\ref{eq_dds_int_u_int_square_u})
\begin{equation} \label{eq_L2_Ric_basic}
 2 \int_{ - \tau_1}^{ - \tau_0} \int_M |{\Ric}|^2 d\nu_t dt
= \int_{ - \tau_1}^{ - \tau_0} \frac{d}{dt} \int_M  R \, d\nu_t dt
= \int_M R \, d\nu_t \bigg|_{t =  - \tau_1}^{t =  - \tau_0}
\leq C(Y, R_{\min}, \tau_0). 
\end{equation}
Lastly, we will derive an $L^2$-bound on the Hessian $\nabla^2 f$.
For this purpose, we use (\ref{eq_ddt_2_NN})
\[ 2 \int_{ - \tau_1}^{0} \tau \int_M \bigg| \Ric + \nabla^2 f - \frac1{2 \td\tau} g \bigg|^2 d\nu_{t} dt = \WW [g, f, \tau] \bigg|_{t =  - \tau_1}^{t = 0}
\leq  -\WW [g_{-\tau_1}, f(\cdot, -\tau_1), \tau_1 ]. \]
Assuming $\tau_1 \in (-\frac12 T, 0]$, we find using (\ref{eq_ddt_2_NN}), (\ref{eq_NN_basic_2})
\begin{multline*}
 -\WW [g_{-\tau_1}, f(\cdot, -\tau_1), \tau_1 ] 
\leq - \frac{1}{\tau_1} \int_{\tau_1}^{2\tau_1} \WW [g_{-\td\tau}, f(\cdot, -\td\tau), \td\tau ] \, d\td\tau \\
\leq  - \frac{1}{\tau_1} \int_{0}^{2\tau_1} \WW [g_{-\td\tau}, f_{ -\td\tau}, \td\tau ] \, d\td\tau
=  - \frac{1}{\tau_1} \big( 2\tau_1 \NN_{x,0} (2\tau_1) \big) \leq 2 Y. 
\end{multline*}
So
\[ \int_{ - \tau_1}^{0} \tau \int_M \bigg| \Ric + \nabla^2 f - \frac1{2 \td\tau} g \bigg|^2 d\nu_{t} dt \leq Y. \]
Combining this with (\ref{eq_L2_Ric_basic}) implies
\[ \int_{ - \tau_1}^{ - \tau_0} \int_M |\nabla^2 f |^2 d\nu_tdt  \leq C(Y, R_{\min}, \tau_0). \]

In the following proposition we establish similar $L^p$-bounds, plus an $L^4$-bound on $\nabla f$, with an additional weight of the form $e^{2\alpha f}$.

\begin{Proposition} \label{Prop_improved_L2}
There is a purely dimensional constant $\ov\alpha > 0$ such that the following holds if $\alpha \in [0, \ov\alpha]$.
Let $(M, (g_t)_{t \in [-2r^2,0]})$, $r> 0$, be a Ricci flow on a compact manifold and denote by $d\nu = (4\pi \tau)^{-n/2} e^{-f} dg$ the conjugate heat kernel based at $(x,0)$ for some $x \in M$.
Assume that $\NN_{x,0} (2r^2) \geq - Y$.
Then for any $0<\theta \leq \frac12$
\begin{align} \label{eq_improved_L2_bound_1}
 \int_{-r^2}^{-\theta r^2} \int_M \big( \tau |{\Ric}|^2 + \tau |\nabla^2 f|^2 +  |\nabla f|^2 + \tau |\nabla f|^4 + \tau^{-1} e^{\alpha f} + \tau^{-1} \big) e^{2\alpha f} d\nu_t dt &\leq C (Y) |\log \theta|,  \\
  \int_M \big( \tau |R| + \tau |\triangle f| + \tau |\nabla f|^2 + e^{\alpha f}  + 1 \big) e^{2\alpha f} d\nu_{-r^2}  &\leq C (Y) .  \label{eq_improved_L2_bound_2}
\end{align}
\end{Proposition}

We remark that due to Proposition~\ref{Prop_L_infty_HK_bound} we have $f \geq - C(Y)$ on $M \times [-r^2, 0)$.
So an integral bound on $e^{\alpha f}$ implies an integral bound on $|f|^p$ for any $p < \infty$.

We also have the following (surprisingly general) result.

\begin{Proposition} \label{Prop_int_ealphaf}
Let $(M, (g_t)_{t \in (-T,0]})$ be a Ricci flow on a compact manifold and denote by $d\nu = (4\pi \tau)^{-n/2} e^{-f} dg$ the conjugate heat kernel based at $(x,0)$ for some $x \in M$.
Then for all $t \in (-T, 0)$, $\alpha \in [0, \frac12]$ we have
\[ \int_M e^{\alpha f} d\nu_t \leq \exp \big( (n - \tau \min_M  R (\cdot, t)) \alpha \big). \]
\end{Proposition}

\begin{proof}
Fix some $t \in (-T, 0)$ and let $R_0 := \max_M R_- (\cdot, t)$.
Recall that by \cite[Sec.~9]{Perelman1} we have
\[ \tau (2\triangle f - |\nabla f|^2 + R ) + f -n \leq 0. \]
Therefore at time $t$ and for $\alpha \in [0, \frac12]$
\begin{align*}
 \frac{d}{d\alpha} \int_M e^{\alpha f} d\nu_t  
&= \int_M f e^{\alpha f} d\nu_t
\leq  \int_M \big( \tau (-2\triangle f + |\nabla f|^2 - R ) + n \big) e^{\alpha f} d\nu_t \\
&\leq  \int_M \big(  \tau (-2\triangle f + |\nabla f|^2  ) + n - \tau R_0  \big)  
(4\pi \tau)^{-n/2} e^{-(1-\alpha) f} dg_t \\
&\leq  \int_M \big(  \tau (-2(1-\alpha) |\nabla f|^2 + |\nabla f|^2  ) + n - \tau R_0  \big)  
(4\pi \tau)^{-n/2} e^{-(1-\alpha) f} dg_t  \\
&\leq (n-\tau R_0)  \int_M e^{\alpha f} d\nu_t .
\end{align*}
Integrating this differential inequality over $\alpha$ and using $\int_M d\nu_1 = 1$ implies the proposition.
\end{proof}
\bigskip

\begin{proof}[Proof of Proposition~\ref{Prop_improved_L2}.]
We will show (\ref{eq_improved_L2_bound_1}), (\ref{eq_improved_L2_bound_2}) with $2\alpha$ replaced by $\alpha$.
We will determine the bound $\ov\alpha$ in the course of the proof.
By parabolic rescaling, we may assume that $r = 1$. 
By Lemma~\ref{Lem_lower_scal} we have $R \geq - C$ on $M \times [-1.9,0]$.

Set
\begin{equation} \label{eq_u_w_identities}
  u := (4\pi \tau)^{-n/2} e^{-f}, \qquad w := \tau (2\triangle f - |\nabla f|^2 + R ) + f -n,  
\end{equation}
and recall that by \cite[Sec.~9]{Perelman1} we have
\[ \square^* (wu) = - 2 \tau \Big| \Ric + \nabla^2 f - \frac1{2\tau} g \Big|^2 u, \qquad w \leq 0. \]
We will also use the fact that
\[ \square f = \partial_t f - \triangle f = -2\triangle f + |\nabla f|^2 - R + \frac{n}{2\tau} = - \tau^{-1} w +  \tau^{-1} f- \frac{n}{2\tau}. \]
We can now compute
\begin{align}
-\frac{d}{dt} \int_M &w e^{\alpha f} d\nu_t
= -\frac{d}{dt} \int_M  e^{\alpha f}w u \, dg_t
= \int_M \big(- ( \square e^{\alpha f}) w u + e^{\alpha f} \square^* (wu) \big) dg_t \notag \\
&=  \int_M \Big( -(\alpha  \square  f e^{\alpha f} -  \alpha^2 |\nabla f|^2 e^{\alpha f} ) w u - 2\tau e^{\alpha f} \Big| \Ric + \nabla^2 f - \frac1{2\tau} g \Big|^2 u \Big) dg_t \notag \\
&=  \int_M \Big( - \Big(- \alpha \tau^{-1} w  - \frac{\alpha n}{2 \tau} + \alpha \tau^{-1} f   - \alpha^2 |\nabla f|^2 \Big) w - 2\tau  \Big|\Ric + \nabla^2 f - \frac1{2\tau} g \Big|^2 \Big) e^{\alpha f} d\nu_t \notag \\
&\leq  \int_M \Big(  \alpha \tau^{-1} \Big( w^2 +  \frac{n}2  w -  f  w \Big)  - 2\tau  \Big|\Ric + \nabla^2 f - \frac1{2\tau} g \Big|^2 \Big) e^{\alpha f} d\nu_t \notag \\
&\leq  \int_M \Big( C  \alpha \tau^{-1} \Big( w^2 + f^2 + 1  \Big)  - 2\tau  \Big|\Ric + \nabla^2 f - \frac1{2\tau} g \Big|^2 \Big) e^{\alpha f} d\nu_t \notag \\
&\leq  \int_M \Big( C  \alpha  \Big( \tau ((\triangle f)^2 + |\nabla f|^4 + R^2) + \tau^{-1} (f^2 + 1)  \Big)  - 2\tau  \Big|\Ric + \nabla^2 f - \frac1{2\tau} g \Big|^2 \Big) e^{\alpha f} d\nu_t. \label{eq_int_M_ealph_computation}
\end{align}
and
\begin{align}
-\frac{d}{dt} \int_M &\tau R e^{\alpha f} d\nu_t
= - \int_M \square (\tau R e^{\alpha f}) d\nu_t
= - \int_M \big( ( \square (\tau R) ) e^{\alpha f} + \tau R \square e^{\alpha f} - 2 \tau  ( \nabla R \cdot \nabla  e^{\alpha f}) \big) d\nu_t \notag \\
&= \int_M \big( R e^{\alpha f} - 2\tau |{\Ric}|^2 e^{\alpha f} - \tau R (\alpha  \square  f - \alpha^2 |\nabla f|^2)  e^{\alpha f} - 2  \tau R \triangle e^{\alpha f} + 2 \tau R \nabla f \cdot \nabla e^{\alpha f} \big)  d\nu_t \notag \\
&= \int_M \Big( R  - 2\tau |{\Ric}|^2  - \tau R \Big(\alpha  \Big(- \tau^{-1} w + \tau^{-1} f - \frac{n}{2\tau} \Big) - \alpha^2 |\nabla f|^2 \Big) \notag \\
&\qquad\qquad\qquad\qquad\qquad\qquad\qquad - 2  \alpha \tau R \triangle f  - \alpha^2 \tau R |\nabla f|^2  + 2 \alpha \tau R |\nabla f|^2 \Big) e^{\alpha f}  d\nu_t \notag \\
&\leq \int_M \big( C \alpha \big( \tau^{-1} w^2 + \tau (\triangle f)^2 + \tau |\nabla f|^4 + \tau R^2 + \tau^{-1} (f^2 + 1) \big) + R - 2 \tau |{\Ric}|^2 \big) e^{\alpha f}  d\nu_t \notag \\
&\leq \int_M \Big( C  \alpha  \Big( \tau ((\triangle f)^2 + |\nabla f|^4 + R^2) + \tau^{-1} (f^2 + 1)  \Big) + R  - 2\tau |{\Ric}|^2 \Big) e^{\alpha f} d\nu_t. \label{eq_int_tauR_computation}
\end{align}
Since
\begin{align*}
(1-\alpha) \int_M |\nabla f|^4 e^{\alpha f} d\nu_t 
&=(1-\alpha) (4\pi \tau)^{-n/2} \int_M |\nabla f|^2 \nabla f \cdot \nabla f e^{-(1-\alpha ) f} dg_t \\
&= (4\pi \tau)^{-n/2} \int_M \big(  2 \nabla^2 f \cdot (\nabla f \otimes \nabla f) + |\nabla f|^2 \triangle f \big) e^{-(1-\alpha ) f} dg_t \\
&\leq C (4\pi \tau)^{-n/2} \int_M | \nabla^2 f | |\nabla f|^2  e^{-(1-\alpha ) f} dg_t \\
&\leq C  \int_M | \nabla^2 f |^2  e^{\alpha f} d\nu_t +\frac1{2}  \int_M  |\nabla f|^4  e^{\alpha f} d\nu_t,
\end{align*}
we obtain that for $\alpha \leq \frac14$
\begin{equation} \label{eq_nabf4_nab2f2}
 \int_M |\nabla f|^4 e^{\alpha f} d\nu_t \leq C  \int_M | \nabla^2 f |^2  e^{\alpha f} d\nu_t. 
\end{equation}
Moreover, by Proposition~\ref{Prop_L_infty_HK_bound} and (\ref{eq_u_w_identities}) we have
\[ - C(Y) \leq f = w - \tau (2 \triangle f - |\nabla f|^2 + R) +n\leq   - \tau (2 \triangle f - |\nabla f|^2 + R) +n, \]
which implies that
\begin{equation} \label{eq_f_square_bound}
 f^2 \leq C(Y) + C \tau^2 ( |\nabla^2 f|^2 + |\nabla f|^4 + R^2) 
\end{equation}
Combining this with (\ref{eq_int_M_ealph_computation}), (\ref{eq_int_tauR_computation}), (\ref{eq_nabf4_nab2f2}) and Proposition~\ref{Prop_int_ealphaf}, yields that for $\alpha \leq \ov\alpha$
\begin{align*}
-\frac{d}{dt}& \int_M ( w + \tau R) e^{\alpha f} d\nu_t \\
&\leq \int_M \Big( C \alpha \big( \tau |\nabla^2 f|^2 +\tau R^2  \big) + C(Y) \alpha \tau^{-1}  - 2\tau  \Big|\Ric + \nabla^2 f - \frac1{2\tau} g \Big|^2 - 2\tau |{\Ric}|^2 + R \Big) e^{\alpha f}  d\nu_t \\
&\leq \int_M \Big( C \alpha \big( \tau |\nabla^2 f|^2 +\tau R^2  \big) + C(Y) \alpha \tau^{-1}  - \frac{\tau}2  \Big| \nabla^2 f - \frac1{2\tau} g \Big|^2 - \tau |{\Ric}|^2 + R \Big) e^{\alpha f}  d\nu_t \\
&\leq \int_M \Big( C \alpha \big( \tau |\nabla^2 f|^2 +\tau R^2  \big) + C(Y) \alpha \tau^{-1}  - \frac{\tau}4  | \nabla^2 f |^2 - \frac{\tau}2 |{\Ric}|^2 + C \tau^{-1} \Big) e^{\alpha f}  d\nu_t \\
&\leq C(Y) \tau^{-1}  - \frac{\tau}8  \int_M \big(  |\nabla^2 f |^2 +  |{\Ric}|^2 \big) e^{\alpha f}  d\nu_t.
\end{align*}
Let now $\eta : [-2, -\theta/2] \to [0,1]$ be a cutoff function with compact support, $\eta \equiv 0$ on $[-2,-1.9]$, $\eta \equiv 1$ on $[-1.1, -\theta]$ and $|\eta'| \leq 10 \tau^{-1}$.
Then
\begin{align}
 \int_{-2}^{-\theta/2} \eta^2(t)  \int_M & \tau \big(  |\nabla^2 f |^2 +  |{\Ric}|^2 \big) e^{\alpha f}  d\nu_t dt \notag \\
& \leq C(Y) \int_{-2}^{-\theta/2} \eta^2(t) \tau^{-1} dt + 8 \int_{-2}^{-\theta/2} \eta^2(t) \frac{d}{dt} \int_M ( w + \tau R) e^{\alpha f} d\nu_t  dt \notag \\
 &\leq C(Y) |{\log \theta}| - 16 \int_{-2}^{-\theta/2} \eta' (t) \eta(t)  \int_M ( w + \tau R) e^{\alpha f} d\nu_t  dt . \label{eq_nab2f_Ric_ealphf}
 \end{align}
For any $b > 0$ and $t \in [-1.9,-1] \cup [-\theta, -\theta/2]$, we obtain using (\ref{eq_f_square_bound}), (\ref{eq_nabf4_nab2f2}) and Proposition~\ref{Prop_int_ealphaf} for $\alpha \leq \ov\alpha$
\begin{align}
  \int_M  |  \tau R | e^{\alpha f} d\nu_t  
 & \leq \int_M ( b(\tau R)^2 + b^{-1} ) e^{\alpha f} d\nu_t  
  \leq Cb\tau \int_M \tau |{\Ric}|^2 e^{\alpha f} d\nu_t   + C b^{-1}, \label{eq_tauR_ts} \\
  \int_M  | w | e^{\alpha f} d\nu_t   
&\leq  \int_M (b w^2 + b^{-1} )e^{\alpha f} d\nu_t \notag  \\
&\leq C b   \tau \int_M \big( \tau ( |\nabla^2 f|^2 + |\nabla f|^4 + R^2 )+  C(Y) \tau^{-1} \big) e^{\alpha f} d\nu_t   + C b^{-1}   \int_M    e^{\alpha f} d\nu_t  \notag \\
&\leq C b \tau   \int_M  \tau \big( |\nabla^2 f|^2 + |{\Ric}|^2 \big) e^{\alpha f} d\nu_t  dt +  C(Y) (b + b^{-1})  .  \label{eq_w_ts}
\end{align}
Combining this with (\ref{eq_nab2f_Ric_ealphf}) for $b \leq \ov{b}$ yields that
\[  \int_{-1.1}^{-\theta}  \int_M \tau \big(  |\nabla^2 f |^2 +  |{\Ric}|^2 \big) e^{\alpha f}  d\nu_t dt \leq  C(Y) |{\log \theta}| + C(Y) \leq C(Y) |{\log \theta}|. \]
Thus by (\ref{eq_nabf4_nab2f2}) we have
\[  \int_{-1.1}^{-\theta}  \int_M \tau |\nabla f|^4 e^{\alpha f}  d\nu_t dt  \leq C(Y) |{\log \theta}|. \]
Using Proposition~\ref{Prop_int_ealphaf} again, we obtain for $\alpha \leq \ov\alpha$
\begin{multline*}
  \int_{-1.1}^{-\theta}  \int_M  \big(  |\nabla f |^2 + \tau^{-1} e^{\alpha f/ 2} + \tau^{-1} \big) e^{\alpha f}  d\nu_t dt 
\leq \int_{-1.1}^{-\theta}  \int_M  \big( \tau |\nabla f |^4 + \tau^{-1} e^{\alpha f/ 2} + 2 \tau^{-1}  \big) e^{\alpha f}  d\nu_t dt \\
\leq C(Y) |{\log \theta}|. 
\end{multline*}
This proves (\ref{eq_improved_L2_bound_1}) on a slightly larger domain.

Next, we prove (\ref{eq_improved_L2_bound_2}).
By combining (\ref{eq_tauR_ts}), (\ref{eq_w_ts}) with (\ref{eq_improved_L2_bound_1}), we obtain a bound of the form
\[  \int_{-1.1}^{-.5}  \int_M  \big(  |  \tau R | +  | w |  \big) e^{\alpha f}  d\nu_t dt   \leq C(Y). \]
So by integrating (\ref{eq_int_M_ealph_computation}), (\ref{eq_int_tauR_computation}) and using the fact that $R \geq  -C$, $w \leq 0$, we obtain
\[ \int_M \big(  |   R | +  | w |  \big)e^{\alpha f}  d\nu_{-1} \leq C(Y). \]
Using Proposition~\ref{Prop_int_ealphaf}, this implies
\[ \int_M \big| 2 \triangle f - |\nabla f|^2 \big| e^{\alpha f}  d\nu_{-1} \leq C(Y). \]
So
\begin{multline*}
 (1- 2\alpha ) \int_M  |\nabla f|^2 e^{\alpha f}  d\nu_{-1}
= \int_M  \big(2 (1-\alpha) |\nabla f|^2 -  |\nabla f|^2) e^{\alpha f}  d\nu_{-1} \\
= \int_M (2\triangle f - |\nabla f|^2)e^{\alpha f}  d\nu_{-1}  \leq C(Y). 
\end{multline*}
Combining the above bounds implies (\ref{eq_improved_L2_bound_2}). 
\end{proof}

\section{Geometric bounds near almost selfsimilar points} \label{sec_geom_bounds_almost_ss}
\subsection{Statement of the results}
In this section we derive bounds relating to the notion of almost selfsimilarity (see Definition~\ref{Def_almost_self_similar}).
We first show that almost selfsimilarity holds if the pointed Nash-entropy is almost constant and, vice versa, almost selfsimilarity implies almost constancy of the pointed Nash-entropy.
Second, we derive further analytical bounds under the almost selfsimilarity condition. 
Third, we establish an almost monotonicity property of the integral of the scalar curvature, which will be needed later.

Our first main result is the following:

\begin{Proposition} \label{Prop_NN_almost_constant_selfsimilar}
If $Y < \infty$, $\eps > 0$ and $\delta \leq \ov\delta (Y, \eps)$, then the following holds.
Let $(M, (g_t)_{t \in I})$ be a Ricci flow on a compact manifold and let $(x_0, t_0) \in M \times I$, $r > 0$ such that $[t_0 - \delta^{-1} r^2, t_0] \subset I$.
Suppose that $\NN_{x_0, t_0} ( r^2) \geq - Y$.

If
\[  \NN_{x_0, t_0} (\delta^{-1} r^2) \geq \NN_{x_0, t_0} (\delta r^2) - \delta , \]
then $(x_0, t_0)$ is $(\eps, r)$-selfsimilar.

Vice versa, if $(x_0, t_0)$ is $(\delta, r)$-selfsimilar, then for all $\tau_1, \tau_2 \in [ \eps r^2,  \eps^{-1} r^2]$
\begin{equation} \label{eq_prop_NN_diff_small}
 | \NN_{x_0, t_0} (\tau_1) - \NN_{x_0, t_0} (\tau_2) | \leq \eps. 
\end{equation}
\end{Proposition}

Our second main result provides further analytical bounds near almost selfsimilar points.
The important aspect here is the additional $e^{\alpha f}$ weight.
Note that several of the following bounds are essentially equivalent to each other and follow using the evolution equation $- \partial_t f = \triangle f - |\nabla f|^2 + R - \frac{n}{2\tau}$.
They are listed here for easy reference.
It is helpful to compare the following bounds with the soliton identities in Subsection~\ref{subsec_soliton_identities} --- the integrands in (\ref{eq_alm_ss_identity_1})--(\ref{eq_alm_ss_identity_5}) vanish on a gradient shrinking soliton with potential $f$.

\begin{Proposition} \label{Prop_almost_soliton_identities}
If $Y < \infty$, $\alpha \in [0, \ov\alpha (Y)]$, $\eps > 0$ and $\delta \leq \ov\delta (Y, \eps)$, then the following holds.
Let $(M, (g_t)_{t \in I})$ be a Ricci flow on a compact manifold and let $(x_0, t_0) \in M \times I$, $r > 0$ such that $[t_0 - \delta^{-1} r^2, t_0] \subset I$.
Suppose that $\NN_{x_0, t_0} ( r^2) \geq - Y$.
Denote by $d\nu = (4\pi \tau)^{-n/2} e^{-f} dg$ the conjugate heat kernel based at $(x_0, t_0)$.

If $(x_0, t_0)$ is $(\delta, r)$-selfsimilar, then we have for $W := \NN_{x_0,t_0} (r^2)$:
\begin{align}
 \int_{t_0 - \eps^{-1} r^2}^{t_0 - \eps r^2} \int_M \tau \Big| \Ric + \nabla^2 f - \frac1{2\tau} g \Big|^2 e^{\alpha f} d\nu_t dt &\leq \eps,  \label{eq_alm_ss_identity_1} \\
r^{-2} \int_{t_0 - \eps^{-1} r^2}^{t_0 - \eps r^2} \int_M  \big| \tau ( - |\nabla f|^2 + \triangle f) + f - \frac{n}2 - W \big| e^{\alpha f} d\nu_t dt &\leq \eps, \label{eq_alm_ss_identity_2} \displaybreak[1] \\
r^{-2} \int_{t_0 - \eps^{-1} r^2}^{t_0 - \eps r^2} \int_M  \Big| \square (\tau f) + \frac{n}2 + W \Big| e^{\alpha f} d\nu_t dt &\leq \eps, \label{eq_alm_ss_identity_3} \displaybreak[1]  \\
r^{-2} \int_{t_0 - \eps^{-1} r^2}^{t_0 - \eps r^2} \int_M  \Big|{ - \tau ( |\nabla f|^2 +R) + f  - W} \Big| e^{\alpha f} d\nu_t dt &\leq \eps, \label{eq_alm_ss_identity_4} \\
 \int_M \big| \tau (2\triangle f - |\nabla f|^2 + R) + f - n - W \big| e^{\alpha f} d\nu_t  \leq \eps \qquad  \text{for all} \quad t &\in [t_0 - \eps^{-1}r^2, t_0 - \eps r^2] . \label{eq_alm_ss_identity_5} 
\end{align}
\end{Proposition}

Our third result addresses the expectation that the geometry near almost selfsimilar points evolves via dilations.
More specifically, on a shrinking soliton, the integral $\int_M \tau R \, d\nu_{t}$ remains constant in time.
We will not be able to obtain an almost version of this statement at this point, however, we will show the almost monotonicity direction that will be important later.
Almost constancy will follow a posteriori from the main results of this paper.

\begin{Proposition} \label{Prop_scalar curvature_almost_ss}
If $\eps > 0$, $Y <  \infty$ and $\delta \leq \ov\delta (Y, \eps)$, then the following holds.

Let $(M, (g_t)_{t \in I})$ be a Ricci flow on a compact manifold and denote by $d\nu = (4\pi \tau)^{-n/2} e^{-f} dg$ the conjugate heat kernel based at some point $(x_0,t_0) \in M \times I$ .
Suppose that $\NN_{x_0,t_0} (r^2) \geq - Y$ and that $(x_0,t_0)$ is $(\delta, r)$-selfsimilar for some $r > 0$.
Then for any $t_1, t_2 \in [t_0 - \eps^{-1} r^2,t_0 - \eps r^2]$ with $t_1 \leq t_2$ we have
\[ \int_M \tau R\, d\nu_{t_1} \leq \int_M \tau R \, d\nu_{t_2} + \eps. \]
\end{Proposition}

\subsection{Proofs}
We will first focus on the proof of Proposition~\ref{Prop_NN_almost_constant_selfsimilar}.
For this purpose, recall from Subsection~\ref{subsec_soliton_identities} that on a normalized gradient shrinking soliton we have
\[ \tau (2\triangle f - |\nabla f|^2 + R) + f - n  \equiv W = \WW [g, f, \tau]. \]
The following lemma shows that this identity holds in an $L^1$-sense if $\WW$ is almost constant.

\begin{Lemma} \label{Lem_integral_soliton_id}
Let $(M, (g_t)_{t \in [-\tau_0,0]})$, $\tau_0 > 0$, be a Ricci flow on a compact manifold and $x_0 \in M$.
Denote by $d\nu = (4\pi \tau)^{-n/2} e^{-f} dg$ the conjugate heat kernel based at $(x_0, 0)$ and assume that $\NN_{x_0, 0} (\tau_0) \geq - Y$.
Suppose that for some $W \in [-Y,0]$ we have for $\WW_{x_0, 0} (\tau) := \WW [g_{-\tau}, f(\cdot, -\tau), \tau]$
\[ | \WW_{x_0, 0} (\tau') - W | \leq \delta \qquad \text{for all} \quad \tau' \in [\delta \tau_0, \tau_0]. \]
Then at time $t =  - \tau_0$ we have
\[  \int_M \big| \tau_0 (2\triangle f - |\nabla f|^2 + R) + f - n - W \big| d\nu_{- \tau_0} \leq \Psi (\delta |  Y). \] 
\end{Lemma}

\begin{proof}
By parabolic rescaling and application of a time-shift, we may assume that $t_0 = 0$ and $\tau_0 = 1$.

Let $u := (4\pi \tau)^{-n/2} e^{-f}$, so $d\nu = u \, dg$ and $\square^* u = 0$, and recall that by \cite[Sec.~9]{Perelman1}
\[ w := \tau(2\triangle f - |\nabla f|^2 + R) + f - n \leq 0,  \]
\[ \square^* (wu) = - 2 \tau \bigg| \Ric + \nabla^2 f - \frac1{2\tau} g \bigg|^2 u \leq 0, \]
\[ \int_M w \, d\nu_t = \WW_{x_0, 0} (|t|).  \]
Now if $h \in C^\infty (M \times [-1, 0])$ is a solution to the heat equation $\square h = 0$ with $|h(\cdot, -1) | \leq 1$, then $|h| \leq 1$ on all of $M \times [-1, 0]$ and therefore
\begin{multline*}
 \frac{d}{dt} \int_M h (w - W) u \, dg_t 
 = \int_M ( \square h) (w- W) u \, dg_t - \int_M h \square^* ((w-W) u) dg_t \\
 =  -\int_M h \square^* (w u) dg_t 
 \geq \int_M  \square^* (w u) dg_t 
= -\frac{d}{dt} \int_M w u \, dg_t = \frac{d}{dt} \WW_{x_0, 0} (|t|).
\end{multline*}
Thus for any $\tau_1 \in [\delta, 1]$ we have
\begin{equation} \label{eq_hW_hW}
 \int_M  h (w - W) u \, dg_{-1} \leq \int_M h (w - W) u \, dg_{-\tau_1} - \WW_{x_0, 0} (1) +  \WW_{x_0, 0} (\tau_1)
 \leq  \int_M h (w - W) d\nu_{-\tau_1} + 2\delta. 
\end{equation}
By duality, it remains to bound the last integral for a suitable $\tau_1 \in [\delta, \frac12]$, which we will determine later.

By \cite[\HKThmGradientPhiEstimate]{Bamler_HK_entropy_estimates} and Proposition~\ref{Prop_improved_L2}, or the preceding discussion in Section~\ref{sec_improved_Lp}, we have by Proposition~\ref{Prop_L_infty_HK_bound}
\begin{equation} \label{eq_nab_h_C_WW}
 |\nabla h| \leq C \qquad \text{on} \quad M \times [-\tfrac12, 0], 
\end{equation}
\[ \delta^{-1} \int_{-2\delta}^{-\delta} \int_M \big( \tau^2 |{\Ric}|^2  + \tau^2 |\nabla f|^4 + \tau^2 |\nabla^2 f|^2 + f^2 \big) d\nu_t dt \leq C(Y). \]
So we can choose a $\tau_1 \in [\delta, 2\delta]$ with the property that
\[ \int_M \big( \tau^2_1 |{\Ric}|^2  + \tau^2_1 |\nabla f|^4 + \tau_1^2 |\nabla^2 f|^2 + f^2 \big) d\nu_{-\tau_1} \leq C(Y). \] 
It follows that
\begin{equation*}
 \int_{M} (w-W)^2 d\nu_{-\tau_1} 
 \leq C(Y) + \int_{M} \big( \tau_1^2 ( |\nabla^2 f|^2 +  |\nabla f|^4 +  R^2 )+ f^2 \big) d\nu_{-\tau_1} \leq C(Y) 
 \end{equation*}
By the $L^2$-Poincar\'e inequality (see Proposition~\ref{Prop_Poincare}) and (\ref{eq_nab_h_C_WW}) we have for $a := \int_M h \, d\nu_{-\tau_1} \in [-1,1]$:
\[ \int_M |h - a|^2 \, d\nu_{-\tau_1} \leq 2 \tau_1 \int_M |\nabla h|^2 d\nu_{-\tau_1} \leq C \tau_1 \leq C \delta. \]
 Therefore,
 \begin{align*}
 \bigg| \int_{M} h (w-W) d\nu_{-\tau_1} \bigg|
&\leq \bigg| a \int_M (w-W) d\nu_{-\tau_1} \bigg|  + \int_{M} |h - a | |w-W|   d\nu_{-\tau_1} \\
&\leq  \bigg| \int_M w\, d\nu_{-\tau_1} - W \bigg| + \bigg( \int_{M} |h - a|^2  d\nu_{-\tau_1} \bigg)^{1/2} \bigg( \int_{M} |w-W|^2   d\nu_{-\tau_1} \bigg)^{1/2} \\
&\leq | \WW_{x_0, 0} (\tau_1) - W | + C(Y) \delta^{1/2}
\leq \delta + C(Y) \delta^{1/2}.
\end{align*}
Combining this with (\ref{eq_hW_hW}) implies
\[  \int_M  h (w - W) u \, dg_{-1} \leq \Psi (\delta | Y). \]
Since $h(\cdot, -1) : M \to [-1,1]$ can be chosen arbitrarily, this finishes the proof of the Lemma.
\end{proof}
\bigskip

\begin{proof}[Proof of Proposition~\ref{Prop_NN_almost_constant_selfsimilar}.]
By parabolic rescaling and application of a time-shift, we may assume that $t_0 = 0$ and $r = 1$.
Fix $\eps, Y$ and let $\delta > 0$ be a constant whose value we will determine later.
Write $d\nu = (4\pi \tau)^{-n/2} e^{-f} dg$ for the conjugate heat kernel based at $(x_0, 0)$, write $\WW_{x_0, 0} (\tau) := \WW [g_{-\tau}, f(\cdot, -\tau), \tau]$ as in Lemma~\ref{Lem_integral_soliton_id} and set $W := \NN_{x_0, 0} (1) \geq - Y$.

Due to the monotonicity of $\NN_{x_0, 0} (\tau')$ we have
\[ |\NN_{x_0, 0} (\tau') - W | \leq \delta \qquad \text{for all} \quad \tau' \in [\delta, \delta^{-1}]. \]
So by (\ref{eq_NN_geq_WW}) after Proposition~\ref{Prop_NN_basic_properties} we have for all $\tau' \in [\delta,\delta^{-1}]$
\begin{equation} \label{eq_WW_NN_in_proof}
 \WW_{x_0, 0} (\tau')  \leq \NN_{x_0, 0}(\tau') \leq W + \delta. 
\end{equation}
On the other hand, we obtain that if $\delta' \geq \delta$ is a constant whose value we will determine later, then for any $\tau' \in [\delta',\delta^{\prime -1}]$
\begin{multline} \label{eq_WW_NN_reverse_in_proof}
 \WW_{x_0, 0} (\tau') 
\geq \frac1{\delta^{-1} - \tau'} \int^{\delta^{-1}}_{\tau'} \WW_{x_0, 0} (\tau'') d\tau''
=  \frac1{\delta^{-1} - \tau'} \big( \delta^{-1} \NN_{x_0, 0} (\delta^{-1}) - \tau' \NN_{x_0, 0} (\tau') \big) \\
\geq  \frac{\delta^{-1}}{\delta^{-1} - \tau'}  \NN_{x_0, 0} (\delta^{-1})
\geq  \frac{\delta^{-1}}{\delta^{-1} - \tau'}  (W - \delta)
\geq W - \Psi (\delta | Y, \delta').
\end{multline}
We can therefore apply Lemma~\ref{Lem_integral_soliton_id} for $\delta' \leq \ov\delta' (Y, \eps)$ and $\delta \leq \ov\delta (Y, \delta', \eps)$ to conclude identity (\ref{eq_almost_self_similar_2}) of Definition~\ref{Def_almost_self_similar}.
 For identity (\ref{eq_almost_self_similar_1}) observe that by Proposition~\ref{Prop_NN_basic_properties}
\[ 2 \int_{ - \eps^{-1}}^{- \eps } \int_M \tau \Big| \Ric + \nabla^2 f - \frac1{2\tau} g \Big|^2 d\nu_t dt = \WW_{x_0, 0} (\eps) -  \WW_{x_0, 0} (\eps^{-1}) \leq \Psi (\delta |Y, \eps) \]
and identity (\ref{eq_almost_self_similar_3}) is a consequence of Lemma~\ref{Lem_lower_scal} for $\delta \leq \ov\delta (\eps)$.

For the last statement, assume that $(x_0, 0)$ is $(\delta, 1)$-selfsimilar.
Set $W := \NN_{x_0, 0} (1) \geq - Y$.
Integrating (\ref{eq_almost_self_similar_2}) implies that for all $\tau' \in [\delta, \delta^{-1}]$
\[ |\WW_{x_0, 0} (\tau') - W | \leq \delta. \]
So as in (\ref{eq_WW_NN_in_proof}) we obtain for $\tau' \in [\eps, \eps^{-1}]$
\[ \NN_{x_0, 0} (\tau') \geq \WW_{x_0, 0} (\tau') \geq W - \delta \]
and as in (\ref{eq_WW_NN_reverse_in_proof}) we obtain
\[ W+\delta
 \geq \WW_{x_0, 0} (\delta) 
\geq \frac1{ \tau' -\delta} \int_{\delta}^{\tau'} \WW_{x_0, 0} (\tau'') d\tau''
\geq  \frac{\tau'}{\tau' - \delta}  \NN_{x_0, 0} (\tau').
\]
This implies (\ref{eq_prop_NN_diff_small}) for $\delta \leq \ov\delta(Y, \eps)$.
\end{proof}
\bigskip

\begin{proof}[Proof of Proposition~\ref{Prop_almost_soliton_identities}.]
By parabolic rescaling and application of a time-shift, we may assume that $t_0 = 0$ and $r = 1$.
Fix $Y, \eps$ and let $0 < \eps' \leq \eps$, $\delta$ be constants whose values we will determine later.
Choose $\ov\alpha$ to be the same constant as in Proposition~\ref{Prop_improved_L2}.

We first prove the proposition in the case in which $\alpha = 0$, assuming that the right-hand sides of (\ref{eq_alm_ss_identity_1})--(\ref{eq_alm_ss_identity_5}) are replaced with $\eps'$, for $\delta \leq \ov\delta (Y, \eps')$.
Identities (\ref{eq_alm_ss_identity_1}) and (\ref{eq_alm_ss_identity_5}) are clear, assuming $\delta \leq \eps$.
By tracing  (\ref{eq_alm_ss_identity_1}) we can also show that
\[ \int_{ - \eps^{-1} }^{- \eps } \int_M \tau \Big( R + \triangle f - \frac{n}{2\tau}  \Big)^2  d\nu_t dt \leq \eps' .\]
Combining this with (\ref{eq_alm_ss_identity_5}) and the evolution equation $- \partial_t f = \triangle f - |\nabla f|^2 + R - \frac{n}{2\tau}$, we obtain (\ref{eq_alm_ss_identity_2})--(\ref{eq_alm_ss_identity_4}),  possibly after adjusting $\eps'$.

Next, we consider the case $\alpha \in (0, \ov\alpha]$.
The integrand in all inequalities (\ref{eq_alm_ss_identity_1})--(\ref{eq_alm_ss_identity_5}) is of the form $X e^{\alpha f}$ for some $X \geq 0$.
By Propositions~\ref{Prop_improved_L2}, \ref{Prop_L_infty_HK_bound} we know that
\[ \int_{ - \eps^{-1} }^{- \eps } \int_M X e^{2\alpha f} d\nu_t dt \leq C(Y, \eps), \qquad \int_M X e^{2\alpha f} d\nu_t  \leq C(Y,\eps), \]
where for the latter integral inequality we assumed $t \in [-\eps^{-1}, -\eps]$.
Therefore, for any $b > 0$ we have
\[ \int_{ - \eps^{-1} }^{- \eps } \int_M X e^{\alpha f} d\nu_t dt
\leq b \int_{ - \eps^{-1} }^{- \eps } \int_M X e^{2\alpha f} d\nu_t dt + b^{-1} \int_{ - \eps^{-1} }^{- \eps } \int_M X  d\nu_t dt
\leq C(Y, \eps) b + \eps' b^{-1}. \]
Since $\eps'$ can be chosen arbitrarily small, assuming that $\delta \leq \ov\delta (Y, \eps')$, we can  make the right-hand side smaller than $\eps$.
Choose for example $b := \frac12 C^{-1} (Y, \eps) \eps$ and $\eps' := \frac12 b \eps$.
This proves (\ref{eq_alm_ss_identity_1})--(\ref{eq_alm_ss_identity_4}).
The bound (\ref{eq_alm_ss_identity_5}) follows similarly, by dropping the time integral.
\end{proof}
\bigskip

\begin{proof}[Proof of Proposition~\ref{Prop_scalar curvature_almost_ss}.]
As before, we may assume without loss of generality that $r = 1$ and $t_0 = 0$.
Let $W := \NN_{x_0, 0} (1) \geq - Y$ and let $b, \delta > 0$ be constants whose values will be determined in the course of the proof.
Let $\alpha := \ov\alpha$ from Proposition~\ref{Prop_improved_L2}.
Then
\begin{multline*}
\tau \frac{d}{dt} \int_M \tau R \, d\nu_t
= \tau \int_M (2 \tau |{\Ric}|^2 - R) d\nu_t
= \tau \int_M \Big( 2 \tau \Ric \cdot \Big( \Ric + \nabla^2 f - \frac1{2\tau}g \Big)
- 2 \tau \Ric \cdot \nabla^2 f \Big) d\nu_t \\
\geq - b \tau^2 \int_M |{\Ric}|^2 d\nu_t 
-  b^{-1} \tau^2 \int_M  \Big| \Ric + \nabla^2 f - \frac1{2\tau}g \Big|^2 d\nu_t - 2\tau^2 \int_M \Ric \cdot \nabla^2 f \, d\nu_t.
\end{multline*}
The last term can be bounded as follows:
\begin{align*}
2\tau^2 \int_M \Ric \cdot \nabla^2 f \, d\nu_t 
&= 2\tau^2 \int_M \big( - \DIV \Ric \cdot \nabla f + \Ric (\nabla f, \nabla f) \big) d\nu_t  \\
&=2\tau^2 \int_M \Big( \Ric  + \nabla^2 f  -\frac1{2\tau} g \Big) \cdot ( \nabla f \otimes \nabla f ) d\nu_t \\
&\qquad + \tau^2 \int_M \Big( -\nabla R - 2\nabla_{\nabla f} \nabla f   + \frac1{\tau} \nabla f \Big)\cdot \nabla f  \, d\nu_t \displaybreak[1] \\
&\leq  b \tau^2 \int_M |\nabla f|^4 d\nu_t + b^{-1} \tau^2  \int_M  \Big| \Ric + \nabla^2 f - \frac1{2\tau}g \Big|^2 d\nu_t  \\
&\qquad + \tau \int_M \nabla \Big( - \tau (|\nabla f|^2 + R )  +  f  -  W \Big)\cdot \nabla f \,  d\nu_t.
\end{align*}
We can be bound the last term as follows, using Proposition~\ref{Prop_L_infty_HK_bound}: 
\begin{align*}
\tau \int_M \nabla \big( -& \tau ( |\nabla f|^2 + R)  +  f  - W \big)\cdot \nabla f  \, d\nu_t
=  \tau \int_M  \big( { - \tau ( |\nabla f|^2 + R)  +  f  -  W } \big) (|\nabla f|^2 -  \triangle f)  d\nu_t  \\
&= - \tau^2 \int_M  \Big( \Ric + \nabla^2 f - \frac1{2\tau} g \Big) \cdot \big( (|\nabla f|^2 -  \triangle f) g \big) d\nu_t \\
&\qquad\qquad +  \tau \int_M  \Big(  \tau (- |\nabla f|^2 + \triangle f    )  +  f -\frac{n}2 -  W \Big) (|\nabla f|^2 -  \triangle f)  d\nu_t \displaybreak[1] \\
&\leq n \tau^2 b \int_M   (|\nabla f|^2 -  \triangle f)^2 d\nu_t 
+ \tau^2 b^{-1} \int_M  \Big| \Ric + \nabla^2 f - \frac1{2\tau} g \Big|^2 d\nu_t \\ 
&\qquad\qquad -   \int_M  \Big(  \tau (- |\nabla f|^2 + \triangle f    )  +  f -\frac{n}2 -  W \Big)^2  d\nu_t \\
&\qquad\qquad +   \int_M  \Big(  \tau (- |\nabla f|^2 + \triangle f    )  +  f -\frac{n}2 -  W \Big) \Big( f - \frac{n}2 - W \Big)  d\nu_t \displaybreak[1] \\
&\leq 2n \tau^2 b \int_M   |\nabla f|^4  d\nu_t  + 2n^2 \tau^2 b \int_M   |\nabla^2 f|^2  d\nu_t 
+ \tau^2 b^{-1} \int_M  \Big| \Ric + \nabla^2 f - \frac1{2\tau} g \Big|^2 d\nu_t \\ 
&\qquad\qquad + \Big| \frac{n}2 + W \Big|  \int_M  \Big|  \tau (- |\nabla f|^2 + \triangle f    )  +  f -\frac{n}2 -  W \Big|  d\nu_t \\
&\qquad\qquad +   \int_M  \Big(  \tau (- |\nabla f|^2 + \triangle f    )  +  f -\frac{n}2 -  W \Big) f \,  d\nu_t \displaybreak[1] \\
&\leq 2n \tau^2 b \int_M   |\nabla f|^4  d\nu_t  + 2n^2 \tau^2 b \int_M   |\nabla^2 f|^2  d\nu_t 
+ \tau^2 b^{-1} \int_M  \Big| \Ric + \nabla^2 f - \frac1{2\tau} g \Big|^2 d\nu_t \\ 
&\qquad\qquad + C(Y)  \int_M  \Big|  \tau (- |\nabla f|^2 + \triangle f    )  +  f -\frac{n}2 -  W \Big|  e^{\alpha f}  d\nu_t.
\end{align*}
Observe that in the second inequality we have dropped the negative of a square; this is the reason why our proof does not imply the opposite monotonicity statement.
Combining all inequalities above and integrating yields, using Proposition~\ref{Prop_improved_L2},
\begin{align*}
\int_M \tau R d\nu_t \bigg|_{t = t_1}^{t = t_2}
\geq &- C b \int_{t_1}^{t_2} \int_M \tau \big(  |{\Ric}|^2 + |\nabla f|^4 + |\nabla^2 f|^2 \big) d\nu_t dt \\
&- C b^{-1} \int_{t_1}^{t_2} \int_M  \tau  \Big| \Ric + \nabla^2 f - \frac1{2\tau}g \Big|^2 d\nu_t dt \\
&- C(Y)  \int_{t_1}^{t_2} \int_M  \tau^{-1} \Big|  \tau (- |\nabla f|^2 + \triangle f    )  +  f -\frac{n}2 -  W \Big|  e^{\alpha f}  d\nu_t dt \\
&\geq - C(Y, \eps) b - C b^{-1} \delta - C(Y) \Psi (\delta | Y, \eps).
\end{align*}
Setting $b = \delta^{1/2}$ implies the proposition for $\delta \leq \ov\delta (Y, \eps)$.
\end{proof}

\section{Comparing conjugate heat kernel measures} \label{sec_comparing_CHK}
In this paper, we will often encounter the following problem.
Suppose that we are given a bound on an integral of the form $\int_M X \, d\nu_{x_0, t_0; t}$ over a geometric quantity $X \geq 0$ on a Ricci flow $(M, (g_t)_{t \in I})$, where $d\nu_{x_0, t_0; t}$ is the conjugate heat kernel measure based at some point $(x_0, t_0)$.
We would like to use this bound to conclude a similar bound on an integral of the form $\int_M X\, d\nu_{x_1, t_1; t}$, for a conjugate heat kernel measure $d\nu_{x_1, t_1; t}$ based at a different point $(x_1, t_1)$.
This is, in general, not possible since these measures need not satisfy a bound of the form $d\nu_{x_1, t_1; t} \leq C d\nu_{x_0, t_0; t}$.
In the following we will show, however, that under certain assumptions we have a bound of the form $d\nu_{x_1, t_1; t} \leq C(\alpha) e^{\alpha f_{x_0, t_0}} d\nu_{x_0, t_0; t}$ for arbitrarily small $\alpha > 0$, where $f_{x_0, t_0}$ is the potential of $d\nu_{x_0, t_0; t}$.
This will allow us to deduce bounds on $\int_M X \, d\nu_{x_1, t_1; t}$ from bounds on modified integrals of the form $\int_M X \, e^{\alpha f_{x_0, t_0}} d\nu_{x_0, t_0; t}$.

\begin{Proposition} \label{Prop_inheriting_bounds}
If $Y, D, A < \infty$, $\beta > 0$, then the following holds.

Let $(M, (g_t)_{t \in I})$ be a Ricci flow on a compact manifold and consider times $s, t^*, t_0, t_1 \in I$ with  $s < t^* \leq t_0, t_1$.
Suppose that 
\[ R (\cdot, s) \geq - A (t^* - s)^{-1}  \]
and that for some constants $-\infty < \alpha_1 < \alpha_0 < 1$ we have
\begin{equation} \label{eq_t1_tstar_theta}
 t_0 - t^* \leq D^2 (t^* - s), \qquad  t_1 - t^* \leq \theta(A) \, \frac{\alpha_0 - \alpha_1}{1-\alpha_0}( t^* - s). 
\end{equation}
Let $x_0, x_1 \in M$ and assume that $\NN_{x_0, t_0}(t_0 - s) \geq - Y$ or $\NN_{x_1, t_1}(t_0 - s) \geq - Y$.
Denote by $d\nu_{x_i, t_i} = (4\pi \tau_i)^{-n/2} e^{-f_i} dg$ the conjugate heat kernels based at $(x_i, t_i)$, $i=0,1$, and assume that
\[ d^{g_{t^*}}_{W_1} (\nu_{x_0, t_0;t^*}, \nu_{x_1,t_1;t^*} ) \leq D \sqrt{t^* - s}. \]
Then
\[ e^{\alpha_1 f_1} d\nu_{x_1,t_1; s} \leq C(Y, D, A,\alpha_0, \alpha_1) e^{\alpha_0 f_0} d\nu_{x_0, t_0; s} . \]
\end{Proposition}

\begin{Remark}
By inspecting the following proof, we obtain a slightly stronger bound of the form
\[ e^{\alpha_1 f_1} d\nu_{x_1,t_1; s} \leq C(D, A,\alpha_0, \alpha_1)  e^{(\alpha_0 - \alpha_1) \NN_{x_0,t_0}(t_0 - s)} e^{\alpha_0 f_0} d\nu_{x_0, t_0; s} ,\]
which is collapsing-independent.
\end{Remark}

\begin{proof}
In the following proof $C(a, b, \ldots)$ will denote a generic constant that depends only on certain parameters $a, b, \ldots$.
Moreover, we will work with a sequence of indexed constants $C_0, C_1, \ldots$, which will not be generic.
After introducing each constant, we will indicate what other constants this constant depends on, and then assume it as fixed for the remainder of this proof.
Let $\theta \in (0,1)$ be some constant whose value we will determine later based on $A$.

By parabolic rescaling, we may assume without loss of generality that $s = 0$ and $t^* = 1$.
Note that this implies using the maximum principle that
\[ R \geq - A  \qquad \text{on} \quad M \times (I \cap [0, \infty)) \]
and
\begin{equation} \label{eq_rescaled_bounds_theta_times}
 t_0 \leq D^2 + 1, \qquad t_1 - 1 \leq \theta \frac{\alpha_0 - \alpha_1}{1-\alpha_0}, \qquad  d^{g_{1}}_{W_1} (\nu_{x_0, t_0;1}, \nu_{x_1,t_1;1} ) \leq D. 
\end{equation}
We will use the notation $\NN^*_0 (x',t') := \NN_{x',t'} (t')$ from \cite[\HKDefNN]{Bamler_HK_entropy_estimates}.
Using Proposition~\ref{Prop_NN_variation_bound} we obtain that $|\NN^*_0 (x_0, t_0) - \NN^*_0 (x_1, t_1)| \leq C( D, A, \alpha_0, \alpha_1)$.
So after adjusting $Y$ depending on $D, A, \alpha_0, \alpha_1$, we may assume that $\NN^*_0 (x_1, t_1) \geq - Y$.

Set $\nu^i := \nu_{x_i, t_i; 1}$ and let $(z_i,1)$ be an $H_n$-center of $(x_i, t_i)$, $i =0,1$.
Again by Proposition~\ref{Prop_NN_variation_bound} and by \cite[\HKThmNabNNSquareNN]{Bamler_HK_entropy_estimates} we obtain that there are constants $C_0 (Y, D,A), C_1 (A) < \infty$ such that
\begin{equation} \label{eq_NN_nab_NN}
 \NN^*_0 (z_1,1) \geq - C_0, 
\qquad |\nabla \NN^*_0(\cdot, 1)| \leq C_1. 
\end{equation}
We also have
\begin{multline} \label{eq_d_1_z_0_z_1}
 d_1 (z_0, z_1) \leq d^{g_{1}}_{W_1} (\delta_{z_0}, \nu^0) + d^{g_{1}}_{W_1} (\nu^0, \nu^1) + d^{g_{1}}_{W_1} (\nu^1, \delta_{z_1}) 
 \leq \sqrt{{\Var}_1  (\delta_{z_0}, \nu^0) } + D + \sqrt{{\Var}_1 (\nu^1, \delta_{z_1})} \\
\leq \sqrt{H_n (t_0 - 1)} + D + \sqrt{H_n (t_1 - 1)} 
\leq D' (  D, \alpha_0, \alpha_1 ). 
\end{multline}

We will now reduce the proposition to a more basic bound in several steps.
For this purpose, fix some $y \in M$.
We need to show that
\begin{equation} \label{eq_Kx1_Kx0_la}
 K (x_1, t_1; y,0) \leq C(D, A,\alpha_0, \alpha_1) \big( K(x_0, t_0; y,0) \big)^{\lambda} , 
\end{equation}
where $\lambda = ( 1-\alpha_0)/ (1-\alpha_1) < 1$.
Set $u := K(\cdot, 1; y,0) \in C^\infty (M)$.
Then (\ref{eq_Kx1_Kx0_la}) is equivalent to
\begin{equation} \label{eq_int_u_1_int_u_0_la}
\int_M u \, d\nu^1 \leq C(D, A,\alpha_0, \alpha_1) \bigg( \int_M u \, d\nu^0 \bigg)^{\lambda} , 
\end{equation}

By \cite[\HKThmNabKboundThmNabKbound]{Bamler_HK_entropy_estimates}, we can find constants $C_2(A) , C_3(A) < \infty$ such that
\begin{align*}
 u &\leq \frac{1}{10} C_2 \exp ( - \NN_0^* (\cdot,1) ), \\
 \frac{|\nabla u|}{u} & \leq C_3 \sqrt{ \log \bigg ( \frac{C_2 \exp (- \NN^*_0 (\cdot, 1))}{ u} \bigg)}.
\end{align*}
So if we set
\[ v :=C_2^{-1} u \, \exp ( \NN^*_0 (\cdot, 1)) \leq \frac1{10}, \]
then using (\ref{eq_NN_nab_NN}), we find
\begin{equation} \label{eq_nab_v_v_C_5}
 \frac{|\nabla v|}{v} \leq \frac{|\nabla u|}{u} + |\nabla \NN^*_0 (\cdot, 1)|
\leq C_3 \sqrt{ - \log v} + C_1 \leq 2C_4 \sqrt{- \log v}, 
\end{equation}
where $C_4 = C_4(C_1, C_3) = C_4(A) < \infty$.
Consider the invertible function
\[  F : [0, \infty) \longrightarrow (0, 1], \qquad a \longmapsto  \exp (- a^2) \]
and write $v =: F \circ w$ for some $w \in C^\infty(M)$.
Then (\ref{eq_nab_v_v_C_5}) implies that
\begin{equation} \label{eq_nab_w_C_5}
 |\nabla w| \leq C_4 
\end{equation}
and to see (\ref{eq_int_u_1_int_u_0_la}), we need to show that
\begin{equation} \label{eq_int_Fw_0_int_Fw_1_la}
 \int_M (F \circ w) \exp ( - \NN^*_0 (\cdot, 1) ) \, d\nu^1 \leq C(Y,D, A,\alpha_0, \alpha_1) \bigg( \int_M (F \circ w) \exp ( - \NN^*_0 (\cdot, 1) )  \, d\nu^0 \bigg)^{\lambda} .
\end{equation}

Let us now simplify (\ref{eq_int_Fw_0_int_Fw_1_la}) even further.
Set
\[ B_0 := B\big(z_0, 1, \sqrt{2H_n (t_0 - 1)} \big). \]
Note that $\sqrt{2H_n (t_0 - 1)} \leq \sqrt{2H_n} D$.
By Lemma~\ref{Lem_mass_ball_Var} we have $\nu^0 ( B_0) \geq \frac12$.
Next, observe that by (\ref{eq_NN_nab_NN}) we have
\[ 0 \leq - \NN^*_0 (\cdot, 1) \leq C_0 + C_1 d_1 (z_1, \cdot). \]
So in order to prove (\ref{eq_int_Fw_0_int_Fw_1_la}), it suffices to show that
\begin{equation} \label{eq_int_Fw_Fw_y}
 \int_M (F \circ w) \exp ( C_1 d_1 (z_1, \cdot) )  \, d\nu^1 \leq C(Y,D, A,\alpha_0, \alpha_1)  \big( F(w(p)) \big)^\la \qquad \text{for all} \quad p \in B_0. 
\end{equation}

To see (\ref{eq_int_Fw_Fw_y}), we first set
\[ \theta := (8 C_4)^{-2}. \]
So (\ref{eq_rescaled_bounds_theta_times}) implies that if $t_1 > 1$, then
\[ \frac1{16(t_1 - 1)} \geq \frac{2C_4^2 \la}{1-\la}. \]
Now fix some $p \in B_0$.
By (\ref{eq_nab_w_C_5}), (\ref{eq_d_1_z_0_z_1}) we have for any $q \in M$
\begin{multline*}
 w(q)
 \geq w(p)  - C_4 d_1 (q,p) 
\geq   w(p) - C_4 d_1 (q, z_1) - C_4 d_1 (z_1, z_0) - C_4 \sqrt{2 H_n (t_0 - 1)} \\
\geq w(p) - C_4 (d_1 (q, z_1)+1) - C_5,
\end{multline*}
where $C_5 = C_5(C_4, D, D') = C_5 (A, D, \alpha_0, \alpha_1) < \infty$.
Next, for any $d > 0$ we have by \cite{Bamler_HK_entropy_estimates} if $t_1 > 1$
\begin{multline*}
 \nu^1 \big( M \setminus B(z_1, 1, d) \big) 
\leq C \exp \bigg({ - \frac1{8(t_1 - 1)} \big( d - \sqrt{2 H_n (t_1 - 1)} \big)_+^2 } \bigg) \\
\leq C \exp \bigg({ - \frac1{16(t_1 - 1)}  d^2 } \bigg)
\leq C \exp \bigg({ - \frac{2C_4^2 \la}{1-\la}  d^2 } \bigg). 
\end{multline*}
The same bound holds trivially if $t_1 = 1$.
It now follows, using the inequality $a^2 + \frac{\la}{1-\la} b^2 \geq \la (a+b)^2$, that
\begin{align*}
 \int_M (F \circ w) & \exp ( C_1 d_1 (z_1, \cdot) ) \, d\nu^1
= \sum_{j=1}^\infty \int_{B(z_1, 1,  j) \setminus B(z_1, 1,  j-1)} (F \circ w)  \exp ( C_1 d_1 (z_1, \cdot) ) \, d\nu^1 \\
&\leq \sum_{j=1}^\infty F \big( (w(p) - C_4 j - C_5)_+ \big) \exp (C_1 j) \nu^1 ( M \setminus B(z_1, 1, j-1)) \\
&\leq C \sum_{j=1}^\infty \exp \bigg({ - \big( w(p)  - C_4 j - C_5 \big)_+^2 + C_1 j  - \frac{2\la}{1-\la} (C_4 j-C_4 )^2 }\bigg) \\
&\leq C (\la, C_1, C_4, C_5) \sum_{j=1}^\infty \exp \bigg({ - \big( w(p) - C_4 j - C_5 \big)_+^2  - \frac{\la}{1-\la} (C_4 j +C_5 )^2 }\bigg) \\
&\leq C (\la, C_1, C_4, C_5) \sum_{j=1}^\infty \exp \Big({ - \la \big( ( w(p) -  C_4 j - C_5 )_+  + C_4 j + C_5 \big)^2  }\Big) \\
&\leq C (\la, C_1, C_4, C_5) \sum_{j=1}^\infty \exp \big({ - \la  ( w(p) )^2  } \big) \\ 
&\leq C (\la, C_1, C_4, C_5) \big( F (w(p)) \big)^{\la}.
\end{align*}
This shows (\ref{eq_int_Fw_Fw_y}), which finishes the proof.
\end{proof}

\section{Distance expansion bound near almost selfsimilar points} \label{sec_dist_expansion}
Most estimates concerning the change of the distance function in time are based on the monotonicity of the $W_1$-distance and the variance (see Propositions \ref{Prop_monotonicity_W1}, \ref{Prop_monotonicity_Var}).
These can be viewed as bounds on the amount that  distances may shrink as we move forward in time.
The opposite direction --- a distance expansion bound --- is generally harder to come by.
The following proposition shows that such a bound indeed holds near almost selfsimilar points.

\begin{Proposition} \label{Prop_dist_expansion_almost_ss}
If $Y, D < \infty$, $\beta > 0$ and $\delta \leq \ov\delta ( Y,D, \beta) \leq \beta$, then the following holds.

Let $(M, (g_t)_{t \in I})$ be a Ricci flow on a compact manifold and consider two points $(x_0, t_0),\lb (x_1, t_1) \in M \times I$ such that $[t_0 - \delta^{-1} r^2, t_0] \subset I$ and $\NN_{x_0, t_0} ( r^2)\geq - Y$.
Assume that $|t_0 - t_1| \leq \beta^{-1} r^2$ and that $(x_0, t_0)$ is $(\delta, r)$-selfsimilar.
Consider two times $s_1, s_2 \in [t_{\min} - \beta^{-1} r^2, t_{\min} - \beta r^2 ]$, where $t_{\min} := \min \{ t_0, t_1 \}$.
If
\begin{equation} \label{eq_Var_nu0_nu1_D}
 d_{W_1}^{g_{s_1}} (\nu_{x_0, t_0} (s_1), \nu_{x_1,t_1} (s_1) ) \leq Dr, 
\end{equation}
then
\[ d_{W_1}^{g_{s_2}} (\nu_{x_0, t_0} (s_2), \nu_{x_1,t_1} (s_2) ) \leq C(Y, D, \beta) r. \]
\end{Proposition}

\begin{proof}
The constant $\delta$ will be determined in the course of the proof.
After parabolic rescaling, we may assume that $r = 1$.
We may assume that $s_1 \leq s_2$, because otherwise the statement is a consequence of Proposition~\ref{Prop_monotonicity_W1}.
Moreover, by iterating the proposition, we may restrict to the case $s_2 - s_1 \leq \frac12 \beta$.

Using Proposition~\ref{Prop_NN_variation_bound}, we obtain that $\NN_{x_1, t_1} ( r^2)\geq - C(Y, D, \beta)$.
Write $d\nu^0 := d\nu_{x_0, t_0} = (4\pi \tau)^{-n/2} e^{-f} dg$, $d\nu^1 := d\nu_{x_1, t_1} = (4\pi \tau_1)^{-n/2} e^{-f_1} dg$.
By Proposition~\ref{Prop_L_infty_HK_bound}, and assuming that $\delta \leq \frac12 \beta$, we have 
\begin{equation} \label{eq_ff1_geq_mC}
 f, f_1 \geq - C(Y, D, \beta) \qquad \text{on} \quad [t_{\min} - \beta^{-1}, t_{\min}].
\end{equation}

Let $s'_1 := s_1 - \tfrac14 \beta$.
We will first show:

\begin{Claim}
There is a measurable subset $S \subset M$ such that $f (\cdot, s'_1), f_1(\cdot, s'_1) \leq C^*(Y,D, \beta)$ on $S$ and $|S|_{s'_1} \geq c(Y,D, \beta) > 0$.
\end{Claim}

\begin{proof}
Choose an $H_n$-center $(z',s'_1) \in M \times I$ of $(x_0, t_0)$.
Recall that by Lemma~\ref{Lem_mass_ball_Var} we have for $C_0 := \sqrt{2H_n}$.
\begin{equation} \label{eq_nuBz_12}
 \int_{B(z',s'_1, C_0 \sqrt{t_0 - s'_1})} (4\pi (t_0 - s'_1 ))^{-n/2} e^{-f} dg_{s'_1} = \nu^0_{s'_1 } \big( B(z',s'_1, C_0 \sqrt{t_0 - s'_1}) \big) \geq \frac12. 
\end{equation}
By \cite[\HKThmUpperVolBound]{Bamler_HK_entropy_estimates} we have
\begin{equation} \label{eq_Bzsp1C0_vol_bound}
 \big| B(z', s'_1, C_0 \sqrt{t_0 - s'_1}) \big|_{ s'_1 } \leq C(Y, D, \beta) (t_0 - s'_1)^{n/2}. 
\end{equation}
Let $C^* < \infty$ be a constant whose value we will determine later and set
\[ S' := \{ f(\cdot, s'_1) \leq C^* \} \cap B(z', s'_1, C_0 \sqrt{t_0 - s'_1}). \]
By (\ref{eq_nuBz_12}), (\ref{eq_Bzsp1C0_vol_bound}) we have
\begin{multline} \label{eq_int_sp_geq}
\int_{S'} d\nu^0_{s'_1} 
 \geq \frac12 -  \int_{B(z',s'_1, C_0 \sqrt{t_0 - s'_1}) \setminus S'} (4\pi (t_0 - s'_1 ))^{-n/2} e^{-f} dg_{s'_1} \\
\geq \frac12 - C(\beta) e^{-C^*}  \big| B(z', s'_1, C_0 \sqrt{t_0 - s'_1}) \big|_{ s'_1 }   
\geq \frac12 - C(Y, D, \beta) e^{-C^*}.
\end{multline}
So if $C^* \geq \underline{C}^* (Y,D, \beta)$, then by (\ref{eq_ff1_geq_mC})
\begin{equation} \label{eq_intSp_volSp}
 \int_{S'} d\nu^0_{s'_1} \geq \frac14, \qquad |S'|_{s'_1} \geq c(Y,D, \beta) > 0. 
\end{equation}

Next, let $\Phi_t : \IR \to (0,1)$ be the antiderivative of $\Phi'_t (x) = (4\pi t)^{-1/2} \exp (- x^2 / 4t )$ with $\lim_{x \to -\infty} \Phi_t (x) \lb = 0$, as in \cite[\HKSecGradientEstimate]{Bamler_HK_entropy_estimates}.
Consider the function
\[ h : M \longrightarrow \IR, \qquad  y \longmapsto \int_{S'} d\nu_{y, s_1; s'_1}. \]
By \cite[\HKThmGradientPhiEstimate]{Bamler_HK_entropy_estimates} the function $\Phi^{-1}_{\beta/ 4} (h)$ is $1$-Lipschitz at time $s_1$.
By the reproduction formula and (\ref{eq_intSp_volSp}) we have
\[ \int_M h \, d\nu^0_{s_1}  \geq \frac14. \]
We claim that
\begin{equation} \label{eq_intSp_nu1}
 \int_{S'} (4\pi (t_1 - s'_1 ))^{-n/2} e^{-f_1} dg_{s'_1} = \int_{S'} d\nu^1_{s'_1 } = \int_M h \, d\nu^1_{s_1} \geq c(Y,D,\beta) > 0. 
\end{equation}
To see the second last inequality, choose $H_n$-centers $(z^*_0, s_1), (z^*_1, s_1)$ of $(x_0,t_0), (x_1, t_1)$, respectively.
Since $h \leq 1$, we can use Lemma~\ref{Lem_mass_ball_Var} to conclude that there is a constant $C_1(Y,D,\beta) < \infty$ such that
\[ \int_{B(z^*_0, s_1, C_1 )}  h \, d\nu^0_{s_1} \geq c(Y,D,\beta) - \nu^0_{s_1} \big( M \setminus   B(z^*_0, s_1, C_1 ) \big) \geq \tfrac12 c(Y,D,\beta). \]
Thus, there is a point $z^{**} \in B(z^*_0, s_1, C_1 )$ with $h(z^{**}) \geq c(Y,D,\beta) > 0$.
Since $\Phi^{-1}_{\beta/ 4} (h)$ is $1$-Lipschitz, this implies that $h \geq c(Y,D,\beta) > 0$ on $B(z^*_1, s_1, \sqrt{2H_n (t_1 - s_1)})$ and therefore, by  Lemma~\ref{Lem_mass_ball_Var}
\[ \int_M h \, d\nu^1_{s_1} \geq \int_{B(z^*_1, s_1, \sqrt{2H_n (t_1 - s_1)})} h \, d\nu^1_{s_1} \geq c(Y,D,\beta), \]
proving (\ref{eq_intSp_nu1}).

Set now
\[  S := \{ f_1(\cdot, s'_1) \leq C^* \} \cap S' \subset B(z', s'_1, C_0 \sqrt{t_0 - s'_1}). \]
Using (\ref{eq_intSp_volSp}), (\ref{eq_Bzsp1C0_vol_bound}), we can estimate as in (\ref{eq_int_sp_geq}) that
\begin{multline*}
\int_{S} d\nu^1_{s'_1} 
 \geq \frac14 -  \int_{S' \setminus S} (4\pi (t_1 - s'_1 ))^{-n/2} e^{-f_1} dg_{s'_1} \\
\geq \frac14 - C(\beta) e^{-C^*}  \big| B(z', s'_1, C_0 \sqrt{t_0 - s'_1}) \big|_{ s'_1 }   
\leq \frac14 - C(Y, D, \beta) e^{-C^*}.
\end{multline*}
So if $C^* \geq \underline{C}^* (Y,D, \beta)$, then $\int_{S} d\nu^1_{s'_1} \geq \frac18$, which, together with (\ref{eq_ff1_geq_mC}), implies the claim for an appropriate choice of $c(Y,D, \beta) > 0$.
\end{proof}

Let $s''_1 := s'_1 - \frac14 \beta = s_1 - \frac12 \beta$, $W := \NN_{x_0, t_0}(r^2)$ and consider the function $u \in C^\infty (M \times [s'_1, s_2])$ defined by
\[ u := (4\pi (t - s''_1))^{-n/2} \exp \Big( {-  \frac{\tau (f  -W)}{t-s''_1}} \Big). \]
Since
\begin{equation} \label{eq_tau_t_spp1}
 \frac{\tau}{t - s''_1} \bigg|_{t = s'_1} = \frac{t_0 - s_1 + \frac14 \beta}{\frac14 \beta} \leq \frac{2\beta^{-1} + \frac14 \beta}{\frac14 \beta},
\end{equation}
we get from the Claim that
\begin{equation} \label{eq_int_u_sp1}
 \int_M u\, d\nu^1_{s'_1} \geq \int_{S} u \, d\nu^1_{s'_1}  \geq c(Y,D, \beta) > 0. 
\end{equation}

For $t \in [s'_1, s_2]$ we have
\begin{equation} \label{eq_tau_t_sp1_geq_1}
  \frac{\tau}{t - s''_1} \geq \frac{\beta}{s_2 - s_1 + \frac12\beta} \geq 1. 
\end{equation}
So using (\ref{eq_tau_t_spp1}), (\ref{eq_ff1_geq_mC}) we have for $t \in [s'_1, s_2]$
\begin{multline*}
 \exp \Big( {-  \frac{\tau (f-W)}{t-s''_1}} \Big) 
=  \exp \Big( {-  \frac{\tau }{t-s''_1}} (f + C(Y, D, \beta)) \Big)  \exp \Big( { \frac{\tau}{t-s''_1} C(Y, D, \beta)} \Big)  \\
\leq e^{-f - C(Y,D, \beta)}  \exp \Big( {  \frac{2\beta^{-1} +\frac14 \beta }{\frac14 \beta} C(Y, D, \beta)} \Big) 
\leq C(Y,D, \beta) e^{-f},
\end{multline*}
\[ \tau \leq 2 \beta^{-1} \leq C(\beta) \tfrac14 \beta \leq C(\beta) (t - s''_1), \]
and thus
\begin{equation} \label{eq_u_leq_C_emf}
 u \leq C(Y, D,\beta) (4\pi \tau)^{-n/2} e^{-f}. 
\end{equation}
We therefore obtain
\begin{align*}
 \square u &= -\frac{n}{2 (t - s''_1)} u - \frac{1}{t-s''_1} \partial_t (\tau ( f-W)) u \\
 &\qquad\qquad\qquad\qquad +  \frac{1}{(t -s''_1)^2} \tau (  f    - W )u + \frac{1}{t - s''_1} \triangle (\tau f) u - \frac{1}{(t - s''_1)^2}  |\nabla \tau f|^2 u \\
 &= -  \frac{1}{t - s''_1} \Big( \square (\tau f)  + \frac{n}2 + W \Big)u
 + \frac{1}{(t-s''_1)^2} \big( \tau f - \tau^2( |\nabla  f|^2 +  R)   - \tau W \big) u + \frac{1}{(t-s''_1)^2} \tau^2 R u \\
 &\geq - C(Y,D, \beta)  \Big( \Big|  \square (\tau f)  + \frac{n}2 + W \Big| + \Big| \tau f - \tau^2( |\nabla  f|^2 +  R)   - \tau W  \Big| + \delta \Big)  (4\pi \tau)^{-n/2} e^{-f} ,
\end{align*}
which implies, using (\ref{eq_ff1_geq_mC}),
\begin{align*}
 \frac{d}{dt} \int_M & u \, d\nu^1_t 
= \int_M \square u  \, d\nu^1_t \\
&\geq - C(Y, D,\beta) \int_M   \Big( \Big|  \square (\tau f)  + \frac{n}2 + W \Big| + \Big| \tau f - \tau^2( |\nabla  f|^2 +  R)   - \tau W  \Big| + \delta \Big)  (4\pi \tau)^{-n/2} e^{-f} d\nu^1_t \\
&= - C(Y,D, \beta) \Big( \frac{\tau_1}{\tau} \Big)^{n/2} \int_M  \Big( \Big|  \square (\tau f)  + \frac{n}2 + W \Big| \\
&\qquad\qquad\qquad\qquad\qquad\qquad\qquad\qquad + \Big| \tau f - \tau^2( |\nabla  f|^2 +  R)   - \tau W  \Big| + \delta \Big)   (4\pi \tau_1)^{-n/2} e^{-f_1} d\nu_t \\
&\geq - C(Y, D,\beta) \int_M  \Big( \Big|  \square (\tau f)  + \frac{n}2 + W \Big| + \Big| \tau f - \tau^2( |\nabla  f|^2 +  R)   - \tau W  \Big| + \delta \Big)  d\nu_t.
\end{align*}
So if $\delta \leq \ov\delta ( Y, D,\beta)$, then by (\ref{eq_int_u_sp1}) and Proposition~\ref{Prop_almost_soliton_identities}
\[ \int_M u \, d\nu^1_{s_2 } \geq \frac12 \int_M u \, d\nu^1_{s'_1 } \geq c(Y,D,\beta) > 0. \]
Thus, by (\ref{eq_u_leq_C_emf}) we have
\begin{equation} \label{eq_int_eff1}
 \int_M (4\pi)^{-n} (\tau \tau_1)^{-n/2} e^{-f-f_1} dg_{s_2 }   = \int_M (4\pi \tau)^{-n/2} e^{-f} d\nu^1_{s_2 }  \geq c(Y,D,\beta) \int_M u \, d\nu^1_{s_2 }  \geq c(Y,D,\beta) > 0. 
\end{equation}

Choose centers $(z_0, s_2)$, $(z_1, s_2)$ of $(x_0, t_0)$, $(x_1, t_1)$, respectively.
Since
\begin{align*}
 d_{W_1}^{g_{s_2}} (\nu_{x_0, t_0} (s_2), \nu_{x_1,t_1} (s_2) )
 &\leq d_{W_1}^{g_{s_2}} (\nu_{x_0, t_0} (s_2), \delta_{z_0, s_2} ) + d_{s_2} (z_0, z_1) + d_{W_1}^{g_{s_2}} (\delta_{z_1}, \nu_{x_1,t_1} (s_2) ) \\
&\leq \sqrt{{\Var}_{s_2} (\nu_{x_0, t_0} (s_2), \delta_{z_0, s_2} )} + d_{s_2} (z_0, z_1) + \sqrt{{\Var}_{s_2} (\delta_{z_1, s_2}, \nu_{x_1,t_1} (s_2) )}  \\
&\leq \sqrt{H_n (t_0 - s_2)} + \sqrt{H_n (t_1 - s_2)} + d_{s_2} (z_0, z_1), 
\end{align*}
it suffices to bound $d:=d_{s_2} (z_0, z_1)$.
For this purpose, we split the integral in (\ref{eq_int_eff1}) into two integrals whose domains cover $M$ and obtain, using (\ref{eq_ff1_geq_mC}),
\begin{align}
\nu_{s_2} ( M \setminus & B(z_0, s_2, d/2) ) + \nu^1_{s_2} ( M \setminus B(z_1, s_2, d/2)) \notag \\
&=  \int_{M \setminus B(z_0, s_2, d/2)}  (4\pi \tau)^{-n/2} e^{-f} dg_{s_2 }  + \int_{M \setminus B(z_1, s_2, d/2)}  (4\pi \tau_1)^{-n/2}  e^{-f_1} dg_{s_2 } \notag \\
&\geq c(Y,D,\beta) \bigg( 
 \int_{M \setminus B(z_0, s_2, d/2)}  (4\pi)^{-n} (\tau \tau_1)^{-n/2} e^{-f-f_1} dg_{s_2 } \notag \\ 
 &\qquad\qquad\qquad\qquad + \int_{M \setminus B(z_1, s_2, d/2)}  (4\pi)^{-n} (\tau \tau_1)^{-n/2} e^{-f-f_1} dg_{s_2 } \bigg) \notag \\   &\geq c(Y,D,\beta) > 0. \label{eq_nunu_geq}
\end{align}
On the other hand,
\begin{multline}
 \nu^0_{s_2} ( M \setminus  B(z_0, s_2, d/2) ) + \nu^1_{s_2} ( M \setminus B(z_1, s_2, d/2)) 
\leq \frac{{\Var}_{s_2} (\nu^0_{s_2},\delta_{z_0}) + {\Var_{s_2}} (\nu^1_{s_2}, \delta_{z_1})}{(d/2)^2} \\
\leq \frac{H_n(t_0 - s_2) +  H_n(t_1 - s_2)}{(d/2)^2} \leq \frac{C(\beta)}{d^2}. \label{eq_nunu_leq}
\end{multline}
Combining (\ref{eq_nunu_geq}), (\ref{eq_nunu_leq}) shows that $d \leq C(Y,D, \beta)$, as desired.
\end{proof}

\section{Almost splitting theorems} \label{sec_alm_splitting}
In this section we derive splitting theorems, which are similar to the standard splitting theorem for gradient shrinking solitons, see for instance \cite{Gianniotis_Type_I,Gianniotis_2019}.
A main complication here is that we are not working with precise solitons, but with almost selfsimilar points.
So we will only obtain almost versions of such splitting theorems.
Note also that since we don't have a sufficiently powerful convergence and partial regularity theory for Ricci flows at this point, these almost splitting theorems cannot be obtained by a limit argument.

In summary, we obtain the following two splitting theorems:
\begin{itemize}
\item If two points $(x_1, t_1), (x_2, t_2)$ at bounded distance are  each sufficiently selfsimilar and the times $t_1, t_2$ are sufficiently different, then the flow near both points is $(\eps,r )$-static.
\item If a set of points $(x_0, t_0), \ldots, (x_N, t_N)$ at bounded distance and  with $t_1 \approx \ldots \approx t_N$ are each sufficiently selfsimilar and are at least $\la$-far apart from each other, then whenever $N \geq C \la^{-k}$, these points are  weakly $(k+1, \eps, r)$-split.
\end{itemize}

\subsection{Detection of almost static points}
The following proposition states that if a flow is almost self-similar at two points based at sufficiently different times, then both points are in fact almost static.

\begin{Proposition}[Almost static cone splitting] \label{Prop_almost_static}
If $Y, D < \infty$, $\beta, \eps > 0$ and $\delta \leq \ov\delta (Y, D, \beta, \eps)$, then the following is true.

Let $(M, (g_t)_{t \in I})$ be a Ricci flow on a compact manifold.
Consider $(\delta, r)$-selfsimilar points $(x_1, t_1), \lb (x_2, t_2)  \in M \times I$ such that $\NN_{x_1, t_1}(r^2) \geq - Y$ and $t_1 \leq t_2$, $\beta r^2 \leq t_2 - t_1 \leq \beta^{-1} r^2$.
Assume that for some $s \in [t_1 - \beta^{-1} r^2, t_1] $ we have
\[ d_{W_1}^{g_s} (\nu_{x_1, t_1} (s), \nu_{x_2,t_2} (s) ) \leq Dr. \]
Then $(x_1,t_1)$ is $(\eps, r)$-static (in the sense of Definition~\ref{Def_static}).
\end{Proposition}

\begin{proof}
By Proposition~\ref{Prop_NN_variation_bound} we have $\NN_{x_2, t_2}(r^2) \geq - C(Y, D, \beta)$.
So after adjusting $Y$, we may assume in the following that $W_i := \NN_{x_i, t_i}(r^2) \geq - Y$ for $i = 1,2$.
Next, after parabolic rescaling and application of a time-shift, we may assume that $r=1$ and $t_2 = 0$, $t_1 \leq - \beta$.
We may also assume that $\beta \leq \eps/10$.

Denote by $d\nu^i = (4\pi \tau_i)^{-n/2} e^{-f_i} dg$, $i = 1,2$, the conjugate heat kernel based at $(x_i, t_i)$.
Let $\theta, \theta', \alpha > 0$ be constants whose values we will determine in the course of the proof.
If $\delta  \leq \ov\delta (Y, \beta, \theta', D)$, then Proposition~\ref{Prop_dist_expansion_almost_ss} implies
\[ d_{W_1}^{g_s} (\nu^1 (s), \nu^2 (s) ) \leq D' (Y,D, \beta, \theta') \qquad \text{for all} \quad s \in [t_1 - \beta^{-1}, t_1 - \theta']. \]
Combining this bound with Proposition~\ref{Prop_inheriting_bounds} for $\theta' \leq \ov\theta' (\theta, \beta, \alpha)$ implies that for $\{ i, j \} = \{ 1,2 \}$, $k \in \{ 1,2 \}$
\begin{equation} \label{eq_dnuidnuj}
 d\nu^j \leq C(Y,D, \beta, \theta, \alpha) e^{\alpha f_k} d\nu^i \qquad \text{on} \quad M \times [t_1 - \beta^{-1}, t_1 - \theta]. 
\end{equation}
So if $\alpha \leq \ov\alpha$ is fixed, then by Proposition~\ref{Prop_almost_soliton_identities}
\begin{align}
\int_{t_1 - \beta^{-1}}^{t_1 - \theta} \int_M \Big| \square (\tau_2 f_2) + \frac{n}2 + W_2 \Big| d\nu^1_t dt &\leq \Psi (\delta | Y, D, \beta, \theta), \label{eq_squaretau1f1} \\
 \int_{t_1 - 2\theta}^{t_1 - \theta} \int_M | -  \tau_2 (R + |\nabla f_2|^2) + f_2 - W_2 | d\nu^1_t dt  &\leq \Psi (\delta | Y, D, \beta, \theta), \label{eq_tau1Rnabf1} \\
 \int_M f_2 \, d\nu^1_{ t_1 - 1} \leq C(Y,D, \beta) + (1-\alpha)^{-1} \int_M f_1 \, d\nu^1_{ t_1 - 1} &\leq C(Y, D, \beta).  \label{eq_f1int}
\end{align}
Note that (\ref{eq_f1int}) is a direct consequence of (\ref{eq_dnuidnuj}) and the constant $C$ in (\ref{eq_f1int}) can be chosen independently of $\theta$.
Combining (\ref{eq_squaretau1f1}), (\ref{eq_f1int}) implies that if $\delta \leq \ov\delta (Y, D, \beta, \theta)$, then for any $t^* \in [t_1 - 1, t_1 - \theta]$
\[  \int_M f_2 \, d\nu^1_{t^*}
= \int_M f_2 \, d\nu^1_{ t_1 - 1} + \int_{t_1-1}^{t^*} \int_M \square f_2 \, d\nu^1_t dt
 \leq C(Y, D, \beta) . \]
 Again, note that the constant $C$ can be chosen independently of $\theta$.
Thus by (\ref{eq_tau1Rnabf1}) and assuming $\delta \leq \ov\delta (Y, D, \beta, \theta)$, there is a $t^* \in [t_1 - 2\theta, t_1 - \theta]$ with
\[ \int_M \tau_2 R \, d\nu^1_{ t^*}  \leq C(Y, D, \beta). \]
So since at time $t = t^*$ we have $\tau_1 \leq 2\theta$ and $\tau_2 \geq \beta$, we obtain using Proposition~\ref{Prop_scalar curvature_almost_ss} that if $\delta \leq \ov\delta (Y, \eps, \theta)$, then
\begin{equation} \label{eq_x2t2_R_bound}
  \int_M  \tau_1 R \, d\nu^2_t  
  \leq \int_M  \tau_1 R \, d\nu^2_{t^*} + \theta
  \leq   C(Y, D, \beta) \theta \qquad \text{for all} \quad t \in [t_1 - \eps^{-1}, t^*]. 
\end{equation}
Using this bound for $t = t_1 - \eps$ implies
\begin{equation} \label{eq_x2t2_Ric_bound}
 2 \int_{t_1 - \eps^{-1}}^{t_1 - \eps} \int_M |{\Ric}|^2 d\nu^1_t dt 
\leq \int_M R \, d\nu^1_t \bigg|_{t = t_1 - \eps^{-1}}^{t = t_1 - \eps}  
\leq  C(Y, D, \beta) \frac{\theta}{\eps} + \delta . 
\end{equation}
If $\theta \leq \ov\theta  (Y, D, \beta, \eps)$ and $\delta \leq \ov\delta (\eps)$, then the right-hand sides in (\ref{eq_x2t2_R_bound}), (\ref{eq_x2t2_Ric_bound}) can be made less than $\eps$, proving that $(x_1, t_1)$ is $(\eps, 1)$-static.
%
\end{proof}
\bigskip

\subsection{\texorpdfstring{Almost splitting off an $\IR^k$-factor}{Almost splitting off an R\^{}k-factor}}
The following proposition states that if a flow is almost self-similar with respect to $> C \la^{-k}$ many points whose parabolic distance is $> D_0 \la$ in a certain sense, then the flow must weakly split off an $\IR^k$-factor in the sense of Definition~\ref{Def_weak_splitting_map}.

\begin{Proposition}[Almost $\IR^k$-splitting] \label{Prop_almost_split_Rk}
If
\begin{multline*}
Y, D < \infty, \qquad
\eps > 0, \qquad
\beta \leq \ov\beta, \qquad 
D_0 \geq \underline{D}_0 (Y), \qquad
C_0 \geq \underline{C}_0 (Y, D), \qquad \\
 \la \leq \ov\la (\eps), \qquad 
\delta \leq \ov\delta (Y, D, \eps, \la),  
\end{multline*}
then the following is true.
Let $(M, (g_t)_{t \in I})$ be a Ricci flow on a compact manifold.
Consider $(\delta, r)$-selfsimilar points $(x_0, t_0), \lb \ldots, \lb (x_{N-1}, t_{N-1}) \lb \in M \times I$ with $t_0 \leq \ldots \leq t_{N-1}$, where for some $k \in \{ 0, \ldots, n-1 \}$ we have
\[ N \geq C_0 \la^{-k}. \]
Denote by $d\nu^i = (4\pi \tau_i)^{-n/2} e^{-f_i} dg$ the conjugate heat kernel based at $(x_i, t_i)$ and write $d\nu = d\nu^0 = (4\pi \tau)^{-n/2} e^{-f} dg$.
Assume that for all $i, j = 0, \ldots, N-1$ we have $\NN_{x_i, t_i} (r^2) \geq - Y$ and
\begin{align*}
 t_i &\in [t_0- \beta (\la r)^2, t_0 + \beta (\la r)^2], \\
 d_{W_1}^{g_{ t_0-r^2}} (\nu^i_{ t_0-r^2}, \nu^0_{t_0 -10r^2}) &\leq Dr, \\
 d_{W_1}^{g_{t_0 - 2\la^2 r^2}} (\nu^i_{t_0 - 2\la^2 r^2}, \nu^j_{ t_0 - 2\la^2 r^2}) &\geq  D_0 \la r \qquad \text{if} \quad i \neq j. 
\end{align*}
Then $(x_0, t_0)$ is weakly $(k+1, \eps, r)$-split (in the sense of Definition~\ref{Def_weak_splitting_map}).
\end{Proposition}

\begin{Remark}
With a little more effort it is possible to make the lower bound $\underline{D}_0$ independent of $Y$.
This will, however, not be needed in the sequel. 
\end{Remark}

The following lemmas will be used in the proof of Proposition~\ref{Prop_almost_split_Rk}

\begin{Lemma} \label{Lem_nabu_almost_constant}
If $Y, A  < \infty$, $\eps, \delta^* > 0$ and $\delta \leq \ov\delta (Y, A,\eps, \delta^*)$, then the following holds.

Let $(M, (g_t)_{t \in [-\delta^{-1} r^2,0]})$, $r > 0$, be a Ricci flow on a compact manifold and denote by $d\nu = (4\pi \tau)^{-n/2} e^{-f} dg$ the conjugate heat kernel based at $(x,0)$ for some $x \in M$.
Suppose that we have $W := \NN_{x,0} ( r^2) \geq - Y$ and that $(x,0)$ is $(\delta, r)$-selfsimilar.

Consider a scalar function $u \in C^\infty (M \times [t_1, t_2])$, where $[t_1, t_2] \subset [- \eps^{-1} r^2, - \eps r^2]$
\begin{align}
  \int_{t_1}^{t_2} \int_M (\tau^{-1/2} |\partial_t u| +  |\nabla^2 u|^2) d\nu_t dt &\leq \delta, \label{eq_u_bound_delta} \\
 \int_{t_1}^{t_2} \int_M (\tau^{-5} u^8 + \tau^{-1} |\nabla u|^4 + (\partial_t u)^2) d\nu_t dt &\leq A. \label{eq_u_bound_A}
\end{align}
Then there is some $q \in \IR$ such that
\[ \int_{t_1}^{t_2} \int_M \tau^{-1} \big| |\nabla u|^2 - q \big| d\nu_t dt  \leq \delta^*. \]
\end{Lemma}

Note we have chosen the powers of $\tau$ such that $u$ has the dimension of length.

\begin{proof}
After parabolic rescaling of the flow by $r^{-2}$, and multiplying $u$ by $r^{-1}$, we may assume that $r=1$.
Let $[t_1, t_2] \subset [-\eps^{-1}, -\eps]$, $t_1 < t_2$, be an arbitrary interval.
Our constants will be independent of $t_1, t_2$, but possibly depend on $\eps$.
We will use the following bounds:
\begin{Claim} For any $b > 0$ and any $t \in [- \eps^{-1}, - \eps ]$ we have
\begin{align}
\int_{t_1}^{t_2} \int_M (|u| + |\nabla u| ) |\nabla^2 u| d\nu_t dt &\leq \Psi (\delta |A, \eps), \label{eq_intint_bounds_1} \\
\int_{t_1}^{t_2} \int_M u^2 \big| \tau (2\triangle f - |\nabla f|^2 + R )+f-n-W \big| d\nu_t dt  &\leq \Psi (\delta |Y, A, \eps), \label{eq_intint_bounds_2} \\
\int_{t_1}^{t_2} \int_M u^2 \Big| \tau (R + \triangle f) - \frac{n}{2} \Big| d\nu_t dt  &\leq \Psi (\delta |Y, A, \eps), \label{eq_intint_bounds_3} \\
\int_{t_1}^{t_2} \int_M |u \square u|  \Big|  f - \frac{n}{2}  - W \Big| d\nu_t dt  &\leq \Psi (\delta |Y, A, \eps), \label{eq_intint_bounds_4} \\
\int_M \bigg| |\nabla u|^2 - \int_M |\nabla u|^2 d\nu_t \bigg|  \Big|  f - \frac{n}{2}  - W \Big| d\nu_t   &\leq b^{-1} \int_M  \bigg| |\nabla u|^2 - \int_M |\nabla u|^2 d\nu_t \bigg|  d\nu_t \notag \\ &\qquad\qquad\qquad\qquad\qquad + C(Y,A, \eps)b. \label{eq_intint_bounds_5}
\end{align}
\end{Claim}

\begin{proof}
To see (\ref{eq_intint_bounds_1}) we use (\ref{eq_u_bound_delta}), (\ref{eq_u_bound_A}) to see that
\begin{multline*}
 \int_{t_1}^{t_2} \int_M (|u| + |\nabla u|) |\nabla^2  u| d\nu_t dt 
\leq C \bigg( \int_{t_1}^{t_2} \int_M (u^2 + |\nabla u|^2) d\nu_t dt \bigg)^{1/2} \bigg( \int_{t_1}^{t_2} \int_M |\nabla^2 u|^2 d\nu_t dt \bigg)^{1/2} \\
 \leq \Psi (\delta | A , \eps). 
\end{multline*}
To see (\ref{eq_intint_bounds_2})--(\ref{eq_intint_bounds_5}), denote by $v_1 \in C^\infty (M \times [t_1, t_2])$ one of the following terms
\[ \big| \tau (2\triangle f - |\nabla f|^2 + R )+f-n-W \big|, \; \Big| \tau (R + \triangle f )- \frac{n}{2} \Big|, \; |\square u|, \;  \bigg| |\nabla u|^2 - \int_M |\nabla u|^2 d\nu_t \bigg|  \]
and let correspondingly $v_2 \in C^\infty (M \times [t_1, t_2])$ denote
\[ u^2, \; u^2\;, |u| \Big|   f - \frac{n}{2}  - W \Big|, \; \Big|   f - \frac{n}{2}  - W \Big| \]
so that the integral to be bounded is of the form $\int_{t_1}^{t_2} \int_M v_1 v_2 d\nu_t dt$ or $\int_M v_1 v_2 d\nu_t$ in (\ref{eq_intint_bounds_5}).
To see (\ref{eq_intint_bounds_2})--(\ref{eq_intint_bounds_4}), we use the almost selfsimilar property of $(x,0)$, Proposition~\ref{Prop_improved_L2}, (\ref{eq_u_bound_delta}), (\ref{eq_u_bound_A}) and the bound $|u| \, |f - \frac{n}2 - W| \leq u^2 + (f - \frac{n}2 - W)^2$ to conclude
\begin{multline*}
 \int_{t_1}^{t_2} \int_M v_1 v_2 \, d\nu_t \leq b^{-1} \int_{t_1}^{t_2} \int_M v_1 \,  d\nu_t dt + b \int_{t_1}^{t_2} \int_M v_1 v_2^2 \, d\nu_t dt \\
\leq \Psi (\delta |\eps,b) + b \int_{t_1}^{t_2} \int_M v_1^2 \,  d\nu_t dt + b \int_{t_1}^{t_2} \int_M  v_2^4\, d\nu_t dt 
\leq \Psi (\delta |\eps,b) + C(Y, A,\eps) b.
\end{multline*}
The inequalities (\ref{eq_intint_bounds_2})--(\ref{eq_intint_bounds_4}) now follow by choosing $b$ appropriately.
Similarly, to see (\ref{eq_intint_bounds_5}), we conclude
\[
 \int_M v_1 v_2 \, d\nu_t \leq b^{-1}  \int_M v_1 \, d\nu_t  + b  \int_M v_1^2 \, d\nu_t  + b  \int_M  v_2^4 \, d\nu_t 
\leq b^{-1}  \int_M v_1 \, d\nu_t  + C(Y, A,\eps) b. \qedhere
\]
\end{proof}
\medskip

Let us now focus on the last inequality of the Claim.
For any $t \in [- \eps^{-1}, - \eps ]$ we have
\begin{multline*}
2 \int_M |\nabla u|^2 d\nu_t 
=  \int_M \triangle u^2 d\nu_t - 2 \int_M u \triangle u \, d\nu_t 
=  \int_M u^2 (-\triangle f + |\nabla f|^2 ) d\nu_t - 2 \int_M u \triangle u \, d\nu_t  \\
=\int_M  \tau^{-1}  u^2 \Big(f- \frac{n}2 - W \Big) d\nu_t -  \int_M \tau^{-1} u^2 \big(\tau (2\triangle f - |\nabla f|^2 + R )+f-n-W \big) d\nu_t  \\  +  \int_M u^2 \Big( R + \triangle f - \frac{n}{2\tau} \Big) d\nu_t - 2 \int_M u \triangle u \, d\nu_t .
\end{multline*}
Integration and application of the Claim implies
\begin{equation} \label{eq_nabu_approx_0_order}
  \int_{t_1}^{t_2}  \bigg| 2\int_M |\nabla u|^2 d\nu_t - \int_M  \tau^{-1} u^2 \Big(f- \frac{n}2 - W \Big) d\nu_t  \bigg| dt \leq  \Psi (\delta | Y,A, \eps). 
\end{equation}
By the $L^1$-Poincar\'e inequality (Proposition~\ref{Prop_Poincare}) and the Claim we have
\begin{equation} \label{eq_nabu_m_nabu_average}
  \int_{t_1}^{t_2}  \int_M \bigg| |\nabla u|^2 - \int_M |\nabla u|^2 d\nu_t \bigg| d\nu_t dt \leq C \int_{t_1}^{t_2}  \tau^{1/2} \int_M |\nabla^2 u| |\nabla u| d\nu_t dt   \leq \Psi (\delta | A, \eps) .
\end{equation}
Next we compute, using integration by parts and (\ref{eq_potential_evolution_equation}),
\begin{align}
\tau^2 \frac{d}{dt} \bigg(  \int_M & \tau^{-1} u^2 \Big( f- \frac{n}2 - W \Big) d\nu_t \bigg)
= \int_M u^2 \Big( f- \frac{n}2 - W \Big) d\nu_t +\tau \int_M \square \Big(u^2 \Big(f- \frac{n}2 - W \Big) \Big) d\nu_t \notag \\
&= \int_M \Big( (u^2 + 2 \tau u \square u - 2 \tau |\nabla u|^2) \Big( f- \frac{n}2 - W \Big) - 4 \tau u \nabla u \cdot \nabla f + \tau u^2 \square f \Big) d\nu_t \notag  \\
&= \int_M \Big( (u^2 + 2 \tau u \square u - 2 \tau |\nabla u|^2) \Big( f- \frac{n}2 - W \Big) \notag  \\
&\qquad\qquad\qquad + 2 \tau u^2 ( \triangle f - | \nabla f|^2) + \tau u^2 \Big( -2 \triangle f + |\nabla f|^2 - R + \frac{n}{2\tau} \Big) \Big) d\nu_t \notag  \\
&= \int_M u^2  \big(\tau (2\triangle f - |\nabla f|^2 + R )+f-n-W \big) d\nu_t  
- 2\tau \int_M u^2  \Big(R + \triangle f -  \frac{n}{2\tau} \Big) d\nu_t \notag  \\
&\qquad+2 \tau \int_M u \square u \Big( f  - \frac{n}{2} - W  \Big) d\nu_t 
 - 2\tau \int_M \bigg( |\nabla u|^2 - \int_M |\nabla u|^2 d\nu_t \bigg) \Big( f  - \frac{n}{2} - W  \Big) d\nu_t \notag \\
 &\qquad - 2\tau  \bigg( \int_M |\nabla u|^2 d\nu_t \bigg) \int_M  \Big( f  - \frac{n}{2} - W  \Big) d\nu_t . \label{eq_ddtau_f_n2_W_u2}
\end{align}
By Proposition~\ref{Prop_NN_almost_constant_selfsimilar}, the very last integral can be bounded as follows
\begin{equation} \label{eq_int_NN_bound}
 \bigg|  \int_M  \Big( f  - \frac{n}{2} - W  \Big) d\nu_t \bigg| = \big| \NN_{x,0} (|t|) - W \big|  = \big| \NN_{x,0} (|t|) - \NN_{x,0} (1) \big| \leq \Psi (\delta | Y, \eps). 
\end{equation}
So integration of (\ref{eq_ddtau_f_n2_W_u2}) and application of the Claim, (\ref{eq_u_bound_A}), (\ref{eq_int_NN_bound}) and (\ref{eq_nabu_m_nabu_average}) implies that for any $[t^*_1, t^*_2] \subset [t_1, t_2]$
\begin{multline*}
 \bigg| \int_M \tau^{-1} u^2 \Big( f- \frac{n}2 - W \Big) d\nu_t \bigg|_{t=t^*_1}^{t=t^*_2} \bigg| \leq b^{-1} \int_{t^*_1}^{t^*_2} \int_M 2\tau^{-1} \bigg| |\nabla u|^2 -  \int_M |\nabla u|^2  d\nu_t \bigg|  d\nu_t dt \\ + C(Y,A, \eps)  b + \Psi (\delta |Y, A, \eps) 
   \leq  C(Y,A, \eps)  b +   \Psi (\delta | Y,A, \eps, b).
\end{multline*}
Choosing $b$ appropriately, implies that there is a $q \in \IR$ such that for all $t \in [t_1, t_2]$
\[ \bigg|  \int_M \tau^{-1} u^2 \Big( f- \frac{n}2 - W \Big) d\nu_t - 2 q \bigg| \leq \Psi (\delta | Y,A, \eps). \]
The lemma now follows by combining this bound with (\ref{eq_nabu_approx_0_order}), (\ref{eq_nabu_m_nabu_average}).
\end{proof}
\bigskip

\begin{Lemma} \label{Lem_Eucl_vs_covering}
Let $(V, (\cdot, \cdot))$ be a Euclidean vector space and consider vectors $v_1, \ldots, v_N \in V$ with
\[ |v_i| \leq A, \qquad \text{and} \qquad  |v_i - v_j| \geq \la, \quad \text{if} \quad i \neq j, \]
where $0 < \la \leq 1$ and $A \geq 1$.
If $N \geq C_0 (A,k) \la^{-k}$ for some $k \geq 0$, then there are coefficients $(a_{li})_{ l=1,\ldots, k+1, i = 1, \ldots, N}$ with $|a_{li}| \leq C_1 (A, \la, k)$ such that
\[ \Big(  \sum_{i=1}^{N} a_{li} v_i \Big)_{l=1, \ldots, k+1} \]
is orthonormal.
\end{Lemma}

\begin{proof}
The proof is by induction on $k$.
The lemma is trivial for $k =0$, so assume that $k \geq 1$ and that the lemma holds for $k-1$.
Assuming that $C_0 (A, k+1) \geq C_0 (A,k)$, we can find an orthonormal system of the form
\[ \Big( u_l := \sum_{i=1}^{N} a_{li} v_i \Big)_{l=1, \ldots, k}. \]
Let $W \subset V$ be its span, consider the orthogonal decomposition $V = W \oplus W^\perp$ and express every vector as $v_i =: (w_i, w^\perp_i)$.
By a simple volume comparison argument, we obtain that for sufficiently large $C_0 (A, k+1)$ we must have $b := |w^\perp_i| \geq \la/10$ for some $i \in \{ 1, \ldots, N \}$.
Then
\[ u_{k+1} := b^{-1} ( \vec 0 , w^\perp_i ) := b^{-1} \Big( v_i - \sum_{l=1}^k (v_i, u_l) u_l \Big) \]
completes $(u_l)_{l =1, \ldots, k}$ to an orthonormal system of cardinality $k+1$ and it can be expressed as a linear combination in the vectors $v_1, \ldots, v_N$ whose coefficients are bounded by a constant of the form $C_1 (A, \la, k+1)$.
\end{proof}
\bigskip

\begin{proof}[Proof of Proposition~\ref{Prop_almost_split_Rk}.]
After parabolic rescaling, we may assume without loss of generality that $t_0 = 0$ and $r = 1$; note that this implies that $t_1, \ldots, t_{N-1} \geq 0$.
We may also assume that $\la^2 \leq \eps$ and $\beta \leq \frac1{10}$.

\begin{Claim} \label{Cl_dnui_dnuj_splitting}
If $-\infty < \eta_1 < \eta_2 < 1$, $\theta \in [\la^2, 1]$, $\eps' \in [\eps,1]$, $\beta \leq \ov\beta( \eta_1, \eta_2)$, $\delta \leq \ov\delta (Y, D, \eps' , \theta, \eta_1, \eta_2)$, then for all $i, j = 0, \ldots, N-1$
\begin{equation}
 e^{\eta_1 f_i} d\nu^i \leq C(Y,D, \eps', \theta, \eta_1, \eta_2) e^{\eta_2 f_j} d\nu^j \qquad  \text{on} \quad M \times [-\eps^{\prime -1}, -\theta]. \label{eq_dnu_bound}
\end{equation}
\end{Claim}

\begin{proof}
Let $\beta' > 0$ be a constant whose value we will determine later.
Using Proposition~\ref{Prop_dist_expansion_almost_ss} we find that if $\delta \leq \ov\delta (Y, D, \eps', \theta, \beta')$, then
\[
 d_{W_1}^{g_t} (\nu^i_{ t}, \nu^0_{t}) \leq D'_1 (Y, D, \eps', \theta, \beta') \qquad \text{for all} \quad t \in [-\eps^{\prime -1}, - \beta' \theta ].
\]
Therefore, Proposition~\ref{Prop_inheriting_bounds} implies (\ref{eq_dnu_bound}) if $\beta \leq \ov\beta (\eta_1, \eta_2)$, $\delta \leq \ov\delta (\eps')$ and $\beta' \leq \ov\beta' (\eta_1, \eta_2)$.
\end{proof}

For all $i = 0, \ldots, N-1$ set
\[ W_i := \NN_{x_i, t_i} (r^2) \]
and consider the functions
\[ u_{i} :=  \tau_i f_i  - \tau f  - (W_i - W_0) \tau, \qquad i = 1, \ldots, N-1. \]

\begin{Claim} \label{Cl_bounds_zi}
If $\alpha \leq \ov\alpha$, $\beta \leq \ov\beta$ and $\delta \leq \ov\delta (Y, D, \la)$, then for all $i = 1, \ldots, N-1$
\begin{align}
  \int_{-\eps^{-1}}^{-\la^2} \int_M (\tau^{-1/2} |\partial_t u_i| +  |\nabla^2 u_i|^2) e^{\alpha f} d\nu_t dt &\leq \Psi(\delta |Y, D, \la), \label{eq_zi_delta_bound} \\
 \int_{-\eps^{-1} }^{-\la^2} \int_M ( \tau^{-5} u_i^8 +\tau^{-1} |\nabla u_i|^4 + (\partial_t u_i)^2) e^{ \alpha f}  d\nu_t dt &\leq C(Y, D, \la) . \label{eq_zi_C_bound} 
\end{align}
\end{Claim}

\begin{proof}
By Proposition~\ref{Prop_improved_L2} and (\ref{eq_dnu_bound}) we have, assuming $\alpha \leq \ov\alpha$, $\beta \leq \ov\beta$ and $\delta \leq \ov\delta (Y, D, \la)$
\begin{equation} \label{eq_zi_better_C_bound}
  \int_{-\eps^{-1} }^{-\la^2} \int_M (|\partial_t u_i| +  |\nabla^2 u_i|^2 + u_i^8 + |\nabla u_i|^4 + (\partial_t u_i)^2) e^{ 2\alpha f}  d\nu_t dt  \leq C(Y, D, \la). 
\end{equation}
This proves (\ref{eq_zi_C_bound}).
In order to establish (\ref{eq_zi_delta_bound}) we only have to show that 
\begin{equation} \label{eq_zi_desired_bound}
  \int_{-\eps^{-1}}^{-\la^2} \int_M (|\partial_t u_i| +  |\nabla^2 u_i|^2)  d\nu_t dt \leq \Psi(\delta |Y, D, \la), 
\end{equation}
because for any $b > 0$ the bounds (\ref{eq_zi_better_C_bound}), (\ref{eq_zi_desired_bound}) imply
\begin{multline*}
   \int_{-\eps^{-1}}^{-\la^2} \int_M (\tau^{-1/2} |\partial_t u_i| +  |\nabla^2 u_i|^2) e^{\alpha f} d\nu_t dt \\
\leq b  \int_{-\eps^{-1}}^{-\la^2} \int_M ( \tau^{-1/2} |\partial_t u_i| +  |\nabla^2 u_i|^2) e^{2\alpha f} d\nu_t dt + b^{-1}  \int_{-\eps^{-1}}^{-\la^2} \int_M (\tau^{-1/2} |\partial_t u_i| +  |\nabla^2 u_i|^2)  d\nu_t dt \\
\leq C(Y, D, \la) b + \Psi(\delta |Y, D, \la) b^{-1}.
\end{multline*}
To see (\ref{eq_zi_desired_bound}) we first observe that for $\eta \leq \ov\eta$, $\beta \leq \ov\beta (\eta)$ and $\delta \leq \ov\delta (Y, D, \la, \eta)$ we have by Claim~\ref{Cl_dnui_dnuj_splitting} and Proposition~\ref{Prop_almost_soliton_identities}
\begin{multline*}
  \int_{-\eps^{-1}}^{-\la^2} \int_M |\square u_i|   d\nu_t dt \leq
 C(Y,D, \la, \eta)  \int_{-\eps^{-1}}^{-\la^2} \int_M \Big| \square (\tau_i f_i ) + \frac{n}2 + W_i \Big|  e^{\eta f_i}  d\nu^i_t dt  \\ + \int_{-\eps^{-1}}^{-\la^2} \int_M \Big| \square (\tau f ) + \frac{n}2 + W_0 \Big|  d\nu_t dt 
\leq  \Psi(\delta |Y, D, \la, \eta).
\end{multline*}
Similarly, using Claim~\ref{Cl_dnui_dnuj_splitting} and Proposition~\ref{Prop_almost_soliton_identities},
\begin{align*}
 \int_{-\eps^{-1}}^{-\la^2} \int_M  &  |\nabla^2 u_i|^2 d\nu_t dt \\
&= \int_{-\eps^{-1}}^{-\la^2} \int_M   \Big| \tau_i \Big( \nabla^2 f_i + \Ric - \frac1{2\tau_i} g \Big) -  \tau \Big( \nabla^2 f + \Ric - \frac1{2\tau} g \Big) + (\tau - \tau_i) \Ric  \Big|^2 d\nu_t dt \\
&\leq C(Y, D, \la, \eta) \int_{-\eps^{-1}}^{-\la^2} \int_M  \tau_i^2  \Big| \nabla^2 f_i + \Ric - \frac1{2\tau_i} g \Big|^2 e^{\eta f_i} d\nu^i_t dt \\
&\qquad +  4\int_{-\eps^{-1}}^{-\la^2} \int_M  \tau^2  \Big| \nabla^2 f + \Ric - \frac1{2\tau} g \Big|^2  d\nu_t dt 
+ 4 |t_i|^2 \int_{-\eps^{-1}}^{-\la^2} \int_M |{\Ric}|^2 d\nu_t dt \\
&\leq \Psi (\delta | Y, D, \la, \eta) + 4 |t_i|^2 \int_{-\eps^{-1}}^{-\la^2} \int_M |{\Ric}|^2 d\nu_t dt.
\end{align*}
Let $\beta_0 > 0$ be a constant whose value we will determine in a moment.
If $0 \leq t_i \leq \beta_0$, then by Proposition~\ref{Prop_improved_L2}
\[  \int_{-\eps^{-1}}^{-\la^2} \int_M   |\nabla^2 u_i|^2 d\nu_t dt \leq \Psi (\delta |Y, D, \lambda, \eta) + C(Y,  \la) \beta_0^2. \]
If $t_i \geq \beta_0$, then by Proposition~\ref{Prop_almost_static} we have
\[ \int_{-\eps^{-1}}^{-\la^2} \int_M |{\Ric}|^2 d\nu_t dt \leq \Psi (\delta | Y, D, \la, \beta_0). \]
Therefore,
\[ \int_{-\eps^{-1}}^{-\la^2} \int_M    |\nabla^2 u_i|^2 d\nu_t dt
 \leq C(Y,  \la) \beta_0^2 +  \Psi (\delta | Y, D, \la, \eta, \beta_0). \]
Choosing $\beta_0$ appropriately, implies that the left-hand side is bounded by a term of the form $\Psi(\delta |Y, D, \la, \eta)$.
Since the only restriction on $\eta$ was a bound of the form $\eta \leq \ov\eta$, this proves the claim.
\end{proof}

\begin{Claim} \label{Cl_nabzizj_qij}
There are constants $q_{ij} \in \IR$ such that
\[ \int_{-\eps^{-1}}^{-\la^2} \int_M \big| \nabla u_i \cdot \nabla u_j - q_{ij} \big| d\nu_t dt \leq \Psi (\delta | Y, D, \la). \]
\end{Claim}

\begin{proof}
This follows by applying Lemma~\ref{Lem_nabu_almost_constant} to $u = u_i \pm u_j$.
\end{proof}

\begin{Claim} \label{Cl_nabzizj_geq_la2}
If $D_0 \geq \underline{D}_0 (Y)$, $\delta \leq \ov\delta (Y, D, \la)$, then for any $i \neq j$ we have
\[ \int_{-2}^{-1} \int_M  |\nabla (u_{i} - u_{j}) |^2 d\nu_t dt \geq  \la^2. \]
\end{Claim}

\begin{proof}
Assume that the claim was wrong for some $i \neq j$.
Then by Claim~\ref{Cl_nabzizj_qij} if $\delta \leq \ov\delta (Y, D,  \la)$, we have for some $q \leq 2 \la^2$
\[ \int_{-2\la^2}^{-\la^2} \int_M \big| |\nabla (u_{i} - u_{j}) |^2 -q \big| d\nu_t dt \leq \Psi (\delta | Y,D, \la). \]
Therefore, using Proposition~\ref{Prop_almost_soliton_identities}, we can find some time $t^* \in [-2\la^2, -\la^2]$ at which for $l = i,j$
\begin{align}  \int_M  \big| | \nabla ( \tau_i f_{i} - \tau_j f_{j}) |^2 -q \big| d\nu_{ t^*} &= \int_M \big| |\nabla (u_{i} - u_{j}) |^2 -q \big| d\nu_{ t^*} \leq \Psi (\delta |Y, D, \la), \label{eq_tauifitaujfj} \\
 \int_M \big| - \tau_l (  |\nabla f_l|^2 + R) + f_l  - W_l \big|  d\nu^l_{t^*}  &\leq \Psi (\delta | Y,  \la).  \label{eq_flRflWl}
 \end{align}
Let $(z_i, t^*), (z_j, t^*)$ be $H_n$-centers of $(x_i, t_i)$, $(x_j, t_j)$, respectively.
Then if $D_0 \geq \underline{D}_0$
\begin{multline} \label{eq_dtstart_12_D_0}
 d_{t^*} (z_i , z_j) \geq d_{W_1}^{g_{t^*}}(\nu^i_{t^*}, \nu^j_{t^*}) -d_{W_1}^{g_{t^*}} (\delta_{z_i}, \nu^i_{t^*}) - d_{W_1}^{g_{t^*}} (\delta_{z_j}, \nu^j_{t^*}) \\
 \geq d_{W_1}^{g_{-2\la^2}}(\nu^i_{-2\la^2}, \nu^j_{-2\la^2}) -\sqrt{{\Var}_{t^*} (\delta_{z_i}, \nu^i_{t^*})} - \sqrt{{\Var}_{t^*} (\delta_{z_j}, \nu^j_{t^*})}  
\geq \big(D_0 - 2 \sqrt{2H_n} \big) \la \geq \tfrac12 D_0 \la.
\end{multline}

After possibly switching the roles of $i, j$, we may assume that
\begin{equation} \label{eq_taui_geq_tau_j}
 \tau_i \geq \tau_j. 
\end{equation}
Set $B_i := B(z_i, t^*,2 \sqrt{ H_n } \la)$.
Then by (\ref{eq_dtstart_12_D_0}), if $D_0 \geq \underline{D}_0$ we have
\begin{equation} \label{eq_nu_Bi}
 \nu^j_{t^*} (  B_{i} ) \leq \nu^j_{t^*} \big( M \setminus B(z_j, t^*, \tfrac14 D_0 \la) \big) 
\leq   \frac1{(\tfrac14 D_0 \la)^2} {\Var}_{t^*} (\nu^j_{t^*}, \delta_{z_j})
\leq   \frac{32 H_n}{D_0^2} , \qquad \nu^i_{t^*} (  B_{i} ) 
 \geq \frac12 . 
\end{equation}
Let $Z < \infty$ be a constant whose value we will determine later and set
\[ S_i := \{ f_i \leq Z \} \cap B_i \]
Since by \cite[\HKThmUpperVolBound]{Bamler_HK_entropy_estimates}
\[ |B_i|_{t^*} \leq C  \la^n, \]
we obtain that if $Z \geq \underline{Z}$, then
\[ \nu^i_{t^*} ( B_i \setminus S_i ) 
\leq C \la^{-n} e^{-Z} |B_i|_{t^*}
\leq C  e^{-Z} \leq \frac18. \]
Using Claim~\ref{Cl_dnui_dnuj_splitting} and assuming $\beta \leq \ov\beta$, $\delta \leq \ov\delta (Y, D, \la)$, this implies that 
\begin{equation} \label{eq_ffifj_upper_bound}
 f, f_i, f_j \leq C(Y, D, \la, Z) \qquad \text{on} \quad S_i .
\end{equation}
Since by Proposition~\ref{Prop_L_infty_HK_bound}, we have $f_i , f_j \geq - C(Y)$ at time $t^*$, we obtain that on $S_i$ we have for all $l' \in \{0, i, j \}$
\[ c(Y, D, \la, Z)dg_{t^*} \leq d\nu^{l'}_{t^*} \leq C(Y, D, \la, Z)dg_{t^*}. \]
Therefore, by (\ref{eq_tauifitaujfj}), (\ref{eq_flRflWl}), we can find a further  subset $S'_i \subset S_i$ with the property that on $S'_i$ we have for $l =i,j$ at time $t^*$
\begin{align*}
 \big| |\tau_i \nabla f_i - \tau_j \nabla f_j|^2 - q \big| &\leq \Psi (\delta | Y, D, \la, Z), \\
 \big| - \tau_l (  |\nabla f_l|^2 + R) + f_l  - W_l  \big| &\leq \Psi (\delta | Y, D, \la, Z), \\
 \nu^l_{t^*} \big( S_i \setminus S'_i \big) 
&\leq \frac18. 
\end{align*}
By combining these bounds with (\ref{eq_ffifj_upper_bound}) and the bounds $-Y \leq W_l \leq 0$, $\frac1{10} \la^2 \leq \tau_l \leq 10 \la^2$, we obtain that on $S'_i$ the following holds at time $t^*$ if $\delta \leq \ov\delta (Y, D, \la, Z)$.
\[ \tau_i |\nabla f_i |^2
 \leq  \tau_i (|\nabla f_i |^2 + R) + W_i + C(Y)
 \leq f_i + C(Y) 
 \leq C(Y, Z), \]
 \begin{multline*}
 \tau_j |\nabla f_j |^2
\leq 10 \la^{-2} \tau_j^2 |\nabla f_j |^2 
\leq 20 \la^{-2} \big(   \tau_i^2 |\nabla f_i |^2 +  | \tau_i \nabla f_i - \tau_j \nabla f_j|^2 \big) \\
\leq 20  \la^{-2} \big( C(Y,Z) \la^2 + 2\la^2 \big) 
\leq C(Y,Z),
\end{multline*}
\begin{align*}
 | \tau_i f_i - \tau_j f_j - (\tau_i^2 - \tau_j^2) R |
& \leq 
 \big| \tau_i^2 |\nabla f_i|^2 - \tau_j^2 |\nabla f_j|^2 \big| + 
 |\tau_i W_i - \tau_j W_j | +  \Psi (\delta | Y, D, \la, Z) \\
 &\leq   C(Y, Z) \la \, \big| \tau_i |\nabla f_i| - \tau_j |\nabla f_j| \big| + C(Y) \la^2 +  \Psi (\delta | Y, D, \la, Z) \\
&\leq
 C(Y, Z) \la q^{1/2} + C(Y) \la^2 + \Psi (\delta | Y, D, \la, Z) \leq C(Y,Z) \la^2. 
\end{align*}
By (\ref{eq_taui_geq_tau_j}) we have $(\tau_i^2 - \tau_j^2 )R \geq - 10 \delta \la^2$.
Since $\tau_i / \tau_j \leq 100$ we obtain that $f_j \leq C(Y,Z) \leq f_i + C(Y,Z)$ on $S'_i$.
Therefore, using (\ref{eq_nu_Bi}),
\[ \frac14 \leq \nu^i_{t^*} (S'_i) =  \int_{S'_i} (4\pi \tau_i)^{-n/2} e^{-f_i} dg_{t^*}
\leq  e^{C(Y,Z)} \int_{S'_i} (4 \pi \tau_j)^{-n/2} e^{-f_j} dg_{t^*}
= e^{C(Y,Z)} \nu^j_{t^*} (S'_i) 
\leq e^{C(Y,Z)}  \frac{32 H_n}{D_0^2}, \]
which yields the desired contradiction for $D_0 \geq \underline{D}_0(Y,Z)$.
\end{proof}

\begin{Claim} \label{Cl_nab_ui_L2_bound}
For all $i = 1, \ldots, N-1$ we have
\[ \int_{-2}^{-1} \int_M  |\nabla u_{i}  |^2 d\nu_t dt \leq  C(Y, D). \]
\end{Claim}

\begin{proof}
Using Proposition~\ref{Prop_improved_L2} and Claim~\ref{Cl_dnui_dnuj_splitting}, we have for $\alpha \leq\ov\alpha$
\begin{multline*}
 \int_{-2}^{-1} \int_M  |\nabla u_{i}  |^2 d\nu_t dt 
 \leq  C \int_{-2}^{-1} \int_M  |\nabla f_{i}  |^2 d\nu_t dt 
 + C \int_{-2}^{-1} \int_M  |\nabla f  |^2 d\nu_t dt \\
 \leq  C(Y,D) \int_{-2}^{-1} \int_M  |\nabla f_{i}  |^2 e^{\alpha f_i} d\nu^i_t dt 
 + C \int_{-2}^{-1} \int_M  |\nabla f  |^2 d\nu_t dt 
\leq C(Y, D). \qedhere
\end{multline*}
\end{proof}
\medskip

Now consider the vector space $V \subset C^\infty (M \times [-2,-1])$ spanned by the functions
\[ \td{u}_i := u_i \big|_{M \times [-2,-1]} - \int_{-2}^{-1} \int_M u_i d\nu_t dt, \]
equipped with the scalar product
\[ (v_1,v_2) := \int_{-2}^{-1} \int_M \nabla v_1 \cdot \nabla v_2 \, d\nu_t dt \]
and norm $| v | := \sqrt{(v,v)}$.
By Claim~\ref{Cl_nab_ui_L2_bound} we have $| \td{u}_i | \leq C(Y, D)$ and by Claim~\ref{Cl_nabzizj_geq_la2} we have $| \td{u}_i - \td{u}_j | \geq  \la$ if $i \neq j$.
By Lemma~\ref{Lem_Eucl_vs_covering}, assuming $C_0 \geq \underline{C}_0 (Y, D)$, there are coefficients $(a_{li})_{ l=1,\ldots, k+1, i = 1, \ldots, N-1}$ with $|a_{li}| \leq C(Y,D,\la)$ such that
\[ \Big( y_l := \sum_{i=1}^{N-1} a_{li} \td{u}_i \Big)_{l=1, \ldots, k+1} \]
is orthonormal.
By Claim~\ref{Cl_nabzizj_qij} for any $l, m = 1, \ldots, k+1$ we have
\begin{align*}
 \int_{-\eps^{-1}}^{-\eps} \int_M & \big| \nabla y_l \cdot \nabla y_m - \delta_{lm} \big| d\nu_t dt \\
&\leq   \eps^{-1} \bigg|  \sum_{i,j=1}^{N-1} a_{li} a_{mj} q_{ij} - \delta_{lm} \bigg| 
+ \int_{-\eps^{-1}}^{-\eps} \int_M \Big| \nabla y_l \cdot \nabla y_m - \sum_{i,j=1}^{N-1} a_{li} a_{mj} q_{ij}  \Big| d\nu_t dt \\
&\leq  \Psi (\delta | Y, D, \la)
+ \eps^{-1} \bigg|  \sum_{i,j=1}^{N-1} a_{li} a_{mj} \int_{-2}^{-1} \int_M  \nabla \td u_i \cdot \nabla \td u_j  \, d\nu_t dt - \delta_{lm} \bigg| \\
&\qquad\qquad +\sum_{i,j=1}^{N-1} a_{li} a_{mj}   \int_{-\eps^{-1}}^{-\eps} \int_M \Big| \nabla u_i \cdot \nabla u_j - q_{ij}   \Big|  d\nu_t dt
\leq \Psi (\delta | Y, D,\la). 
\end{align*}
By Claim~\ref{Cl_bounds_zi}, we moreover have
\[  \int_{-\eps^{-1}}^{-\eps} \int_M  | \square y_l | d\nu_t dt \leq \Psi (\delta | Y, D,\la). \]
So $(y_1, \ldots, y_{k+1})$ is a weak $(k+1, \Psi (\delta | Y, D, \la), 1)$-splitting map at $(x_0, t_0)$.
Therefore, the proposition follows for $\delta \leq \ov\delta (Y, D, \eps, \la)$.
\end{proof}

\section{A preliminary quantitative stratification result} \label{sec_prelim_q_strat}
In this section we will prove a weak version of the quantitative stratification result, Theorem~\ref{Thm_Quantitative_Strat}.
The main difference will be that the quantitative strata $\SS^{\eps, k}_{\sigma r, r}$ are replaced by strata $\td\SS^{\eps, k}_{\sigma r, r}$ that are defined using the integral almost notions from Section~\ref{sec_int_almost_prop} instead of $\IF$-closeness to a metric soliton with certain properties.

\begin{Definition} \label{Def_SS_weak}
Let $(M, (g_t)_{t \in I})$ be a Ricci flow.
For $\eps > 0$ and $0 < r_1 < r_2 \leq \infty$ the effective strata
\[ \td\SS^{\eps, 0}_{r_1, r_2} \subset \td\SS^{\eps, 1}_{r_1, r_2} \subset \td\SS^{\eps,2}_{r_1, r_2} \subset \ldots \subset \td\SS^{\eps,n+2}_{r_1, r_2} \subset  M \times I \]
are defined as follows:
$(x',t') \in \td\SS^{\eps, k}_{r_1, r_2}$ if and only for all $r' \in (r_1, r_2)$ none of the following two properties hold:
\begin{enumerate}
\item $(x',t')$ is $(\eps, r')$-selfsimilar and weakly $(k+1, \eps, r')$-split.
\item $(x',t')$ is $(\eps, r')$-selfsimilar, $(\eps, r')$-static and weakly $(k-1, \eps, r')$-split.
\end{enumerate}
\end{Definition}

We recall that by definition the property of being $(k, \eps, r')$-split is vacuous for $k \leq 0$.
Also, by definition, if a point $(x,t) \in M \times I$ is $(\eps, r)$-selfsimilar, then $[t - \eps^{-1} r^2, t] \subset I$.

Note that, in contrast to Definition~\ref{Def_SS_quantitative}, the strata $\td\SS^{\eps, k}_{r_1, r_2}$ are defined up to the  maximal index $n+2$, since we lack a proper local regularity theory at this moment.

The following proposition is the main result of this section.
Recall the notion of $P^*$-parabolic balls from Subsection~\ref{subsec_parabo_nbhd}.

\begin{Proposition}
\label{Prop_Quantitative_Strat_preliminary}
If $Y < \infty$, $\eps > 0$, 
then the following holds.
Let $(M, (g_t)_{t \in I})$ be a Ricci flow on a compact manifold, $(x_0, t_0) \in M \times I$ a point and $r > 0$ a scale with $[t_0 - 2 r^2, t_0] \subset I$.
Assume that $\NN_{x_0,t_0} (r^2) \geq - Y$ and consider the effective strata $\td\SS^{\eps,k}_{r_1 ,r_2} \subset M \times I$ from Definition~\ref{Def_SS_weak}.
Then for any $\sigma \in (0, \eps)$ there are points $(x_1, t_1), \ldots, (x_N, t_N) \in \td\SS^{\eps,k}_{\sigma r, \eps r} \cap P^* (x_0, t_0; r)$ with  $N \leq C(Y, \eps) \sigma^{-k-\eps}$ and
\[ \td\SS^{\eps, k}_{\sigma r, \eps r} \cap P^* (x_0, t_0; r) \subset \bigcup_{i=1}^N P^* (x_i, t_i; \sigma r). \]
\end{Proposition}

So in order to derive Theorem~\ref{Thm_Quantitative_Strat} from Proposition~\ref{Prop_Quantitative_Strat_preliminary}, we need to relate the almost properties from Definition~\ref{Def_SS_weak} to properties about $\IF$-closeness as in Definition~\ref{Def_SS_quantitative} and prove an $\eps$-regularity theorem for the points in the complement of $\td\SS^{\eps,n-2}_{\sigma r, \eps r}$.

Let us now focus on the proof of Proposition~\ref{Prop_Quantitative_Strat_preliminary}, which will be an adaptation of the arguments from \cite{Cheeger-Naber-quantitative}.
The following lemma is a consequence of Propositions~\ref{Prop_almost_static}, \ref{Prop_almost_split_Rk}.

\begin{Lemma} \label{Lem_quant_strat_preparation}
If
\begin{equation} \label{eq_splitting_cor_bounds}
Y < \infty, \qquad
\eps > 0, \qquad
 \la \leq \ov\la (\eps), \qquad 
\delta \leq \ov\delta (Y,  \eps, \la),  
\end{equation}
then the following holds.
Let $(M, (g_t)_{t \in I})$ be a Ricci flow on a compact manifold and $k \geq 0$, $r > 0$.
Denote by $S \subset M \times I$ the set of points that are $(\delta, r)$-selfsimilar and not:
\begin{enumerate}[label=(\roman*)]
\item weakly $(k+1, \eps, r)$-split or 
\item weakly $(k-1, \eps,  r)$-split and $(\eps, r)$-static.
\end{enumerate}
Consider a point $(x^*, t^*) \in M \times I$ such that $\NN_{x, t} (r^2) \geq - Y$ for all $(x,t) \in P^* (x^*, t^*; r)$.
Then there are points $(x_1, t_1), \ldots, (x_{N}, t_{N}) \in S \cap P^* (x^*, t^*; r)$ with  $N \leq C_1 (Y) \la^{-k}$ and
\begin{equation} \label{eq_S_covered_parab_nbhd}
 S \cap P^* (x^*, t^*; r) \subset \bigcup_{i=1}^N P^* (x_i, t_i; \la r). 
\end{equation}
\end{Lemma}

\begin{proof}
Without loss of generality, we may assume that $r = 1$ and $t^* = 0$.
If $S \cap P^* (x^*, t^*; r) = \emptyset$, then we are done.
Otherwise, by the definition of selfsimilarity we have $[-4, 0] \subset [-\delta^{-1}, 0] \subset I$ for $\delta \leq \ov\delta$.
After adjusting $Y$, we may also assume that $\NN_{x, t}(2) \geq - Y$ for all $(x,t) \in P^* (x^*, t^*; r)$.
Let $\beta > 0$ be a constant whose value we will determine in the course of the proof.
The following claim is a consequence of Proposition~\ref{Prop_almost_split_Rk}.

\begin{Claim} \label{Cl_case_all_ti_close}
Suppose that $\beta \leq \ov\beta$ and that bounds of the form (\ref{eq_splitting_cor_bounds}) hold.
Consider $(\delta, r)$-selfsimilar points $(x_1, t_1), \ldots, (x_{N}, t_{N}) \in P^* (x^*, t^*; 1)$ with $t_1 \leq \ldots \leq t_N$.
Assume that $|t_i - t_j | \leq  \beta \la^2$ for all $i,j = 1, \ldots, N$ and assume that the parabolic balls $P^* (x_i, t_i; \tfrac1{10} \la)$, $i = 1, \ldots, N$, are pairwise disjoint.
If $N \geq 10^{-2} \beta C_1(Y,  \beta) \la^{-k'}$ for some $k' \in \{ 0, \ldots,n \}$, then the point $(x_1, t_1)$ is weakly $(k'+1, \eps,1)$-split.
\end{Claim}

\begin{proof}
By Proposition~\ref{Prop_monotonicity_W1}, we have for any $i = 1, \ldots, N$
\[ d^{g_{t_1 - 2}}_{W_1} (\nu_{x_i,t_i; t_1 - 2}, \nu_{x_1, t_1; t_1 - 2} ) 
\leq d^{g_{-1}}_{W_1} (\nu_{x_i,t_i;-1}, \nu_{x_1, t_1; -1} ) 
\leq 1. \]
Fix $D_0 \geq \underline{D}_0 (Y)$ according to Proposition~\ref{Prop_almost_split_Rk} and assume that $D_0 \geq 100$.
If $\beta \leq \ov\beta$, then for all $i, j = 1, \ldots, N$ with $i \neq j$, the fact that $(x_j, t_j) \not\in P^* (x_i, t_i; \tfrac1{10} \la)$ and Proposition~\ref{Prop_monotonicity_W1} imply that 
\begin{multline*}
 d_{W_1}^{g_{ t_1 - 4 (\la/100D_0)^2}} (\nu_{x_i,t_i; t_1 - 4 (\la/100D_0)^2}, \nu_{x_j, t_j; t_1 - 4 (\la/100D_0)^2} ) 
\geq d_{W_1}^{g_{ t_i -  (\la/10)^2}} (\nu_{x_i,t_i; t_i -  (\la/10)^2}, \nu_{x_j, t_j; t_i -  (\la/10)^2} ) \\
\geq \la / 10 \geq D_0 \cdot ( \la/ 100 D_0) \cdot \sqrt{2}. 
\end{multline*}
So assuming $\beta \leq \ov\beta$ and bounds of the form (\ref{eq_splitting_cor_bounds}), we may apply Proposition~\ref{Prop_almost_split_Rk} at $(x_1, t_1)$ with $r = \sqrt{2}$, $\la$ replaced with $\la / 100 D_0$ and $D=1$.
We obtain that if $N \geq  C_0 (Y,  1)(100 D_0)^{k'} \la^{-k'}$, then the points $(x_1, t_1)$ is weakly $(k'+1, \eps, 1)$-split.
Lastly, we set $C_1 := 100  C_0 (Y,  1)\cdot (100 D_0)^{k'} \beta^{-1} C_0$.
\end{proof}

Choose $(x_1, t_1), \ldots, (x_N, t_N) \in S \cap P^* (x^*, t^*; 1)$ with $N$ maximal such that the parabolic neighborhoods $P^* (x_i, t_i; \tfrac1{10}\la )$ are pairwise disjoint.
By reindexing, we may also assume that $t_1 \leq \ldots \leq t_N$.
We first claim that (\ref{eq_S_covered_parab_nbhd}) holds.
To see this consider a point $(x, t) \in S \cap P^*(x^*, t^*; 1)$.
By maximality of $N$ we have $P^* (x, t; \tfrac1{10}\la ) \cap P^* (x_i, t_i; \tfrac1{10}\la ) \neq \emptyset$ for some $i \in \{ 1, \ldots, N \}$, which implies $(x,t) \in P^* (x_i, t_i; \la )$ by Proposition~\ref{Prop_basic_parab_nbhd}.

Assume now by contradiction that $N > C_1 \la^{-k}$, where $C_1$ is the constant from the Claim~\ref{Cl_case_all_ti_close}.
Then by Claim~\ref{Cl_case_all_ti_close} there cannot be a subset $\mathcal{I}' \subset \{ 1, \ldots, N \}$ of size $|\mathcal{I}'| \geq N/2$ such that $|t_i - t_j| \leq \beta\la^2$ for all $i,j \in \mathcal{I}'$.
This implies that for any $i,j = 1, \ldots, N$ with $j - i \geq N/2$ we have $|t_i - t_j| \geq \beta\la^2$.
Therefore, assuming $\delta \leq \ov\delta (Y,  \beta, \la,  \eps)$, Proposition~\ref{Prop_almost_static} applied with $r =  1$ and $\beta$ replaced with $\beta \la^2$ implies that all points $(x_i, t_i)$, $i = 1, \ldots, \lfloor N/2 \rfloor$, must be $(\eps, 1)$-static.
Thus $k > 2$ and all these points are not $(k-2, \eps, r)$-split.
Choose a subset $\mathcal{I} \subset \{ 1, \ldots, \lfloor N/2 \rfloor \}$ such that $|t_i  - t_j | \leq \beta \la^2$ for all $i,j \in \mathcal{I}$ and
\[ |\mathcal{I} | \geq 10^{-1} \beta \la^2 \lfloor N/2 \rfloor \geq  10^{-2} \beta C_1 \la^{-(k-2)}. \]
Applying again Claim~\ref{Cl_case_all_ti_close} to all points $(x_i, t_i)$, $i \in \mathcal{I}$, for $k' = k-2$, yields the desired contradiction.
\end{proof}
\bigskip

\begin{proof}[Proof of Proposition~\ref{Prop_Quantitative_Strat_preliminary}.]
By parabolic rescaling, we may assume without loss of generality that $r = 1$.
Moreover, by Proposition~\ref{Prop_NN_variation_bound} and after adjusting $Y$ if necessary, we may assume that $\NN_{x,t}(1) \geq - Y$ for all $(x,t) \in P^* (x_0, t_0; 1)$.
Now fix $Y, \eps$ and choose constants $C_1$ and $ \la,\delta \leq \eps$ such that we can apply Lemma~\ref{Lem_quant_strat_preparation}.
For the remainder of this proof we will assume $C_1= C_1(Y)$ to be fixed.
We will impose further assumptions on $\la$ of the form $\la \leq \ov\la ( Y,  \eps)$.
This will also influence the precise choice of $\delta$.

Choose $m \in \IN$ such that $\sigma \in [\lambda^{m}, \la^{m-1}]$. 

By Proposition~\ref{Prop_NN_almost_constant_selfsimilar} for any $(x,t) \in P^* (x_0, t_0; 1)$ the number of powers $j \in \IN_0$ for which $(x,t)$ is not $(\delta, \la^j)$-selfsimilar is bounded by some constant $Q(Y,  \la, \delta) < \infty$.
For any finite sequence $o_0, \ldots, o_{m-1} \in \{ 0, 1 \}$, we set
\[ X(o_0, \ldots, o_{m-1}) := \left\{ (x,t) \in P^* (x_0, t_0; 1)  \;\; : \;\;  \text{$(x,t)$ is $(\delta, \la^j)$-selfsimilar if and only if $o_j = 0$} \right\}. \]
So $X(o_0, \ldots, o_{m-1}) = \emptyset$ if $o_0 + \ldots + o_{m-1} > Q$ and therefore, $P^* (x_0, t_0; 1)$ is the union of at most
\[ {m-1 \choose Q} \leq m^Q  \]
many non-empty subsets of the form $X(o_0, \ldots, o_{m-1})$.

\begin{Claim}
There is a constant $C_2( \la) < \infty$ such that for any $j = 0, \ldots, m-1$, any $o_0, \ldots, o_{m-1} \in \{ 0, 1 \}$ and any $(x^*, t^*) \in  P^*(x_0, t_0;1)$ with $(x^*, t^*) = (x_0, t_0)$ if $j=0$ there are points
\[ (x'_1, t'_1), \ldots, (x'_{N'}, t'_{N'})\in \td\SS^{\eps, k}_{\sigma, \eps} \cap X(o_0, \ldots, o_{m-1}) \cap P^* (x^*, t^*; \la^{j})  \]
with
\[ \td\SS^{\eps, k}_{\sigma, \eps} \cap X(o_0, \ldots, o_{m-1}) \cap P^*(x^*, t^*; \la^{j})  \subset \bigcup_{i=1}^{N'} P^*(x'_i, t'_i; \la^{j+1}), \]
where
\[ N' \leq \begin{cases} C_1 \la^{-k} & \text{if $o_j = 0$ and $j \geq 1$} \\ C_2 & \text{if $o_j = 1$ or $j =0$} \end{cases}. \]
\end{Claim}

\begin{proof}
If $o_{j} = 0$ and $j \geq 1$, then the claim follows from Lemma~\ref{Lem_quant_strat_preparation} with $S:= \td\SS^{\eps, k}_{\sigma, \eps} \cap X(o_0, \lb \ldots,  \lb o_{m-1})$.
If $o_{j} = 1$ or $j =0$, then the claim follows from \cite[\HKBasicCovering]{Bamler_HK_entropy_estimates} for some constant of the form $C_2( \la) < \infty$.
\end{proof}

Using the claim, we obtain by induction that $\td\SS^{\eps, k}_{\sigma, \eps} \cap X(o_0, \ldots, o_{m-1}) \cap P^*(x_0, t_0, 1)$ can be covered by at most
\[ C_2^{1+\sum o_j} (C_1 \la^{-k} )^{m-1- \sum o_j} \]
many parabolic balls of the form $P^* (x', t'; \sigma ) \supset P^* (x', t'; \la^{m})$.
So by the argument above the claim, we find that $\td\SS^{\eps, k}_{\sigma, \eps}  \cap P^* (x_0, t_0; 1)$ can be covered by at most
\[ N \leq m^{Q} C_2^{Q+1} (C_1 \la^{-k} )^{m} 
\leq C (Y,  \la, \delta) m^{Q} C_1^m \la^{-mk} \]
many such parabolic balls.
Choose now $\la \leq \ov\la ( Y, \eps)$ such that
\[ \la^{-\eps/2} \geq C_1. \]
Then
\[ N \leq C (Y,  \la, \delta) m^{Q} \la^{-mk - m\eps/2}
\leq C (Y,  \la, \delta) \la^{-mk - m\eps  }
\leq C (Y,  \la, \delta) \sigma^{-k-\eps}. \]
This finishes the proof.
\end{proof}

\part{Local regularity theory} \label{part_loc_reg}
In this section we study $(k, \eps, r)$-splitting maps in more detail.
One of our main goals will be to deduce a preliminary $\eps$-regularity theorem for points that are split and static to a sufficient degree.

This part is structured as follows.
In Section~\ref{sec_improving_splitting_maps} we first show that weak splitting maps can be upgraded to strong splitting maps.
Then we show that strong splitting maps possess several useful regularity properties.
In Section~\ref{sec_alm_radial} we construct and analyze almost radial functions; these are similar to splitting maps, but instead of approximating the coordinates on a Euclidean factor they approximate the radial coordinate on a Cartesian factor that is a metric cone.
We will see that approximate radial functions exist near points that are sufficiently well split, static and selfsimilar.
In addition, we will show that away from the vertex of the cone (in an integral sense) we gain an additional coordinate in our splitting map.
Finally, in Section~\ref{sec_prelim_eps_reg} we derive a preliminary $\eps$-regularity theorem, which provides curvature bounds near points that are strongly $(n,\eps,r)$-split and $(\eps, r)$-static for sufficiently small $\eps$.

\section{Splitting maps} \label{sec_improving_splitting_maps}
In this section we study weak and strong splitting maps.
We first that the existence of a weak splitting map implies the existence of a strong splitting map.
Then we present some improved regularity properties of strong splitting maps.
We refer to Section~\ref{sec_int_almost_prop} for definitions of strong and weak splitting maps and a preliminary discussion.

\subsection{Weak splitting map to strong splitting map} 
Our first goal is to construct a splitting map that is close to a given weak splitting map.
We remark that, by combining the following proposition with Proposition~\ref{Prop_Quantitative_Strat_preliminary}, we obtain an improved quantitative stratification result, in that we can replace the weak splitting property in Definition~\ref{Def_SS_weak} by the strong splitting property.

The following proposition is the main result of this subsection.

\begin{Proposition} \label{Prop_weak_splitting_map_to_splitting_map}
If $\eps > 0$ and $\delta \leq \ov\delta (\eps)$, then the following holds.
Let $(M, (g_t)_{t \in I})$ be a Ricci flow on a compact manifold and let $r > 0$.
Suppose that $\vec y = (y_1, \ldots, y_k)$ is a weak $(k, \delta, r)$-splitting map at some point $(x_0, t_0) \in M \times I$.
Then there is a strong $(k, \eps, r)$-splitting map $\vec y^{\, \prime} = (y'_1, \ldots, y'_k)$ at $(x_0, t_0)$ and numbers $a_1, \ldots, a_k \in \IR$ such that $y'_i + a_i$ is close to $y_i$ in the following sense.
For all $i = 1, \ldots, k$
\begin{equation} \label{eq_yi_prime_yi_closeness}
\sup_{t \in [t_0-\eps^{-1} r^2,t_0-\eps r^2]}  r^{-2} \int_M   (y'_i+a_i -y_i)^2  d\nu_t
+ r^{-2}  \int_{t_0-\eps^{-1} r^2}^{t_0-\eps r^2}  \int_M    |\nabla y'_i -   \nabla y_i |^2  d\nu_t dt \leq \eps. 
\end{equation}
\end{Proposition}

\begin{Remark} \label{Rmk_splitting_map_L_infty}
As we will see in the course of the proof of Proposition~\ref{Prop_weak_splitting_map_to_splitting_map} (see also Lemma~\ref{Lem_improve_coordinate}), the new splitting map $(y'_1, \ldots, y'_k)$ can even be chosen in the following way.
Suppose that $Z \geq 1$ and $\delta \leq \ov\delta (\eps, Z)$.
Then we can replace the right-hand side of (\ref{eq_yi_prime_yi_closeness}) with
\[ C(\eps) e^{-\eps Z^2/64}  \]
and we obtain in addition that
\[ \osc y'_i \leq 2Z, \qquad |\nabla y'_i| \leq CZ. \]
This would allow us to deduce higher $L^p$-bounds on $\nabla y'_i$ directly.
Such bounds, however, will also follow a posteriori from the properties of a strong splitting map (see Proposition~\ref{Prop_properties_splitting_map}).
\end{Remark}

The following lemma, which will be used in the proof of Proposition~\ref{Prop_weak_splitting_map_to_splitting_map}, gives us Gaussian bounds for sub-level sets of almost linear functions.

\begin{Lemma} \label{Lem_almost_linear_Gaussian}
Let $(M, (g_t)_{t \in [-r^2,0]})$ be a Ricci flow on a compact manifold, denote by $d\nu = (4\pi \tau)^{-n/2} e^{-f} dg$ the conjugate heat kernel based at $(x,0)$ for some $x \in M$.
Suppose that $y \in C^\infty(M)$ satisfies at time $-r^2$
\[  \int_M  ||\nabla y|^2 - 1|  d\nu_{-r^2} \leq \delta \leq 1,  \qquad \int_M y \, d\nu_{-r^2} = 0. \]
Then for any $Z \geq 0$
\[\int_{\{ |y| \geq Z r \}} d\nu_{-r^2} \leq C e^{-Z^2/4} + \Psi (\delta | Z). \]
Here $C < \infty$ denotes a dimensional constant.
\end{Lemma}

\begin{proof}
We adapt an argument from \cite{Hein-Naber-14}.
By parabolic rescaling we may assume without loss of generality that $r = 1$.
We may also assume that $Z \geq 10$.
We will abbreviate $d\nu := d\nu_{-1}$.

First observe that by the $L^2$-Poincar\'e inequality (Proposition~\ref{Prop_Poincare})
\[  \int_M y^2 d\nu \leq 2 \int_M |\nabla y|^2  d\nu \leq 2+2\delta \leq 4. \]
Choose $a \in \IR$ such that 
\[ \td y := \begin{cases} - Z + a & \text{if $y \leq -Z$} \\ y + a & \text{if $-Z \leq y \leq Z$} \\ Z +a & \text{if $Z \leq y$} \end{cases} \]
satisfies
\begin{equation} \label{eq_tdy_average_zero}
 \int_M \td y \, d\nu = 0. 
\end{equation}
It follows that
\begin{multline} \label{eq_a_bound_tdyy}
 |a| 
= \bigg| \int_M ( y- \td y + a) d\nu \bigg| 
\leq \int_{\{ |y| \geq Z \}} | |y| -  Z | d\nu
\leq Z^{-1}  \int_{\{ |y| \geq Z \}} Z | |y| - Z | d\nu  \\
\leq Z^{-1}  \int_{\{ |y| \geq Z \}} y^2  d\nu
\leq  4Z^{-1}
\leq 1 
\end{multline}
and
\begin{equation} \label{eq_nabtdy_m_1}
  \int_M \big| |\nabla \td y |^2 - 1 \big|  d\nu
\leq  \int_M \big| |\nabla y |^2 - 1 \big|  d\nu
\leq \delta. 
\end{equation}
For $\lambda > 0$ set
\[ U(\lambda) := \frac1\lambda \log \bigg(  \int_M e^{\lambda \td y} d\nu \bigg). \]
Note that (\ref{eq_tdy_average_zero}) implies that $\lim_{\lambda \to 0} U(\lambda) = 0$.
An application of the Log-Sobolev inequality \cite[Theorem~1.10]{Hein-Naber-14} to $\phi^2 := e^{\lambda \td y - \lambda U(\lambda)}$ yields, using (\ref{eq_a_bound_tdyy}), (\ref{eq_nabtdy_m_1}),
 \begin{multline*}
 e^{-\lambda U(\lambda)}  \int_M \td y e^{\lambda \td y} d\nu -   U(\lambda) 
 \leq \lambda e^{-\lambda U(\lambda)}  \int_M  |\nabla \td y|^2 e^{\lambda \td y} d\nu \\
 \leq \lambda e^{-\lambda U(\lambda)} \bigg(  \int_M  e^{\lambda \td y} d\nu +  \int_M \big| |\nabla \td y|^2 - 1 \big| e^{\lambda (Z+1)} d\nu \bigg) 
 \leq  \lambda  + \lambda e^{\lambda (Z+1)- \lambda U(\lambda)} \delta. 
\end{multline*}
 Thus
 \[ \frac{dU}{d\lambda} \leq 1 +e^{\lambda Z - \lambda U(\lambda)} \delta . \]
 It follows that
 \[ U(\lambda) \leq \lambda + \Psi (\delta | Z, \lambda), \]
and therefore
\begin{equation*}
  \int_{\{ \td y \geq 2\la \}} d\nu
 \leq e^{-   2\la^2}  \int_M e^{ \la \td y} d\nu 
 = e^{- 2  \la^2} e^{\la U(\la)}
  \leq e^{-\la^2} + \Psi (\delta | \la).
\end{equation*}
Setting $\la := \frac12 (Z - 4Z^{-1})$ and using (\ref{eq_a_bound_tdyy}) implies
\[ \int_{\{ y \geq Z \}} d\nu 
\leq \int_{\{ \td y \geq Z - 4Z^{-1} \}} d\nu 
\leq e^{-\frac14 (Z - 4Z^{-1})^2}  + \Psi (\delta | Z)
\leq e^{-\frac14 Z^2 + 8 - 16Z^{-2} }  + \Psi (\delta | Z). \]
Repeating the same procedure for $y$ replaced by $-y$ yields
\[ \int_{\{ | y| \geq Z \}} d\nu 
\leq 2 e^{-\frac14 Z^2 + 8 - 16Z^{-2} }  + \Psi (\delta | Z)
\leq 2 e^8 e^{-\frac14 Z^2 }  + \Psi (\delta | Z). \]
This finishes the proof.
\end{proof}

Proposition~\ref{Prop_weak_splitting_map_to_splitting_map} will now follow by applying the next lemma to each coordinate function. 

\begin{Lemma} \label{Lem_improve_coordinate}
Let $(M, (g_t)_{t \in [-T,0]})$, $T \geq 2$, be a Ricci flow on a compact manifold, denote by $d\nu = (4\pi \tau)^{-n/2} e^{-f} dg$ the conjugate heat kernel based at $(x,0)$ for some $x \in M$ and let $\theta \in (0,\frac12)$.
Suppose that there is a function $y \in C^\infty (M \times [-T, -\theta])$ such that
\begin{equation} \label{eq_Lem_improve_coordinate_1}
 \int_{-T}^{-\theta} \int_M \big( | \square y | +  ||\nabla y|^2 - 1|  \big) d\nu_t dt \leq \delta \leq 1. 
\end{equation}
Then for any $Z \geq 1$ there is a function $y' \in C^\infty(M \times [-T+1, - \theta])$ solving the heat equation $\square y' = 0$ such that
\begin{equation} \label{eq_Lem_improve_coordinate_2}
\sup_{t \in [-T+1,-\theta]} \int_M (y'-y)^2 d\nu_t + \int_{-T+1}^{-\theta} \int_M   |\nabla y' -   \nabla y |^2   d\nu_t dt \leq  C(T) e^{-Z^2/32 T} + \Psi (\delta | T, \theta, Z),
\end{equation}
\[ \osc y' \leq 2Z, \qquad  |\nabla y' | \leq C Z. \]
\end{Lemma}

\begin{proof}
After adding a suitable constant to $y$, we may assume without loss of generality that
\[ \int_M y \, d\nu_{-1}  = 0. \]
This implies that at every $t^* \in [-T, -\theta]$ 
\begin{equation} \label{eq_int_y_average_almost_0}
\bigg| \int_M y \, d\nu_{t^*}  \bigg|  
= \bigg| \int_{-1}^{t^*}  \frac{d}{dt}  \int_M \, y \, d\nu_{t}  dt \bigg|
= \bigg|  \int_{-1}^{t^*} \int_M \square y \, d\nu_{t}  dt \bigg| 
\leq \int_{-T}^{-\theta} \int_M  |\square y|  d\nu_{t} dt
\leq \delta .
\end{equation}
Therefore, by the $L^2$-Poincar\'e inequality (Proposition~\ref{Prop_Poincare})
\begin{equation} \label{eq_bound_y_L2}
 \int_{-T}^{-\theta} \int_M y^2 \, d\nu_t dt  \leq 2\int_{-T}^{-\theta}  \tau \int_M |\nabla y|^2 d\nu_t dt + \Psi( \delta | T )  \leq CT^2 + \Psi( \delta | T). 
\end{equation}

Let $\eta_Z : \IR \to [-Z,Z]$ be a smooth cutoff function such that $\eta_Z (x) = x$ if $x \in [-Z/2, Z/2]$, $\eta_Z (x) = - Z$ if $x \leq -Z$, $\eta_Z (x) = Z$ if $x \geq Z$ and $0 \leq \eta'_Z \leq 1$, $|\eta''_Z| \leq 10 Z^{-1}$.
Then
\[ \td y := \eta_Z \circ y \]
satisfies
\begin{equation} \label{eq_td_y_identities}
|\td y | \leq Z, \qquad 
 |\nabla \td y| \leq |\nabla y|, \qquad 
|\square \td y| \leq |\square y| + 10 Z^{-1} |\nabla y|^2. 
\end{equation}
Moreover, we have the following properties:

\begin{Claim}
\begin{align}
 \int_{-T}^{-\theta} \int_{\{ |y| \geq Z/2 \}} |\nabla y|^2  d\nu_t dt &\leq  C(T) e^{-Z^2/16 T} + \Psi (\delta | T, \theta, Z),\label{eq_naby_eH16} \\
 \int_{-T}^{-\theta} \int_M \big( (y- \td y)^2 + |\nabla y - \nabla \td y|^2 \big)  d\nu_t dt &\leq  C(T) e^{-Z^2/16 T} + \Psi (\delta | T, \theta, Z). \label{eq_ynabynaby_eH16}
\end{align}
\end{Claim}

\begin{proof}
By Lemma~\ref{Lem_almost_linear_Gaussian} and (\ref{eq_int_y_average_almost_0}) we have the following conditional statement for any $t \in [-T, -\theta]$
\[ \int_M  | |\nabla y|^2 - 1 |  d\nu_t \leq \delta^{1/2} \qquad \Longrightarrow \qquad
\int_{\{ |y| \geq Z/2 \}} d\nu_t \leq C e^{-Z^2/16|t|} + \Psi (\delta | T, \theta, Z). \]
Since the last integral is always bounded by $1$, this implies
\[ \int_{\{ |y| \geq Z/2 \}}  d\nu_t \leq \delta^{-1/2} \int_M  | |\nabla y|^2 - 1 | d\nu_t  + C e^{-Z^2/16 |t|} + \Psi (\delta | T, \theta, Z). \]
Integration over $t$ implies
\begin{equation} \label{eq_int_T_theta_1_exp}
 \int_{-T}^{-\theta} \int_{\{ |y| \geq Z/2 \}}  d\nu_t dt \leq CT e^{-Z^2/16 T} + \Psi (\delta | T, \theta, Z). 
 \end{equation}
By the definition of $\td y$ and (\ref{eq_td_y_identities}), (\ref{eq_Lem_improve_coordinate_1}) we have
\begin{multline} \label{eq_nab_y_td_y_16}
  \int_{-T}^{-\theta} \int_M |\nabla ( y - \td y)|^2  d\nu_t dt
= \int_{-T}^{-\theta} \int_{\{ |y| \geq Z/2 \}} |\nabla y|^2  d\nu_t dt \\
\leq \delta + \int_{-T}^{-\theta} \int_{\{ |y| \geq Z/2 \}}  d\nu_t dt
\leq CT e^{-Z^2/16 T} + \Psi (\delta | T, \theta, Z). 
\end{multline}
This proves (\ref{eq_naby_eH16}) and the second part of (\ref{eq_ynabynaby_eH16}).
For the first part of (\ref{eq_ynabynaby_eH16}) observe that
\[  \bigg| \frac{d}{dt}  \int_M (\td y   - y) d\nu_t \bigg|
\leq \bigg| \int_M \square (\td y - y) d\nu_t  \bigg| 
\leq \int_{\{ |y| \geq Z/2  \}} \big( 2|\square y|  + 10 Z^{-1} |\nabla y|^2 \big) d\nu_t, \]
which implies that for any $t_1, t_2 \in [-T, -\theta]$
\begin{equation*}
 \bigg| \int_M (y - \td y) d\nu_t \bigg|_{t=t_1}^{t=t_2}  \bigg| 
\leq \int_{-T}^{-\theta} \int_{\{ |y| \geq Z/2 \}}  \big(  2 |\square y | + 10 Z^{-1} |\nabla y|^2 \big) d\nu_t dt 
\leq CT e^{-Z^2/16 T} + \Psi (\delta | T, \theta, Z).  
\end{equation*}
Since by (\ref{eq_int_T_theta_1_exp}) and (\ref{eq_bound_y_L2})
\begin{multline*}
 \bigg| \int_{-T}^{-\theta} \int_M (y - \td y) d\nu_t dt \bigg|
\leq \bigg( \int_{-T}^{-\theta} \int_{\{ |y| \geq Z/2 \}} d\nu_t dt \bigg)^{1/2}  \bigg( \int_{-T}^{-\theta} \int_M (y - \td y)^2 d\nu_t dt \bigg)^{1/2} \\
\leq \bigg( \int_{-T}^{-\theta} \int_{\{ |y| \geq Z/2 \}} d\nu_t dt \bigg)^{1/2}  \bigg( \int_{-T}^{-\theta} \int_M y^2 d\nu_t dt \bigg)^{1/2} \leq CT^{3/2} e^{-Z^2/32 T} + \Psi (\delta | T, \theta, Z),
\end{multline*}
it follows that for all $t \in [-T, -\theta]$
\[ \bigg| \int_M (y - \td y) d\nu_t   \bigg| 
\leq CT^{3/2} e^{-Z^2/32 T} + \Psi (\delta | T, \theta, Z).  \]
Using the $L^2$-Poincar\'e inequality and (\ref{eq_nab_y_td_y_16}), we therefore obtain
\begin{multline*}
  \int_{-T}^{-\theta} \int_M  ( y - \td y)^2  d\nu_t dt
\leq  CTe^{-Z^2/16 T} + \Psi (\delta | T, \theta, Z) + \int_{-T}^{-\theta} \bigg( \int_M ( y - \td y) d\nu_t \bigg)^{2} dt \\
\leq  CT^4 e^{-Z^2/16 T} + \Psi (\delta | T, \theta, Z) . \qedhere
\end{multline*}
\end{proof}
\medskip

Combining (\ref{eq_Lem_improve_coordinate_1}) and (\ref{eq_ynabynaby_eH16}) implies that there is a time $t^*_1 \in [-T, -T+\frac12]$ such that
\begin{align}
  \int_M   ||\nabla y|^2 - 1|  d\nu_{ t^*_1} &\leq 2 \delta, \label{eq_choice_tstar1_1} \\
 \int_M \big( (y- \td y)^2 + |\nabla y - \nabla \td y|^2 \big)  d\nu_{t^*_1} &\leq C(T) e^{-Z^2/16 T} + \Psi (\delta | T, \theta, Z). \label{eq_choice_tstar1_2}
\end{align}
Let $y' \in C^\infty (M \times [t^*_1, -\theta])$ be the solution to the heat equation $\square y' = 0$ with initial condition $y' (\cdot, t^*_1) := \td y (\cdot, t^*_1)$.
Note that by the maximum principle and (\ref{eq_td_y_identities}) we have $|y'| \leq Z$.

Let us now verify that $y'$ satisfies the remaining assertions.
We have
\[ \int_M (y' - \td y)^2 d\nu_{t^*_1}= \int_M (\td y -  \td y)^2 d\nu_{t^*_1} = 0. \]
Using (\ref{eq_td_y_identities}) and $|y'-\td y| \leq |y'| + |\td y| \leq 2Z$, we obtain
\begin{multline*}
 \frac{d}{dt} \int_M (y'- \td y)^2 d\nu_t 
= \int_M  \square ( y'- \td y)^2  d\nu_t 
= \int_M \big( {- 2 \square \td y \, (y'- \td y) - 2 |\nabla (y'- \td y)|^2 } \big) d\nu_t  \\
\leq \int_M  4 Z \big(| \square  y|  + 10 Z^{-1} |\nabla y|^2 \chi_{\{ |y| \geq Z/2 \}}  \big) d\nu_t    - 2 \int_M |\nabla (y'- \td y)|^2  d\nu_t .
\end{multline*}
Integration over $t$ and using (\ref{eq_naby_eH16}) implies that for any $t^{*}_2 \in [t^*_1, -\theta]$
\[
  \int_M (y'- \td y)^2 d\nu_{t^{*}_2} + 2 \int_{t^*_1}^{t^{*}_2}  \int_M |\nabla (y'- \td y)|^2  d\nu_t dt \\
\leq C(T) Z e^{-Z^2/16 T} + \Psi (\delta | T, \theta, Z).
\]
Combination with (\ref{eq_ynabynaby_eH16}) implies (\ref{eq_Lem_improve_coordinate_2}) since $Z e^{-Z^2/16T} \leq C(T) e^{-Z^2/32T}$.

The bound $|\nabla y' | \leq CZ$ on $M \times [-T+1, -\theta]$ follows from $|y'| \leq Z$ via a gradient estimate (see for example \cite[\HKThmGradientPhiEstimate]{Bamler_HK_entropy_estimates}).
\end{proof}

\begin{proof}[Proof of Proposition~\ref{Prop_weak_splitting_map_to_splitting_map}.]
By parabolic rescaling, we may assume that $r = 1$ and $t_0 = 0$.
Set $T = 2\eps^{-1}$, $\theta = \eps$, apply Lemma~\ref{Lem_improve_coordinate} to each coordinate function $y_i$ and call the resulting function $y'_i$.
By choosing $\delta \leq \ov{\delta}(\eps, Z)$, we can make the right-hand side of (\ref{eq_Lem_improve_coordinate_2}) $\leq C(\eps) e^{-\eps Z^2/64 }$.
So by choosing $Z \geq \underline{Z}(\eps)$, we can ensure that (\ref{eq_yi_prime_yi_closeness}) holds for $a_i = 0$.
To verify the properties of a strong splitting map, observe that for all $i, j = 1, \ldots, k$
\begin{align*}
 \int_{-\eps^{-1}}^{-\eps} \int_M & | \nabla y'_i \cdot \nabla y'_j - \delta_{ij} | d\nu_t dt
\leq \int_{-\eps^{-1}}^{-\eps} \int_M | \nabla y_i \cdot \nabla y_j - \delta_{ij} | d\nu_t dt \\
&\qquad + \int_{-\eps^{-1}}^{-\eps} \int_M | ( \nabla y'_i - \nabla y_i) \cdot \nabla y_j | d\nu_t dt  + \int_{-\eps^{-1}}^{-\eps} \int_M | \nabla y'_i \cdot (\nabla y'_j - \nabla y_j) | d\nu_t dt \\
&\leq \delta + \bigg( \int_{-\eps^{-1}}^{-\eps} \int_M |  \nabla y'_i - \nabla y_i|^2 d\nu_t dt \bigg)^{1/2} \bigg( \int_{-\eps^{-1}}^{-\eps} \int_M | \nabla y_j |^2 d\nu_t dt \bigg)^{1/2} \\
&\qquad +  \bigg( \int_{-\eps^{-1}}^{-\eps} \int_M | \nabla y'_i |^2 d\nu_t dt \bigg)^{1/2} \bigg( \int_{-\eps^{-1}}^{-\eps} \int_M |  \nabla y'_j - \nabla y_j|^2 d\nu_t dt \bigg)^{1/2} \\
&\leq \delta + \Psi (Z | \eps).
\end{align*}

So if we choose  $\delta \leq \ov\delta (\eps)$ and $Z \geq \underline{Z}(\eps)$, then Properties~\ref{Def_strong_splitting_map_1}, \ref{Def_strong_splitting_map_2} of Definition~\ref{Def_strong_splitting_map} hold.
So after replacing $y'_i$ with $y'_i - a_i$ for some suitable $a_i \in \IR$ according to Lemma~\ref{Lem_prop_3_strong_splitting}, $\vec y^{\,\prime} = (y'_1, \ldots, y'_k)$ is indeed a strong $(k, \eps, 1)$-splitting map.
\end{proof}

\subsection{Strong splitting maps}

The following proposition summarizes important properties that follow a posteriori from the definition of a strong splitting map (Definition~\ref{Def_strong_splitting_map}).

\begin{Proposition} \label{Prop_properties_splitting_map}
If $Y< \infty$, $\eps > 0$ and $\delta \leq \ov\delta (Y,\eps)$, then the following holds.
Let $(M, (g_t)_{t \in I})$ be a Ricci flow on a compact manifold and let $r > 0$.
Consider a point $(x_0, t_0) \in M \times I$ with $[t_0 - \delta^{-1} r^2, t_0] \subset I$ and $\NN_{x_0, t_0} ( r^2) \geq -Y$ and let $d\nu = (4\pi \tau)^{-n/2} e^{-f} dg$ be the conjugate heat kernel at $(x_0, t_0)$.
Let $\vec y = (y_1, \ldots, y_k) : M \times [t_0 - \delta^{-1} r^2, t_0 - \delta r^2] \to \IR^k$ be a $(k,  \delta, r)$-splitting map at $(x_0, t_0)$. Then the following holds for $\alpha \in [0, \ov\alpha]$ and $i,j = 1, \ldots, k$:
\begin{enumerate}[label=(\alph*)]
\item \label{Prop_properties_splitting_map_a} At any time $t \in [t_0 - \eps^{-1} r^2, t_0 -\eps r^2]$ and for all $p \in [1, \eps^{-1}]$, $m = 0, \ldots, [\eps^{-1}]$ we have
\begin{align}
  \int_M |\nabla y_i \cdot \nabla y_j - \delta_{ij} |^p e^{\alpha f} d\nu_{t}  &\leq \eps. \label{eq_nabyinabyj_p} \\
  \bigg| \int_M y_i^{2m}  \, d\nu_{t} - \frac{(2m)!}{m!} (t_0 - t)^m \bigg|  &\leq \eps.  \label{eq_yi_2m_bound}
\end{align}
\item \label{Prop_properties_splitting_map_b} For any $p, q \in [0, \eps^{-1}]$ we have
\begin{equation} \label{eq_nab2_y_y_p_e_alph}
 r^{2-q} \int_{t_0 - \eps^{-1} r^2}^{t_0 - \eps r^2} \int_M \Big( \sum_{i=1}^k |\nabla^2 y_i |^2  \Big) \Big( \sum_{j=1}^k |\nabla y_j |^{p} \Big) \Big( \sum_{l=1}^k  |y_l |^{q} \Big) e^{\alpha f} d\nu_{t} dt \leq \eps. 
\end{equation}
\end{enumerate}
\end{Proposition}

\begin{proof}
Without loss of generality, we may assume that $t_0 = 0$ and $r =1$.
Fix $\eps > 0$.
We will determine $\delta \leq 10^{-3}$ in the course of the proof.

We first derive an $L^p$-bound on $\nabla y_i$ over time-slices for any $p \in [1, \eps^{-1}]$.
For this purpose, observe that since $\square |\nabla y_i| \leq 0$, we have by \cite[\HKThmHypercontractivity]{Bamler_HK_entropy_estimates} for any $t_1, t_2 \in [- \delta^{-1}, -\delta]$, $t_1 \leq t_2$ with $\frac{t_1}{t_2} \geq p-1$
\[ \bigg(  \int_M |\nabla y_i|^{p} d\nu_{t_2} \bigg)^{1/p} 
 \leq \bigg(  \int_M |\nabla y_i|^{2} d\nu_{t_1} \bigg)^{1/2}. \]
 So if $\delta \leq \ov\delta(\eps)$, then integration with respect to $t_1$ implies that for any $t_2 \in [-\eps^{-1}, -\eps]$
\begin{equation} \label{eq_nab_y_i_L_p_bound}
 \bigg(  \int_M |\nabla y_i|^{p} d\nu_{t_2} \bigg)^{2/p} \leq \int_{-\delta^{-1}}^{-\delta^{-1}+1}  \int_M |\nabla y_i|^{2} d\nu_{t} dt \leq 1+ \delta. 
\end{equation}

Next, we show that (\ref{eq_nab2_y_y_p_e_alph}) holds for $p = q = \alpha = 0$.
For this purpose, observe that
\[ \frac{d}{dt} \int_M \big( |\nabla y_i|^2 - 1 \big) \, d\nu_t
= \int_M \square |\nabla y_i|^2 \, d\nu_t
= - 2 \int_M  |\nabla^2 y_i|^2 \, d\nu_t \]
So if $\eta : [-\delta^{-1}, -\delta] \to [0,1]$ denotes a cutoff function that vanishes near the endpoints of the domain and satisfies $\eta \equiv 1$ on $[-\eps^{-1}, -\eps]$ and $|\eta'| \leq 10 \eps^{-1}$, then we obtain
\begin{multline*}
 2\int_{-\eps^{-1}}^{-\eps} \int_M |\nabla^2 y_i|^2 \, d\nu_t dt
\leq 2 \int_{-\delta^{-1}}^{-\delta} \eta(t) \int_M |\nabla^2 y_i|^2 \, d\nu_t dt \\
= -  \int_{-\delta^{-1}}^{-\delta} \eta'(t) \int_M \big( |\nabla y_i|^2 - 1 \big) \, d\nu_t dt
\leq 10 \eps^{-1}  \int_{-\delta^{-1}}^{-\delta}  \int_M \big| |\nabla y_i|^2 - 1 \big| \, d\nu_t dt
\leq \Psi(\delta | \eps). 
\end{multline*}
This shows that (\ref{eq_nab2_y_y_p_e_alph}) holds for $p = q = \alpha = 0$ if $\delta \leq \ov\delta (\eps)$.
After adjusting $\eps$ and $\delta$, we may assume in the following that (\ref{eq_nab2_y_y_p_e_alph}) holds for $p = q = \alpha = 0$ and for $\eps$ replaced with $\delta$.

Next we prove (\ref{eq_nabyinabyj_p}) for $\alpha = 0$.
Set $u := |\nabla y_i \cdot \nabla y_j - \delta_{ij}|$ and observe that $\square u \leq | \square ( \nabla y_i \cdot \nabla y_j)|$ in the weak sense (see Propositions~\ref{Prop_weak_square}, \ref{Prop_mu_square_rules}), so
\begin{equation*}
  \frac{d}{dt} \int_M  u \, d\nu_t 
= \int_M \square u \, d\nu_t 
\leq  \int_M |{-2\nabla^2 y_i \cdot \nabla^2 y_j }|\, d \nu_t
\leq   \int_M \big( |\nabla^2 y_i|^2 + |\nabla^2 y_j|^2 \big) d\nu_t.
\end{equation*}
Integration over time and application of our previous assumption concerning (\ref{eq_nab2_y_y_p_e_alph}) implies that for any $t_1, t_2 \in [-\delta^{-1}, -\delta]$ with $t_1 \leq t_2$
\[ \int_M  u \, d\nu_t  \bigg|_{t=t_1}^{t=t_2} \leq  2 \delta. \]
Integrating this again over $[-\delta^{-1}, \lb - \delta^{-1} + 1]$ with respect to $t_1$ implies that at every time $t_2 \in [-\delta^{-1}+1, -\delta]$
\[ \int_M  u \, d\nu_t  \bigg|_{t=t_2} \leq 2 \delta + \delta. \]
So by Cauchy-Schwarz and (\ref{eq_nab_y_i_L_p_bound}), we obtain that if $\delta \leq \ov\delta (\eps)$, then for any $p \in [1, \eps^{-1}]$, $t \in [-\eps^{-1}, -\eps]$
\begin{multline*}
 \int_M  u^{p} \, d\nu_{t} 
 \leq  \bigg( \int_M  u \, d\nu_{t} \bigg)^{1/2}  \bigg( \int_M  u^{2p-1} \, d\nu_{t} \bigg)^{1/2} \\
 \leq  \Psi(\delta)  \bigg( \int_M  |\nabla y_i|^{4p-2} \, d\nu_{t} \bigg)^{1/4}
 \bigg( \int_M  |\nabla y_j|^{4p-2} \, d\nu_{t} \bigg)^{1/4}
 \leq \Psi(\delta ). 
\end{multline*}
This proves (\ref{eq_nabyinabyj_p}) for $\alpha = 0$.
If $0 < \alpha \leq \ov\alpha$, then we have by Proposition~\ref{Prop_int_ealphaf}
\[ \int_M  u^{p} e^{\alpha f} d\nu_t   \leq \bigg( \int_M  u^{2p}  d\nu_t \bigg)^{1/2}  \bigg( \int_M  e^{2\alpha f} d\nu_t \bigg)^{1/2}
\leq \Psi (\delta |  \eps), \]
which proves (\ref{eq_nabyinabyj_p}) in its general form.
As before, let us assume from now on that (\ref{eq_nabyinabyj_p}) holds for $\eps$ replaced with $\delta$.

Our next goal is to show (\ref{eq_yi_2m_bound}) by induction on $m$.
Observe that (\ref{eq_yi_2m_bound}) holds trivially for $m = 0$.
After adjusting $\eps, \delta$ we may assume without loss of generality,  that (\ref{eq_yi_2m_bound}) holds for $\eps$ replaced with $\delta$ for some $m \leq [\eps^{-1}] -1$.
Our goal is to show that then (\ref{eq_yi_2m_bound}) holds for $m +1$ if $\delta \leq \ov\delta (\eps)$.
For this purpose, observe that for $t \in [-\eps^{-1}, -\delta]$
\begin{multline*}
  \frac{d}{dt} \int_M y_i^{2m+2} \, d\nu_t  
= \int_M \square y_i^{2m+2} \, d\nu_t 
= -(2m+2)(2m+1) \int_M |\nabla y_i|^2 y_i^{2m} \, d\nu_t \\
= -(2m+2)(2m+1) \int_My_i^{2m} \, d\nu_t 
- (2m+2)(2m+1) \int_M \big( |\nabla y_i|^2 - 1 \big) y_i^{2m} \, d\nu_t .
\end{multline*}
Therefore, using the $L^{4m}$-Poincar\'e inequality (Proposition~\ref{Prop_Poincare}), abbreviating $\tau = -t$,
\begin{align*}
  \bigg| \frac{d}{dt} \int_M y_i^{2m+2} \, & d\nu_t + (2m+2)(2m+1) \frac{(2m)!}{m!} \tau^{m} \bigg| \\
&\leq C(m) \delta + C(m) \bigg( \int_M \big| |\nabla y_i|^2 - 1 \big|^2  \, d\nu_t \bigg)^{1/2}\bigg( \int_M  y_i^{4m} \, d\nu_t \bigg)^{1/2} \\
&\leq C(m) \delta +  C(m) \tau^{m} \bigg( \int_M \big| |\nabla y_i|^2 - 1 \big|^2  \, d\nu_t \bigg)^{1/2}\bigg( \int_M  |\nabla y_i|^{4m} \, d\nu_t \bigg)^{1/2}
\leq \Psi (\delta | \eps, m).
\end{align*}
Integrating this over time and applying the $L^{2m+2}$-Poincar\'e inequality implies that for any $t^* \in [-\eps^{-1}, -\eps]$ and $\theta \in [\delta, \eps]$, using (\ref{eq_nabyinabyj_p}),
\begin{multline*}
 \bigg| \int_M y_i^{2m + 2} \, d\nu_t - \frac{(2m +2)!}{(m+1)!} \tau^{m+1} \bigg| \, \bigg|_{t = t^*} 
\leq \Psi (\delta | \eps, m, \theta) + \bigg| \int_M y_i^{2m + 2} \, d\nu_t - \frac{(2m +2)!}{(m+1)!} \tau^{m+1} \bigg| \, \bigg|_{t =  - \theta}  \\
\leq \Psi (\delta | \eps, m, \theta) + C(m) \theta^{m +1} \int_M |\nabla y_i|^{2m + 2} \, d\nu_{-\theta} + C(m) \theta^{m+1}
\leq  \Psi (\delta | \eps, m, \theta) + C(m) \theta^{m +1} . 
\end{multline*}
This implies (\ref{eq_yi_2m_bound}) for $m+1$ if $\theta \leq \ov\theta (\eps, m)$ and $\delta \leq \ov\delta (\eps, m, \theta)$.
Let us assume from now on that Assertion~\ref{Prop_properties_splitting_map_a} holds for $\eps$ replaced with $\delta$.

It remains to prove (\ref{eq_nab2_y_y_p_e_alph}) in its full form.
Since for any $b > 0$ we have, using our previous assumption involving (\ref{eq_nab2_y_y_p_e_alph}) for $p = q = \alpha = 0$,
\begin{align*}
  \int_{ - \eps^{-1} }^{ - \eps } \int_M & \Big( \sum_{i=1}^k |\nabla^2 y_i |^2  \Big) \Big( \sum_{j=1}^k |\nabla y_j |^{p} \Big) \Big( \sum_{l=1}^k  |y_l |^{q} \Big) e^{\alpha f}  d\nu_{t} dt  \\
&\leq C(\eps) b \sum_{i,j,l=1}^k  \int_{ - \eps^{-1} }^{ - \eps } \int_M |\nabla^2 y_i |^2  |\nabla y_j |^{2p}   |y_l |^{2q} e^{2\alpha f}  d\nu_{t} dt  +
b^{-1} \sum_{i=1}^k  \int_{ - \eps^{-1} }^{ - \eps } \int_M   |\nabla^2 y_i |^2   d\nu_{t} dt \\
&\leq C(\eps) b \sum_{i,j,l=1}^k  \int_{ - \eps^{-1} }^{ - \eps } \int_M |\nabla^2 y_i |^2  |\nabla y_j |^{4p}   e^{4\alpha f}  d\nu_{t} dt \\
&\qquad + C(\eps) b \sum_{i,j,l=1}^k  \int_{ - \eps^{-1} }^{ - \eps } \int_M |\nabla^2 y_i |^2     |y_l |^{4q}   d\nu_{t} dt  + b^{-1} \delta, 
\end{align*}
it suffices to show, after adjusting $\alpha, p, q$, that for any $p, q \geq 2$, $i, j, l = 1, \ldots, k$
\begin{align} \label{eq_nab2yypealphC}
 \int_{ - \eps^{-1} }^{ - \eps } \int_M  |\nabla^2 y_i |^2   |\nabla y_j |^{p}     e^{\alpha f} d\nu_{t} dt &\leq C(Y,\eps, p), \\
  \int_{ - \eps^{-1} }^{ - \eps } \int_M  |\nabla^2 y_i |^2     |y_l |^{q}  d\nu_{t} dt &\leq C(Y,\eps, q).  \label{eq_nab2_no_nab}
\end{align}
For this purpose, observe that for any $p \geq 2$ and $i, j = 1, \ldots, k$ we have for $\delta^*_{ij} := 1-\delta_{ij}$
\begin{align*}
 \square \big( |\nabla y_i|^{2} |\nabla y_j|^{p} \big)
&= \big( \square  |\nabla y_i|^{2} \big) |\nabla y_j|^{p} +  |\nabla y_i|^{2} \big(  \square |\nabla y_j|^{p} \big) - 2 \nabla |\nabla y_i|^{2} \cdot \nabla |\nabla y_j|^{p} \\
&\leq - c ( p) \big( |\nabla^2 y_i|^2  |\nabla y_j|^{p} + |\nabla^2 y_j|^2 |\nabla y_i|^{2} |\nabla y_j|^{p-2} \big) \\
&\qquad + C(p ) \delta^*_{ij} |\nabla^2 y_i| \, | \nabla^2 y_j| \, |\nabla y_i| \, |\nabla y_j|^{p-1} \\
&\leq - c ( p) \big( |\nabla^2 y_i|^2  |\nabla y_j|^{p} + |\nabla^2 y_j|^2 |\nabla y_i|^{2} |\nabla y_j|^{p-2} \big) \\
&\qquad + C( p ) \delta^*_{ij} \big( |\nabla^2 y_i|^2 | \nabla y_i|^{2} +  |\nabla^2 y_j|^2 |\nabla y_j|^{2p-2} \big).
\end{align*}
Moreover, using (\ref{eq_potential_evolution_equation}),
\[ \square e^{\alpha f} = \alpha (\square f) e^{\alpha f} - \alpha^2 |\nabla f|^2 e^{\alpha f}
= \alpha \Big( -2 \triangle f + (1-\alpha) |\nabla f|^2 - R + \frac{n}{2\tau} \Big) e^{\alpha f} . \]
It follows that
\begin{align}
 e^{-\alpha f} \square \big(  |\nabla y_i|^{2} & |\nabla y_j|^{p} e^{\alpha f} \big) \notag \\
&= \square \big( |\nabla y_i|^{2} |\nabla y_j|^{p} \big) + e^{-\alpha f} |\nabla y_i|^{2} |\nabla y_j|^{p} \square e^{\alpha f} +2\alpha \nabla \big( |\nabla y_i|^{2} |\nabla y_j|^{p} \big) \cdot \nabla  f \notag \\
&\leq - c ( p) \big( |\nabla^2 y_i|^2  |\nabla y_j|^{p} + |\nabla^2 y_j|^2 |\nabla y_i|^{2} |\nabla y_j|^{p-2} \big) \notag \\
 &\qquad + C( p ) \delta^*_{ij} \big( |\nabla^2 y_i|^2 |\nabla y_i|^{2} +  |\nabla^2 y_j|^2 |\nabla y_j|^{2p-2} \big) \notag \\
&\qquad + C(\eps) |\nabla y_i|^{2} |\nabla y_j|^{p} \big( |\nabla^2 f| +  |\nabla f|^2 + |R| + 1 \big) \notag \\
&\qquad + C(\eps,  p) \big( |\nabla^2 y_i| \, |\nabla y_i| \, |\nabla y_j|^{p} + |\nabla^2 y_j | \, |\nabla y_i|^{2} |\nabla y_j|^{p - 1} \big) |\nabla f| \notag \\
&\leq - \tfrac12 c ( p)  |\nabla^2 y_i|^2  |\nabla y_j|^{p} + C( p ) \delta^*_{ij} \big( |\nabla^2 y_i|^2 |\nabla y_i|^{2} +  |\nabla^2 y_j|^2 |\nabla y_j|^{2p-2} \big) \notag \\
&\qquad + C(\eps, p) \big(  |\nabla y_i|^{8} + |\nabla y_j|^{4p} +  |\nabla^2 f|^2 +  |\nabla f|^4 + R^2 + 1 \big)  . \label{eq_ealphf_square_nab2_nab_p}
\end{align}
Therefore, we obtain using Proposition~\ref{Prop_improved_L2}  and Assertion~\ref{Prop_properties_splitting_map_a} that if $\alpha \leq \ov\alpha$, $\delta \leq \ov\delta(\eps,p)$, then
\begin{align}
 \int_{ - \eps^{-1} }^{ - \eps } \int_M &  |\nabla^2 y_i|^2  |\nabla y_j|^{p}  e^{\alpha f} d\nu_t dt \notag \\
& \leq - C(\eps,  p) \int_M |\nabla y_i |^2 |\nabla y_j|^p e^{\alpha f} d\nu_t  \bigg|_{t = -\eps^{-1}}^{t=-\eps} + C(Y, \eps,  p) \notag \\
&\qquad  + C(\eps,  p) \delta^*_{ij} \int_{ - \eps^{-1} }^{ - \eps } \int_M  \big( |\nabla^2 y_i|^2 |\nabla y_i|^{2} +  |\nabla^2 y_j|^2 |\nabla y_j|^{2p-2} \big) e^{\alpha f} d\nu_t dt \notag \\
& \leq  C(Y, \eps,  p) + C(\eps,  p) \delta^*_{ij} \int_{ - \eps^{-1} }^{ - \eps } \int_M  \big( |\nabla^2 y_i|^2 |\nabla y_i|^{2} +  |\nabla^2 y_j|^2 |\nabla y_j|^{2p-2} \big) e^{\alpha f} d\nu_t dt. \label{eq_int_nab2nab_p}
\end{align}
This shows (\ref{eq_nab2yypealphC}) if $i = j$.
Given (\ref{eq_nab2yypealphC}) for $i = j$, we can also bound the last integral in (\ref{eq_int_nab2nab_p}) and conclude (\ref{eq_nab2yypealphC}) for $i \neq j$.

Lastly, integrating the bound
\begin{multline*}
 \square |\nabla y_i|^2 |y_l|^q
\leq - 2 |\nabla^2 y_i|^2 |y_l|^q - q(q-1) |\nabla y_i|^2 |\nabla y_l|^2 |y_l|^{q-2} + C(q) |\nabla^2 y_i| \, |\nabla y_i| \, |\nabla y_l| \, | y_l|^{q-1} \\
\leq -  |\nabla^2 y_i|^2 |y_l|^q + C(q)  |\nabla y_i|^2 \, |\nabla y_l|^2 \, | y_l|^{q-2} 
\leq -  |\nabla^2 y_i|^2 |y_l|^q + C(q) \big(  |\nabla y_i|^6 + |\nabla y_l|^6 +  | y_l|^{3q-6} \big),
\end{multline*}
as in (\ref{eq_int_nab2nab_p}) implies (\ref{eq_nab2_no_nab}), which finishes the proof.
\end{proof}
\bigskip

\section{Almost radial functions} \label{sec_alm_radial}
In this section we characterize the geometry near points $(x_0, t_0)$ that are simultaneously $(r, \eps)$-selfsimilar, $(r, \eps)$-static and strongly $(k,  \eps, r)$-split.
Near such points the flow is expected to be close to a static flow of the form $(C^{n-k} \times \IR^k, (v , \vec 0))$, where $(C^{n-k}, v)$ is a metric cone with vertex $v$.
If $d\nu = (4\pi \tau)^{-n/2} e^{-f} dg$ denotes the conjugate heat kernel based at $(x_0, t_0)$, then the function
\[ 4\tau (f-W) \]
is expected to approximate the square of the distance function, $d^2_{C^{n-k} \times \IR^k} ( (v, \vec 0), \cdot)$ from the basepoint $(v, \vec 0)$.
Our goal in Subsection~\ref{subsec_construction_almost_radial} will be to construct an \emph{almost radial} function $q$.
This function approximates the square of the distance function to a vertex on the $C^{n-k}$-factor, i.e. $d^2_{C^{n-k} } ( v, \cdot)$ and is almost constant in the $\IR^k$-direction.
If $(y_1, \ldots, y_k)$ denotes a strong splitting map representing the $\IR^k$-factor, then this function will be of the form
\[ q := 4\tau (f-W) - \sum_{i=1}^k y_i^2. \]
We will show that this function indeed satisfies the expected properties in an integral sense.

In Subsection~\ref{subsec_almost_radial_splitting_extension} we will show that near points where $q \gtrsim d^2 > 0$ --- i.e. points that correspond to points on $C^{n-k} \times \IR^k$ of distance $\gtrsim d$ from $\{ v \} \times \IR^{k}$ --- a suitable rescaling of $q$ can be used to extend the splitting map $(y_1, \ldots, y_k)$ at scale $\approx c d$ by an additional coordinate $y_{k+1}$.
This will allow us to conclude that such points are in fact strongly $(k+1, \Psi(\delta), cd)$-split.
In other words, away from the set of vertices $\{ v \} \times \IR^k$ we obtain an additional $\IR$-symmetry at small scales.

\subsection{Construction of an almost radial function} \label{subsec_construction_almost_radial}
The following proposition asserts the existence of an almost radial function with the expected properties on its first and second derivatives.

\begin{Proposition} \label{Prop_construction_almost_radial}
If $\alpha \in [0, \ov\alpha]$, $\eps > 0$, $Y < \infty$ and $\delta \leq \ov\delta (Y,\eps)$, then the following holds.
Let $(M, (g_t)_{t \in I})$ be a Ricci flow on a compact manifold and let $r > 0$ and $(x_0, t_0) \in M \times I$ with $[t_0 - \delta^{-1} r^2, t_0] \subset I$.
Suppose that $\vec y = (y_1, \ldots, y_k) \in C^\infty ( M \times [t_0 - \delta^{-1} r^2, t_0 - \delta r^2])$ is a strong $(k, \delta, r)$-splitting map at $(x_0, t_0)$.
Assume moreover that $(x_0, t_0)$ is $( \delta, r)$-selfsimilar and $(\delta, r)$-static and that $W := \NN_{x_0, t_0} (r^2) \geq - Y$.
Denote by $d\nu = (4\pi \tau)^{-n/2} e^{-f} dg$ the conjugate heat kernel based at $(x_0, t_0)$.
Then the function $q \in C^\infty ( M \times [t_0 - \eps^{-1} r^2, t_0 - \eps r^2])$,
\begin{equation} \label{eq_def_almost_radial}
 q := 4 \tau ( f -  W) - \sum_{i=1}^k y_i^2 
\end{equation}
satisfies the bounds
\begin{align}
 r^{-2} \int_{t_0 - \eps^{-1} r^2}^{t_0 - \eps r^2} \int_M  \Big| \nabla^2 q - 2 \Big(  g -   \sum_{i=1}^k d y_i \otimes d y_i \Big) \Big|^2  e^{\alpha f} d \nu_t dt  &\leq \eps \label{eq_almost_radial_1}  \\
\sum_{i=1}^k r^{-3} \int_{t_0 - \eps^{-1} r^2}^{t_0 - \eps r^2} \int_M |\nabla q \cdot \nabla y_i | e^{\alpha f} d\nu_t dt &\leq \eps \label{eq_almost_radial_2}  \\
 r^{-4} \int_{t_0 - \eps^{-1} r^2}^{t_0 - \eps r^2} \int_M \big| |\nabla q|^2 - 4 q \big| e^{\alpha f} d\nu_t dt &\leq \eps \label{eq_almost_radial_3}  \\
 r^{-2} \int_{t_0 - \eps^{-1} r^2}^{t_0 - \eps r^2} \int_M |\partial_t q |  e^{\alpha f} d\nu_t dt &\leq \eps \label{eq_almost_radial_4}  
\end{align}
Moreover, if $k = n$, then
\begin{equation} \label{eq_almost_radial_5}
\int_{t_0 - \eps^{-1} r^2}^{t_0 - \eps r^2} \int_M \big( r^{-4} | q | + r^{-3} |\nabla q|^2  \big) e^{\alpha f} d\nu_t dt \leq \eps
\end{equation}
\end{Proposition}

\begin{proof}
Without loss of generality, we may assume that $r = 1$ and $t_0 = 0$.
We first reduce the proposition to the case in which $\alpha = 0$.
To see this, denote by $h e^{\alpha f}$ one of the integrands in (\ref{eq_almost_radial_1})--(\ref{eq_almost_radial_5}).
For any $b > 0$ we have
\[  \int_{- \eps^{-1} }^{ - \eps } \int_M h  e^{\alpha f} d \nu_t dt 
\leq b  \int_{- \eps^{-1} }^{ - \eps } \int_M h  e^{2\alpha f} d \nu_t dt  + b^{-1}  \int_{- \eps^{-1} }^{ - \eps } \int_M h \,  d \nu_t dt . \]
So it suffices to show that the first double integral on the right-hand side is $\leq C(Y, \eps)$ if $\alpha \leq \ov\alpha$ and that the second double integral is $\leq \Psi (\delta | Y,\eps)$.
The boundedness of the first double integral can be seen as follows.
In all four cases we have
\begin{multline*}
 h \leq C(Y,\eps) \sum_{i=1}^k \big( 1+ |\nabla^2 f|^2 + |\nabla^2 y_i|^2 |y_i|^2 + |\nabla y_i|^4 + |\nabla f|^2 +   |\nabla y_i|^2 |y_i|^2
 +|f| + y_i^2  + |\partial_t f| + |y_i| \, |\partial_t y_i| \big) \\
 \leq C(Y,\eps) \sum_{i=1}^k \big( 1+ |\nabla^2 f|^2 + |R| + |\nabla f|^4 +  |\nabla^2 y_i|^2  |y_i|^2 + |\nabla y_i|^4 + |y_i|^4  + |f| +   |\nabla^2 y_i|^2 \big).
\end{multline*}
Therefore, using Propositions~\ref{Prop_improved_L2}, \ref{Prop_properties_splitting_map} if $\alpha \leq \ov\alpha$, then
\[   \int_{- \eps^{-1} }^{ - \eps } \int_M h  e^{2\alpha f} d \nu_t dt \leq C(Y, \eps). \]
So we may assume for the remainder of the proof that $\alpha = 0$.

In the following we will frequently use Proposition~\ref{Prop_properties_splitting_map} in order to control the strong splitting map $(y_1, \ldots, y_k)$ without further reference.
We will also use the following bound for $p \in [1, 4]$, which follows from Propositions~\ref{Prop_improved_L2}, \ref{Prop_properties_splitting_map}:
\begin{equation}
\int_{ - \eps^{-1} }^{ - \eps } \int_M |\nabla q|^p d\nu_t dt
\leq C(\eps) \int_{ - \eps^{-1} }^{ - \eps } \int_M \Big(  |\nabla f|^p + \sum_{i=1}^k \big( y_i^{2p} + |\nabla y_i |^{2p}  \big)\Big) d\nu_t dt \leq C(Y,\eps) .\label{eq_nabq2_bound_C}
\end{equation}

Using the $(\delta, 1)$-static and $(\delta, 1)$-selfsimilar conditions, we can estimate
\begin{multline} \label{eq_nab2_q_computation}
  \int_{ - \eps^{-1} }^{ - \eps } \int_M \Big| \nabla^2 q - 2 \Big(  g -   \sum_{i=1}^k d y_i \otimes d y_i \Big) \Big|^2 d \nu_t dt  
  \leq \Psi (\delta )   + C \sum_{i=1}^k \int_{ - \eps^{-1} }^{- \eps } \int_M \big|   y_i  \nabla^2 y_i \big|^2 d \nu_t dt  \\
  \leq  \Psi (\delta | Y,\eps).
\end{multline}
This proves (\ref{eq_almost_radial_1}) for $\delta \leq \ov\delta (Y,\eps)$.
Next, we estimate
\begin{align*}
  \int_{-\eps^{-1}}^{-\eps} \bigg| \int_M & \nabla q \cdot \nabla y_i \, d\nu_t \bigg| dt 
=  \int_{-\eps^{-1}}^{-\eps}\bigg|  \int_M \Big( 4 \tau \nabla f \cdot \nabla y_i - 2 \sum_{j=1}^k y_j \nabla y_j \cdot \nabla y_i  \Big) d\nu_t \bigg| dt  \\
&= \int_{-\eps^{-1}}^{-\eps} \bigg|  \int_M \big( 4 \tau \triangle y_i - 2 y_i - 2 \sum_{j=1}^k y_j ( \nabla y_j \cdot \nabla y_i - \delta_{ij} )  \big) d\nu_t \bigg|  dt  \\
&\leq \Psi(\delta | Y, \eps) + C\sum_{j=1}^k \bigg( \int_{-\eps^{-1}}^{-\eps} \int_M y_j^2d\nu_t dt  \bigg)^{1/2} \bigg( \int_{-\eps^{-1}}^{-\eps} \int_M  |\nabla y_j \cdot \nabla y_i - \delta_{ij}|^2 d\nu_t dt \bigg)^{1/2} \\
&\leq \Psi(\delta | Y, \eps) .
\end{align*}
So by the $L^1$-Poincar\'e inequality (Proposition~\ref{Prop_Poincare}) and (\ref{eq_nabq2_bound_C}), (\ref{eq_nab2_q_computation}) we have
\begin{align}
  \int_{-\eps^{-1}}^{-\eps} \int_M & |\nabla q \cdot \nabla y_i | d\nu_t dt 
\leq C(\eps) \int_{-\eps^{-1}}^{-\eps} \int_M |\nabla (\nabla q \cdot \nabla y_i )| d\nu_t dt  + \Psi (\delta | Y, \eps ) \notag \\
&\leq C(\eps) \int_{-\eps^{-1}}^{-\eps} \int_M |\nabla^2 q \cdot \nabla y_i | d\nu_t dt 
+ C(\eps) \int_{-\eps^{-1}}^{-\eps} \int_M |\nabla q \cdot \nabla^2 y_i | d\nu_t dt  
+ \Psi (\delta | Y, \eps ) \notag \displaybreak[1] \\
&\leq C(\eps) \int_{-\eps^{-1}}^{-\eps} \int_M \Big| \nabla^2 q - 2\Big(  g -   \sum_{j=1}^k d y_j \otimes d y_j \Big) \Big| \, |\nabla y_i | d\nu_t dt \notag \\
&\qquad + C(\eps) \int_{-\eps^{-1}}^{-\eps} \int_M \Big| \nabla y_i - \sum_{j=1}^k (\nabla y_j \cdot \nabla y_i ) \nabla y_j \Big| d\nu_t dt \notag \\
&\qquad + C(\eps) \bigg( \int_{-\eps^{-1}}^{-\eps} \int_M |\nabla q|^2  d\nu_t dt \bigg)^{1/2} \bigg( \int_{-\eps^{-1}}^{-\eps} \int_M | \nabla^2 y_i |^2 d\nu_t dt \bigg)^{1/2}
+ \Psi (\delta | Y, \eps ) \notag \displaybreak[1] \\
&\leq C(\eps) \bigg( \int_{-\eps^{-1}}^{-\eps} \int_M \Big| \nabla^2 q - 2\Big(  g -   \sum_{i=1}^k d y_i \otimes d y_i \Big) \Big|^2  d\nu_t dt \bigg)^{1/2}  \bigg( \int_{-\eps^{-1}}^{-\eps} \int_M  |\nabla y_i |^2 d\nu_t dt \bigg)^{1/2} \notag \\
&\qquad + C(\eps)\sum_{j=1}^k \bigg( \int_{-\eps^{-1}}^{-\eps} \int_M \big|\nabla y_j \cdot \nabla y_i  - \delta_{ij} \big|^2 d\nu_t dt \bigg)^{1/2} \bigg( \int_{-\eps^{-1}}^{-\eps} \int_M |\nabla y_j|^2 d\nu_t dt \bigg)^{1/2} \notag \\
&\qquad 
+ \Psi (\delta | Y, \eps ) 
\notag \\ &\leq \Psi (\delta | Y, \eps ) , \label{eq_nabq_nabyi_computation}
\end{align}
which proves (\ref{eq_almost_radial_2}) for $\delta \leq \ov\delta (Y, \eps)$.

From (\ref{eq_nabq_nabyi_computation}) and (\ref{eq_nabq2_bound_C}) we obtain that for any $b > 0$
\begin{align*}
 \int_{-\eps^{-1}}^{-\eps} \int_M & |y_i| \, |\nabla q \cdot \nabla y_i | d\nu_t dt 
\leq b^{-1} \int_{-\eps^{-1}}^{-\eps} \int_M  |\nabla q \cdot \nabla y_i | d\nu_t dt + b \int_{-\eps^{-1}}^{-\eps} \int_M |y_i|^{2} |\nabla q \cdot \nabla y_i | d\nu_t dt \\
&\leq b^{-1} \Psi (\delta | Y, \eps) + b \bigg( \int_{-\eps^{-1}}^{-\eps} \int_M |\nabla q |^{2} d\nu_t dt \bigg)^{1/2} \bigg( \int_{-\eps^{-1}}^{-\eps} \int_M |y_i|^{4} | \nabla y_i |^2 d\nu_t dt \bigg)^{1/2} \\
&\leq b^{-1} \Psi (\delta | Y, \eps) + C(Y,\eps) b.
\end{align*}
Thus
\begin{equation} \label{eq_yi_nabq_naby}
  \int_{-\eps^{-1}}^{-\eps} \int_M |y_i| \, |\nabla q \cdot \nabla y_i | d\nu_t dt 
\leq \Psi (\delta | Y, \eps). 
\end{equation}
Similarly, we can argue that for any $b > 0$
\begin{align*}
 \int_{-\eps^{-1}}^{-\eps} \int_M & |\nabla y_i| \, |\nabla q \cdot \nabla y_i | d\nu_t dt 
\leq b^{-1} \int_{-\eps^{-1}}^{-\eps} \int_M  |\nabla q \cdot \nabla y_i | d\nu_t dt + b \int_{-\eps^{-1}}^{-\eps} \int_M |\nabla y_i|^{2} |\nabla q \cdot \nabla y_i | d\nu_t dt \\
&\leq b^{-1} \Psi (\delta | Y, \eps) + b \bigg( \int_{-\eps^{-1}}^{-\eps} \int_M |\nabla q |^{2} d\nu_t dt \bigg)^{1/2} \bigg( \int_{-\eps^{-1}}^{-\eps} \int_M | \nabla y_i |^6 d\nu_t dt \bigg)^{1/2} \\
&\leq b^{-1} \Psi (\delta | Y, \eps) + C(Y,\eps) b.
\end{align*}
Thus
\begin{equation} \label{eq_nabyi_nabq_naby}
  \int_{-\eps^{-1}}^{-\eps} \int_M |\nabla y_i| \, |\nabla q \cdot \nabla y_i | d\nu_t dt 
\leq \Psi (\delta | Y, \eps). 
\end{equation}

Next, we estimate, using (\ref{eq_yi_nabq_naby}),  (\ref{eq_nab2_q_computation}),
\begin{align}
 \int_{-\eps^{-1}}^{-\eps}  \bigg| \int_M & \big( |\nabla q|^2 - 8 \tau (n-k) \big) d\nu_t  \bigg| dt 
\notag \\
& =  \int_{-\eps^{-1}}^{-\eps}\bigg| \int_M \Big( 4 \tau \nabla q \cdot \nabla f - 2 \sum_{i=1}^k y_i \nabla q \cdot \nabla y_i - 8 \tau (n-k) \Big) d\nu_t \bigg| dt \notag \\
&\leq  \int_{-\eps^{-1}}^{-\eps} \bigg| \int_M \big( 4 \tau \triangle q   - 8\tau (n-k) \big) d\nu_t \bigg| dt  + \Psi (\delta | Y, \eps) \notag \\
&\leq  \int_{-\eps^{-1}}^{-\eps} \bigg| \int_M \big( 4 \tau \triangle q   - 8\tau n + 8\tau \sum_{i=1}^k |\nabla y_i|^2 \big) d\nu_t \bigg| dt \notag \\
&\qquad  + 8 \int_{-\eps^{-1}}^{-\eps} \tau \bigg| \int_M  \sum_{i=1}^k (|\nabla y_i|^2 -1) d\nu_t \bigg| dt    + \Psi (\delta | Y, \eps) \notag \\
&\leq \Psi (\delta | Y, \eps)  \label{eq_nabq_8nk}
\end{align}
and, using Propositions~\ref{Prop_NN_almost_constant_selfsimilar}, \ref{Prop_properties_splitting_map},
\begin{multline} \label{eq_q_8nk}
 \int_{-\eps^{-1}}^{-\eps}\bigg| \int_M \big(4 q - 8 \tau (n-k) \big) d\nu_t \bigg| dt 
=  \int_{-\eps^{-1}}^{-\eps}  \bigg| \int_M \Big( 16 \tau (f - W) - 8 \tau n - 4 \sum_{i=1}^k (y_i^2 - 2 \tau) \Big)   d\nu_t  \bigg| dt \\
\leq \int_{-\eps^{-1}}^{-\eps} 16 \tau | \NN_{x_0, 0} (\tau) - W| + \sum_{i=1}^k 4  \int_{-\eps^{-1}}^{-\eps} \bigg| \int_M (y_i^2 - 2\tau) d\nu_t \bigg| dt 
\leq \Psi (\delta | Y,\eps).
\end{multline}
Combining (\ref{eq_nabq_8nk}), (\ref{eq_q_8nk}) yields
\[   \int_{-\eps^{-1}}^{-\eps} \bigg| \int_M \big( |\nabla q|^2 - 4 q \big) d\nu_t \bigg| dt  \leq \Psi (\delta | Y, \eps). \]
So by the $L^1$-Poincar\'e inequality, (\ref{eq_nabq2_bound_C}), (\ref{eq_nab2_q_computation}), (\ref{eq_nabyi_nabq_naby})
\begin{align}
  \int_{-\eps^{-1}}^{-\eps} \int_M & \big| |\nabla q|^2 - 4 q \big| d\nu_t dt 
\leq  C(\eps) \int_{-\eps^{-1}}^{-\eps} \int_M \big| 2\nabla^2 q \cdot \nabla q - 4 \nabla q \big| d\nu_t dt  + \Psi (\delta | Y, \eps) \notag \\
&\leq  C(\eps) \int_{-\eps^{-1}}^{-\eps} \int_M \Big| \nabla^2 q - 2 \Big(  g -   \sum_{i=1}^k d y_i \otimes d y_i \Big) \Big| \, |\nabla q| d\nu_t dt  \notag \\
&\qquad + C(\eps) \int_{-\eps^{-1}}^{-\eps} \int_M \Big|  \sum_{i=1}^k (\nabla y_i \cdot \nabla q) \nabla y_i \Big| d\nu_t dt + \Psi (\delta | Y, \eps) \notag \\
&\leq C(\eps) \bigg( \int_{-\eps^{-1}}^{-\eps} \int_M \Big| \nabla^2 q - 2\Big( g -   \sum_{i=1}^k d y_i \otimes d y_i \Big) \Big|^2 d\nu_t dt \bigg)^{1/2} \bigg( \int_{-\eps^{-1}}^{-\eps} \int_M  |\nabla q|^2 d\nu_t dt \bigg)^{1/2} \notag  \\
&\qquad + C \sum_{i=1}^k \int_{-\eps^{-1}}^{-\eps} \int_M \big|   \nabla y_i \cdot \nabla q \big| \, |\nabla y_i| d\nu_t dt + \Psi (\delta | Y, \eps) \notag  \\
& \leq \Psi (\delta | Y, \eps) . \label{eq_nabq2_4q_small_proof}
\end{align}
This proves (\ref{eq_almost_radial_3}) for $\delta \leq \ov\delta (Y,\eps)$.

To see (\ref{eq_almost_radial_4}), we observe that since $(x_0,0)$ is $(\delta, 1)$-selfsimilar and $(\delta,1)$-static, we have by Proposition~\ref{Prop_almost_soliton_identities}
\begin{multline} \label{eq_partialt_f_W}
  \int_{-\eps^{-1}}^{-\eps} \int_M |\partial_t (\tau (f - W))| d\nu_t dt
  \leq
    \int_{-\eps^{-1}}^{-\eps} \int_M \Big| \square (\tau f) + W + \frac{n}2 \Big| d\nu_t dt 
    + \int_{-\eps^{-1}}^{-\eps} \int_M \Big| \tau \triangle  f - \frac{n}2 \Big| d\nu_t dt \\
\leq \Psi (\delta | Y,\eps).
\end{multline}
Moreover,
\begin{multline} \label{eq_partialt_yi2}
  \int_{-\eps^{-1}}^{-\eps} \int_M |\partial_t y_i^2| d\nu_t dt 
\leq 2 \int_{-\eps^{-1}}^{-\eps} \int_M |y_i| \, |\triangle y_i| d\nu_t dt \\
\leq 2 \bigg( \int_{-\eps^{-1}}^{-\eps} \int_M |y_i|^2 d\nu_t dt \bigg)^{1/2}  \bigg( \int_{-\eps^{-1}}^{-\eps} \int_M  |\triangle y_i|^2 d\nu_t dt \bigg)^{1/2}
\leq \Psi (\delta | Y,\eps).
\end{multline}
Combining (\ref{eq_partialt_f_W}), (\ref{eq_partialt_yi2}) implies (\ref{eq_almost_radial_4}) for $\delta \leq \ov\delta (Y,\eps)$.

Lastly, consider the case $k = n$.
Let $\eta > 0$ be a constant whose value we will determine in a moment and let $S \subset M \times [-\eps^{-1}, -\eps]$ be the subset where
\[ \sum_{i,j=1}^n |\nabla y_i \cdot \nabla y_j - \delta_{ij} | \geq \eta. \]
Then
\begin{equation} \label{eq_int_chi_X_small}
 \int_{-\eps^{-1}}^{-\eps} \int_M \chi_{S} \, d\nu_t dt \leq \Psi (\delta | \eps, \eta). 
\end{equation}
If $\eta \leq \ov\eta$, then on $M \times [-\eps^{-1}, -\eps] \setminus S$ we have
\[ |\nabla q| \leq C \sum_{i=1}^n |\nabla q \cdot \nabla y_i|. \]
It follows using  (\ref{eq_nabq_nabyi_computation}), (\ref{eq_int_chi_X_small}), (\ref{eq_nabq2_bound_C}) that
\begin{multline*}
 \int_{-\eps^{-1}}^{-\eps} \int_M |\nabla q| \, d\nu_t dt
\leq C \sum_{i=1}^n \int_{-\eps^{-1}}^{-\eps} \int_M |\nabla q \cdot \nabla y_i| \, d\nu_t dt + \int_{-\eps^{-1}}^{-\eps} \int_M |\nabla q| \chi_S \, d\nu_t dt \\
\leq \Psi (\delta | Y, \eps) + \bigg( \int_{-\eps^{-1}}^{-\eps} \int_M  \chi_S \, d\nu_t dt \bigg)^{1/2} \bigg( \int_{-\eps^{-1}}^{-\eps} \int_M |\nabla q|^2  d\nu_t dt \bigg)^{1/2} 
\leq \Psi (\delta | Y, \eps) .
\end{multline*}
So, again using (\ref{eq_nabq2_bound_C}),
\[ \int_{-\eps^{-1}}^{-\eps} \int_M |\nabla q|^2 \, d\nu_t dt
\leq \bigg( \int_{-\eps^{-1}}^{-\eps} \int_M |\nabla q| \, d\nu_t dt \bigg)^{2/3} \bigg( \int_{-\eps^{-1}}^{-\eps} \int_M |\nabla q|^4 \, d\nu_t dt \bigg)^{1/3}
\leq \Psi (\delta | Y, \eps) .  \]
Combining this with (\ref{eq_nabq2_4q_small_proof}) implies a similar bound for $|q|$.
This shows (\ref{eq_almost_radial_5}) and concludes the proof.
\end{proof}

\subsection{Extending splitting maps away from almost vertices} \label{subsec_almost_radial_splitting_extension}
The next result roughly states the following.
Suppose that $q$ is an almost radial function, which is constructed based on a strong $(k, \eps, r)$-splitting map as in (\ref{eq_def_almost_radial}).
Then points where $q \gtrsim (\la r)^2$ are in fact strongly $(k+1, \eps, \la r)$-split.

\begin{Proposition} \label{Prop_extending_splitting_maps_almost_radial}
If 
\[ \eps, \zeta > 0, 
\qquad  0 < \la < 1, \qquad Y, A < \infty, \qquad \beta \leq \ov\beta (Y, A, \eps), \qquad \delta \leq \ov\delta (Y,A, \eps, \la, \beta, \zeta), \]
then the following holds.
Let $(M, (g_t)_{t \in I})$ be a Ricci flow on a compact manifold and let $r > 0$  and $(x_0, t_0) \in M \times I$ with $[t_0 - \delta^{-1} r^2, t_0] \subset I$.
Suppose that $\vec y = (y_1, \ldots, y_k) : M \times [t_0 - \delta^{-1} r^2, t_0 - \delta r^2] \to \IR^k$ is a $(k, \delta, r)$-splitting map at $(x_0, t_0)$.
Assume moreover that $(x_0, t_0)$ is $( \delta, r)$-selfsimilar and $( \delta, r)$-static and that $W := \NN_{x_0, t_0} (r^2) \geq - Y$.
Denote by $d\nu = (4\pi \tau)^{-n/2} e^{-f} dg$ the conjugate heat kernel based at $(x_0, t_0)$ and define $q \in  C^\infty ( M \times [t_0 - \delta^{-1} r^2, t_0 - \delta r^2])$ as in (\ref{eq_def_almost_radial}).

Suppose that $(x_1,t_1) \in P^* (x_0, t_0; A r)$, $t_1 \leq t_0 - \eps r^2$ and that for some $\tau^* \in [\zeta (\la r)^2, (8 n)^{-1} (\la r)^2]$ the positive part $q_+$ satisfies
\[ \frac1{\tau^*} \int_{t_1 - 2\tau^*}^{t_1 - \tau^*} \int_M q_+ \, d\nu_{x_1, t_1;t} dt \geq (\la r)^2. \]
Then the point $(x_1, t_1)$ is strongly $(k+1, \eps, \beta \la r)$-split.
\end{Proposition}

\begin{Remark}
In fact, with some more effort it is possible to show the existence of a strong splitting map of the form $(y_1, \ldots, y_k, y'_{k+1})$ for some additional function $y'_{k+1}$.
\end{Remark}

\begin{proof}
We may assume that $t_0 = 0$ and $r = 1$.
Due Propositions~\ref{Prop_NN_variation_bound}, \ref{Prop_weak_splitting_map_to_splitting_map}, it suffices to show that $(x_1, t_1)$ is only weakly $(k+1, \eps, \beta \la r)$-split --- after adjusting $\eps, \beta$ depending on $Y, A$.

Fix $Y, A, \eps$ and choose $\alpha \leq \ov\alpha$ according to Propositions~\ref{Prop_properties_splitting_map}, \ref{Prop_construction_almost_radial}.
Without loss of generality, we may assume that $\beta \leq \eps$ and $\zeta \leq \eps \beta^2$.
Let $0<  \zeta' < 1$ be a constant whose value we will determine in the course of the proof.
Set $\nu^1_t := \nu_{x_1,t_1;t}$.
By Proposition~\ref{Prop_dist_expansion_almost_ss} we have for $\delta \leq \ov\delta (Y, A, \la, \zeta, \zeta')$
\[ d_{W_1}^{g_t} ( \nu (t), \nu^1 (t) ) \leq C(Y, A,\la, \zeta, \zeta'), \qquad \text{for all} \qquad t \in [t_1-1, t_1-\zeta' \zeta \la^2]. \]
So by Proposition~\ref{Prop_inheriting_bounds} we have that if $\zeta' \leq \ov\zeta' (\alpha)$, then
\[ d\nu^1 \leq C(Y, A, \la, \zeta, \zeta', \alpha) e^{\alpha f} d\nu \qquad \text{on} \quad M \times [t_1 - 1, t_1 - \zeta \la^2]. \]
It follows using Propositions~\ref{Prop_properties_splitting_map}, \ref{Prop_construction_almost_radial} that, after discarding the constants $\alpha, \zeta'$,
\begin{multline} \label{eq_bounds_wrt_x1t1}
\int_{t_1 -  \la^2}^{t_1 - \zeta \la^2} \int_M \bigg( \sum_{i,j=1}^k |\nabla y_i \cdot \nabla y_j - \delta_{ij}|^2 +  \Big| \nabla^2 q -2 \Big(  g -   \sum_{i=1}^k d y_i \otimes d y_i \Big) \Big|^2  \\
+ \sum_{i=1}^k |\nabla q \cdot \nabla y_i |
+ \big| |\nabla q|^2 - 4 q \big| 
+ |\partial_t q | \bigg) d\nu^1_t dt \leq \Psi (\delta | Y, A, \la,  \zeta). 
\end{multline}

The bound (\ref{eq_bounds_wrt_x1t1}) implies that for any $t^*, t^{**} \in [t_1 -  \la^2, t_1 -  \zeta \la^2]$ we have
\[ \bigg| \int_M ( q - 2(n-k) (t_1 - t))  d\nu^1_t \bigg|_{t=t^*}^{t=t^{**}} \bigg|
\leq \int_{t_1 - \la^2}^{t_1 - \zeta \la^2} \int_M |\square q + 2(n-k)| d\nu^1_t dt \leq \Psi (\delta | Y, A, \la, \zeta). \]
Combining this with the following bound for the negative part of $q$
\[ \int_{t_1 -  \la^2}^{t_1 - \zeta \la^2} \int_M q_- \, d\nu^1_t dt
\leq \int_{t_1 -  \la^2}^{t_1 - \zeta \la^2} \int_M \big| |\nabla q|^2 - 4 q \big| d\nu^1_t dt
\leq \Psi (\delta | Y, A, \la,  \zeta), \]
we obtain, using the assumption of the proposition, that if $\delta \leq \ov\delta (Y,A, \la,\zeta)$, then
\begin{equation} \label{eq_a_geq_12la2}
  a := \int_M q \, d\nu^1_{t_ 1 - (\beta \la)^2} \geq \la^2 - 2(n-k)(8 n)^{-1} \la^2 - \tfrac14 \la^2 \geq \tfrac12 \la^2. 
\end{equation}
Moreover, for all $t^{**} \in [t_1 - \eps^{-1} (\beta \la)^2, t_1 - \eps (\beta \la)^2] \subset [ t_1 - \la^2, t_1 - \zeta \la^2]$
\[ \bigg| \frac1a \int_M q \, d\nu^1_{t^{**}} - 1 \bigg| \leq \frac{2(n-k) \eps^{-1} (\beta \la)^2 + \Psi (\delta | Y, A, \la,  \zeta)}{a} 
\leq 4n \eps^{-1} \beta^2 + \Psi (\delta | Y, A, \la,  \zeta). \]
So by (\ref{eq_bounds_wrt_x1t1}) we have
\begin{equation} \label{eq_nabq_4a_1}
 (\beta \la)^{-2} \int_{t_1 -  \eps^{-1} (\beta \la)^2}^{t_1 - \eps( \beta \la)^2} \bigg| \int_M \bigg( \frac{|\nabla q|^2}{4a} -1 \bigg) d\nu^1_t \bigg| dt \leq 4n \eps^{-2} \beta^2 + \Psi (\delta | Y, A, \eps, \la, \zeta) .
\end{equation}
Thus by the $L^1$-Poincar\'e inequality (Proposition~\ref{Prop_Poincare}), (\ref{eq_nabq_4a_1}), (\ref{eq_bounds_wrt_x1t1}) we have
\begin{align}
(\beta \la)^{-2} &\int_{t_1 -  \eps^{-1} (\beta \la)^2}^{t_1 - \eps( \beta \la)^2}  \int_M \bigg| \frac{|\nabla q|^2}{4a} -1 \bigg| d\nu^1_t dt \notag \\
&\leq C(\eps) (\beta \la)^{-1} \int_{t_1 -  \eps^{-1} (\beta \la)^2}^{t_1 - \eps( \beta \la)^2} \int_M \frac{|\nabla^2 q| \, |\nabla q|}{4a}  d\nu^1_t dt
+ C(\eps) \beta^2 + \Psi (\delta | Y, A, \eps, \la, \zeta) \displaybreak[1] \notag \\
&\leq (\beta \la)^{-2} \beta \int_{t_1 -  \eps^{-1} (\beta \la)^2}^{t_1 - \eps( \beta \la)^2} \int_M \frac{ |\nabla q|^2}{4a}  d\nu^1_t dt
+ C(\eps)  \beta^{-1} \int_{t_1 -  \eps^{-1} (\beta \la)^2}^{t_1 - \eps( \beta \la)^2} \int_M \frac{|\nabla^2 q|^2}{4a}  d\nu^1_t dt
\notag \\ 
&\qquad + C(\eps) \beta^2 + \Psi (\delta | Y, A,  \eps, \la,  \zeta) \displaybreak[1] \notag \\
&\leq  (\beta \la)^{-2} \beta \int_{t_1 -  \eps^{-1} (\beta \la)^2}^{t_1 - \eps( \beta \la)^2}  \int_M  \bigg( \frac{ |\nabla q|^2}{4a} -1 \bigg) d\nu^1_t dt + C(\eps) \beta \notag \\
&\qquad
+ C(\eps) \beta^{-1} \la^{-2} \int_{t_1 -  \eps^{-1} (\beta \la)^2}^{t_1 - \eps( \beta \la)^2} \int_M  \Big| \nabla^2 q -2 \Big(  g -   \sum_{i=1}^k d y_i \otimes d y_i \Big) \Big|^2  d\nu^1_t dt \notag \\
&\qquad + C(\eps) \beta^{-1} \la^{-2} \int_{t_1 -  \eps^{-1} (\beta \la)^2}^{t_1 - \eps( \beta \la)^2} \int_M  \Big( 1 + \sum_{i=1}^k \big| |\nabla y_i|^2- 1|^2 \Big)  d\nu^1_t dt
\notag \\ 
&\qquad + C(\eps) \beta^2 + \Psi (\delta | Y, A,  \eps, \la,  \zeta) \displaybreak[1] \notag \\
&\leq  C(\eps) \beta^3 + C(\eps) \beta  + C(\eps) \beta^2 + \Psi (\delta | Y, A,  \eps, \la, \beta, \zeta) \notag \\
&\leq \Psi (\beta | \eps) + \Psi(\delta | Y, A, \eps, \la, \beta, \zeta). \label{eq_nabq_4a_abs_value}
\end{align}

Set now
\[ y_{k+1} := \frac{q}{2\sqrt{a}}. \]
Then using (\ref{eq_bounds_wrt_x1t1}), (\ref{eq_a_geq_12la2}), (\ref{eq_nabq_4a_1}), (\ref{eq_nabq_4a_abs_value}) we obtain that if $\beta \leq \ov\beta (\eps)$ and $\delta \leq \ov\delta (Y, A, \eps, \la, \beta, \zeta)$, then for any $i,j = 1, \ldots, k+1$
\begin{multline*}
 (\beta \la)^{-1} \int_{t_1 -  \eps^{-1} (\beta \la)^2}^{t_1 - \eps( \beta \la)^2}  \int_M \big| \square y_{k+1} \big| d\nu^1_t dt 
 \leq  C \beta^{-1} \la^{-2} \int_{t_1 -  \eps^{-1} (\beta \la)^2}^{t_1 - \eps( \beta \la)^2}  \int_M \Big| \triangle q - 2n + 2 \sum_{i=1}^k |\nabla y_i|^2 \Big| d\nu^1_t dt \\
 + C \beta^{-1} \la^{-2} \int_{t_1 -  \eps^{-1} (\beta \la)^2}^{t_1 - \eps( \beta \la)^2}  \int_M \Big( | \partial_t q | +1 + \sum_{i=1}^k \big| |\nabla y_i|^2 - 1 \big| \Big)  d\nu^1_t dt  
 \leq \frac{\eps}2 + C(\eps) \beta^{-1} \la^{-2} (\beta \la)^2
 \leq
 \eps, 
 \end{multline*}
\begin{equation*}
 (\beta \la)^{-2} \int_{t_1 -  \eps^{-1} (\beta \la)^2}^{t_1 - \eps( \beta \la)^2}  \int_M \big| \nabla y_i \cdot \nabla y_j - \delta_{ij} \big| d\nu^1_t dt \leq \eps. 
\end{equation*}
Since $\square y_i = 0$ for $i = 1, \ldots, k$, this shows that $(y_1, \ldots, y_{k+1})$ is a weak $(k+1, \beta \la, \eps)$-splitting map, as desired.
\end{proof}

\section{\texorpdfstring{A preliminary $\eps$-regularity theorem}{A preliminary {\textbackslash}eps-regularity theorem}} \label{sec_prelim_eps_reg}
In the following we will show the first $\eps$-regularity theorem involving strong splitting maps.
The following proposition provides local curvature bounds near points that are $(\eps, r)$-static and strongly $(n,\eps,r)$-split for a sufficiently small $\eps$.
This proposition will provide us a first tool for a preliminary partial regularity of $\IF$-limits in Section~\ref{Sec_basic_limits}.
This theory will, in turn, be used to establish more general $\eps$-regularity results in the subsequent sections.

Recall that the curvature scale $\rrm(x,t)$ at a point $(x,t)$ of Ricci flow $(M,(g_t)_{t \in I})$ is defined as the supremum over all $r > 0$ with the property that $|{\Rm}| \leq r^{-2}$ on the parabolic ball $P(x,t;r)$.

\begin{Proposition} \label{Prop_eps_regularity_np2}
Let $Y< \infty$ and suppose that $0 <\eps \leq \ov\eps (Y)$. Then the following holds.

Let $(M, (g_t)_{t \in I})$ be Ricci flow on a compact manifold, $r > 0$ and $(x_0,t_0) \in M \times I$ such that $\NN_{x_0,t_0} (r^2) \geq - Y$.
Suppose that $(x_0, t_0)$ is $(\eps, r)$-static and strongly $(n, \eps, r)$-split.
Then $\rrm (x_0, t_0) \geq \eps r$.
\end{Proposition}

\begin{proof}
Using Proposition~\ref{Prop_NN_almost_constant_selfsimilar}, we may assume, after adjusting $r$ and $\eps$ by a suitable constant, that $(x_0, t_0)$ is in addition $(\eps, r)$-selfsimilar.
Next, we may assume by parabolic rescaling that $r = 1$ and $t_0 = 0$.
Choose a strong $(n, \eps, r)$-splitting map $\vec y = (y_1, \ldots, y_n): M \times [-\eps^{-1}, -\eps] \to \IR^k$.
Denote by $d\nu = (4\pi \tau)^{-n/2} e^{-f} dg$ the conjugate heat kernel based at $(x_0, 0)$, set $W := \NN_{x_0, 0} (1)$ and
\[ q := 4 \tau ( f -  W) - \sum_{i=1}^n y_i^2 . \]
It follows from Proposition~\ref{Prop_construction_almost_radial} that
\[ \int_{-2}^{-1} \int_M |q| d\nu_t dt \leq \Psi (\eps |Y). \]
Let $\delta > 0$ be a constant whose value we will determine later.
Using Propositions~\ref{Prop_properties_splitting_map}, \ref{Prop_improved_L2}, \ref{Prop_L_infty_HK_bound} and assuming that $\eps \leq \ov\eps (\delta, Y)$ it follows that there is a $\tau_0 \in [1,2]$ and $A \leq \ov{A}(Y)$ for which the bounds (\ref{eq_z_integral_bounds})--(\ref{eq_f_lower_mA}) in the following lemma hold.
So $\NN_{x_0, 0} (1) = W \geq - \Psi (\delta |Y)$ and the proposition follows from \cite[\HKThmEpsRegularity]{Bamler_HK_entropy_estimates} if $\delta \leq \ov\delta (Y)$, $\eps \leq \ov\eps$.
\end{proof}
\bigskip

\begin{Lemma} \label{Lem_lower_bound_WW}
If  $\eps > 0$, $A < \infty$ and $\delta \leq \ov\delta ( \eps, A)$, then the following is true.

Let $(M, g)$ be a compact, $n$-dimensional Riemannian manifold equipped with a probability measure $\nu$ of the form $d\nu= (4\pi \tau)^{-n/2} e^{-f} dg$ for $f \in C^\infty (M)$, $\tau_0 > 0$.
Let $y_1, \ldots, y_n \in C^\infty(M)$ and assume that we have for some $W \leq 0$:
\begin{align}
   \int_M \bigg( \tau_0^{1/2} \sum_{i=1}^n |\nabla^2 y_i | + \sum_{i,j=1}^n | \nabla y_i \cdot \nabla y_j - \delta_{ij} | + \bigg| f - \frac1{4\tau_0} \sum_{i=1}^n y_i^2 - W \bigg| \bigg)   d\nu  &\leq \delta \label{eq_z_integral_bounds}\displaybreak[1] \\
 \int_M  |\nabla f |^2  d\nu  &\leq A \label{eq_nab_f_bound_lemma_W} \displaybreak[1] \\
 \int_M y_i^2 \, d\nu &\leq A, \label{eq_z_squared_int}  \displaybreak[1] \\
 f &\geq  - A.  \label{eq_f_lower_mA}
\end{align}
Then $W \geq - \eps$.
\end{Lemma}

\begin{proof}
After rescaling the metric by $\tau_0^{-1}$ and replacing each function $y_i$ with $\tau_0^{-1/2} y_i$, we may assume in the following that $\tau_0 = 1$.

Let $Z < \infty$, $\alpha > 0$ be constants whose values we will determine later.
Let $\eta_1 : \IR \to [0,1]$ be a cutoff function such that $\eta_1 \equiv 1$ on $(-\infty, Z+1]$ and $\eta_1 \equiv 0$ on $[Z+2, \infty)$ and $|\eta'_1| \leq 10$.
Let $\eta_2 : \IR \to [0,1]$ be a cutoff function such that $\eta_2 \equiv 1$ on $[1-\alpha, 1+\alpha]$ and $\eta_2 \equiv 0$ outside $[1-2\alpha, 1+2\alpha]$ and $|\eta'_2| \leq 10/\alpha$.
Consider the map $\vec y := (y_1, \ldots, y_n ) : M \to \IR^n$ and set
\[ w := \eta_1 (f) \eta_2 \Big( \sum_{i,j=1}^n (\nabla y_i \cdot \nabla y_j - \delta_{ij})^2 \Big). \] 

Since $W \leq 0$, we have by (\ref{eq_z_integral_bounds})
\begin{align}
 \int_{ \{ Z+1 < f  \} \cap \{ \frac14 \sum_i y_i^2 < Z \} }d\nu
\leq \int_{ \{ Z+1 < f  \} \cap \{ \frac14 \sum_i y_i^2 < Z \} } \Big| f - \frac14 \sum_i y_i^2 - W \Big|d\nu 
&\leq \Psi(\delta ), \label{eq_eta1_supp}  \\
  \int_{  \{  \eta_2  < 1 \} } d\nu \leq  \int_M C(\alpha)\sum_{i,j=1}^n | \nabla y_i \cdot \nabla y_j - \delta_{ij} |d\nu &\leq \Psi( \delta |\alpha), \label{eq_eta2_supp}
\end{align}
where we used the shorthand notation $\{ \eta_2 < 1 \} := \{  \eta_2 ( \sum_{i,j=1}^n (\nabla y_i \cdot \nabla y_j - \delta_{ij})^2) < 1 \}$.
Therefore, if we use the Riemannian background measure $dg$, then by (\ref{eq_nab_f_bound_lemma_W})
\begin{multline} \label{eq_nab_eta1}
\int_{\{ \frac14 \sum_i y_i^2 < Z \}} |\nabla (\eta_1 (f))| dg
\leq 10 (4\pi)^{n/2} e^{Z+2} \int_{ \{ Z+1 < f < Z+2 \} \cap \{ \frac14 \sum_i y_i^2 < Z \} } |\nabla f| d\nu \\
 \leq C e^{Z} \bigg( \int_{ \{ Z+1 < f < Z+2 \} \cap \{ \frac14 \sum_i y_i^2 < Z \} }  |\nabla f|^2 d\nu \bigg)^{1/2} \bigg( \int_{ \{ Z+1 < f < Z+2 \} \cap \{ \sum_i y_i^2 < Z \} }   d\nu \bigg)^{1/2} 
 \leq    \Psi (\delta |A,Z) ,
\end{multline}
and, using (\ref{eq_z_integral_bounds}), we have for $\alpha \leq \ov\alpha$
\begin{multline}
 \int_{\{ f < Z+2 \}}  \Big|\nabla \Big( \eta_2 \Big( \sum_{i,j=1}^n (\nabla y_i \cdot \nabla y_j - \delta_{ij})^2 \Big)\Big)\Big| dg  \leq
C\alpha^{-1} \int_{\{ f < Z+2 \} \cap  \{ 0 < \eta_2 < 1 \}} |\nabla^2 y_i| dg \\
\leq C(\alpha) e^{Z+2} \sum_{i=1}^n \int_{\{ f < Z+2 \} }  |\nabla^2 y_i|  d\nu  
 \leq \Psi(\delta|Z, \alpha) . \label{eq_nab_eta2}
\end{multline}
Combining (\ref{eq_nab_eta1}), (\ref{eq_nab_eta2}) yields
\begin{equation} \label{eq_nab_w_bound}
 \int_{\{ \frac14 \sum_i y_i^2 <Z \}} |\nabla w| \, dg \leq  \Psi (\delta | A, Z, \alpha) . 
\end{equation}
If $\alpha \leq \ov\alpha$, then $\vec y$ restricted to $\{ w > 0 \}$ is a local diffeomorphism whose differential has eigenvalues within $[ 1- \Psi(\alpha), 1+\Psi(\alpha)]$.
Define $\td{w} : \IR^n \to [0, \infty]$ by
\[ \td{w} (x_1, \ldots, x_n ) := \sum_{y_i (p) = x_i} w(p). \]
Then (\ref{eq_nab_w_bound}) implies that $\td{w}|_{B(\vec 0, 2 \sqrt{Z})} \in W^{1,1} (B(\vec 0, 2 \sqrt{Z}))$ and if $\alpha \leq \ov\alpha$
\[ \int_{B(\vec 0, 2 \sqrt{Z})} |\nabla \td{w} | 
\leq  2\int_{\{ \frac14 \sum_i y_i^2 <Z \}} |\nabla w| \, dg \leq \Psi (\delta | A, Z, \alpha) . \]
So by the $L^1$-Poincar\'e inequality on $\IR^n$ there is an $m \geq 0$ such that
\begin{equation} \label{eq_tdw_poincare}
 \int_{B(\vec 0, 2 \sqrt{Z})}  | \td{w} - m | \leq  \Psi (\delta | A, Z, \alpha). 
\end{equation}

\begin{Claim} \label{Cl_m_12}
If $Z \geq \underline{Z}(A)$, $\alpha \leq \ov\alpha$, $\delta \leq \ov\delta ( A, Z, \alpha)$, then $m > \frac12$.
\end{Claim}

\begin{proof}
Combining (\ref{eq_eta1_supp}), (\ref{eq_eta2_supp}) implies
\begin{equation} \label{eq_int_wl1c14ylZ}
 \int_{  \{ w < 1 \} \cap \{ \frac14 \sum_i y_i^2 < Z \} } d\nu \leq \Psi(\delta | \alpha). 
\end{equation}
It follows that for $\alpha \leq \ov\alpha$
\begin{multline*}
 \big|  \{ 0 < \td{w} < 1 \} \cap B(\vec 0, 2 \sqrt{Z})  \big|
 \leq C \int_{ \{ 0 < w < 1 \} \cap \{ \frac14 \sum_i y_i^2 < Z \}} dg 
 \leq C e^{Z+2} \int_{ \{ 0 < w < 1 \} \cap \{ \frac14 \sum_i y_i^2 < Z \}} d\nu \\
\leq  \Psi (\delta | Z, \alpha). 
\end{multline*}
Therefore, if $m \leq \frac12$, then (\ref{eq_tdw_poincare}) implies
\[ \int_{ \{ \frac14 \sum_i y_i^2 < Z \} } w \, dg 
= \int_{B(\vec 0, 2 \sqrt{Z})}  \td{w} 
\leq \Psi (\delta | Z, \alpha) + 2 \int_{B(\vec 0, 2 \sqrt{Z})} | \td{w} - m |  
\leq  \Psi (\delta | A, Z, \alpha). \]
So, using (\ref{eq_f_lower_mA}), (\ref{eq_int_wl1c14ylZ}),
\begin{multline*}
 \int_{ \{ \frac14 \sum_i y_i^2 < Z \} } d\nu 
 \leq \int_{  \{ w < 1 \} \cap \{ \frac14 \sum_i y_i^2 < Z \} } d\nu  
 + \int_{  \{ w = 1 \} \cap \{ \frac14 \sum_i y_i^2 < Z \} } w \, d\nu \\
 \leq \Psi(\delta| \alpha) + C e^A \int_{ \{ \frac14 \sum_i y_i^2 < Z \} } w \, dg
 \leq
 \Psi (\delta | A, Z, \alpha).
\end{multline*}
This implies
\[ \int_M \bigg( \frac14\sum_i y_i^2 \bigg) d\nu \geq Z \big( 1 -  \Psi (\delta | A, Z, \alpha)\big),  \]
in contradiction to (\ref{eq_z_squared_int}) if $Z \geq \underline{Z}(A)$ and $\delta \leq \ov\delta ( A, Z, \alpha)$.
\end{proof}

Claim~\ref{Cl_m_12} combined with (\ref{eq_tdw_poincare}) implies that
\begin{equation} \label{eq_tdw0_small}
 \big|  \{  \td{w}  = 0 \} \cap B(\vec 0, 2 \sqrt{Z})  \big| \leq   \Psi (\delta | A, Z, \alpha). 
\end{equation}
Since
\[ \exp \bigg({ f - \frac14 \sum_{i=1}^n y_i^2 } - W \bigg) \leq 1 + e^{f-W}  \bigg| {f - \frac14 \sum_{i=1}^n y_i^2 } - W \bigg| \leq 1 + e^{Z-W}  \bigg|{ f - \frac14 \sum_{i=1}^n y_i^2 } - W \bigg|, \]
we conclude using (\ref{eq_tdw0_small}), (\ref{eq_z_integral_bounds})
\begin{align*}
 \int_{B(\vec 0, 2 \sqrt{Z})} &(4\pi)^{-n/2}  \exp \Big( {- \frac14 \sum_{i=1}^n x_i^2 } \Big) d\vec{x} -  \Psi (\delta | A, Z, \alpha) \\
&\leq  \int_{\{ \td w > 0 \} \cap B(\vec 0, 2 \sqrt{Z})} (4\pi)^{-n/2} \exp \Big( {- \frac14 \sum_{i=1}^n x_i^2 } \Big) d\vec{x} \\
&\leq  (1+\Psi(\alpha)) \int_{ \{ w > 0 \} \cap \{  \frac14 \sum_{i=1}^n y_i^2 < Z \} }(4\pi)^{-n/2} \exp \Big( {- \frac14 \sum_{i=1}^n y_i^2 } \Big) dg \\
&\leq (1+\Psi(\alpha)) e^W  \int_{ \{  f < Z+2 \} \cap \{ \frac14 \sum_{i=1}^n y_i^2 < Z \} }\exp \Big( {f - \frac14 \sum_{i=1}^n y_i^2 } - W \Big) (4\pi)^{-n/2} e^{-f} dg \\
&\leq (1+\Psi(\alpha)) e^W  \int_{M } \bigg( 1 + e^{Z + 2-W}  \bigg|{ f - \frac14 \sum_{i=1}^n y_i^2 } - W \bigg| \bigg) d\nu \\
&\leq (1+\Psi(\alpha)) e^W \big( 1 + e^{Z+2-W} \delta \big)
\leq (1+\Psi(\alpha)) e^W + \Psi(\delta |Z, \alpha).
\end{align*}
Therefore
\[ e^{W} \geq \frac{1 - \Psi (Z)-  \Psi (\delta | A, Z, \alpha)}{1+\Psi(\alpha) } \geq 1 - \Psi(\alpha) -  \Psi(Z ) -\Psi(\delta |A, Z, \alpha) . \]
Choosing $\alpha \leq \ov\alpha (\eps)$, $Z \geq \underline{Z} (\eps)$ and $\delta \leq \ov\delta(A,Z, \eps)$ yields $W \geq  - \eps$, which finishes the proof.
\end{proof}

\part{\texorpdfstring{Partial regularity theory of $\IF$-limits}{Partial regularity theory of F-limits}} \label{part_partial_reg}
In this part we combine the quantitative stratification result from Part~\ref{part_prel_quant_strat} with the local regularity theory from Part~\ref{part_loc_reg}.
In Section~\ref{Sec_basic_limits} we obtain a preliminary partial regularity theory of non-collapsed $\IF$-limits.
This theory will depend on a working assumption that concerns an $\eps$-regularity property and involves a parameter $\Delta \in \{ 1, 2,3,4 \}$; the $\eps$-regularity theorem from Part~\ref{part_loc_reg} implies this working assumption for $\Delta = 1$.
We will find that the singular part of non-collapsed $\IF$-limits has codimension $\geq \Delta$.
In Sections~\ref{sec_eps_reg_codim_2}, \ref{sec_eps_reg_codim_4} we will use this preliminary partial regularity theory to successively show that the working assumption holds for $\Delta =2$ and $\Delta = 4$, respectively.

\section{A preliminary partial regularity theory} \label{Sec_basic_limits}
In this section, we combine the results obtained so far to derive a partial regularity and structure theory for non-collapsed $\IF$-limits.
This theory for itself will still be relatively weak, however, under a working assumption, which concerns $\eps$-regularity theorems and depends on a parameter $\Delta \in \{ 1,2,3,4 \}$, we will be able to derive improved results.
The parameter $\Delta$ will turn out to give a bound on the codimension of the singular set of the limit.
The $\eps$-regularity theorem from Section~\ref{sec_prelim_eps_reg} corresponds to $\Delta = 1$.
Our goal in the subsequent sections will be to verify this working assumption for $\Delta = 2$ and $\Delta =4$.
Once this is accomplished, the results in this section will imply a structure theory with optimal dimensional bounds

\subsection{Setup and working assumption} \label{subsec_setup_limit}
Throughout this section we will consider the setting from Subsection~\ref{subsec_part_reg_limit_intro}.
Let us recall this setting and provide some further details. 

Consider a sequence of pointed Ricci flows $(M_i, (g_{i,t})_{t \in I_i}, x_i)$ on compact, $n$-dimensional manifolds and over time-intervals of the form $I_i = (-T_i,0]$, $T_i \in (0, \infty]$, where $\lim_{i \to \infty} T_i =: T_\infty \in (0, \infty]$.
We will view the point $x_i$ to be located at time $0$ and assume the non-collapsing condition
\[ \NN_{x_i, 0} (\tau_0) \geq - Y_0 \]
for some uniform $Y_0 < \infty$, $\tau_0 > 0$.

As discussed in \cite{Bamler_RF_compactness}, these flows correspond to $H_n := ( \frac{(n-1)\pi^2}{2} +4)$-concentrated metric flows $\XX^i$ with full regular part $\RR^i = \XX^i = M_i \times I_i$, which are equipped with a Ricci flow spacetime structure $(\RR^i, \tf, \partial^i_{\tf}, g^{i})$ that corresponds to the original flows $(M_i, (g_{i,t})_{t \in I_i})$.
In the following, we will often write $(M_i, (g_{i,t})_{t \in I_i})$ instead of $\XX^i$ if there is no chance of confusion.
We will also frequently use terminology relating to Ricci flow spacetimes, as defined in \cite[\SYNSubsecRFST]{Bamler_RF_compactness}. 

Consider now a future continuous, $H_n$-concentrated metric flow $\XX$ of full support over $I_\infty := (-T_\infty,0]$ whose final time-slice is of the form $\XX_0 = \{ x_\infty \}$.
Suppose that there is a correspondence $\CF$ over $I_\infty$ between the metric flows $\XX^i$ associated to $(M_i, (g_{i,t})_{t \in I_i})$, $i \in \IN$ and $\XX$, such that on compact time-intervals:
\begin{equation} \label{eq_setup_Mi_to_XX}
 (M_i, (g_{i,t})_{t \in I_i}, (\nu_{x_i,0;t})_{t \in I_i}) \xrightarrow[i\to \infty]{\quad \IF, \CF \quad} (\XX, (\nu_{x_\infty;t})_{t \in I_\infty}).
\end{equation}
We remark that it would be more accurate to replace ``$M_i, (g_{i,t})_{t \in I_i}$'' by ``$\XX^i$'' in (\ref{eq_setup_Mi_to_XX}).

Before discussing some of the basic consequences of (\ref{eq_setup_Mi_to_XX}), let us briefly recall that given \emph{any} sequence of Ricci flows $(M_i, (g_{i,t})_{t \in I_i})$ as described above, we can apply \cite[\SYNCorCompactness, \SYNThmCompactnessFutCont]{Bamler_RF_compactness} to construct $\CF$, $\XX$ such that (\ref{eq_setup_Mi_to_XX}) holds.
We also emphasize that we are always free to choose $\XX$ to be future continuous.
This property will become useful later.

For the remainder of this section we will fix the convergent subsequence of Ricci flows $(M_i,\lb  (g_{i,t})_{t \in I_i},\lb x_i)$, its limit $\XX$ and the correspondence $\CF$ such that (\ref{eq_setup_Mi_to_XX}) holds.
It will not be important how these objects have been constructed.
Note that we have dropped the index ``$\infty$'' from the limit $\XX$ for convenience; the same will apply to most objects relating to $\XX$, such as its regular part and the heat kernel.
Similarly, we will often drop the superscript ``$i$'' for objects relating to the metric flows $\XX^i$ as long as there is no chance of confusion.

Since we assumed $\XX$ to be future continuous, we have:

\begin{Lemma}
$\XX$ is intrinsic.
\end{Lemma}

\begin{proof}
This follows from \cite[\SYNThmsIntrinsic]{Bamler_RF_compactness}.
\end{proof}

As discussed in \cite[\SYNSubsecRegPart]{Bamler_RF_compactness}, we can decompose $\XX$ into its regular ($\RR$) and singular ($\SS := \XX \setminus \RR$) part
\begin{equation} \label{eq_reg_sing_dec}
  \XX = \RR {\,\, \dotcup \,\,} \SS, 
\end{equation}
where $\RR$ is open and $\SS$ is closed.
Here $\RR$ carries the structure of a Ricci flow space-time $(\RR, \mathfrak{t}, \partial_{\mathfrak{t}}, g)$ with the property that for every time $t \in I_\infty$ the topology of $\RR_t = \XX_t \cap \RR$ agrees with the subspace topology from $\XX$ and $\XX_t$ and the inclusion map $(\RR_t, d_{g_t}) \to (\XX, d_t)$ is a local isometry; here and in the sequel $d_{g_t}$ denotes the length metric of $g_t$ and $d_t$ denotes the metric on $\XX_t$ coming from the metric flow structure.
Note that this implies that $d_t \leq d_{g_t}$ and that the intersection of any $d_t$-geodesic with $\RR_t$ is a Riemannian geodesic.

By \cite[\SYNThmRegPartProperties]{Bamler_RF_compactness} we can express conjugate heat kernels restricted to $\RR$ as follows:
\[ d\nu_{x;s} = K(x, \cdot) dg_s, \qquad x \in \XX, s \in I_\infty, s < \tf (x), \]
where
\[ K : \{ (x,y) \in \XX \times \RR \;\; : \;\; \tf(x) > \tf(y) \} \longrightarrow \IR_+ \]
is a continuous function with the following properties, called the \emph{heat kernel} (see also \cite[\SYNDefHK]{Bamler_RF_compactness}).
For any $x \in \XX$ the function $K(x, \cdot ) : \RR_{< \tf(x)} \to \IR_+$ is smooth and satisfies the conjugate heat equation $\square^* K(x, \cdot) = 0$.
The map $x \mapsto K(x,0)$ is continuous in the $C^\infty_{\loc}$-topology.
Moreover, $K$ restricted to $\{ (x,y) \in \RR \times \RR \;\; : \;\; \tf(x) > \tf(y) \}$. For any $y \in \RR$, the function $K(\cdot, y)$ is a heat flow on $\XX_{>\tf(y)}$ and $\square K(\cdot, y) = 0$ on $\RR_{>\tf(y)}$.

In the following, we will sometimes express $K(x, \cdot)$, for $x \in \XX$, as
\[ K(x, \cdot) =: (4\pi \tau_x)^{-n/2} e^{-f_x}, \]
where $f_x \in C^\infty(\RR_{< \tf(x)})$ and $\tau_x := \tf(x) - \tf$.
For any $x \in \XX$ and $\tau > 0$ we define the pointed Nash-entropy at $x$ as usual (see also the discussion in Subsection~\ref{subsec_Nash_in_limit_intro})
\begin{equation*}
 \NN_x (\tau) := \int_{\RR_{\tf (x) - \tau}} f_x \, d\nu_{x; \tf(x) - \tau} - \frac{n}2, 
\end{equation*}
where integration is only performed over the regular part and $\RR_{\tf (x) - \tau}$ may a priori be empty.

Next, let $\RR^* \subset \RR$ be the set of points where the convergence (\ref{eq_setup_Mi_to_XX}) is smooth, as defined in  \cite[\SYNSubsecSmoothConv]{Bamler_RF_compactness}, and set $\SS^* := \XX \setminus \RR^* \supset \SS$.
So we have
\begin{equation} \label{eq_reg_sing_dec_star}
 \XX = \RR^* {\,\, \dotcup \,\,} \SS^*, 
\end{equation}
We will later see that $\RR = \RR^*$, meaning that the decompositions (\ref{eq_reg_sing_dec}), (\ref{eq_reg_sing_dec_star}) agree (see Corollary~\ref{Cor_RRstar_RR}).
At this moment, note that $\RR, \SS$ depend only on the limit $\XX$, while $\RR^*, \SS^*$ also depend on the correspondence $\CF$ describing the convergence (\ref{eq_setup_Mi_to_XX}).
We warn the reader that $\RR^*$ may a priori become larger if we pass to a subsequence in (\ref{eq_setup_Mi_to_XX}).

Let us now recall some of the properties that follow from the smooth convergence on $\RR^*$ and that we will be using most frequently throughout this section.
For a more comprehensive discussion see \cite[\SYNThmSmoothConv]{Bamler_RF_compactness}.
The subset $\RR^* \subset \RR$ is open and we can find an increasing sequence of open subsets $U_1 \subset U_2 \subset \ldots \subset \RR^*$ with $\bigcup_{i=1}^\infty U_i = \RR^*$, open subsets $V_i \subset \RR^i = M_i \times I_i$ and time-preserving diffeomorphisms $\psi_i : U_i \to V_i$ such that on $\RR^*$
\[ \psi^*_i g^i \xrightarrow[i \to \infty]{\;\; C^\infty_{\loc} \;\;} g, \qquad
\psi^*_i \partial_{\tf}^i \xrightarrow[i \to \infty]{\;\; C^\infty_{\loc} \;\;} \partial_{\tf}, \qquad 
K^i (x_i, 0; \cdot) \circ \psi_i\xrightarrow[i \to \infty]{\;\; C^\infty_{\loc} \;\;} K(x_\infty; \cdot) , \]
where $K^i$ denotes the heat kernel on $(M_i, (g_{i,t})_{t \in I_i})$.
Moreover, the following is true.
For any sequence $(x'_i, t'_i) \in M_i \times I_i$ for which we have $(x'_i, t'_i) \to x'_\infty \in \XX$ within $\CF$ in the sense of \cite[\SYNDefConvPtsWithin]{Bamler_RF_compactness} (recall that by definition this is equivalent to the convergence of the heat kernels based at the corresponding points) we have on $\RR^*$
\[ K^i (x'_i, t'_i; \cdot) \circ \psi_i\xrightarrow[i \to \infty]{\;\; C^\infty_{\loc} \;\;} K(x'_\infty; \cdot) . \]
On $\{ (x', y') \in \RR^* \times \RR^* \;\; : \;\;  \tf(x') > \tf(y') \}$ we have
\[ K^i \circ (\psi_i, \psi_i)\xrightarrow[i \to \infty]{\;\; C^\infty_{\loc} \;\;} K . \]
Lastly, for any sequence $(x'_i, t'_i) \in M_i \times I_i$ and any point $x'_\infty \in \RR^*$ we have $(x'_i, t'_i) \to x'_\infty$ within $\CF$ if and only if $(x'_i, t'_i) \in V_i$ for large $i$ and $\psi_i^{-1} (x'_i, t'_i) \to x'_\infty$ in $\RR^*$.

Let us now fix the sequence of diffeomorphisms $\psi_i : U_i \to V_i$ for the remainder of this section.
Motivated by the last property of the maps $\psi_i$, we will make the following definition.

\begin{Definition}[Pointwise and local uniform convergence]
If $f_i : U_i \to Y$, $U_i \subset M_i \times I_i$, is a sequence of maps into some topological space $Y$ and $f_\infty : U_\infty \to Y$ for some open subset $U_\infty \subset \RR^*$, then we say that we have {\bf pointwise} or {\bf local uniform convergence} $f_i \to f_\infty$ if $f_i \circ \psi_i \to f$ converges pointwise or locally uniform.
If $Y$ is a smooth manifold, then we define $C^\infty_{\loc}$-convergence similarly.
\end{Definition}

We will now introduce the following working assumption involving a fixed parameter $\Delta \in \{ 1,2,3,4 \}$, which we will assume for the remainder of this section.

\begin{Assumption} \label{Aspt_working_ass}
For any $Y' < \infty$ there is a constant $\eps' (Y') > 0$ such that the following holds for all $i$.
Suppose that $(M, (g_t)_{t \in I})$ is a Ricci flow on a compact manifold and $(x,t) \in M \times I$, $r > 0$ satisfy $[t- \eps^{\prime -1} r^2, t] \subset I$ and $\NN_{x,t} (r^2) \geq - Y'$.
Suppose that one of the following is true:
\begin{enumerate}[label=(\roman*)]
\item $(x,t)$ is strongly $(n+3-\Delta, \eps', r)$-split or
\item $(x,t)$ is $(\eps', r)$-static and strongly $(n+1-\Delta, \eps', r)$-split.
\end{enumerate}
Then $\rrm (x,t) \geq \eps' r$.
\end{Assumption}

Note that we have the following:

\begin{Lemma}
Assumption~\ref{Aspt_working_ass} holds at least for $\Delta = 1$.
\end{Lemma}

\begin{proof}
This follows from Proposition~\ref{Prop_eps_regularity_np2} and the fact that for sufficiently small $\eps'$ no point in a Ricci flow on an $n$-dimensional manifold is strongly $(n+2, \eps', r)$-split if $\eps' \leq \ov\eps'$.
\end{proof}

As it will turn out later in Proposition~\ref{Prop_reg_codim_4}, Assumption~\ref{Aspt_working_ass} holds even for $\Delta = 4$.
This choice will be optimal, as one may see by considering a blow-down sequence of the Eguchi-Hanson metric.

\subsection{Preparatory results}
In this subsection we prove a few lemmas that will become useful later.

The first lemma addresses the limiting behavior of $P^*$-parabolic neighborhoods.

\begin{Lemma} \label{Lem_conv_P_star}
Consider sequences of points $(x'_i, t'_i), (x''_i, t''_i) \in M_i \times I_i$ such that $(x'_i, t'_i) \to x'_\infty \in \XX$ and $(x''_i, t''_i) \to x''_\infty \in \XX$ within $\CF$.
Let $A_1 > A_0 > 0$, $T^-_1 > T^-_0  >0$, $T^+_1 > T^+_0 \geq 0$.
Then the following is true, given that the corresponding $P^*$-parabolic neighborhoods exist:
\begin{enumerate}[label=(\alph*)]
\item \label{Lem_conv_P_star_a} If for infinitely many $i$ we have $(x''_i, t''_i) \in P^*(x'_i, t'_i; A_0, -T^-_0, T^+_0)$, then $x''_\infty \in P^*(x'_\infty; \lb A_1, \lb -T^{-}_1,\lb T^{+}_1)$.
\item \label{Lem_conv_P_star_b} If $x''_\infty \in P^*(x'_\infty; A_0, -T^{-}_0, T^{+}_0)$, then $(x''_i, t''_i) \in P^*(x'_i, t'_i; A_1, -T^{-}_1, T^{+}_1)$ for large $i$.
\end{enumerate}
\end{Lemma}

\begin{proof}
Suppose that Assertion~\ref{Lem_conv_P_star_a} or \ref{Lem_conv_P_star_b} is false.
Then we may pass to a subsequence such that the assumption in Assertion~\ref{Lem_conv_P_star_a} holds for all $i$ or the claim in Assertion~\ref{Lem_conv_P_star_b} is violated for all $i$.
Next, we may pass to a further subsequence such that the convergence (\ref{eq_setup_Mi_to_XX}) is time-wise at almost every time (see \cite[\SYNLemSubseqTimewiseAE]{Bamler_RF_compactness}).
Choose $T^-_2 \in (T^-_0, T^-_1)$ such that the convergence (\ref{eq_setup_Mi_to_XX}) is time-wise at time $t^* := \tf(x'_\infty) - T^-_2$.
In Assertion~\ref{Lem_conv_P_star_a} we have $(x''_i, t''_i) \in P^*(x'_i, t'_i; A_0, -T^-_2, T^+_0)$ for all $i$ and it suffices to show that $x''_\infty \in P^*(x'_\infty; \lb A_1, \lb -T^{-}_2,\lb T^{+}_1)$.
In Assertion~\ref{Lem_conv_P_star_b} we have $x''_\infty \in P^*(x'_\infty; A_0, -T^{-}_2, T^{+}_0)$ and it suffices to show that $(x''_i, t''_i) \in P^*(x'_i, t'_i; A_1, -T^{-}_2, T^{+}_1)$ for large $i$.

By \cite[\SYNThmConvImpliesStrict]{Bamler_RF_compactness} and the time-wise convergence at time $t^*$ we have strict convergence
\[ \nu_{x'_i,t'_i;t^*} \longrightarrow \nu_{x'_\infty;t^*}, \qquad
\nu_{x''_i,t''_i;t^*} \longrightarrow \nu_{x''_\infty;t^*}. \]
This implies that
\[ \lim_{i \to \infty} d^{g_{i, t^*}}_{W_1} \big( \nu_{x'_i,t'_i;t^*}, \nu_{x''_i,t''_i;t^*} \big) = d^{\XX_{t^*}}_{W_1} ( \nu_{x'_\infty;t^*}, \nu_{x''_\infty;t^*} ), \]
which allows us to reach the desired contradiction.
\end{proof}
\bigskip

Next, we pass upper heat kernel and volume bounds to the limit.

\begin{Lemma} \label{Lem_limit_HK_bound}
Let $A < \infty$, $-T \in I_\infty$ and consider a point $x'_\infty \in P^*(x_\infty; A, -T)$ and set $t' := \tf (x'_\infty)$.
Then the following is true:
\begin{enumerate}[label=(\alph*)]
\item \label{Lem_limit_HK_bound_a} For any $y_\infty \in \RR^*_{s}$, $s \in [-T, t')$, and $\eps > 0$ we have 
\[ K (x'_\infty; y_\infty) 
\leq C(Y_0, \tau_0, A, T, \eps) (t'-s)^{-n/2} \exp \bigg( - \frac{\big( d^{\XX_{s}}_{W_1} (\nu_{x'_\infty; s}, \delta_{y_\infty})\big)^2}{(8+\eps) (t'-s)} \bigg). \]
Moreover, if $(x'_i, t'_i) \to x'_\infty$ within $\CF$ and $T+s \geq \eps$, then
\[ K (x'_\infty; y_\infty) \leq C(\eps) \exp \bigg( -\limsup_{i \to \infty} \NN_{x'_i,t'_i} (t-s) - \frac{\big( d^{\XX_{s}}_{W_1} (\nu_{x'_\infty; s}, \delta_{y_\infty})\big)^2}{(8+\eps) (t'-s)} \bigg).\]
\item \label{Lem_limit_HK_bound_b} For any $r > 0$ with $t' -  r^2 > -T$ and $t \in I_\infty$ we have $|P^*(x'_\infty; r) \cap \RR^*_t|_{g_t} \leq C(Y_0, \tau_0, A, T) r^n$.
\end{enumerate}
\end{Lemma}

\begin{proof}
By \cite[\SYNLemSubseqTimewiseAE]{Bamler_RF_compactness} we may again pass to a subsequence such that the convergence (\ref{eq_setup_Mi_to_XX}) is time-wise at almost every time of $I_\infty$.
By \cite[\SYNThmExistencConvSeq]{Bamler_RF_compactness} and Lemma~\ref{Lem_conv_P_star} we can find a sequence $(x'_i, t'_i) \in P(x_i,0;A',-T')$ for any $A' > A$ and $T' > T$ such that $(x'_i, t'_i) \to x'_\infty$ within $\CF$.
By Proposition~\ref{Prop_NN_variation_bound} we have $\NN_{x'_i, t'_i} (\tau'_0 (T)) \geq - Y'_0(Y_0, A, T)$ for all $i$.

Passing \cite[\HKThmsHKboundGaussPstarVolBound]{Bamler_HK_entropy_estimates} to the limit and using \cite[\SYNThmConvImpliesStrict]{Bamler_RF_compactness} and Lemma~\ref{Lem_conv_P_star}, we obtain Assertion~\ref{Lem_limit_HK_bound_b} as well as Assertion~\ref{Lem_limit_HK_bound_a} under the additional assumption that the convergence (\ref{eq_setup_Mi_to_XX}) is time-wise at $s$.
To prove Assertion~\ref{Lem_limit_HK_bound_a} in its full form, we use the continuity of $K$.
Let $s'_j \searrow s$ be a sequence of times at which the convergence (\ref{eq_setup_Mi_to_XX}) is time-wise and choose $y'_j \in \RR^*_{s'_j}$ with $y'_j \to y_\infty$.
Then $\lim_{j \to \infty} K(x'_\infty; y'_j) = K(x'_\infty; y_\infty)$ and
\[ \liminf_{j \to \infty} d^{\XX_{s'_j}}_{W_1} (\nu_{x'_\infty; s'_j}, \delta_{y'_j}) 
\geq \liminf_{j \to \infty} d^{\XX_{s}}_{W_1} (\nu_{x'_\infty; s}, \nu_{y'_j;s})
= d^{\XX_{s}}_{W_1} (\nu_{x'_\infty; s}, \delta_{y_\infty}). \]
Combining this with Assertion~\ref{Lem_limit_HK_bound_a} at times $s'_j$ finishes the proof.
\end{proof}
\bigskip

\begin{Lemma} \label{Lem_nu_Pstar_bound}
Let $A < \infty$, $-T \in I_\infty$, consider a point $x'_\infty \in P^*(x_\infty; A, -T)$ and set $t' := \tf (x'_\infty)$.
For any $r > 0$ with $t' -  r^2 > -T$ and any $t \in I_\infty$ we have
\[ \nu_{x'_\infty;t} (P^*(x'_\infty; r) \cap \XX_t) \leq C(Y_0, \tau_0, A, T) r^n. \]
\end{Lemma}

\begin{proof}
As in the beginning of the proof of Lemma~\ref{Lem_limit_HK_bound}, we may pass to a subsequence such that the convergence (\ref{eq_setup_Mi_to_XX}) is time-wise at almost every time of $I_\infty$ and choose a sequence $(x'_i, t'_i) \in P(x_i,0;A',-T')$ for any $A' > A$ and $T' > T$ such that $(x'_i, t'_i) \to x'_\infty$ within $\CF$ and $\NN_{x'_i, t'_i} (\tau'_0 (T)) \geq - Y'_0(Y_0, A, T)$.

We first prove the lemma under the additional assumption that (\ref{eq_setup_Mi_to_XX}) is time-wise at time $t$.
Choose $r_1 > r$ such that $t' -  r_1^2 > -T$.
Recall that by \cite[\SYNThmConvImpliesStrict]{Bamler_RF_compactness} and the time-wise convergence implies that there is a metric space $(Z_t, d^{Z}_t)$ and isometric embeddings $\varphi^i_t : (M_i, d_{g_{i,t}}) \to (Z_t, d^Z_t)$, $\varphi^\infty_t : (\XX_t,d_t) \to (Z_t, d^Z_t)$ such that
\begin{equation} \label{eq_nu_xp_i_in_Z_W1}
 (\varphi^i_t)_* \nu_{x'_i, t'_i; t} \xrightarrow[i \to \infty]{\quad W_1 \quad} (\varphi^\infty_t)_* \nu_{x'_\infty; t}. 
\end{equation}
Let $S_\infty := P^*(x'_\infty; r) \cap \XX_t$ and $S_i := P^* (x'_i, t'_i; r_1) \cap M_i \times \{ t \}$.
By Lemma~\ref{Lem_conv_P_star} for any $y_\infty  \in S_\infty$ and any sequence $y_i \in M_i$ with $\varphi^i_t (y_i) \to \varphi^\infty_t (y_\infty)$ we have $(y_i, t) \in S_i$ for large $i$.
Thus, for any compact subset $S'_\infty \subset S_\infty$ there is an open neighborhood $\varphi^\infty_t (S'_\infty) \subset U \subset Z_t$ with the property that $U \cap \varphi^i_t (M_i) \subset \varphi^i_t(S_i)$ for large $i$.
It follows, using (\ref{eq_nu_xp_i_in_Z_W1}), Lemma~\ref{Lem_limit_HK_bound} and \cite[\HKPstarVolBound]{Bamler_HK_entropy_estimates}, that
\begin{multline*}
 \nu_{x'_\infty;t} ( S'_\infty ) 
\leq \limsup_{i \to \infty} \nu_{x'_i, t'_i; t} (S_i) 
\leq \limsup_{i \to \infty} \int_{S_i} K(x'_i,t'_i; \cdot, t) dg_{i,t} \\
\leq C(Y_0, \tau_0, A, T)\limsup_{i \to \infty}  |S_i|_{g_{i,t}}
\leq C(Y_0, \tau_0, A, T) r_1^n. 
\end{multline*}
Letting $S'_\infty \to S_\infty$ and $r_1 \to r$ implies the lemma if (\ref{eq_setup_Mi_to_XX}) is time-wise at time $t$.

Next, consider the case in which the convergence (\ref{eq_setup_Mi_to_XX}) is not time-wise at time $t$.
Choose $r_1 > r$ such that $t' -  r_1^2 > -T$ and fix a sequence of times $t_j \searrow t$ with the property that the convergence (\ref{eq_setup_Mi_to_XX}) is time-wise at each time $t_j$.
By our previous conclusion
\begin{equation} \label{eq_nu_xp_r_1}
 \nu_{x'_\infty;t_j} (P^*(z_\infty; r_1) \cap \XX_{t_j}) \leq C(Y_0, \tau_0, A, T) r^n_1. 
\end{equation}
Let $u : \XX_{\geq t} \to [0,1]$ be the heat flow with initial condition $u_t = \chi_{P^*(z_\infty; r) \cap \XX_{t}}$.
Then for all $j \geq 1$
\begin{equation} \label{eq_XXt_j_u_t}
  \int_{\XX_{t_j}} u_{t_j} \,  d\nu_{x'_\infty;t_j} = \nu_{x'_\infty;t} ( P^*(x'_\infty; r) \cap \XX_{t} ) . 
\end{equation}
We claim that
\begin{equation} \label{eq_sup_u_Pstar}
 \sup_{\XX_{t_j} \setminus P^*(x'_\infty; r_1)} u_{t_j} \xrightarrow[j \to \infty]{} 0. 
\end{equation}
If this is the case, then the lemma follows by combining (\ref{eq_nu_xp_r_1}), (\ref{eq_XXt_j_u_t}), (\ref{eq_sup_u_Pstar}).
Suppose by contradiction that (\ref{eq_sup_u_Pstar}) was false.
Then, after passing to a subsequence, we could find a sequence $y_j \in \XX_{t_j} \setminus P^*(x'_\infty; r_1)$ with
\begin{equation} \label{eq_upj_nu_geq_c}
 u(y_j) = \nu_{y_j;t}(P^*(x'_\infty; r) \cap \XX_{t})\geq c > 0. 
\end{equation}
Choose $H_n$-centers $z_j \in \XX_t$ of $y_j$.
By (\ref{eq_upj_nu_geq_c}) and Lemma~\ref{Lem_mass_ball_Var} we must have $d_t (z_j , P^*(x'_\infty; r) \cap \XX_{t}) \to 0$.
So if we choose $r_2 \in (r, r_1)$, then $z_j \in P^*(x'_\infty; r_2)$ for large $j$.
On the other hand, for any small $r'' > 0$ and $t'' := t -r^{\prime\prime 2}$ we have for large $j$
\[ d^{\XX_{t''}}_{W_1} ( \nu_{z_j;t''}, \nu_{y_j;t''} )
\leq d^{\XX_{t}}_{W_1} ( \delta_{z_j}, \nu_{y_j;t} )
\leq \sqrt{\Var ( \delta_{z_j}, \nu_{y_j;t} )}
\leq \sqrt{H_n (t_j - t)} < r'', \]
which implies that $y_j \in P^*(z_j; r'')$.
Combining this with the fact that $z_j \in P^*(x'_\infty; r_2)$ and choosing $r''$ sufficiently small implies that $y_j \in P^*(x'_\infty; r_1)$ for large $j$, in contradiction to our assumption.
This finishes the proof of the lemma.
\end{proof}
\bigskip

\subsection{The curvature scale in the limit}
In this subsection we consider the convergence of the curvature scales $\rrm^{M_i \times I_i}$ on the Ricci flows $(M_i, (g_{i ,t})_{t \in I_i})$.
We will show that, after passing to a subsequence, we have pointwise convergence $\rrm^{M_i \times I_i} \to \tdrrm$ for a certain curvature scale function $\tdrrm$ on $\XX$.
The following is our main result.

\begin{Lemma} \label{Lem_tdrrm}
Consider an arbitrary subsequence of the given sequence of Ricci flows $(M_i, \lb (g_{i,t})_{i \in I_i})$ in (\ref{eq_setup_Mi_to_XX}) (note that in doing so we may enlarge $\RR^* \subset \RR$ and shrink $\SS^* \supset \SS$).
Then we can pass to a further subsequence and find a continuous function $\tdrrm : \XX_{<0} \to [0, \infty]$ with the following properties:
\begin{enumerate}[label=(\alph*)]
\item \label{Lem_tdrrm_a} For any sequence $(x'_i, t'_i) \in M_i \times I_i$ such that $(x'_i, t'_i) \to x' \in \XX_{<0}$ within $\CF$ we have 
\[ \lim_{i \to \infty} \rrm (x'_i, t'_i) = \tdrrm (x'). \]
\item \label{Lem_tdrrm_b} $\tdrrm = 0$ on $\SS^*_{<0}$.
\item \label{Lem_tdrrm_c} $\tdrrm  > 0$ on $\RR^*$.
\item \label{Lem_tdrrm_d} $\tdrrm$ is locally Lipschitz on $\RR^*$
\item \label{Lem_tdrrm_e} $\tdrrm$ restricted to every time-slice $\XX_t$ is $1$-Lipschitz.
\item \label{Lem_tdrrm_f} $|\partial_{\tf} \tdrrm^2| \leq C$ in the barrier sense, where $C < \infty$ is a dimensional constant.
\item \label{Lem_tdrrm_g} For any $x' \in \RR^*_{t'}$ there is a time $T^* \in (0, \tdrrm^2(x')]$ with the following property.
The open, two-sided parabolic neighborhood 
\begin{multline*}
\qquad\qquad P^\circ := P^\circ (x';\tdrrm(x'), - \tdrrm^2(x'), a') \\ := P^* (x';\tdrrm(x'), - \tdrrm^2(x'), a' ) \cap \XX_{(\tf(x') - \tdrrm^2(x'), \tf(x') + T^*)} 
\end{multline*}
satisfies $P^\circ \subset \RR^*$ and is unscathed in the usual sense (i.e. $B(x', r') \subset \RR^*$ is relatively compact for all $r' \in (0, \tdrrm(x'))$ and any point in $B(x', \tdrrm(x'))$ survives until any time in the time-interval $(\tf(x') - \tdrrm^2(x'), \tf(x') + T^*)$).

Moreover, if $T^* < \tdrrm^2(x')$ and $\tf(x') + T^* < 0$, then no point in $P^\circ$ survives past time $\tf(x') + T^*$ within $\RR$ and for any $x'' \in \XX_{t''}$ with $t'' > t' + T^*$ and $r' < \tdrrm (x')$ we have
\[ \lim_{t \nearrow t' + T^*} \sup_{(B(x', r'))(t)} K(x''; \cdot) = 0. \]
\item \label{Lem_tdrrm_h} For any $A < \infty$, $-T \in I_\infty$, $r > 0$ and $t \in [-T, 0)$ the set 
\[ S_{A, T, t, r} := \{ \tdrrm > r \} \cap P^* (x_\infty; A, -T) \cap \RR_t \]
is relatively compact in $\RR_t$.
\end{enumerate}
\end{Lemma}

\begin{proof}
We first establish some basic properties concerning $\rrm$ and $\RR^*$.

\begin{Claim} \label{Cl_rrm_RRstar_conv}
Consider an arbitrary subsequence in (\ref{eq_setup_Mi_to_XX}).
Suppose that such that $(x'_i, t'_i) \to x' \in \XX_{< 0}$ within $\CF$, where $(x'_i, t'_i) \in M_i \times I_i$.
Then the following is true:
\begin{enumerate}[label=(\alph*)]
\item \label{Cl_rrm_RRstar_conv_a} If $\liminf_{i \to \infty} \rrm (x'_i, t'_i) > 0$, then $x' \in \RR^* \subset \RR$.
\item \label{Cl_rrm_RRstar_conv_b} If $x' \in \RR^*$ and $(x''_i, t''_i) \to x' \in \XX$ within $\CF$ for some other sequence $(x''_i, t''_i) \in M_i \times I_i$, then 
\begin{equation} \label{eq_liminfsup_xp_xpp}
\qquad \liminf_{i \to \infty} \rrm (x'_i, t'_i) = \liminf_{i \to \infty} \rrm (x''_i, t''_i), \qquad \limsup_{i \to \infty} \rrm (x'_i, t'_i) = \limsup_{i \to \infty} \rrm (x''_i, t''_i). 
\end{equation}
\end{enumerate}
\end{Claim}

\begin{proof}
To see Assertion~\ref{Cl_rrm_RRstar_conv_a}, note that, using Lemma~\ref{Lem_conv_P_star}, we can find some $A < \infty$ and $-T \in  I_\infty$ such that $(x'_i, t'_i) \in P^*(x_i, 0; A, - T)$ for large $i$.
Therefore, by Proposition~\ref{Prop_NN_variation_bound} and \cite[\HKThmNLC]{Bamler_HK_entropy_estimates} we have $\limsup_{i \to \infty} |B(x'_i, t'_i, r')| > 0$ for some $0 < r' < \liminf_{i \to \infty} \rrm (x'_i, t'_i)$.
So by \cite[\SYNDefSmoothConv]{Bamler_RF_compactness} we have $x' \in \RR^*$.

To see Assertion~\ref{Cl_rrm_RRstar_conv_b}, observe that by the discussion in Subsection~\ref{subsec_setup_limit} we have $(x''_i, t''_i) \in V_i$ for large $i$ and $\psi_i^{-1} (x''_i, t''_i) \to x'$.
So for any $r' > 0$ we have $(x''_i, t''_i) \in P(x'_i, t'_i; r')$ for large $i$.
This shows (\ref{eq_liminfsup_xp_xpp}).
\end{proof}

\begin{Claim} \label{Cl_rrm_conv_pt}
\begin{enumerate}[label=(\alph*)]
\item \label{Cl_rrm_conv_pt_a} Assertion~\ref{Lem_tdrrm_a} of the lemma holds for all $x' \in \SS$ if we set $\tdrrm (x') := 0$, even without passing to a further subsequence.
\item \label{Cl_rrm_conv_pt_b} Let $x' \in \RR$.
Then there is some $\tdrrm(x') \in [0, \infty]$ such that, after passing to a subsequence of the given subsequence of Ricci flows, Assertion~\ref{Lem_tdrrm_a} of the lemma holds for $x'$.
Moreover, if $\tdrrm(x') > 0$, then $x' \in \RR^*$.
\end{enumerate}
\end{Claim}

\begin{proof}
For Assertion~\ref{Cl_rrm_conv_pt_a} assume by contradiction that $\limsup_{i \to \infty} \rrm (x'_i, t'_i) > 0$ for some sequence $(x'_i, t'_i) \in M_i \times I_i$ with $(x'_i, t'_i) \to x'$ within $\CF$.
Then, after passing to a subsequence, we have $\liminf_{i \to \infty} \rrm (x'_i, t'_i) > 0$, which implies via Claim~\ref{Cl_rrm_RRstar_conv}\ref{Cl_rrm_RRstar_conv_a} that $x' \in \RR$ in contradiction to our assumption.

To see Assertion~\ref{Cl_rrm_conv_pt_b}, pick a sequence $(x'_i, t'_i) \in M_i \times I_i$ such that $(x'_i, t'_i) \to x'$ within $\CF$.
This is always possible by \cite[\SYNThmExistencConvSeq]{Bamler_RF_compactness}.
Pass to a subsequence such that $\tdrrm(x') := \lim_{i \to \infty} \rrm (x'_i, t'_i)$ exists.
We may also assume that the sequence $(x'_i, t'_i)$ was chosen such that if $\tdrrm(x') =0$, then no other such sequence $(x''_i, t''_i)$ has $\limsup_{i \to \infty} \rrm (x''_i, t''_i) > 0$, because otherwise we could replace $(x'_i, t'_i)$ by $(x''_i,t''_i)$ and pass to an appropriate subsequence.
Let now $(x''_i, t''_i) \in M_i \times I_i$ be another sequence such that $(x''_i, t''_i) \to x'$ within $\CF$.
If $\tdrrm(x') =0$, then by the choice of $(x'_i, t'_i)$ we have $\lim_{i \to \infty} \rrm (x''_i, t''_i) = 0$.
If $\tdrrm (x') > 0$, then $x' \in \RR^*$ by Claim~\ref{Cl_rrm_RRstar_conv}\ref{Cl_rrm_RRstar_conv_a} and therefore $\lim_{i \to \infty} \rrm (x''_i, t''_i) =  \lim_{i \to \infty} \rrm (x'_i, t'_i)$ by Claim~\ref{Cl_rrm_RRstar_conv}\ref{Cl_rrm_RRstar_conv_b}.
\end{proof}

Fix a countable, dense subset $Q \subset \RR$.
By repeating Claim~\ref{Cl_rrm_conv_pt} for all $x' \in Q$ and passing to a diagonal subsequence, we obtain:

\begin{Claim}\label{Cl_tdrrm_Q}
After passing to a further subsequence of the given subsequence of Ricci flows, there is a function $\tdrrm : Q \cup \SS_{<0} \to [0, \infty]$ such that the following is true:
\begin{enumerate}[label=(\alph*)]
\item Assertions~\ref{Lem_tdrrm_a}, \ref{Lem_tdrrm_g} of the lemma hold for all $x' \in Q \cup \SS_{<0}$.
\item $\tdrrm = 0$ on $\SS_{<0}$.
\item $\{ \tdrrm > 0 \} \subset \RR^*$ (note again that $\RR^*$ may have been enlarged after passing to the subsequence).
\end{enumerate}
\end{Claim}

\begin{proof}
Assertion~\ref{Lem_tdrrm_g} of the lemma is a consequence of \cite[\SYNConvParabNbhd]{Bamler_RF_compactness}.
The rest follows from Claim~\ref{Cl_rrm_conv_pt}.
\end{proof}

\begin{Claim}
The function $\tdrrm$ from Claim~\ref{Cl_tdrrm_Q} can be extended to a continuous function $\tdrrm : \XX \to [0, \infty]$ such that Assertions~\ref{Lem_tdrrm_a}--\ref{Lem_tdrrm_g} of the lemma hold.
\end{Claim}

\begin{proof}
By the definition of $\rrm$ we have $|\nabla \rrm| \leq 1$ and $|\partial_{\tf} \rrm^2| \leq C$ on $M_i \times I_i$ for some dimensional constant $C < \infty$.
Moroever, by the discussion in Subsection~\ref{subsec_setup_limit} we have $\rrm \circ \psi_i \to \tdrrm$ on $Q \cap \RR^*$.
It follows that $\tdrrm |_{Q \cap \RR^*}$ can be extended to a continuous function $\tdrrm^{\,\prime}$ on $\RR^*$ such that $\rrm \circ \psi_i \to \tdrrm^{\,\prime}$ locally uniformly on $\RR^*$ and such that Assertions~\ref{Lem_tdrrm_d}, \ref{Lem_tdrrm_e}, \ref{Lem_tdrrm_f} of the lemma hold.
Set now $\tdrrm := \tdrrm^{\,\prime}$ on $\RR^*$ and $\tdrrm := 0$ on $\SS^*_{<0}$.
Then Assertion~\ref{Lem_tdrrm_b} of the lemma holds trivially and, again by the discussion in Subsection~\ref{subsec_setup_limit}, Assertion~\ref{Lem_tdrrm_a} of the lemma holds for all $x' \in \RR^*$.
This implies Assertion~\ref{Lem_tdrrm_c} of the lemma and Assertion~\ref{Lem_tdrrm_g}, again via \cite[\SYNConvParabNbhd]{Bamler_RF_compactness}.

To see that the full statement of Assertion~\ref{Lem_tdrrm_a} holds, it remains to show that for any sequence $(x'_i, t'_i)  \to x' \in \SS^* \setminus \SS \subset \RR$ we have $\lim_{i \to \infty} \rrm (x'_i, t'_i) = 0$.
Suppose not.
Then by Claim~\ref{Cl_rrm_conv_pt}\ref{Cl_rrm_conv_pt_b} we could pass to a subsequence, which would enlarge $\RR^*$ such that $x' \in \RR^*$ and $\td r := \lim_{i \to \infty} \rrm (x'_i, t'_i) > 0$.
After repeating the argument from the previous paragraph, we would find a sequence $x''_j \in Q$ with $x''_j \to x'$ such that $\tdrrm (x''_j) \to \td r$.
Using Assertion~\ref{Lem_tdrrm_g} of the lemma this would, however, imply that even before passing to the subsequence, we had $x' \in \RR^*$.
This contradicts the fact that, before passing to the subsequence, $x' \in \SS^*$.
\end{proof}

Lastly, we establish Assertion~\ref{Lem_tdrrm_h}.
By Proposition~\ref{Prop_NN_variation_bound} and \cite[\HKThmNLC]{Bamler_HK_entropy_estimates} for any $x' \in S_{A, T, t, r}$ the distance ball $B (x', r)$ is relatively compact in $\RR^*_t$ and we have $|B(x', r)|_t \geq c(Y_0, \tau_0, A, T , r)$ and $B(x', r) \subset P^* (x_\infty; A +  r, -T)$.
Moreover, by the same argument as for Lemma~\ref{Lem_limit_HK_bound}\ref{Lem_limit_HK_bound_b} we have
\[ |P^* (x_\infty; A + r, -T) \cap \RR^*_t |_t < \infty. \]
So the relative compactness follows from a basic covering argument.
\end{proof}
\bigskip

\subsection{Sizes of quantitative strata in the limit}
Our goal in this subsection will be to pass the statement of Proposition~\ref{Prop_Quantitative_Strat_preliminary} to the limit $\XX$.
We also define cutoff functions $\eta_r \in C^\infty (\RR)$, which will be used throughout the remainder of this section.

The following lemma states that we can pass coverings by $P^*$-parabolic balls to the limit.

\begin{Lemma} \label{Lem_pass_covering_limit}
Let $A, A',  N < \infty$, $r > 0$ be constants and $-T, -T' \in I_\infty$.
Assume that $T' > T$, $A'  > A$ and $- T - r^2 \in I_\infty$.
Suppose that there are subsets $S_i \subset M_i \times I_i$ and $S_\infty \subset \XX$ with the following properties:
\begin{enumerate}[label=(\roman*)]
\item \label{Lem_pass_covering_limit_i} For every $x'_\infty \in S_\infty$ there is a sequence $(x'_i, t'_i) \in M_i \times I_i$ with $(x'_i, t'_i) \in S_i$ for large $i$ such that $(x'_i, t'_i) \to x'_\infty$.
\item \label{Lem_pass_covering_limit_ii} For every $i$ there are points $\{ (z_{j,i}, s_{j, i}) \}_{j = 1}^{N_i} \subset M_i \times I_i$ such that
\[ S_i \cap P^* (x_i, 0; A', -T') \subset \bigcup_{j=1}^{N_i} P^* (z_{j, i}, s_{j, i}; \tfrac18 r), \qquad N_i \leq N. \]
\end{enumerate}
Then there are points $\{ z_{j,\infty} \}_{j = 1}^{N_\infty} \subset S_\infty$ such that
\begin{equation} \label{eq_S_infty_covering}
 S_\infty \cap P^* (x_\infty; A, -T) \subset \bigcup_{j=1}^{N_\infty} P^* (z_{j,\infty};  r), \qquad N_\infty \leq N. 
\end{equation}
\end{Lemma}

\begin{proof}
Choose $\{ z_{j,\infty} \}_{j = 1}^{N_\infty} \subset S_\infty \cap P^* (x_\infty; A, -T)$ maximal with the property that the $P^*$-parabolic balls $P^* (z_{j, \infty};  \lb \frac12 r)$ are pairwise disjoint.
If no such maximum exists, then choose points with $N_\infty = N +1$.
Our goal is to show that $N_\infty \leq N$, because then (\ref{eq_S_infty_covering}) follows from the maximality of $N_\infty$ using Proposition~\ref{Prop_basic_parab_nbhd}.
Suppose by contradiction that $N_\infty > N$.

Fix a $j' \in \{ 1, \ldots, N_\infty \}$ for a moment.
By Assumption~\ref{Lem_pass_covering_limit_i} there is a sequence $(x'_{j', i}, t'_{j', i}) \in S_i$ with $(x'_{j',i}, t'_{j',i}) \to z_{j', \infty}$.
For large $i$ we have $(x'_{j',i}, t'_{j',i}) \in P^* (x_i, 0; A', -T')$.
So by Assumption~\ref{Lem_pass_covering_limit_ii}, we have $(x'_{j',i}, t'_{j',i}) \in P^* (z_{ j_{j',i}, i}, s_{j_{j',i}, i}; \frac18 r)$ for large $i$, where $j_{j',i} \in \{ 1, \ldots, N_i \}$ is some sequence of indices.
Since $N_\infty > N$, we may find two indices $j', j'' \in \{ 1, \ldots, N_\infty \}$, $j' \neq j''$ such that, after passing to a subsequence, $j_{j',i} = j_{j'',i}$ for all $i$.
By Proposition~\ref{Prop_basic_parab_nbhd}, we obtain that $(x'_{j',i}, t'_{j',i}) \in P^* (x'_{j'',i}, t'_{j'',i}; \frac14 r)$, which implies that $z_{j',\infty} \in P^*(z_{j'',\infty}; \frac12 r)$, in contradiction to our choice of the points $z_{j,\infty}$.
This finishes the proof.
\end{proof}

The next result expresses Proposition~\ref{Prop_Quantitative_Strat_preliminary} in the limit.
Note that due to Proposition~\ref{Prop_weak_splitting_map_to_splitting_map} we may use the \emph{strong} splitting property instead of the weak splitting property.
Also, a simple covering argument using \cite[\HKBasicCovering]{Bamler_HK_entropy_estimates} allows us to generalize the covering statement to Proposition~\ref{Prop_Quantitative_Strat_preliminary} to a $P^*$-parabolic neighborhood of more general size.

\begin{Proposition} \label{Prop_prelim_quanti_strat_XX}
For any $0 < \sigma < \eps$ there are subsets
\begin{equation} \label{eq_tdSS_subset_tdSS}
 \td\SS^{\eps, 0}_{\sigma, \eps} \subset \td\SS^{\eps, 1}_{\sigma, \eps} \subset \ldots \subset \td\SS^{\eps, n+1}_{\sigma, \eps} \subset \XX 
\end{equation}
with the following properties for every $k = 0, \ldots, n-1$:
\begin{enumerate}[label=(\alph*)]
\item \label{Prop_prelim_quanti_strat_XX_a} For every $x'_\infty \in \XX$ we have $x'_\infty \in \XX \setminus \td\SS^{\eps, k}_{\sigma, \eps}$ if and only if there is a sequence of points $(x'_i, t'_i) \in M_i \times I_i$ with $(x'_i, t'_i) \to x'_\infty$ within $\CF$ and scales $r'_i \in (\sigma, \eps)$ such that for infinitely many $i$ one of the following is true:
\begin{enumerate}[label=(a\arabic*)]
\item \label{Prop_prelim_quanti_strat_XX_a1} $(x'_i, t'_i)$ is $(\eps, r'_i)$-selfsimilar and strongly $(k+1, \eps, r'_i)$-split.
\item \label{Prop_prelim_quanti_strat_XX_a2} $(x'_i, t'_i)$ is $(\eps, r'_i)$-selfsimilar, $(\eps, r'_i)$-static and strongly $(k-1, \eps, r'_i)$-split.
\end{enumerate}
\item \label{Prop_prelim_quanti_strat_XX_b} For every $A < \infty$ and $- T \in I_\infty$ with $-T - \eps^2 \in I$  there are points $\{ x'_j \}_{j=1}^N \subset \td\SS^{\eps, k}_{\sigma, \eps} \cap P^* (x_\infty; \lb A, \lb -T)$ such that
\[ \td\SS^{\eps, k}_{\sigma, \eps} \cap P^* (x_\infty; \lb A, \lb -T) \subset \bigcup_{j=1}^{N} P^*(x'_{j};  \sigma), \qquad N \leq C(Y_0, \tau_0, A, T, \eps) \sigma^{-k-\eps}. \]
\end{enumerate}
\end{Proposition}

\begin{proof}
We may define each subset $\td\SS^{\eps, k}_{\sigma, \eps} \subset \XX$ according to Assertion~\ref{Prop_prelim_quanti_strat_XX_a}.
The identity (\ref{eq_tdSS_subset_tdSS}) is clear.
Assertion~\ref{Prop_prelim_quanti_strat_XX_b} follows by combining Propositions~\ref{Prop_Quantitative_Strat_preliminary}, \ref{Prop_weak_splitting_map_to_splitting_map}, \ref{Prop_NN_variation_bound}, \cite[\HKBasicCovering]{Bamler_HK_entropy_estimates} and Lemma~\ref{Lem_pass_covering_limit}.
\end{proof}
\bigskip

As a consequence of Proposition~\ref{Prop_prelim_quanti_strat_XX} we obtain:

\begin{Proposition} \label{Prop_cover_sublevel_tdrrm}
For any subsequence of the given sequence of Ricci flows that allows the definition of $\tdrrm$ according to Lemma~\ref{Lem_tdrrm} the following is true.
For every $A < \infty$ and $- T \in I_\infty$ with $-T - \eps^2 \in I$ and $r \leq \ov{r} (Y_0, \tau_0, A, T, \eps)$ there are points $\{ x'_j \}_{j=1}^N \subset \{ \tdrrm \leq r \} \cap P^* (x_\infty; \lb A, \lb -T)$ such that
\[ \{ \tdrrm \leq r \} \cap P^* (x_\infty; \lb A, \lb -T) \subset \bigcup_{j=1}^{N} P^*(x'_{j};  r), \qquad N \leq C(Y_0, \tau_0, A, T, \eps) r^{-n-2+\Delta-\eps}. \]
\end{Proposition}

\begin{proof}
Apply Proposition~\ref{Prop_prelim_quanti_strat_XX} to the given subsequence and consider the subsets $\td\SS^{\eps, n+2 - \Delta}_{\sigma, \eps} \subset \XX$ (note that these subsets may become larger after passing to a subsequence).
So for any $x'_\infty \in \XX \setminus \td\SS^{\eps, n+2-\Delta}_{\sigma, \eps}$ there is a sequence $(x'_i, t'_i) \in M_i \times I_i$ such that $(x'_i, t'_i) \to x'_\infty$ within $\CF$ and such that for infinitely many $i$ there is some $r'_i \in (\sigma, \eps)$ such that Property~\ref{Prop_prelim_quanti_strat_XX_a1} or \ref{Prop_prelim_quanti_strat_XX_a2} of Proposition~\ref{Prop_prelim_quanti_strat_XX} holds.
So by Assumption~\ref{Aspt_working_ass}, Proposition~\ref{Prop_NN_variation_bound} and Lemma~\ref{Lem_tdrrm} we have for $\eps \leq \ov\eps(Y, \tau_0, A, T)$
\[ \tdrrm (x'_\infty)
=  \lim_{i \to \infty} \rrm(x'_i, t'_i) 
\geq \limsup_{i \to \infty} c(Y, \tau_0, A, T) r'_i
\geq c(Y, \tau_0, A, T)  \sigma > 0. \]
The proposition now follows by choosing $\sigma := c^{-1} r$ and application of the covering result \cite[\HKBasicCovering]{Bamler_HK_entropy_estimates}.
\end{proof}
\bigskip

The next lemma is a consequence of Proposition~\ref{Prop_cover_sublevel_tdrrm}.
It will provide us with a bound on the volume of sublevel sets $\{ \tdrrm \leq r \}$ and yield cutoff functions $\eta_r \in C^\infty (\RR)$ that are $\equiv 1$ away from a set of small volume.

\begin{Lemma} \label{Lem_eta_r_on_RR}
For any subsequence of the given sequence of Ricci flows that allows the definition of $\tdrrm$ according to Lemma~\ref{Lem_tdrrm} the following holds:
\begin{enumerate}[label=(\alph*)]
\item \label{Lem_eta_r_on_RR_a} For any $A < \infty$, $-T \in I_\infty$, $\eps > 0$ there is a constant $C(Y_0, \tau_0, A, T, \eps) < \infty$ such that for $r \leq \ov{r} (Y_0, \tau_0, A, T, \eps)$
\begin{align*}  \int_{I_\infty} \int_{\{ 0 < \tdrrm \leq r \} \cap P^* (x_\infty; A, -T) \cap \RR_t} dg_t dt &\leq C r^{\Delta - \eps}, \\
  \sup_{t \in I_\infty} \int_{\{ 0 < \tdrrm \leq r \} \cap P^* (x_\infty; A, -T) \cap \RR_t} dg_t &\leq C r^{\Delta - 2 - \eps}. 
\end{align*}
\item \label{Lem_eta_r_on_RR_b} There is a family of smooth functions $(\eta_r \in C^\infty (\RR))_{r > 0}$ that take values in $[0,1]$ such that:
\begin{enumerate}[label=(b\arabic*)]
\item \label{Lem_eta_r_on_RR_b1} $\tdrrm \geq r$ on $\{ \eta_r > 0 \}$.
\item \label{Lem_eta_r_on_RR_b2} $\eta_r = 1$ on $\{ \tdrrm \geq 2r \}$.
\item \label{Lem_eta_r_on_RR_b3} $|\nabla \eta_r | \leq C_0 r^{-1}$ and $|\partial_{\tf} \eta_r | \leq C_0 r^{-2}$ for some dimensional constant $C_0 < \infty$.
\item \label{Lem_eta_r_on_RR_b4} For any $x' \in \XX$, $L < \infty$ and $r > 0$, the set $\{ \eta_r > 0 \} \cap B (x', L)$ is relatively compact in $\RR^*_t$.
\item \label{Lem_eta_r_on_RR_b5} For any $A < \infty$, $-T \in I_\infty$ and $t \in [-T, 0)$ the set $\{ \eta_r > 0 \} \cap P^* (x_\infty; A, -T) \cap \RR_t$ is relatively compact in $\RR_t$.
\item \label{Lem_eta_r_on_RR_b6} For any maximal worldline $\gamma : I' \to \RR$ (i.e. every trajectory $\gamma$ of $\partial_{\tf}$ with $\tf (\gamma(t)) = t$) the following is true.
Either $t_{\min} := \inf I' = - \infty$ or $\eta_r ( \gamma(t)) = 0$ for $t$ close to $t_{\min}$.
\end{enumerate}
\end{enumerate}
\end{Lemma}

\begin{proof}
Assertion~\ref{Lem_eta_r_on_RR_a} follows by combining Lemmas~\ref{Lem_tdrrm}, \ref{Lem_limit_HK_bound}\ref{Lem_limit_HK_bound_b}, Propositions~\ref{Prop_NN_variation_bound}, \ref{Prop_cover_sublevel_tdrrm} and Assumption~\ref{Aspt_working_ass}.

For Assertion~\ref{Lem_eta_r_on_RR_b} fix a cutoff function $\ov\eta : (0, \infty] \to [0,1]$ with $\ov\eta \equiv 0$ on $[0,1.1]$, $\ov\eta \equiv 1$ on $[1.9, \infty]$ and $|\ov\eta'| \leq 10$.
If we let $\eta_r$ be a smoothing of $\ov\eta ( \tdrrm / r)$, then we can ensure that Assertions~\ref{Lem_eta_r_on_RR_b1}, \ref{Lem_eta_r_on_RR_b2} hold and that Assertion~\ref{Lem_eta_r_on_RR_b3} follows using Lemma~\ref{Lem_tdrrm}.
Assertions~\ref{Lem_eta_r_on_RR_b5}, \ref{Lem_eta_r_on_RR_b6} follow from Lemma~\ref{Lem_tdrrm} and Assertion~\ref{Lem_eta_r_on_RR_b4} follows from Assertion~\ref{Lem_eta_r_on_RR_b5} since $B(x', L) \subset P^* (x_\infty; L, -T)$ if $-T \leq \tf (x')$.
\end{proof}

\bigskip

\subsection{The size of the singular set}
The following theorem states that under Assumption~\ref{Aspt_working_ass} we have a bound of the form $\dim_{\mathcal{M}^*} \SS \leq \dim_{\mathcal{M}^*} \SS^* \leq n+2-\Delta$.
It also relates the distance metric $d_t$ with the length metric $d_{g_t}$ on $\RR^*_t$ and characterizes heat flows on $\XX$.

\begin{Theorem} \label{Thm_SS_dimension_bound_limit}
Suppose we are in the setting described in Subsection~\ref{subsec_setup_limit} and consider a subsequence of the given sequence of Ricci flows that allows the definition of $\tdrrm$ according to Lemma~\ref{Lem_tdrrm}.
If $\Delta \geq 4$, then the following is true:
\begin{enumerate}[label=(\alph*)]
\item \label{Thm_SS_dimension_bound_limit_a} 
For every $A < \infty$, $-T \in I_\infty$ and $\eps > 0$ the following holds.
For any $r > 0$ with $-T-r^2 \in I_\infty$ we can find points $x'_i \in \SS \cap P^* (x_\infty; A, -T)$, $i = 1, \ldots, N$ with
\[ \SS^* \cap P^*  (x_\infty; A, -T) \subset \bigcup_{i=1}^N P^* (x'_i; r), \qquad N \leq C(Y_0, \tau_0, A, T, \eps) r^{-n-2+\Delta - \eps}. \]
\item \label{Thm_SS_dimension_bound_limit_b} $\SS^*_t =  \SS^* \cap \XX_t$ is a set of measure zero (in the sense of \cite{Bamler_RF_compactness}) for almost all $t \in I_\infty$ (if $\Delta = 1,2$) and for all $t \in I_\infty \setminus \{ 0 \}$ (if $\Delta \geq 3$).
\item \label{Thm_SS_dimension_bound_limit_c} If $\Delta \geq 4$, then for all $t \in I_\infty$ the length metric $d_{g_t}$ of $g_t$ is equal to the restriction of $d_t$ to $\RR_t$.
\item \label{Thm_SS_dimension_bound_limit_d} If $\Delta \geq 3$, then every uniformly bounded $u \in C^0 ( \RR_{ [t', 0)} ) \cap C^\infty ( \RR_{(t', 0)})$, $t' \in I_\infty$, that satisfies the heat equation $\square u = 0$ on $ \RR_{ (t', 0)}$ is a restriction of the heat flow with initial condition $u  \chi_{\RR_{t'}} : \XX_{t'} \to \IR$ (defined using the metric flow structure on $\XX$) restricted to $\RR_{ [t', 0)}$.
\end{enumerate}
\end{Theorem}

Assertions~\ref{Thm_SS_dimension_bound_limit_c}, \ref{Thm_SS_dimension_bound_limit_d} of Theorem~\ref{Thm_SS_dimension_bound_limit} imply that if $\Delta \geq 4$, then the metric flow $\XX$ is uniquely determined by the Ricci flow spacetime structure on $\RR$:

\begin{Theorem} \label{Thm_XX_uniquely_det_RR}
Suppose $\Delta=4$ and consider an $H'$-concentrated metric flow $\XX'$ of full support and complete time-slices over $I_\infty \setminus \{0 \}$ for some $H' < \infty$.
Denote by $\RR' \subset \XX'$ its regular part.
If there is an isometry $\RR \to \RR'$ as Ricci flow spacetimes, then this isometry can be extended uniquely to an isometry $\XX_{<0} \to \XX'_{<0}$ of metric flows.
\end{Theorem}

\begin{proof}[Proof of Theorem~\ref{Thm_SS_dimension_bound_limit}.] 
Assertion~\ref{Thm_SS_dimension_bound_limit_a} is a direct consequence of Proposition~\ref{Prop_cover_sublevel_tdrrm} since $\SS^* = \{ \tdrrm = 0 \}$.
Assertion~\ref{Thm_SS_dimension_bound_limit_b} follows from Assertion~\ref{Thm_SS_dimension_bound_limit_a} combined with Lemma~\ref{Lem_nu_Pstar_bound}.

For Assertion~\ref{Thm_SS_dimension_bound_limit_c} suppose that $\Delta \geq 4$ and fix $t \in I_\infty$, $t < 0$.
It remains to show that $d_{g_t} \leq d_{t}$ on $\RR_{t}$, because the opposite inequality is clear.
By Assertion~\ref{Thm_SS_dimension_bound_limit_b}, it suffices to show that $d_{g_t} \leq d_{t}$ on $\RR^*_{t}$.
As in \cite{Chen-Wang-II}, we will show that every function $u \in C^0( \RR^*_{t})$ that is $L$-Lipschitz with respect to $d_{g_t}$ is also $L$-Lipschitz with respect to $d_{t}$; taking $u := d_{g_t} (x', \cdot)$ for any $x' \in \RR^*_t$, will then imply the bound $d_{g_t} \leq d_{t}$.
So suppose that $u \in C^0( \RR^*_{t})$ is $L$-Lipschitz.
We claim that $u$ is a pointwise limit of a sequence of smooth and $L_i$-Lipschitz functions $u_i \in C^\infty ( \RR^*_{t})$ whose support is bounded with respect to $d_{g_t}$ and $L_i \to L$.
To see this, multiply $u$ with functions of the form $h_i ( d_{g_t} (x', \cdot))$, for some fixed $x' \in \RR^*_t$, where $h_i$ are slowly decreasing functions with $h_i \to 1$ as $i \to \infty$, and apply a smoothing procedure.
So we may assume in the following that $u \in C^\infty ( \RR^*_{t})$ and that the support of $u$ within $\RR^*_{t}$ is bounded with respect to $d_{g_t}$.
This implies that $u$ is uniformly bounded and that $u \equiv 0$ on $\RR^*_{t} \setminus P^* (x_\infty; A, -T)$ for some $A <\infty$ and $-T \in I_\infty$.
After replacing $u$ with $L^{-1} u$, we may moreover assume that $u$ is $1$-Lipschitz, or equivalently, $|\nabla u| \leq 1$.
For the remainder of the proof of Assertion~\ref{Thm_SS_dimension_bound_limit_c}, $C$ will denote a generic constant that may depend on $t, u,Y_0, \tau_0, A, T$, but not on $i$ or a parameter $r$, which we will introduce later.
Any further dependence will be indicated in parentheses.

Fix some small $r > 0$ and let $\eta_r \in C^\infty (\RR)$ be the function from Lemma~\ref{Lem_eta_r_on_RR}.
Then
\[ u_r := u \eta_r \]
is compactly supported in $\RR_{t}$ and
\begin{equation} \label{eq_nab_ur_m_1}
 \int_{\RR^*_{t}} ( |\nabla u_r| - 1)_+ dg_{t} 
\leq C \int_{P^* (x_\infty; A, -T) \cap \RR^*_t} |\nabla \eta_r| dg_t 
\leq C r^{-1} \int_{\{ 0 < \tdrrm \leq 2 r \}  \cap P^* (x_\infty; A, -T)}  dg_t 
\leq C r^{.5}. 
\end{equation}

\begin{Claim}
For almost every $t^{*} \in (t, 0)$ and $x', x'' \in \XX_{t^*}$ there is a constant $C(t^{*}, x', x'') < \infty$ such that
\begin{equation} \label{eq_ur_dnuxpxpp}
 \bigg| \int_{\RR^*_{t}} u_r \, d\nu_{x'; t} - \int_{\RR^*_{t}} u_r \, d\nu_{x''; t} \bigg| \leq (1+ C(t^{*}, x', x'') r^{.5}) d_{t^{*}} (x', x''). 
\end{equation}
\end{Claim}

\begin{proof}
By passing to a subsequence, we may assume that the convergence in (\ref{eq_setup_Mi_to_XX}) is time-wise at almost every time; suppose that $t^*$ is chosen such that it is time-wise at time $t^*$.
This allows us to fix sequences of the form $(x'_i, t^*), (x''_i, t^*) \in M_i \times I_i$ such that $(x'_i, t^*) \to x'$ and $(x''_i, t^*) \to x''$ within $\CF$ and consider the diffeomorphisms $\psi_i : U_i \to V_i$ from Subsection~\ref{subsec_setup_limit}, describing the smooth convergence on $\RR^*$.
For large $i$ (depending on $r$) we have $\supp u_r \subset U_i$ and so we can define $u_{i, r} \in C^{\infty} (M_i \times \{ t \})$ by
\[ u_{i, r} := \begin{cases}
u_r \circ \psi^{-1}_i & \text{on $M_i \times \{t \}  \cap V_i$} \\
0 & \text{on $M_i \times \{t \} \setminus V_i$}
\end{cases}. \]
By (\ref{eq_nab_ur_m_1}), we have for large $i$
\[ \int_{M_i} ( |\nabla u_{i,r}| - 1)_+ dg_{i,t} \leq C r^{.5}. \]
Consider now the solutions $u^*_{i,r} \in C^\infty (M_i \times [t, 0])$ to the heat equation $\square u^*_{i,r} = 0$ with initial conditions $u^*_{i,r} (\cdot, t) = u_{i, r}$.
Since $\square |\nabla u^*_{i,r}| \leq 0$, we obtain, using \cite[\HKThmHKboundNNbound]{Bamler_HK_entropy_estimates}, that for large $i$, and for any $x''' \in M_i$
\begin{multline*}
 |\nabla u^*_{i,r}| (x''', t^{*}) - 1
 \leq \int_{M_i} ( |\nabla u_{i,r}| - 1)_+ K(x''', t^{*}; \cdot, t) dg_{i, t}
 \leq C \exp ( - \NN_{x''', t^{*}} ( t^{*} - t )) r^{.5} \\
 \leq C(t^*, x') \exp ( C(t^*) d_{t^*} (x'_i, x''') ) r^{.5}. 
\end{multline*}
It follows that along any minimizing time-$t^{*}$-geodesic between $x'_i, x''_i$
\[ |\nabla u^*_{i,r}| (\cdot , t^{*}) - 1 \leq C(t^*, x', x'')  r^{.5} \]
Therefore,
\begin{multline*}
\limsup_{i \to \infty} \bigg| \int_{M_i} u_{i,r} \, d\nu_{x'_i; t} - \int_{M_i} u_{i,r} \, d\nu_{x''_i; t} \bigg| =
 \limsup_{i \to \infty} |u^*_{i,r} (x'_i) - u^{*}_{i,r}(x''_i)| \\
 \leq (1+ C(t^*, x', x'') r^{.5}) d_{t^{*}} (x', x''). 
\end{multline*}
The claim now follows by passing to the limit $i \to \infty$.
\end{proof}

Letting $r \to 0$ in (\ref{eq_ur_dnuxpxpp}) implies that
\begin{equation} \label{eq_intuintu_dXX}
 \bigg| \int_{\RR^*_{t}} u \, d\nu_{x'; t} - \int_{\RR^*_{t}} u \, d\nu_{x''; t} \bigg| \leq d_{t^{*}} (x', x''). 
\end{equation}
Now, by future continuity of $\XX$ (see \cite[\SYNPropFutContTwoPts]{Bamler_RF_compactness}), for any $x', x'' \in \RR_{t}$ there are sequences $t^*_j \searrow t$, $x'_j, x''_j \in \XX_{t^*_j}$ with $\nu_{x'_j; t} \to \delta_{x'}$, $\nu_{x''_j; t} \to \delta_{x''}$ in the $W_1$-sense and $\lim_{j \to \infty} d_{t^*_j} (x'_j, x''_j) = d_{t} (x', x'')$.
Combining this with (\ref{eq_intuintu_dXX}) implies that $| u(x') - u(x'') | \leq d_{t} (x', x'')$, as desired.
This finishes the proof of Assertion~\ref{Thm_SS_dimension_bound_limit_c}.

Lastly, we establish Assertion~\ref{Thm_SS_dimension_bound_limit_d}.
Assume that $\Delta \geq 3$ and consider a function $u \in C^0 ( \RR_{ [t', 0)} ) \cap C^\infty ( \RR_{(t', 0)})$, $t' \in I_\infty$, with $\square u = 0$ on $\RR_{(t', 0)}$.
By replacing $u$ with $u - u'$, where $u' : \XX_{[t', 0]} \to \IR$ is a uniformly bounded solution to the heat flow with initial condition $u \chi_{\RR_{t'}}$, we may assume in the following that $u |_{\RR_{t'}} = 0$.
Then we need to show that $u = 0$ on $\RR_{[t', 0)}$.
To see this, consider some $x^* \in \RR_{t^*}$ and let $v := K(x^*; \cdot) : \RR_{[t', t^*)}  \to [0, \infty)$ be the conjugate heat kernel based at $x^*$.
If we can show that for any $t'' \in [t', t^*)$ we have
\begin{equation} \label{eq_int_uv_tpp}
 \int_{\RR^*_{t''}} u^2 \, v \, dg_{t''} = 0, 
\end{equation}
then, since $u$ was assumed to be uniformly bounded, we obtain $u_{t''} \equiv 0$.
So our goal will be to show (\ref{eq_int_uv_tpp}) for some given $v$ and $t''$, which we will assume to be fixed for the remainder of this proof.
We will view $v$ as a smooth function on $\RR_{[t', t'']}$.
In the following, generic constants may depend on the choice of $t', t'', u, v$, but not on $i$.
Note that $v$ arises as a limit of conjugate heat kernels of the form $v_i := K(x^*_i, t^*_i; \cdot, \cdot)$ for some convergent sequence $(x^*_i, t^*_i) \to x_\infty$ within $\CF$.
We will set in the following
\[  d\nu_{x^*_i,t^*_i} =: d\nu^*_i  =: v_i \, dg_{i} =: (4\pi \tau^*_i)^{-n/2} e^{-f^*_i} dg_i. \]
We record that by Lemma~\ref{Lem_tdrrm}, Propositions~\ref{Prop_NN_variation_bound}, \ref{Prop_improved_L2}, and by standard local derivative estimates we have on $\RR_{[t', t'']}$ for some $C < \infty$, which may depend on $v$, $t'$, $t''$,
\begin{multline} \label{eq_v_bounds_as_limit}
 \square^* v = 0, \qquad
\int_{t'}^{t''} \int_{\RR^*_t} |\nabla v| dg_t dt
 \leq \limsup_{i \to \infty} \int_{t'}^{t''} \int_{M_i} |\nabla f^*_i| (4\pi \tau^*_i)^{-n/2} e^{-f_{i}^*} dg_{i,t} dt \leq C, \\
 0 \leq v \leq C, \qquad
 |\nabla v| \leq C( \tdrrm^{-1}+1), \qquad
 |\nabla^2 v|, |\partial_{\tf} v| \leq C (\tdrrm^{-2}+1).
\end{multline}
Note that the last two derivative bounds in (\ref{eq_v_bounds_as_limit}) follow since $v_i \to v$ in $C^\infty_{\loc}$ and since $\rrm \to \tdrrm$.
By Lemma~\ref{Lem_tdrrm} and local derivative estimates (see also \cite[\SYNLemStdParEst]{Bamler_RF_compactness}), we also obtain that
\begin{equation} \label{eq_u_bounds_loc_der}
 |\nabla u| \leq C( \tdrrm^{-1}+1), \qquad
 |\nabla^2 u|, |\partial_{\tf} u| \leq C( \tdrrm^{-2}+1). 
\end{equation}
Note that here we do not need to use the fact that $u$ is a limit of solutions to the heat equation, because the existence of a \emph{backward} parabolic neighborhood of radius $\tdrrm$ inside $\RR$ is ensured by Lemma~\ref{Lem_tdrrm}.

Next, we prove:

\begin{Claim} \label{Cl_int_parts_v_star}
Let $v^* \in C^\infty (\RR_{[t', t'']})$ be a smooth function such that for some $C^* < \infty$ we have
\[ 0 \leq v^* \leq C^*, \qquad |\nabla v^*| \leq C^* (\tdrrm^{-1}+1), \qquad |\partial_{\tf} v^*|, |\nabla^2 v^*| \leq C^* (\tdrrm^{-2}+1) \]
and such that $v^* \equiv 0$ on $\RR_{[t', t'']} \setminus P$ for $P := P^* (x_\infty; C^*, - t')$.
Then
\begin{equation} \label{eq_claim_int_parts_identity_v_star}
 \int_{\RR^*_{t''}} u^2 \, v^* \, dg_{t''} \leq - \int_{t'}^{t''} \int_{\RR^*_t} u^2 \, \square^* v^* \, dg_t dt. 
\end{equation}
\end{Claim}

\begin{proof}
Fix some small $r > 0$ and let $\eta_r \in C^\infty (\RR)$ be the function from Lemma~\ref{Lem_eta_r_on_RR}.
It follows that the product $u^2 v^* \eta_r$ is compactly supported in $\RR^*_t$ for all $t \in [t', t'']$.
Therefore, using the bounds (\ref{eq_v_bounds_as_limit}), (\ref{eq_u_bounds_loc_der}) and Lemma~\ref{Lem_eta_r_on_RR},
\begin{multline} \label{eq_int_by_parts_triangle}
\bigg| \int_{t'}^{t''}  \int_{\RR^*_{t}} \big( -v^*  \triangle u^2   + u^2 \triangle v^* \big) dg_t dt \bigg|
= \lim_{r \to 0} \bigg| \int_{t'}^{t''} \int_{\RR^*_{t}} \big( -v^*  \triangle u^2   + u^2 \triangle v^* \big) \eta_r dg_t dt \bigg| \\
= \lim_{r \to 0} \bigg| \int_{t'}^{t''} \int_{\RR^*_{t}} \big( -  u^2 \nabla v^* \nabla \eta_r     +2 v^* u \nabla u \nabla \eta_r\big) dg_t dt \bigg| 
= \lim_{r \to 0} C r^{-2} \int_{t'}^{t''} \int_{\{ \tdrrm \leq 2r \} \cap \RR^*_{t}}  dg_t dt 
= 0.
\end{multline}
Next, by integrating along every worldline and using Lemma~\ref{Lem_eta_r_on_RR}\ref{Lem_eta_r_on_RR_b6} and the fact that $u^2 v^* \geq 0$, we obtain
\[  \int_{t'}^{t''}  \int_{\RR^*_{t}} \big(  \partial_{\tf} (u^2 v^* \eta_r) -  u^2 v^* \eta_r R \big)  dg_t dt \geq  \int_{\RR^*_{t''}} u^2 \, v^* \, \eta_r \, dg_{t''} . \]
Therefore, as in (\ref{eq_int_by_parts_triangle}),
\begin{align} 
\int_{t'}^{t''}  \int_{\RR^*_{t}} &\big( v^* \partial_{\tf} u^2 + u^2 \partial_{\tf} v^* - u^2  R v^* \big) dg_t dt
= \lim_{r \to 0} \int_{t'}^{t''} \int_{\RR^*_{t}} \big( v^* \partial_{\tf} u^2 + u^2 \partial_{\tf} v^* -  u^2 v^* R \big) \eta_r \, dg_t dt \notag \\
&= \lim_{r \to 0} \int_{t'}^{t''} \int_{\RR^*_{t}} \big(  \partial_{\tf} (u^2 v^* \eta_r) - u^2 v^* \partial_{\tf}  \eta_r  - u^2 v^* R \eta_r \big) dg_t dt \notag \\
&\geq \limsup_{r \to 0} \bigg( \int_{t'}^{t''}  \int_{\RR^*_{t}}  \big( \partial_{\tf} (u^2 v^* \eta_r) - u^2 v^* R \eta_r \big) dg_t dt - Cr^{-2} \int_{t'}^{t''}  \int_{\{ \tdrrm \leq 2r \} \cap \RR^*_{t}} dg_t dt \bigg) \notag \\
&\geq  \int_{\RR^*_{t''}} u^2 \, v^* \, dg_{t''}. \label{eq_int_by_parts_partial_t}
\end{align}
Combining (\ref{eq_int_by_parts_triangle}) and (\ref{eq_int_by_parts_partial_t}) and using the fact that $\square u^2 \leq 0$ implies (\ref{eq_claim_int_parts_identity_v_star}).
\end{proof}

Let $z \in \XX_{t'}$ be an $H_n$-center of $x_\infty$ and let $L < \infty$ be a constant whose value we will choose later.
Consider the diffeomorphisms $\psi_i : U_i \to V_i$ from Subsection~\ref{subsec_setup_limit}, describing the smooth convergence on $\RR^*$.
Choose a $10 L^{-1}$-Lipschitz function $w^*_L : \XX_{t'} \to [0,1]$ with $w^*_L \equiv 1$ on $B (z,L)$, $w^*_L \equiv 0$ on $\XX_{t'} \setminus B (z, 2L)$.
We may also assume that $w^*_L |_{\RR_{t'}}$ is smooth and that there are smooth functions $w^*_{L,i} : M_i  \times \{ t' \} \to [0,1]$ such that $w^*_{L,i} \circ \psi_i |_{\RR^*_{t'}} \to w^*_L |_{\RR^*_{t'}}$ in $C^\infty_{\loc}$ on $\RR^*_{t'}$ and such that 
\[ |\nabla w^*_{L,i}| \leq 10 L^{-1}. \]
Let $w_{L,i} \in C^\infty( M_i \times [t', 0])$ be solutions the heat equation $\square w_{L,i} = 0$ with initial condition $w_{L,i} (\cdot, t') = w_{L,i}^* (\cdot, t')$ and let $w_L : \XX_{[t', 0)} \to [0,1]$ be the heat flow with initial condition $w_{L,t'} = w^*_L$.
Then $w_{L,i} \circ \psi_i \to w_L$ on $\RR^*_{[t', t'']}$ in $C^\infty_{\loc}$.
It follows that 
\begin{equation} \label{eq_nab_w_L_bounded}
|\nabla w_L| \leq 10 L^{-1} \qquad \text{on} \quad \RR^*_{[t', t'']}
\end{equation}
and
\begin{multline} \label{eq_int_nab_w2_L_bounded}
 2\int_{t'}^{t''} \int_{\RR^*_{t}} |\nabla^2 w_L|^2 v \, dg_t dt 
\leq \limsup_{i \to \infty} \int_{t'}^{t''} \int_{M_i} 2 |\nabla^2 w_{L,i}|^2 v_i  \, dg_{i,t} dt \\
\leq \limsup_{i \to \infty} \int_{M_i} |\nabla w_{L,i}|^2 d\nu^*_{i,t'}
\leq 100 L^{-2}. 
\end{multline}

Let $\eta : [0,1] \to [0,1]$ be a cutoff function with $\eta \equiv 0$ on $[0,.1]$, $\eta \equiv 1$ on $[.9,1]$ and $|\eta'| \leq 10$.
Set $\phi_L := \eta \circ w_L$.
We claim that there is a constant $A$, which may depend on $L$, such that
\begin{equation} \label{eq_supp_phi_L_P_var}
 \phi_L \equiv 0 \qquad \text{on} \quad \RR^*_{[t', t'']} \setminus P^* (x_\infty; A, t'). 
\end{equation}
To see this, consider some point $y' \in \RR^*_{[t', t'']} \setminus P^* (x_\infty; A, t')$ and let $z' \in \XX_{t'}$ be an $H_n$-center of $y'$.
Then
\[ w_L (y') = \int_{\RR_{t'}} w^*_L \, d\nu_{y'; t'}
\leq \nu_{y'; t'} ( B (z, 2L) )
\leq \frac{{\Var} (\nu_{y'; t'}, \delta_{z'})}{(d_{t'} (z, z') - 2L)^{2}} 
\leq  \frac{H_n |t'|}{(d_{t'} (z, z') - 2L)^{2}}  \]
and
\begin{multline*}
 d_{t'} (z, z') 
\geq 
d^{\XX_{t'}}_{W_1} (\nu_{y'; t'}, \nu_{x_\infty; t'})
- d^{\XX_{t'}}_{W_1} (\delta_z, \nu_{x_\infty; t'})
- d^{\XX_{t'}}_{W_1} (\nu_{y'; t'}, \delta_{z'} ) \\
\geq A - \sqrt{\Var (\delta_z, \nu_{x_\infty; t'})} - \sqrt{\Var (\nu_{y'; t'}, \delta_{z'} )}
\geq A - 2\sqrt{H_n |t'|}. 
\end{multline*}
So if $A$ is sufficiently large, depending on $L$, we have $w_L(y') \leq .1$ and therefore $\phi_L (y') = 0$.

Due to (\ref{eq_supp_phi_L_P_var}), (\ref{eq_v_bounds_as_limit}) and derivative estimates on $w_L$, we can apply Claim~\ref{Cl_int_parts_v_star} with $v^* := v \phi_L$.
We obtain using (\ref{eq_nab_w_L_bounded}), (\ref{eq_int_nab_w2_L_bounded}) that
\begin{align*}
 \int_{\RR^*_{t''}} u^2 \,  v \phi_L  \, dg_{t''} 
&\leq - \int_{t'}^{t''} \int_{\RR^*_{t}} u^2 \square^* (v\phi_L ) dg_t dt \\
&=  \int_{t'}^{t''} \int_{\RR^*_{t}}  u^2 (2 \nabla v \cdot \nabla \phi_L + v (\partial_{\tf} + \triangle ) \phi_L) dg_t dt \\
&\leq C \int_{t'}^{t''} \int_{\RR^*_{t}}  \big( L^{-1} |\nabla v|  + L^{-2} v +  |\nabla^2 w_L | v  \big) dg_t dt \\
&\leq C L^{-1} + C \int_{t'}^{t''} \int_{\RR^*_{t}} (L^{-1} +  L |\nabla^2 w_L |^2) v \,  dg_t dt
\leq CL^{-1} ,
\end{align*}
where $C$ is independent of $L$.
Lastly, observe that $w_L, \phi_L \to 1$ pointwise as $L \to \infty$.
So letting $L \to \infty$ implies (\ref{eq_int_uv_tpp}), which finishes the proof. 
\end{proof}
\bigskip

\begin{proof}[Proof of Theorem~\ref{Thm_XX_uniquely_det_RR}.]
Using Theorem~\ref{Thm_SS_dimension_bound_limit}\ref{Thm_SS_dimension_bound_limit_c} we can uniquely extend the isometries $\RR_{t} \to \RR'_{t}$ to isometric embeddings $\phi_t : \XX_t \to \XX'_t$ for all $t \in I_\infty \setminus \{ 0 \}$.
Next let $x \in \RR_t$, $s < t$, $s, t \in I_\infty \setminus \{ 0 \}$, and consider a bounded continuous function $\ov u' \in C^0_c (\RR'_s)$ with compact support.
Let $u' : \XX'_{[s,0)} \to \IR$ be the heat flow on $\XX'$ with initial condition $\ov u'$.
Then $u' |_{\RR'_{(s, 0)}}$ is smooth and $\square u' = 0$.
Thus, if we set $u := u' \circ \phi$, then $\square u = 0$ on $\RR$.
It follows using Theorem~\ref{Thm_SS_dimension_bound_limit}\ref{Thm_SS_dimension_bound_limit_d} that $u |_{\RR_{(s, 0)}}$ is the restriction of the heat flow with initial condition $\ov u := \ov u' \circ \phi_s$ to $\RR_{(s, 0)}$.
Thus, for any $x \in \RR_{(s, 0)}$
\[ \int_{\RR'_s} \ov u' \, d\nu'_{\phi_t(x); s}
= u' (\phi_t(x)) 
= u (x) 
= \int_{\RR_s} \ov u \, d\nu_{x; s} 
= \int_{\RR'_s} \ov u' \, d ((\phi_s)_* \nu_{x;s} ) . \]
It follows that 
\begin{equation} \label{eq_phi_s_nu_nu_p}
(\phi_s)_* \nu_{x;s} = \nu'_{\phi_t(x), s}
\end{equation}
 on $\RR_s$ and since $((\phi_s)_* \nu_{x;s} )(\RR_s) = 1$, we have (\ref{eq_phi_s_nu_nu_p}) on $\XX_s$.
 Since $\XX'$ has full support, we find that $\phi_s$ is also surjective.
 Since $\nu_{x;s}, \nu'_{x;s}$ depend continuously on $x$, we find that (\ref{eq_phi_s_nu_nu_p}) holds for all $x \in \XX_{(s, 0)}$, which finishes the proof.
\end{proof}
\bigskip

\subsection{\texorpdfstring{The pointed-Nash entropy and an $\eps$-regularity theorem in the limit}{The pointed-Nash entropy and an {\textbackslash}eps-regularity theorem in the limit}}
The next theorem concerns the convergence of the Nash-entropy and establishes the known properties in the limit, such as an $\eps$-regularity theorem; compare with \cite{Bamler_HK_entropy_estimates}.

\begin{Theorem}\label{Thm_NN_in_limit}
Suppose we are in the setting described in Subsection~\ref{subsec_setup_limit} and consider a subsequence of the given sequence of Ricci flows that allows the definition of $\tdrrm$ according to Lemma~\ref{Lem_tdrrm}.
Let $x'_\infty \in \XX_{t'_\infty}$ be a point and $r > 0$ a scale satisfying $[t'_\infty - r^2, t'_\infty] \subset I_\infty$ and
\begin{equation} \label{eq_dnu_SS_0_condition}
 d\nu_{x'_\infty; t' - r^2} (\SS^*_{ t'_\infty - r^2} ) = 0. 
\end{equation}
Then the following holds:
\begin{enumerate}[label=(\alph*)]
\item \label{Thm_NN_in_limit_a} For any sequence $(x'_i, t'_i) \in M_i \times I_i$ converging to $x'_\infty$ within $\CF$ we have
\[ \lim_{i \to \infty} \NN_{x'_i, t'_i} (r^2) = \NN_{x'_\infty} (r^2). \]
\item \label{Thm_NN_in_limit_b} There is a dimensional constant $\eps_0 > 0$ such that if $\NN_{x'_\infty} (r^2) \geq - \eps_0$, then $x'_\infty \in \RR^*$ and $\tdrrm(x'_\infty) \geq \eps_0 r$.
\item \label{Thm_NN_in_limit_c} If $x'_\infty \in \RR$, then
\[ \lim_{\substack{r \searrow 0 \\ \text{(\ref{eq_dnu_SS_0_condition}) holds}}} \NN_{x'_\infty} (r^2) = 0. \]
\end{enumerate}
\end{Theorem}

Note that by Theorem~\ref{Thm_SS_dimension_bound_limit} that if $\Delta \geq 3$, then condition (\ref{eq_dnu_SS_0_condition}) holds automatically and if $\Delta \geq 1$, then (\ref{eq_dnu_SS_0_condition}) holds for almost every $r$.
Therefore, we obtain:

\begin{Corollary}\label{Cor_RRstar_RR}
$\RR^* = \RR$ and $\SS^* = \SS$.
Moreover, Theorems~\ref{Thm_SS_dimension_bound_limit}, \ref{Thm_NN_in_limit} hold for the original sequence of Ricci flows.
In other words, it is not necessary to consider a subsequence that allows the definition of $\tdrrm$ according to Lemma~\ref{Lem_tdrrm}.
\end{Corollary}

\begin{proof}[Proof of Theorem~\ref{Thm_NN_in_limit}.]
In this proof $C < \infty$ will denote a generic constant whose value may depend on various geometric data, but not on the index $i$.

We first prove Assertion~\ref{Thm_NN_in_limit_a}.
Set $t'_\infty := \tf (x'_\infty)$.
By Lemma~\ref{Lem_conv_P_star} and Proposition~\ref{Prop_NN_variation_bound} we have $\NN_{x'_i, t'_i} (r^2) \geq -C$ for some uniform $C < \infty$.
Therefore by Proposition~\ref{Prop_L_infty_HK_bound} if $d\nu_{x'_i, t'_i} = (4\pi \tau_{x'_i, t'_i})^{-n/2} e^{-f_{x'_i, t'_i}} dg_i$, then $f_{x'_i, t'_i} \geq  - C$ on $M_i \times [t'_\infty - r^2, t'_i)$.
By Proposition~\ref{Prop_int_ealphaf} we also have 
\[ \int_{M_i} e^{-f_{x'_i, t'_i}/2} dg_{i, t'_\infty - r^2} \leq C < \infty. \]

Consider the diffeomorphisms $\psi_i : U_i \to V_i$ from Subsection~\ref{subsec_setup_limit}.
Then
\begin{equation} \label{eq_fxpinfty_bounds}
 f_{x'_\infty} \geq - C \qquad \text{on} \quad \RR^*_{t'_\infty - r^2}, \qquad
\int_{\RR^*_{t'_\infty -r^2} }   e^{-f_{x'_\infty}/2} dg_{t'_\infty - r^2} \leq C ,
\end{equation}
which implies
\begin{multline*}
 \bigg| \int_{M_i \times \{ t'_\infty - r^2 \} \setminus V_i}  f_{x'_i, t'_i} e^{-f_{x'_i, t'_i}} \, dg_{i, t'_\infty - r^2} \bigg|
\leq C \int_{M_i \setminus V_i} e^{f_{x'_i, t'_i}/4} e^{-f_{x'_i, t'_i}} \, dg_{i, t'_\infty - r^2}\\
\leq C \bigg(\int_{M_i \setminus V_i}  e^{-f_{x'_i, t'_i}/2} \, dg_{i, t'_\infty - r^2} \bigg)^{1/2} \bigg(\int_{M_i \setminus V_i}  e^{-f_{x'_i, t'_i}} \, dg_{i, t'_\infty - r^2} \bigg)^{1/2}
 \to 0,
 \end{multline*}
 \begin{multline*}
 \bigg| \int_{\RR^*_{t'_\infty -r^2} \setminus U_i}  f_{x'_\infty } e^{-f_{x'_\infty}} dg_{t'_\infty - r^2} \bigg|
\leq C \int_{\RR^*_{t'_\infty -r^2} \setminus U_i}  e^{f_{x'_\infty}/4} e^{- f_{x'_\infty}} dg_{t'_\infty - r^2}\\
\leq C\bigg( \int_{\RR^*_{t'_\infty -r^2} \setminus U_i}  e^{-f_{x'_\infty}/2} dg_{t'_\infty - r^2} \bigg)^{1/2} \bigg( \int_{\RR^*_{t'_\infty -r^2} \setminus U_i}   e^{- f_{x'_\infty}} dg_{t'_\infty - r^2} \bigg)^{1/2}
 \to 0 
\end{multline*}
Combining both bounds implies Assertion~\ref{Thm_NN_in_limit_a}.

Assertion~\ref{Thm_NN_in_limit_b} follows from Assertion~\ref{Thm_NN_in_limit_a} combined with Lemma~\ref{Lem_tdrrm}, \cite[\SYNThmExistencConvSeq]{Bamler_RF_compactness} and \cite[\HKThmEpsRegularity]{Bamler_HK_entropy_estimates}.

For Assertion~\ref{Thm_NN_in_limit_c}, let us only consider $r>0$ for which (\ref{eq_dnu_SS_0_condition}) holds.
By Assertion~\ref{Thm_NN_in_limit_a} we know that $\NN_{x'_\infty} (r^2)$ is non-increasing and non-positive for all such $r$.
So $A := \lim_{r \searrow 0} \NN_{x'_\infty}(r^2) \leq 0$ exists.
Suppose by contradiction that $A < 0$.
By (\ref{eq_fxpinfty_bounds}) we can find a constant $C^* < \infty$ such that
\begin{equation} \label{eq_f_geq_m_C_star}
 f_{x'_\infty} \geq - C^* \qquad \text{on} \quad \RR_{t'_\infty - r^2} 
\end{equation}
for each $r$ satisfying (\ref{eq_dnu_SS_0_condition}).
So using the local concentration property of the conjugate heat kernel based at $x'_\infty$ (see \cite{Bamler_RF_compactness}), we obtain that there is some $\delta > 0$ such that for large $L < \infty$  we have
\[ \limsup_{r \searrow 0} \int_{B(x'_\infty(t'_\infty - r^2), L r) \cap \RR_{t'_\infty - r^2} } f_{x'_\infty} e^{-f_{x'_\infty}} dg_{t'_\infty - r^2} - \frac{n}2 \leq - \delta . \]
Therefore, by a blow-up argument near $x'_\infty$, and using (\ref{eq_f_geq_m_C_star}), we obtain a solution $v \lb = (4\pi |t|)^{-n/2} \lb e^{-f''}$ to the backwards heat equation on $\IR^n \times \IR_-$ with
\[ v (\cdot, t) \leq \frac{ C}{ |t|}, \quad
\int_{\IR^n} v(\vec x, t) \, d\vec x = 1, \quad
\int_{\IR^n} |\vec x|^2 v(\vec x, t) \, d\vec x \leq H_n |t| , \quad
\int_{\IR^n} f'' v(\vec x, t) \, d\vec x - \frac{n}2 \leq - \delta < 0. \]
The first three bounds imply that $v = (4\pi |t|)^{-n/2} e^{-|\vec x|^2/4|t|}$, which contradicts the last bound.
\end{proof}
\bigskip

\subsection{\texorpdfstring{Splitting off an $\IR^k$-factor in the limit}{Splitting off an R\^{}k-factor in the limit}}
In the next series of results, we discuss how certain almost-symmetries of the Ricci flows $(M_i, (g_{i,t})_{t \in I_i})$ affect the limiting flow $\XX$.

The first result addresses the case in which the flows are strongly $(k, \eps_i, r)$-split for some $\eps_i \to 0$.

\begin{Theorem} \label{Thm_limit_from_strong_splitting}
Consider the setting described in Subsection~\ref{subsec_setup_limit}.
Suppose that $I_\infty = (-\infty, 0]$ and that there is a sequence of strong $(k,  \eps_i, r)$-splitting maps $\vec y^{\, i}$ at $(x_i,0)$ for $\eps_i \to 0$ and for some fixed $r > 0$.
Then, after passing to a subsequence, we have local smooth convergence $y^i_l \to y^\infty_l$ on $\RR$, $l = 1, \ldots, k$, where $y^\infty_1, \ldots, y^\infty_k \in C^\infty(\RR)$ satisfy for $j, l = 1, \ldots, k$
\begin{equation} \label{eq_y_infty_identities}
 \nabla y_j^\infty \cdot \nabla y_l^\infty  = \delta_{jl}, \qquad  \nabla^2 y_j^\infty = 0, \qquad  \partial_\tf y_j^\infty = 0. 
\end{equation}
Moreover, the heat kernel $K$ on $\RR$ satisfies the following infinitesimal spatial translation property for $l = 1, \ldots, k$
\begin{equation} \label{eq_K_infty_spatial_symm}
 \nabla_{x_1} K (x_1; x_0) \cdot \nabla y_l^\infty + \nabla_{x_0} K (x_1; x_0) \cdot \nabla y_l^\infty  = 0. 
\end{equation}
If $\Delta \geq 2$, then the vector fields $\nabla y_l^\infty$ are complete, i.e. $\RR$ isometrically splits off an $\IR^k$-factor as a Ricci flow spacetime.
If $\Delta \geq 4$, then this splitting induces a splitting of the form $\XX_{<0} \cong \XX' \times \IR^k$ for some metric flow $\XX'$ over $\IR_-$.
\end{Theorem}

\begin{proof}[Proof of Theorem~\ref{Thm_limit_from_strong_splitting}.]
By parabolic rescaling we may assume that $r = 1$; we will use the letter $r$ for other purposes in the following.
Using standard local derivative estimates and Proposition~\ref{Prop_properties_splitting_map}, we obtain local uniform bounds that allow us to conclude that after passing to a subsequence, we have local smooth convergence $y^i_l \to y^\infty_l \in C^\infty (\RR)$ for each $l = 1, \ldots, k$.
The identities (\ref{eq_y_infty_identities}) are clear.

Next, we prove (\ref{eq_K_infty_spatial_symm}).
For this purpose, fix some arbitrary $t_0 < 0$, an arbitrary compactly supported function $u^\infty \in C^\infty_c (\RR_{t_0})$ and some $l \in \{ 1, \ldots, k \}$.
Choose functions $u^i \in C^\infty_c (M_i \times \{ t_0 \})$ such that $u^i \to u^\infty$ in $C^\infty_{\loc}$.
By applying the flow of the vector field $\nabla y^i_l$ we may extend each $u^i$ to a smooth family of functions $u^{\prime, i} \in C^\infty ( M_i \times \{ t_0 \} \times (- \alpha, \alpha))$ for $\alpha > 0$ depending on a parameter $h$ such that
\begin{equation} \label{eq_upi_identities}
  u^{\prime, i} (\cdot, t_0, 0) = u^i (\cdot, t_0), \qquad \nabla y^i_l \cdot \nabla u^{\prime, i} (\cdot, t_0, h) = \frac{\partial u^{\prime, i}}{\partial h} (\cdot, t_0, h) 
\end{equation}
and such that $u^{\prime, i}$ smoothly converges to a smooth function $u^{\prime, \infty} \in C^\infty ( \RR_{t_0} \times (- \alpha, \alpha))$ that extends $u^\infty$ and whose restriction to $\RR_{t_0} \times [- \alpha', \alpha']$ for any $\alpha' < \alpha$ is compactly supported and satisfies (\ref{eq_upi_identities}) for $i = \infty$.
Next, solve the heat equation with initial condition $u^{\prime, i} (\cdot, t_0, h)$ for any $h \in (-\alpha, \alpha)$ and $i \in \IN$ to obtain functions $u^{\prime\prime, i}  \in C^\infty ( M_i \times [ t_0 ,0] \times (- \alpha, \alpha))$ such that
\begin{equation} \label{eq_uppi_HE}
 \square u^{\prime\prime, i} = 0, \qquad 
u^{\prime\prime, i} (\cdot, t_0, \cdot) = u^{\prime, i}. 
\end{equation}
Due to the smooth convergence of the heat kernels on $\RR$, we may pass to a subsequence and assume that we have local smooth convergence of $u^{\prime\prime, i}$ to some smooth function $u^{\prime\prime, \infty} \in C^\infty (\RR_{[t_0, 0)} \times (-\alpha, \alpha))$ satisfying (\ref{eq_uppi_HE}) for $i = \infty$ and
\begin{equation} \label{eq_uppinfity_int_HK}
 u^{\prime\prime, \infty} (\cdot, h) = \int_{\RR_{t_0}} u^{\prime, \infty} (x, h) K (\cdot; x ) dg_{t_0}(x). 
\end{equation}

We can compute that
\[ \square \frac{\partial u^{\prime\prime, i}}{\partial h} = \frac{d}{dh} \square  u^{\prime\prime, i} = 0, \qquad
 \square ( \nabla y_l^i \cdot \nabla u^{\prime\prime, i} ) = -2\nabla^2 y_l^i \cdot \nabla^2 u^{\prime\prime, i}. \]
This implies that in the sense of weak derivatives
\[ \square \Big| \frac{\partial u^{\prime\prime, i}}{\partial h} - \nabla y_l^i \cdot \nabla u^{\prime\prime, i} \Big| \leq 2 |\nabla^2 y_l^i | \, | \nabla^2 u^{\prime\prime, i} |. \]
So for any $t_1 \in [t_0, 0)$ we have for any fixed $h \in (-\alpha, \alpha)$, using Proposition~\ref{Prop_properties_splitting_map} and setting $\nu^i := \nu_{x_i,0}$,
\begin{align*}
  \int_{M_i}  &\Big|  \frac{\partial u^{\prime\prime, i}}{\partial h} - \nabla y_l^i \cdot \nabla u^{\prime\prime, i} \Big| d\nu^i_{t_1}  \leq  2  \int_{t_0}^{t_1} \int_{M_i}   |\nabla^2 y_l^i | \, | \nabla^2 u^{\prime\prime, i} | d\nu^i_{t} dt \\
  &\leq  \bigg(  2\int_{t_0}^{t_1} \int_{M_i}  | \nabla^2 y^i_l |^2 d\nu^i_{t} dt \bigg)^{1/2}
  \bigg(  2\int_{t_0}^{t_1} \int_{M_i}  | \nabla^2 u^{\prime\prime, i} |^2 d\nu^i_{t} dt \bigg)^{1/2} \\
  &\leq  \bigg(  2\int_{t_0}^{t_1} \int_{M_i}  | \nabla^2 y^i_l |^2 d\nu^i_{t} dt \bigg)^{1/2}
  \bigg(   \int_{M_i}  | \nabla u^{\prime\prime, i} |^2 d\nu^i_{t_0}  \bigg)^{1/2} \to 0.
\end{align*}
It follows that
\[ \frac{\partial u^{\prime\prime, \infty}}{\partial h} = \nabla y_l^\infty \cdot \nabla u^{\prime\prime, \infty}. \]
Combining this with (\ref{eq_upi_identities}), (\ref{eq_uppinfity_int_HK}) for $h  = 0$ implies that for any $x_1 \in \RR_{t_1}$
\[ \nabla y^\infty_l \int_{\RR_{t_0}} u^\infty (x_0) \nabla_{x_1} K (x_1; x_0) dg_{t_0} (x_0) =  \int_{\RR_{t_0}} (\nabla y^\infty_l \cdot \nabla u^{\infty}) (x_0) K (x_1; x_0 ) dg_{t_0} (x_0). \]
Integration by parts on the right-hand side and using $\nabla^2 y^\infty_l = 0$ implies that
\[  \int_{\RR_{t_0}} u^\infty (x_0) \nabla_{x_1} K (x_1; x_0) \cdot \nabla y^\infty_l   (x_1) \, dg_{t_0} (x_0) = - \int_{\RR_{t_0}} u^{\infty} (x_0)    \nabla_{x_0} K (x_1; x_0 ) \cdot \nabla y^\infty_l (x_0) \, dg_{t_0} (x_0). \]
Since $u^\infty$ was arbitrary, this implies (\ref{eq_K_infty_spatial_symm}).

Next, assume that $\Delta \geq 2$.
Our goal will be to show that then all trajectories of $\nabla y^\infty_l$, $l = 1, \ldots, k$, are complete.
By Theorem~\ref{Thm_NN_in_limit}\ref{Thm_NN_in_limit_c} we obtain that for any $x'_\infty \in \RR_t$ we have 
\[ \lim_{\tau' \to 0} \frac2{\tau'} \int_{\tau'/2}^{\tau'} \NN_{x'_\infty} (\tau'') d\tau'' = 0. \]

Now fix some $t_1 \in I_\infty \setminus \{ 0 \}$ and consider a maximal trajectory $\gamma : I^* \to \RR_{t_1}$ of $\nabla y_l^\infty$.
We claim that $s \mapsto \int_{\tau'/2}^{\tau'} \NN_{\gamma(s)} (\tau'') d\tau''$ is constant, which implies for small $\tau'$ that $\inf_{s \in I^*} \tdrrm(\gamma(s)) \geq c \sqrt{\tau'}  > 0$ via Theorem~\ref{Thm_NN_in_limit}\ref{Thm_NN_in_limit_b}.

To show the constancy statement, fix some small $\tau' > 0$ and set $t_0 := t_1 - \tau'$, $t'_0 := t_1 - \tau'/2$.
Assume that $\tau'$ is chosen small enough so that there is a point $x_0 \in \RR_{t_0}$ that survives until time $t'_0$.
Write
\[ K (\gamma(s); \cdot) \big|_{\RR_{[t_0,t'_0]}} = (4\pi \tau)^{-n/2} e^{-f_s} \]
for some smooth family $f \in C^\infty ( \RR_{[t_0,t'_0]} \times I^*)$.
Our goal will be to show that
\begin{equation} \label{eq_need_ts_int_ddsNN}
 \frac{d}{ds} \int_{t_0}^{t'_0} \int_{\RR_{t}} f_s e^{-f_s} dg_{t} dt = 0.
\end{equation}

Since $K$ is a limit of heat kernels, there is a constant $C < \infty$ such that for every $t \in [t_0, t'_0]$
\begin{equation} \label{eq_general_fs_bounds}
 - f_s, |f_s e^{-f_s}| \leq  C, \qquad \int_{t_0}^{t'_0} \int_{\RR_{t}}  |f_s e^{-f_s}| dg_{t} dt \leq C. 
\end{equation}
By (\ref{eq_K_infty_spatial_symm}) we have
\[ \partial_s f_s = \nabla y_l^\infty \cdot \nabla f_s. \]
Consider the functions $(\eta_r \in C^\infty (\RR))_{r > 0}$ from Lemma~\ref{Lem_eta_r_on_RR} (which can be applied after passing to a suitable subsequence).
We also fix some $\delta > 0$ and a smooth cutoff function $\ov\eta_\delta : [0, \infty) \to [0, \infty)$ with $\ov\eta_\delta \equiv 0$ on $[0, \delta]$ and $\ov\eta_\delta (a) = a$ on $[2 \delta, \infty)$.
By Lemmas~\ref{Lem_limit_HK_bound}, \ref{Lem_eta_r_on_RR} we know that for any $t \in [t_0, t'_0]$ the support of the functions $\ov\eta_\delta ( e^{-f_s} )$ is contained in a compact subset of $\RR_t$.
Therefore, for any $s_1, s_2 \in I^*$ we have
\begin{multline} \label{eq_fsetaefs}
 \int_{t_0}^{t'_0} \int_{\RR_t} f_s \ov\eta_\delta (e^{-f_s}) \eta_r \, dg_t dt \bigg|_{s = s_1}^{s=s_2}
= \int_{s_1}^{s_2} \int_{t_0}^{t'_0} \int_{\RR_t} \nabla y_l^\infty \cdot \nabla \big(  f_s \ov\eta_\delta (e^{-f_s} )\big) \eta_r \, dg_t dt ds \\
= - \int_{s_1}^{s_2} \int_{t_0}^{t'_0} \int_{\RR_t}  f_s \ov\eta_\delta (e^{-f_s} ) \nabla y^\infty_l \cdot \nabla \eta_r \, dg_t dt ds 
\end{multline}
Since $f_s \ov\eta_\delta (e^{-f_s} )$ is uniformly bounded, the right-hand side of (\ref{eq_fsetaefs}) goes to zero as $r \to 0$ and we obtain that
\[  \int_{t_0}^{t'_0} \int_{\RR_t} f_s \ov\eta_\delta (e^{-f_s}) \, dg_t dt \bigg|_{s = s_1}^{s=s_2} = 0. \]
Letting $\delta \to 0$ and using (\ref{eq_general_fs_bounds}) implies that
\[  \int_{t_0}^{t'_0} \int_{\RR_t} f_s e^{-f_s}  dg_t dt \bigg|_{s = s_1}^{s=s_2} = 0, \]
which shows (\ref{eq_need_ts_int_ddsNN}).

If $\Delta = 4$, then by Theorem~\ref{Thm_SS_dimension_bound_limit}(c), the time-slices $(\XX_t, d_t)$ are the metric completions of $(\RR_t, g_t)$.
Therefore, $\XX_t \cong \XX'_t \times \IR^k$ for some family of complete metric spaces $(\XX'_t, d'_t)$.
Using this identification, we define for any $x' \in \XX'_t$ and $-T_\infty < s < t < 0$:
\[ \nu'_{x'; s} := (\proj_{\XX'_s} )_* \nu_{(x, \vec 0); s} . \]
It can be checked easily that $\XX'$ combined with the probability measures $\nu'_{x';s}$ form an $H_n$-con\-cen\-trat\-ed metric flow.
To see that $\XX_{<0} \cong \XX' \times \IR^k$ as metric flows, it remains to check that for any $(x', \vec b) \in \XX'_t \times \IR^k$
\begin{equation} \label{eq_nu_is_nu_otimes_nu}
 \nu_{(x', \vec b); s} = \nu'_{x';s} \otimes \nu^{\IR^k}_{(\vec b, t);s}, 
\end{equation}
where
\[ \nu^{\IR^k}_{(\vec b, t);s} = (4\pi (t-s))^{-n/2} \exp \Big( -\frac{ |\vec b - \vec x|^2}{4 (t-s)} \Big) d\vec x \]
is the standard Gaussian measure on $\IR^k$.
To see this, note that the heat flows on $\XX$ defined with respect to both families of measures in (\ref{eq_nu_is_nu_otimes_nu}) satisfy the heat equation on $\RR$.
So by Theorem~\ref{Thm_SS_dimension_bound_limit}\ref{Thm_SS_dimension_bound_limit_d} both families of measures induce the same heat flows and therefore they must be the same.
\end{proof}
\bigskip

\subsection{Static limits}
The next result addresses the case in which the flows $(M_i, (g_{i,t})_{t \in I_i})$ are $( \eps_i, r)$-static for some $\eps_i \to 0$.
In this case we expect a static limit.

\begin{Theorem} \label{Thm_limit_from_static}
Suppose that $I_\infty = (-\infty, 0]$ and that $(x_i, 0)$ is $( \eps_i, r)$-static for $\eps_i \to 0$ and fixed $r > 0$, or more generally, that $I_\infty = (-T_\infty, 0]$ and
\begin{equation} \label{eq_intint_Ric_to_0}
 \int_{T'_1}^{T'_2} \int_{M_i} |{\Ric}|^2 d\nu_{x_i,0;t} dt \to 0, \qquad \text{for all} \quad [T'_1,T'_2] \subset (-T_\infty,0). 
\end{equation}
Then on $\RR$ we have
\[ \Ric = 0 \]
Moreover, the heat kernel $K$ on $\RR$ satisfies the following infinitesimal translation property
\begin{equation} \label{eq_K_dt_invariant}
 \partial_{\tf, x_1} K(x_1; x_0)  + \partial_{\tf, x_0} K(x_1; x_0)  = 0, 
\end{equation}
where $\partial_{\tf, x_1}$, $\partial_{\tf, x_0}$ denote the time-derivatives with respect to the first and second parameter, respectively.

If $\Delta \geq 3$, then there is a Ricci flat Riemannian manifold $(M_\infty, g_\infty)$ and an identification
\begin{equation} \label{eq_identification_RR}
 \RR = M_\infty \times (-T_\infty,0)
\end{equation}
such that $g_t = g_\infty$ for all $t \in (-T_\infty,0)$ and such that $\partial_{\tf}$ corresponds to the unit vector field on the second factor.
Moreover, if $\Delta \geq 4$, then (\ref{eq_identification_RR}) can be extended to an identification of the form $\XX_{<0} = X_\infty \times (-T_\infty,0)$ such that $d_t = d_\infty$ for all $t \in (-T_\infty,0)$, where $(X_\infty, d_\infty)$ denotes the metric completion of $(M_\infty, d_{g_\infty})$.
\end{Theorem}

\begin{proof}[Proof of Theorem~\ref{Thm_limit_from_static}.] The vanishing of $\Ric$ follows directly from the static condition or from (\ref{eq_intint_Ric_to_0}).
We will now show (\ref{eq_K_dt_invariant}).
As in the proof of Theorem~\ref{Thm_limit_from_strong_splitting}, fix some $-T_\infty < t_0 < 0$ and a compactly supported $u^\infty \in C^\infty_c (\RR_{t_0})$.
By a similar construction as in this proof and after passing to a subsequence, we can find smooth functions $u_h^{\prime\prime, i} \in C^\infty (M_i \times [h, 0])$ that smoothly depend on a parameter $h \in (t_0 - \alpha, t_0 + \alpha)$ such that
\[ \square u_h^{\prime\prime, i} = 0, \qquad \partial_t u_h^{\prime\prime, i} (\cdot, h) + \partial_h u_h^{\prime\prime, i} (\cdot, h) = 0 \]
and such that we have local smooth convergence $u^{\prime\prime, i} \to u^{\prime\prime, \infty}$, where $u^{\prime\prime, \infty}_{t_0} |_{\RR_{t_0}} = u^\infty$ and
\begin{equation} \label{eq_uhpp_uh_K}
 u^{\prime\prime, \infty}_h  = \int_{\RR_{h}} u_h^{\prime\prime, \infty} (x) K (\cdot; x) dg_{h}(x). 
\end{equation}

We may now compute that
\[ \square \partial_h u^{\prime\prime, i}_h = \frac{d}{dh} \square  u^{\prime\prime, i}_h = 0, \qquad
 \square ( \partial_t u^{\prime\prime, i}_h ) = 2 \Ric \cdot \nabla^2 u^{\prime\prime, i}_h. \]
Therefore, as in the proof of Theorem~\ref{Thm_limit_from_strong_splitting}, we have for any fixed $h \in (t_0 - \alpha, t_0 + \alpha)$ and $t_1 \in (h, 0]$ and $\nu^i := \nu_{x_i, 0}$
\begin{align*}
  \int_{M_i}  \big|  \partial_h u_h^{\prime\prime, i} + \partial_t u_h^{\prime\prime, i} \big| d\nu_{t_1}  &\leq  2  \int_{h}^{t_1} \int_{M_i}   |{\Ric} | \, | \nabla^2 u_h^{\prime\prime, i} | d\nu^i_t dt \\
  &\leq  \bigg(  2\int_{h}^{t_1} \int_{M_i}  | { \Ric} |^2 d\nu^i_t dt \bigg)^{1/2}
  \bigg(  2\int_{h}^{t_1} \int_{M_i}  | \nabla^2 u^{\prime\prime, i}_h |^2 d\nu^i_t dt \bigg)^{1/2} \\
  &\leq \bigg(  2\int_{h}^{t_1} \int_{M_i}  |{\Ric} |^2 d\nu^i_t dt \bigg)^{1/2}
  \bigg(   \int_{M_i}  | \nabla u^{\prime\prime, i}_h |^2 d\nu^i_{ h} \bigg)^{1/2} \to 0.
\end{align*}
So
\[  \partial_h u_h^{\prime\prime, \infty} + \partial_t u_h^{\prime\prime, \infty} = 0. \]
Combining this with (\ref{eq_uhpp_uh_K}) yields that for any $x_1 \in \RR_{t_1}$ (recall that $u^{\prime\prime, \infty}_{t_0} |_{\RR_{t_0}} = u^\infty$ is compactly supported in $\RR_{t_0}$) 
\begin{align*}
 \int_{\RR_{t_0}} u^{\infty} (x) \partial_{\mathfrak{t}, x_1} K (x_1 ; x)  dg_{t_0} (x) 
&= - \int_{\RR_{t_0}}  \bigg( \frac{\partial}{\partial h} \bigg|_{h = t_0} u^{\prime\prime, \infty}_{t_0} (x) \bigg) K (x_1 ; x) dg_{t_0} (x) \\
&=  \int_{\RR_{t_0}} \partial_{\tf}  u^{\prime\prime, \infty}_{t_0} (x) K (x_1 ; x) dg_{t_0} (x) \\
&= \int_{\RR_{t_0}}  \triangle u^{ \infty} (x) K (x_1 ; x) dg_{t_0} (x) \\
&= \int_{\RR_{t_0}}   u^{\infty} (x) \triangle_x K (x_1 ; x) dg_{t_0} (x) \\
&= -\int_{\RR_{t_0}}   u^{\infty} (x) \partial_{\mathfrak{t}, x} K (x_1 ; x) dg_{t_0} (x). 
\end{align*}
Since $u^\infty$ was arbitrary, this implies (\ref{eq_K_dt_invariant}).

Next, we assume that $\Delta \geq 3$ and prove that the trajectories of $\partial_{\mathfrak{t}}$ are complete on $I_\infty \setminus \{ 0 \}$.
This will be a consequence of the following claim:

\begin{Claim} \label{Cl_traj_dt_exist_unti_t_star}
Any maximal trajectory $\gamma : I' \to \RR$ of $\partial_\tf$ with $\tf (\gamma(t)) = t$ is defined on an interval of the form $I' = I_\infty \cap (-\infty, t^*)$ for $t^* \leq 0$.
If $t^* < 0$, then for any $x'' \in \XX_{t''}$ with $t'' > t^*$ we have $\lim_{t \nearrow t^*} K (x'', \gamma(t)) = 0$.
\end{Claim}

In fact, if Claim~\ref{Cl_traj_dt_exist_unti_t_star} holds for some trajectory $\gamma$, then we must have $t^* = 0$, because if $\gamma_0 : [t'', t'' + a_0]  \to \RR$ denotes another trajectory of $\partial_\tf$ with $\tf (\gamma_0 (t)) = t$ and $[t'', t''+a_0] \subset ( t^*, 0)$, then by (\ref{eq_K_dt_invariant}) we have for any $t \in I'$ and $a \in [0, a_0]$
\[ K (\gamma_0 (t'' + a), \gamma(t)) = K (\gamma_0 (t''), \gamma ( t - a)). \]
So $K (\gamma_0 (t''), \gamma (t^*-a)) = 0$ for all $a \in (0, a_0]$, which is impossible by the strong maximum principle and the fact that $d\nu_{\gamma_0(t''); t^*-a}$ has full support.

\begin{proof}[Proof of Claim~\ref{Cl_traj_dt_exist_unti_t_star}.]
By a similar argument as in the proof of Theorem~\ref{Thm_limit_from_strong_splitting}, the claim can be reduced to showing that $\NN_{\gamma(t)} (\tau)$ is constant along wordlines $\gamma$, i.e. trajectories of $\partial_{\tf}$.
See, in particular, Theorem~\ref{Thm_NN_in_limit}\ref{Thm_NN_in_limit_b} and Lemma~\ref{Lem_tdrrm}\ref{Lem_tdrrm_g}.
So fix some $[t_1, t_2] \subset (-T_\infty, 0)$ and $\tau > 0$ with $[t_1 + \tau, t_2 +\tau] \subset (-T_\infty, 0)$ and consider a trajectory $\gamma : [t_1+\tau, t_2+\tau] \to \RR$ of $\partial_{\mathfrak{t}}$ with $\mathfrak{t} (\gamma(t)) = t$.
Write
\[ K (\gamma( \mathfrak{t} (x) + \tau ) ; x) =: (4\pi \tau)^{-n/2} e^{-f(x)} \]
for some $f \in C^\infty ( \RR_{[t_1, t_2]} )$.
So
\begin{equation} \label{eq_partial_tf_f_0}
 \partial_{\mathfrak{t}} f = 0 
\end{equation}
and $f \geq - C$ by Proposition~\ref{Prop_L_infty_HK_bound}.
Our goal is to show that
\begin{equation} \label{eq_need_show_dt_NN_0}
 \int_{\RR_{t_1}} f_{t_1} e^{-f_{t_1}} dg_{t_1} = \int_{\RR_{t_2}} f_{t_2} e^{-f_{t_2}} dg_{t_2}. 
\end{equation}

Denote by $S_1 \subset \RR_{t_1}$ the set of points that  survive until time $t_2$.
Then $S_2 := S_1 (t_2) \subset \RR_{t_2}$ is the set of points the survive until time $t_1$.
It remains to show that $\RR_{t_1} \setminus S_1$ and $\RR_{t_2} \setminus S_2$ have measure zero.
Combining this with (\ref{eq_partial_tf_f_0}) will imply (\ref{eq_need_show_dt_NN_0}) using Fubini's Theorem.

Observe that by (\ref{eq_partial_tf_f_0})
\begin{multline} \label{eq_int_t1_int_t2}
 \int_{\RR_{t_1} \setminus S_1} e^{-f_{t_1}} dg_{t_1} 
= \int_{\RR_{t_1}} e^{-f_{t_1}} dg_{t_1} - \int_{S_1} e^{-f_{t_1}} dg_{t_1}
= (4\pi \tau)^{n/2} - \int_{S_2} e^{-f_{t_2}} dg_{t_2} \\
= \int_{\RR_{t_2}} e^{-f_{t_2}} dg_{t_2} - \int_{S_2} e^{-f_{t_2}} dg_{t_2}
=  \int_{\RR_{t_2} \setminus S_2} e^{-f_{t_2}} dg_{t_2} .   
\end{multline}
So it suffices to show that the last integral vanishes.
To see this, let $\eta_r$ be a cutoff function as in Lemma~\ref{Lem_eta_r_on_RR}.
Let $F < \infty$ be a constant whose value we will determine later.
By Lemmas~\ref{Lem_limit_HK_bound}, \ref{Lem_eta_r_on_RR} we find that for sufficiently large $A < \infty$ we have
\[ \{ f \leq F \} \subset P^* (x_\infty; A, t_2). \]
So by Lemma~\ref{Lem_eta_r_on_RR} there is a constant $C^*  < \infty$, which may depend on $F, A$, such that
\[ \int_{t_1}^{t_2} \int_{\{ f \leq F \} \cap \RR_{t}} |\partial_{\tf} \eta_r| dg_t dt \leq C^* r^{.5}. \]
Therefore, since every trajectory of $-\partial_{\tf}$ starting from $\RR_{t_2} \setminus S_2$ leaves $\{ \eta_r > 0 \}$ before it ceases to exist, we have
\begin{multline*}
  \int_{\{ f \leq F \} \cap (\RR_{t_2} \setminus S_2)} e^{-f_{t_2}} \, dg_{t_2} 
= 
 \lim_{r \to 0}  \int_{\{ f \leq F \} \cap (\RR_{t_2} \setminus S_2)} e^{-f_{t_2}} \eta_r \, dg_{t_2} \\
=  \lim_{r \to 0}  \int_{t_1}^{t_2} \int_{\{ f \leq F \} \cap \RR_{t}} \partial_{\tf} ( e^{-f} \eta_r ) dg_t dt 
=    \lim_{r \to 0}  \int_{t_1}^{t_2} \int_{\{ f \leq F \} \cap \RR_{t}}  e^{-f} \partial_{\tf} \eta_r \, dg_t dt 
= 0. 
\end{multline*}
Letting $F \to \infty$ implies
\[ \int_{\RR_{t_2} \setminus S_2} e^{-f_{t_2}} dg^{\RR}_{t_2} = 0. \]
So by combining this with (\ref{eq_int_t1_int_t2}), we obtain that $\RR_{t_1} \setminus S_1$ and $\RR_{t_2} \setminus S_2$ have measure zero, as desired.
\end{proof}

Finally, if $\Delta \geq 4$, then the remaining assertions of the theorem are a consequence of Theorem~\ref{Thm_SS_dimension_bound_limit}\ref{Thm_SS_dimension_bound_limit_c}, \ref{Thm_SS_dimension_bound_limit_d}.
\end{proof}
\bigskip

\subsection{Metric soliton limits}
The next result concerns the case in which the flows $(M_i, (g_{i,t})_{t \in I_i})$ are $(\eps_i, r)$-selfsimilar for some $\eps_i \to 0$.
In this case we expect a limit that is a metric soliton.

\begin{Theorem} \label{Thm_limit_from_selfsimilar}
Suppose that $I_\infty = (-\infty, 0]$ and that $(x_i, 0)$ is $(\eps_i, r)$-selfsimilar for $\eps_i \to 0$ and fixed $r > 0$ or, more generally, that $I_\infty = (-T_\infty, 0]$ and that $\NN_{x_\infty} (\tau)$ is constant in $\tau \in (0, T_\infty)$.
Then if we write $d\nu_{x_\infty} = (4\pi \tau)^{-n/2} e^{-f}dg$ on $\RR$, we have
\begin{align}
\Ric + \nabla^2 f - \frac1{2\tau} g &= 0, \label{eq_soliton_eq_limit_1} \\
 \tau (2 \triangle f - |\nabla f|^2 + R) + f - n &\equiv W := \NN_{x_\infty} (r^2).  \label{eq_soliton_eq_limit_2}
\end{align}
Moreover, the heat kernel $K$ on $\RR$ satisfies the following infinitesimal translation property
\begin{equation} \label{eq_K_dt_nab_f_invariance}
  (\tau \partial_{\tf}  -  \tau \nabla f )_{x_1} K(x_1; x_0) + (\tau \partial_{\tf} -  \tau \nabla f )_{x_0} K(x_1; x_0)  =  \frac{n}2 K(x_1; x_0).  
\end{equation}

If $\Delta \geq 3$, then the vector field $\partial_{\mathfrak{t}} -   \nabla f$ is complete on $(-T_\infty, 0)$.
If $\Delta \geq 4$, then $(\XX, \lb (\nu_{x_\infty; t})_{t \in (-T_\infty,0)})$ is a metric soliton.
There is an Riemannian manifold $(M_\infty, g_\infty)$ and an identification
\begin{equation} \label{eq_RR_identification_M_infty_sol}
\RR = M_\infty \times (-T_\infty, 0)
\end{equation}
such that $g_t = |t|g_\infty$ for all $t\in (-T_\infty, 0)$ and such that $\partial_{\mathfrak{t}} -   \nabla f$ corresponds to the standard vector field on the second factor.
In addition, (\ref{eq_RR_identification_M_infty_sol}) can be extended to an identification of the form $\XX_{<0} = X_\infty \times (-T_\infty,0)$ such that $d_t = |t|^{1/2} d_\infty$ for all $t\in (-T_\infty, 0)$, where $(X_\infty,d_\infty)$ denotes the metric completion of $(M_\infty, d_{g_\infty})$.
\end{Theorem}

\begin{proof}
In the following, $C$ will denote a generic constant, which may depend on various geometric data, but not on the index $i$.
Write $d\nu^i := d\nu_{x_i, 0} = (4\pi \tau)^{-n/2} e^{-f_i} dg_i$.

In the case in which $(x_i, 0)$ are $(\eps_i, r)$-selfsimilar for some $\eps_i \to 0$ and some uniform $r >0$, the soliton identities (\ref{eq_soliton_eq_limit_1}), (\ref{eq_soliton_eq_limit_2}) follow by passing the $(\eps_i, r)$-self\-sim\-i\-lar condition to the limit.
In the second case, in which $I_\infty = (-T_\infty, 0]$,  there is a $W \leq 0$ such that  for any $-T \in I_\infty$, 
\[ \lim_{i \to \infty} \NN_{x_i, 0} (T) = W. \]
As in the proof of Proposition~\ref{Prop_NN_almost_constant_selfsimilar}, this implies that for any $-T \in I_\infty$
\[ \lim_{i \to \infty} \WW_{x_i, 0} [g_{i,-T}, f_i (\cdot , -T), T] = W. \]
Therefore, the almost selfsimilar identities (\ref{eq_almost_self_similar_1}), (\ref{eq_almost_self_similar_2}) follow from Proposition~\ref{Prop_NN_basic_properties} and Lem\-ma~\ref{Lem_integral_soliton_id}.
Passing these properties to the limit implies again (\ref{eq_soliton_eq_limit_1}), (\ref{eq_soliton_eq_limit_2}).
We also remark that in the second case, we obtain as in the proof of Proposition~\ref{Prop_improved_L2}, that for every $-T \in I_\infty$ and large $i$
\begin{equation} \label{eq_nab_f_4_bound_special}
 \int_{-T}^0 \int_{M_i}( |\nabla^2 f_i|^2 + |\nabla f_i|^4 + R^2 + f^2) d\nu^i_t dt \leq C(T). 
\end{equation}

Next we show (\ref{eq_K_dt_nab_f_invariance}).
We follow the lines of the proof of Theorem~\ref{Thm_limit_from_static}.
Fix some $t_0 < 0$.
As before, after passing to a subsequence, we can find smooth functions $u_h^{\prime\prime, i} \in C^\infty (M_i \times [h, 0])$ that smoothly depend on a parameter $h \in (t_0 - \alpha, t_0 + \alpha)$ such that
\[ \square u_h^{\prime\prime, i} = 0, \qquad (\tau \partial_t - \tau \nabla f_i) u_h^{\prime\prime, i} (\cdot, h) +(\tau \partial_h u_h^{\prime\prime, i}) (\cdot, h) = 0 \]
and such that we have local smooth convergence $u^{\prime\prime, i} \to u^{\prime\prime, \infty}$, where $u^{\prime\prime, \infty}_{t_0} |_{\RR_{t_0}} = u^\infty$ and
\begin{equation} \label{eq_uhpp_uh_K_soliton}
 u^{\prime\prime, \infty}_h  = \int_{\RR_{h}} u_h^{\prime\prime, \infty} (x) K (\cdot; x) dg_{h}(x). 
\end{equation}
We obtain
\[ \square (\tau (h) \partial_h u^{\prime\prime, i}) = 0 \]
and
\begin{align*}
 \square & ( ( \tau \partial_t - \tau \nabla f_i) u^{\prime\prime, i})
= - \partial_t u^{\prime\prime, i} + 2 \tau \Ric \cdot \nabla^2 u^{\prime\prime, i}
+ \nabla f_i \cdot \nabla u^{\prime\prime, i} - \tau \square ( \nabla f_i \cdot \nabla u^{\prime\prime, i} ) \\
&= - \triangle u^{\prime\prime, i} + 2 \tau \Ric \cdot \nabla^2 u^{\prime\prime, i}
+ \nabla f_i \cdot \nabla u^{\prime\prime, i} - \tau   \nabla \square f_i \cdot \nabla u^{\prime\prime, i}   - \tau   \nabla f_i \cdot \nabla \square u^{\prime\prime, i}  + 2 \tau   \nabla^2 f_i \cdot \nabla^2  u^{\prime\prime, i} \\
&= (2\tau \Ric + 2\tau\nabla^2 f_i - g_i) \cdot \nabla^2 u^{\prime\prime, i} 
 -    \nabla \square (\tau f_i) \cdot \nabla u^{\prime\prime, i}    .
\end{align*}
So as in the proof of Theorem~\ref{Thm_limit_from_static} we have for $t_1 \in [h, 0]$
\begin{multline} \label{eq_diff_i_integral_soliton}
  \int_{M_i}  \big| \tau(h) \partial_h u_h^{\prime\prime, i} + (\tau \partial_t - \tau \nabla f_i) u_h^{\prime\prime, i} \big| d\nu_{ t_1}  \leq  2   \int_{h}^{t_1} \int_{M_i} \tau  \Big|{\Ric} + \nabla^2 f_i - \frac1{2\tau} g_i \Big| \, | \nabla^2 u_h^{\prime\prime, i} | \, d\nu^i_t dt  \\
  +  \bigg|   \int_{h}^{t_1} \int_{M_i}   \nabla \square (\tau f_i)\cdot \nabla u_h^{\prime\prime, i}  \, d\nu^i_t dt \bigg| .
\end{multline}
The first term on the right-hand side can be bounded similarly as before: 
\begin{align}
2 \int_{h}^{t_1} \int_{M_i} & \tau  \Big|{\Ric} + \nabla^2 f_i - \frac1{2\tau} g_i \Big| \, | \nabla^2 u_h^{\prime\prime, i} | \, d\nu^i_t dt \notag \\
&\leq \bigg( 2 \int_{h}^{t_1} \int_{M_i} \tau^2  \Big|{\Ric} + \nabla^2 f_i - \frac1{2\tau} g_i \Big|^2 \, d\nu^i_t dt \bigg)^{1/2} \bigg( \int_{h}^{t_1} \int_{M_i} 2 | \nabla^2 u_h^{\prime\prime, i} |^2 \, d\nu^i_t dt \bigg)^{1/2} \notag \\
&\leq \bigg( 2 \int_{h}^{t_1} \int_{M_i} \tau^2  \Big|{\Ric} + \nabla^2 f_i - \frac1{2\tau} g_i \Big|^2 \, d\nu^i_t dt \bigg)^{1/2} \bigg(  2 \int_{M_i} | \nabla u_h^{\prime\prime, i} |^2\,  d\nu^i_{ h} \bigg)^{1/2}  \to 0. \label{eq_first_soliton_diff_to_0}
\end{align}

The next claim addresses the second term of (\ref{eq_diff_i_integral_soliton}).

\begin{Claim} 
As $i \to \infty$ we have
\[ \bigg|   \int_{h}^{t_1} \int_{M_i}   \nabla \square (\tau f_i)\cdot \nabla u_h^{\prime\prime, i}  \, d\nu^i_t dt \bigg| \to 0. \]
\end{Claim}

\begin{proof}
Using integration by parts, we obtain
\begin{multline} \label{eq_nab_square_split}
 \int_{h}^{t_1} \int_{M_i}   \nabla \square (\tau f_i)\cdot \nabla u_h^{\prime\prime, i}  d\nu^i_t dt 
=    \int_{h}^{t_1} \int_{M_i}   \Big( \square (\tau f_i) + \frac{n}2 + W \Big) \nabla u_h^{\prime\prime, i} \cdot \nabla f_i \, d\nu^i_t dt \\
-  \int_{h}^{t_1} \int_{M_i}   \Big( \square (\tau f_i) + \frac{n}2 + W \Big) \, ( \triangle u_h^{\prime\prime, i} ) \, d\nu^i_t dt . 
\end{multline}
In the following we will show that both terms on the right-hand side of (\ref{eq_nab_square_split}) converge to $0$ as $i \to \infty$.
Throughout the proof we will use the fact that
\[ \square |\nabla u_h^{\prime\prime, i}|^2 = - 2 |\nabla^2 u_h^{\prime\prime, i}|^2 \leq 0. \]
This implies, in particular, that we have a uniform bound of the form $|\nabla  u_h^{\prime\prime, i}| \leq C < \infty$.

For the convergence of the first term in (\ref{eq_nab_square_split}), observe first that by Proposition~\ref{Prop_improved_L2} or (\ref{eq_nab_f_4_bound_special})
\begin{multline*}
 \int_{h}^{t_1} \int_{M_i}   \Big| \square (\tau f_i) + \frac{n}2 + W \Big| \,  | \nabla f_i |^2 \, d\nu^i_t dt \\
\leq \bigg( \int_{h}^{t_1} \int_{M_i}   \Big( \square (\tau f_i) + \frac{n}2 + W \Big)^2 \,  d\nu^i_t dt \bigg)^{1/2} \bigg( \int_{h}^{t_1} \int_{M_i}    | \nabla f_i |^4 \, d\nu^i_t dt \bigg)^{1/2}
\leq C. 
\end{multline*}
So since for any $b > 0$ we have
\begin{multline*}
\bigg| \int_{h}^{t_1} \int_{M_i}   \Big( \square (\tau f_i) + \frac{n}2 + W \Big) \nabla u_h^{\prime\prime, i} \cdot \nabla f_i \, d\nu^i_t dt \bigg|
\leq C\int_{h}^{t_1} \int_{M_i}   \Big| \square (\tau f_i) + \frac{n}2 + W \Big| \, |\nabla f_i| \, d\nu^i_t dt \\
\leq C b \int_{h}^{t_1} \int_{M_i}   \Big| \square (\tau f_i) + \frac{n}2 + W \Big| \,  | \nabla f_i |^2 \, d\nu^i_t dt + C b^{-1} \int_{h}^{t_1} \int_{M_i}   \Big| \square (\tau f_i) + \frac{n}2 + W \Big| \,   d\nu^i_t dt 
\end{multline*}
and since the second term on the right-hand side goes to zero, we obtain that the first term in (\ref{eq_nab_square_split}) converges to $0$ as $i \to \infty$.

For the second term in (\ref{eq_nab_square_split}) define similarly as in the proof of Proposition~\ref{Prop_improved_L2}
\[ v_i := (4\pi \tau)^{-n/2} e^{-f_i}, \qquad w_i := \tau (2 \triangle f_i - |\nabla f_i|^2 + R) + f - n = - \square (\tau f_i) - \frac{n}2. \]
and recall that by \cite[Sec.~9]{Perelman1}
\[ \square^* (w_i v_i) \leq 0, \qquad w_i \leq 0. \]
Thus, using Proposition~\ref{Prop_improved_L2},
\begin{multline*}
2 \int_{h}^{t_1} \int_{M_i}   \Big| \square (\tau f_i) + \frac{n}2 + W \Big| \, | \nabla^2 u_h^{\prime\prime, i} |^2 d\nu^i_t dt 
\leq \int_{h}^{t_1} \int_{M_i} (w_i - |W|) v_i  \, \square | \nabla u^{\prime\prime, i}_h |^2 dg_{i,t} dt \\
=  \int_{h}^{t_1} \int_{M_i} \square^* ((w_i - |W|) v_i ) | \nabla u^{\prime\prime, i}_h |^2 dg_{i,t} dt
+ \int_{M_i}  ((w_i - |W|) v_i ) | \nabla  u^{\prime\prime, i}_h |^2 dg_{i,t} \bigg|_{t=h}^{t=t_1}  \\
\leq C \int_{M_i}  (|w_i |+ |W|) d\nu^i_{t_1}  +  C \int_{M_i}  (|w_i |+ |W|) d\nu^i_{h}   \leq C.
\end{multline*}
It follows that for any $b > 0$
\begin{multline*}
 \limsup_{i \to \infty} \bigg| \int_{h}^{t_1} \int_{M_i}   \Big( \square (\tau f_i) + \frac{n}2 + W \Big) \, ( \triangle u_h^{\prime\prime, i} ) d\nu^i_t dt \bigg| \\
\leq  C \limsup_{i \to \infty} \int_{h}^{t_1} \int_{M_i}   \Big| \square (\tau f_i) + \frac{n}2+ W \Big| \big( b^{-1} +  b  | \nabla^2 u_h^{\prime\prime, i} |^2  \big) d\nu^i_t dt \leq C b.
\end{multline*}
Letting $b \to 0$ finishes the proof of the claim.
\end{proof}

Combining (\ref{eq_diff_i_integral_soliton}), (\ref{eq_first_soliton_diff_to_0}) and the claim implies that the left-hand side of (\ref{eq_diff_i_integral_soliton}) converges to $0$ as $i \to \infty$.
So, as in the proof of  Theorem~\ref{Thm_limit_from_static}, we obtain that
\[ \tau (h) \partial_h u_h^{\prime\prime, \infty} + (\tau \partial_t - \tau  \nabla f) u_h^{\prime\prime, \infty} = 0. \]
Combining this with (\ref{eq_uhpp_uh_K_soliton}) implies that for any $x_1 \in \RR_{t_1}$ (recall that $u^{\prime\prime, \infty}_{t_0} |_{\RR_{t_0}} = u^\infty$ is compactly supported in $\RR_{t_0}$)
\begin{align*}
 \int_{\RR_{t_0}} & u^{\infty} (x)  \big(\tau (x_1) \partial_{\mathfrak{t}, x_1} - (\tau \nabla f)_{x_1} \big) K (x_1 ; x) \, dg_{t_0} (x) \\
 &= - \tau(t_0)\int_{\RR_{t_0}}  \bigg( \frac{\partial}{\partial h} \bigg|_{h = t_0} u^{\prime\prime, \infty}_{h} (x) \bigg) K (x_1 ; x) \, dg_{t_0} (x) \displaybreak[1] \\
&=  \tau(t_0) \int_{\RR_{t_0}} \big( (  \partial_{\tf} -  \nabla f )_x u^{\prime\prime, \infty}_{t_0} (x) \big) K (x_1 ; x) \, dg_{t_0} (x) \displaybreak[1] \\
&= \tau(t_0) \int_{\RR_{t_0}} ( \triangle u^{ \infty} -  \nabla f \cdot \nabla u^\infty) (x) K (x_1 ; x) \, dg_{t_0} (x) \displaybreak[1] \\
&= \tau(t_0) \int_{\RR_{t_0}}   u^{\infty} (x) \big( \triangle_x K (x_1 ; x) + \triangle f \, K (x_1 ; x) + \nabla f \cdot \nabla_x K (x_1 ; x)  \big)  dg_{t_0} (x) \displaybreak[1] \\
&= \tau(t_0) \int_{\RR_{t_0}}   u^{\infty} (x) \Big( - \partial_{\tf, x} K (x_1 ; x) + \frac{n}{2\tau(t_0)} K (x_1 ; x) + \nabla f \cdot \nabla_x K (x_1 ; x)  \Big)  dg_{t_0} (x) \displaybreak[1] \\
&=  - \int_{\RR_{t_0}}   u^{\infty} (x) \Big(  \big( \tau \partial_{\mathfrak{t}} - \tau \nabla f \big)_x K (x_1 ; x) - \frac{n}2 K (x_1; x) \Big) dg_{t_0} (x) . 
\end{align*}
Since $u^\infty$ was arbitrary, this implies (\ref{eq_K_dt_nab_f_invariance}).

Next, suppose that $\Delta \geq 3$.
Then, again as in the proof of  Theorem~\ref{Thm_limit_from_static}, we obtain that for any maximal trajectory $\gamma$ of $\tau \partial_t -\tau \nabla f_\infty$ and $a > 0$ we have
\[ \frac{d}{ds} \NN_{\gamma(s)} ( a \tf (\gamma(s))) = 0. \]
As before, using Theorem~\ref{Thm_NN_in_limit}, we obtain that $\gamma$ is defined on all of $\IR$.

Lastly, suppose that $\Delta \geq 4$.
Since the flow of the vector field $\tau \partial_{\tf} -\tau \nabla f_\infty$ is a dilation on $\RR$, by Theorem~\ref{Thm_SS_dimension_bound_limit}\ref{Thm_SS_dimension_bound_limit_c}, \ref{Thm_SS_dimension_bound_limit_d} it extends to a dilation on $\XX$.
Therefore, $(\XX, (\nu_{x_\infty;t})_{t \in (-T_\infty, 0)})$ is a metric soliton and the remaining assertions of the theorem follow.
\end{proof}
\bigskip

\subsection{Limits that are Ricci flat cones}
Lastly, we consider the combined case in which the flows $(M_i, (g_{i,t})_{t \in I_i})$ are at the same time strongly $(k, \eps_i, r)$-split, $(\eps_i, r)$-static and $(\eps_i, r)$-selfsimilar for some sequence $\eps_i \to 0$.
In this case, we expect a static limit $\XX_{<0}$ of the form $(C^{n-k} \times \IR^k )\times \IR_-$, where $(C^{n-k}, v)$ is a Ricci flat cone with vertex $v$.
In the case $k = n-\Delta$, we obtain that this cone is even smooth away from its vertex.
An important aspect of the following result is that if $k = n-\Delta$, then we can drop the lower bound on $\Delta$.
This will be important in subsequent sections, as this theorem will form the basis for improved $\eps$-regularity results.

\begin{Theorem} \label{Thm_limit_from_almost_cone}
Suppose that $(x_i, 0)$ are $(k, \eps_i, r)$-split, $(\eps_i, r)$-static and $(\eps_i, r)$-selfsimilar for some $\eps_i \to 0$ and fixed $r > 0$, $k \geq 0$.
Choose $(k, \eps_i, r)$-splitting maps $\vec y^{\, i}$ at $(x_i, 0)$.
Write $W_i := \NN_{x_i, 0} (r^2)$ for all $i =1, 2 \ldots$, $W_\infty := \lim_{i\to \infty} W_i$ and set for $i =1, \ldots, \infty$
\begin{equation} \label{eq_qi_def}
 q_i := 4 \tau ( f_i - W_i ) - \sum_{j=1}^k (y_j^i)^2. 
\end{equation}
Then, after passing to a subsequence, we have $\vec y^{\, i} \to \vec y^{\, \infty}$, $f_i \to f_\infty$, $q_i \to q_\infty \in C^\infty (\RR)$ locally smoothly on $\RR$ and $q_\infty$ satisfies (\ref{eq_qi_def}) for $i = \infty$ and
\begin{multline} \label{eq_q_infty_idenitities}
 \Ric \equiv 0, \qquad  q_\infty \geq 0, \qquad |\nabla q_\infty|^2 = 4q_\infty, \qquad \nabla q_\infty \cdot \nabla y_j^\infty = 0, \qquad \\
  \nabla^2 q_\infty = 2g -2\sum_{j=1}^k d y^\infty_j \otimes dy^\infty_j, \qquad \partial_\tf q_\infty = 0. 
\end{multline}

If $\Delta \geq 3$ or $k = n-\Delta$, then $\RR$ can be identified with a constant flow of the form $\RR = (M'_\infty \times \IR^k) \times \IR_-$ such that
\[ g = g'_\infty + \sum_{j=1}^k (dy^\infty_j)^2, \]
where $(M'_\infty, g'_\infty)$ is a Riemannian manifold with $\Ric_{g'_\infty} \equiv 0$ and of dimension $n-k$, the functions $y^\infty_1, \ldots, y^\infty_k$ are coordinates for the second factor and $\partial_{\tf}$ is induced by the last factor.
Moreover, $q_\infty$ only depends on the first factor and the subset $\{ q_\infty > 0 \} \subset M'_\infty$ is isometric to a Riemannian cone minus its vertex over a smooth link $(N^{n-k-1}, g^N)$, in such a way that the function  $q_\infty$ restricted to $\{ q_\infty > 0 \} \subset M'_\infty$  equals the square of the radial coordinate function.
More specifically, we can write $\{ q_\infty > 0 \} = \IR_+ \times N$ with
\[ g_\infty = (d\sqrt{q_\infty})^2 + q_\infty g^N. \]
We have
\begin{equation} \label{eq_vol_link}
 \vol (N, g^N) = e^{W_\infty} \vol (S^{n-k-1}, g^{S^{n-k-1}} ). 
\end{equation}
In the case $k = n-\Delta$, the link $N$ is compact, so we have $\rrm \geq c(Y) \sqrt{q}$ on $\RR$.

Lastly, if $\Delta \geq 4$, then the flow $\XX$ is static and the time-slices of $\XX$ are isometric to $X'_\infty \times \IR^k$, where $(X'_\infty, d'_\infty)$ is a metric cone that is the metric completion of $(M'_\infty, g'_\infty)$.
\end{Theorem}

\begin{proof}[Proof of Theorem~\ref{Thm_limit_from_almost_cone}.]
As in the proof of Theorem~\ref{Thm_limit_from_strong_splitting}, assume that $r = 1$.
We will use the letter $r$ for other applications in the following.
By Theorems~\ref{Thm_limit_from_strong_splitting}, \ref{Thm_limit_from_static}, \ref{Thm_limit_from_selfsimilar}, we may pass to a subsequence and assume locally smooth convergence $f_i \to f_\infty$, $y^i_j \to y^\infty_j$ and $W_i \to W_\infty$.
This implies convergence $q_i \to q_\infty$ with
\[ q_\infty = 4 \tau (f_\infty - W_\infty) - \sum_{j=1}^k (y^\infty_j)^2. \]
The identities (\ref{eq_q_infty_idenitities}) are a direct consequence of Proposition~\ref{Prop_construction_almost_radial}.

We now establish the splitting of the Ricci flow spacetime
\begin{equation} \label{eq_RR_M_infty_R_k}
\RR = (M'_\infty \times \IR^k) \times \IR_- 
\end{equation}
for some Ricci-flat Riemannian manifold $(M'_\infty, g'_\infty)$ and the completeness of the vector field $\nabla q_\infty = 4 \tau \nabla f_\infty - 2 \sum_{j=1}^k y^\infty_j \nabla y^\infty_j$.

Consider first the case $\Delta \geq 3$.
In this case the splitting (\ref{eq_RR_M_infty_R_k}) follows from Theorems~\ref{Thm_limit_from_strong_splitting}, \ref{Thm_limit_from_static}.
It follows from (\ref{eq_q_infty_idenitities}) that $q_\infty$ only depends on the first factor.
If we restrict $q_\infty$ to $M'_\infty \times \{ \vec 0 \} \times \{ -1 \}$, then $q_\infty = 4 (f_\infty - W_\infty)$ and $\nabla q_\infty = 4\nabla f_\infty$.
So the completeness of $\nabla q_\infty$ follows from the completeness of $\tau  \partial_{\mathfrak{t}} - \tau  \nabla f$, see Theorem~\ref{Thm_limit_from_selfsimilar}.
If $\Delta \geq 4$, then the last statement of this theorem is a consequence of Theorem~\ref{Thm_SS_dimension_bound_limit}\ref{Thm_SS_dimension_bound_limit_c}.

Next, consider the case $k = n-\Delta$.

\begin{Claim}
If $k =n-\Delta$, then for any $x'_\infty \in \RR$ we have 
\begin{equation} \label{eq_tdrrm_q_bound}
 \tdrrm (x'_\infty) \geq c(Y) \sqrt{q_\infty(x'_\infty)}, \qquad  |B(x'_\infty, \tdrrm (x'_\infty)) \cap \RR_{\tf(x'_\infty)} | \geq c(Y) \tdrrm^{\,n} (x'_\infty). 
\end{equation}
\end{Claim}

\begin{proof}
Consider a point $x'_\infty \in \RR$ with $q_\infty (x'_\infty) > 0$ and choose points $(x'_i, t'_i) \in M_i \times I_i$ such that $(x'_i, t'_i)  \to x'_\infty$ within $\CF$.
By Lemma~\ref{Lem_nu_Pstar_bound} there is an $r < \infty$ such that $(x'_i, t'_i ) \in P^* (x_i, 0; r)$ for large $i$.
By the almost self-similarity condition and Proposition~\ref{Prop_NN_almost_constant_selfsimilar} we have $\NN_{x_i, 0} (r^2) \geq -  Y -1$ for large $i$.
So we may apply Proposition~\ref{Prop_extending_splitting_maps_almost_radial} for $Y$ replaced with $Y+1$ and $A =1$ and obtain the following:
If $\eps > 0$, $0 < \la < 1$ and $\beta \leq \ov\beta (Y, \eps)$ and if for some fixed $\tau' \in (0, (8n)^{-1} (\la r)^2)$ the positive part of $q_i$ satisfies:
\begin{equation} \label{eq_intint_bigger_la2}
 \frac1{\tau'} \int_{t'_i-2\tau'}^{t'_i-\tau'} \int_{M_i} (q_i)_+ \, d\nu^i_{x'_i, t'_i; t}dt \geq (\la r)^2 
\end{equation}
for infinitely many $i$, then for infinitely many $i$ the point $(x'_i, t')$ is strongly $(n+3-\Delta, \beta \la r, \eps)$-split.
Thus by Assumption~\ref{Aspt_working_ass} and Lemma~\ref{Lem_tdrrm} and choosing $\eps \leq \ov\eps (Y)$, we obtain that $\tdrrm(x'_\infty) \geq c(Y) \la r$ and the two-sided parabolic neighborhood around $x'_\infty$ of radius $c(Y, A) \la r$ is relatively compact in $\RR$.
By substituting $r$ for $\la r$ (note that $r$ has to be chosen large and $\la$ can be chosen small) and the locally smooth convergence $q_i \to q_\infty$, we therefore obtain the following statement:
If for some $\tau' \in (0, (8n)^{-1} r^2)$ we have
\begin{equation} \label{eq_intint_bigger_la2_new}
 \frac1{\tau'} \int_{\tf(x'_\infty)-2\tau'}^{\tf(x'_\infty)-\tau'} \int_{\RR_t} (q_\infty)_+ \, d\nu_{x'_\infty;t} dt > r^2 ,
\end{equation}
then $\rrm(x'_\infty) \geq c(Y) r$.
By the local $H_n$-concentration bounds of $\nu_{x'_\infty;t}$ (see also \cite[\SYNPropLocConc]{Bamler_RF_compactness}), the bound (\ref{eq_intint_bigger_la2_new}) holds for small $\tau'$ whenever $q_\infty (x'_\infty) > r^2$.
This shows the first bound in (\ref{eq_tdrrm_q_bound}).
The second bound follows using \cite[\HKThmNLC]{Bamler_HK_entropy_estimates}.
\end{proof}

The identities (\ref{eq_q_infty_idenitities}) combined with the Claim imply that the trajectories of the vector fields $\nabla y^\infty_j, \nabla q_\infty$ are complete.
Moreover, by Lemma~\ref{Lem_tdrrm} every trajectory $\gamma$ of $\partial_{\tf}$ with $\tf(\gamma(t)) = t$ is defined on an interval of the form $(-\infty, t^*)$ for some $t^* \leq 0$ with the property that $\lim_{t \nearrow t^*} f_\infty(t) = \infty$ if $t^* < 0$.
This, however, would contradict the identities (\ref{eq_q_infty_idenitities}) and the fact that $\partial_{\tf} y^\infty_j = 0$, which implies that any such trajectory is defined over $\IR_-$.
We therefore obtain the splitting (\ref{eq_RR_M_infty_R_k}) using Theorems~\ref{Thm_limit_from_strong_splitting}, \ref{Thm_limit_from_static} and (\ref{eq_q_infty_idenitities}).

Let us now discuss both cases $\Delta \geq 3$ and $k = n-\Delta$ at the same time.
The fact that $\{ q_\infty > 0 \} \subset M'_\infty$ is isometric to a Riemannian cone minus its vertex over some link $(N, g^N)$ is a direct consequence of the completeness of $\nabla q_\infty$ and (\ref{eq_q_infty_idenitities}).
If $k = n-\Delta$, then the Claim implies that $(N, g^N)$ is complete and has positive injectivity radius.
So if it has bounded volume, then it is compact.

It remains to establish (\ref{eq_vol_link}) in both cases $\Delta \geq 3$ and $k = n-\Delta$.
For this purpose, we identify $M'_\infty \times \IR^k$ with $\RR_{-1}$ and compute
\begin{align*}
 1 
&= \int_{M'_\infty} \int_{\IR^k} (4\pi)^{-n/2} e^{-f_\infty (x_1,\vec x_2)} d\vec x_2 dg'_\infty (x_1) \\
&= (4\pi)^{-n/2} \int_{M'_\infty} \int_{\IR^k} \exp \Big( - \frac14 \big( q_\infty (x_1,\vec x_2)  + |\vec x_2|^2 \big) - W_\infty \Big) d\vec x_2 dg'_\infty (x_1) \displaybreak[1] \\
&= (4\pi)^{-(n-k)/2} e^{-W_\infty}  \int_{M'_\infty}  \exp \Big( - \frac14 q_\infty   \Big) dg'_\infty \displaybreak[1] \\
&= (4\pi)^{-(n-k)/2} e^{-W_\infty} \vol (N, g^N) \int_0^\infty s^{(n-k-1)} e^{-s^2/4} ds \displaybreak[1] \\
&= e^{-W_\infty} \frac{\vol (N, g^N)}{\vol (S^{n-k-1}, g^{S^{n-k-1}})} \int_{\IR^{n-k}} (4\pi)^{-(n-k)/2} e^{-|\vec x|^2/4} d\vec x \\
&= e^{-W_\infty} \frac{\vol (N, g^N)}{\vol (S^{n-k-1}, g^{S^{n-k-1}})}.
\end{align*}
This finishes the proof
\end{proof}

\section{\texorpdfstring{An improved $\eps$-regularity theorem and codimension 2 of the singular set}{An improved {\textbackslash}eps-regularity theorem and codimension 2 of the singular set}} \label{sec_eps_reg_codim_2}
In this section we improve the $\eps$-regularity theorem of Proposition~\ref{Prop_eps_regularity_np2}.
More specifically, we reduce the strong splitting assumption by one dimension and remove the static assumption.

\begin{Proposition} \label{Prop_eps_reg_codim_2}
Let $Y< \infty$ and suppose that $0 <\eps \leq \ov\eps (Y)$. Then the following holds.

Let $(M, (g_t)_{t \in I})$ be a Ricci flow on a compact manifold, $r > 0$ and $(x_0,t_0) \in M \times I$ such that $\NN_{x_0,t_0} (r^2) \geq - Y$.
Suppose that $(x_0, t_0)$ is strongly $(n-1, \eps, r)$-split.
Then $\rrm (x_0, t_0) \geq \eps r$.
\end{Proposition}

This implies:

\begin{Corollary}
Assumption~\ref{Aspt_working_ass} holds for $\Delta = 2$.
\end{Corollary}

In addition, Proposition~\ref{Prop_eps_reg_codim_2} implies that in order to show that we can choose $\Delta = 4$, it suffices to consider points that are $(n-3, \eps, r)$-split and $(\eps, r)$-static.
This will be carried out in the next section.

\begin{proof}
Without loss of generality, we may assume that $t_0 = 0$ and $r = 1$.
Assume that the statement of the proposition was false for some fixed $Y < \infty$.
Then we can find a sequence of counterexamples $(M_i, (g_{i,t})_{t \in I_i})$, $x_i \in M$ for some sequence $\eps_i \to 0$.
Write $d\nu^i = d\nu_{x_i, 0} = (4\pi \tau_i)^{-n/2} e^{-f_i} dg_i$, $W_i := \NN_{x_i, 0} (1)$ and choose $(n-1,  \eps_i,1)$-splitting maps $\vec y^{\, i}$ at $(x_i, 0)$.
After passing to a subsequence, we may assume that $W_\infty := \lim_{i \to \infty} W_i$ exists.
Since $\rrm(x_i, 0) \to 0$, we must have $W_\infty  < 0$ by \cite[\HKThmEpsRegularity]{Bamler_HK_entropy_estimates}.

Using Proposition~\ref{Prop_NN_almost_constant_selfsimilar}, and after possibly parabolic rescaling and replacing of the sequence $\eps_i$ by some other sequence $\eps'_i \to 0$, we may furthermore assume that $(x_i, 0)$ is $(n-1,  \eps_i,1)$-split and $( \eps_i,1)$-selfsimilar.
As discussed in Section~\ref{Sec_basic_limits}, we may pass to a subsequence, and assume that we have convergence
\begin{equation} \label{eq_F_convergence_codim_2}
 (M_i, (g_{i,t})_{t \in I_i}, \nu^i) \xrightarrow[i \to \infty]{\quad \IF, \CF \quad} (\mathcal{X}, \nu_{x_\infty}), 
\end{equation}
where $\XX$ is an $H_n$-concentrated metric flow with full support and on $(-\infty, 0]$, $\nu^\infty := \nu_{x_\infty}$ is a conjugate heat kernel on $\XX$ and the $\IF$-convergence is understood to hold on compact time-intervals within some correspondence $\CF$.
The limiting space $\XX$ allows a regular-singular decomposition of the form $\XX = \RR {\,\, \dotcup \,\,} \SS$, where $\dim \SS \leq n+1$ in the sense of Theorem~\ref{Thm_SS_dimension_bound_limit}, and for almost all $t < 0$ we have
\[ \nu^\infty_t ( \SS_t ) = 0. \]
Denote by $\psi_i : U_i \to V_i$ the diffeomorphisms describing the smooth convergence on $\RR$.
Theorem~\ref{Thm_limit_from_strong_splitting} implies that, after passing to a subsequence, we have local smooth convergence  $f_i \to f_\infty$, $y^i_j \to y^\infty_j$, $j = 1, \ldots, n-1$, where $f_\infty \in C^\infty (\RR)$ is the potential of $d\nu^\infty =(4\pi \tau )^{-n/2} e^{-f_\infty} dg$ and $y_1^\infty, \ldots, y_{n-1}^\infty \in C^\infty (\RR)$ induce a local splitting, which implies that $\Rm \equiv 0$ on $\RR$.

\begin{Claim} \label{Cl_eps_static}
$(x_i, 0)$ is $(\eps''_i, 1)$-static for some sequence $\eps''_i \to 0$.
\end{Claim}

\begin{proof}
Fix some small $\eps'' > 0$ and suppose by contradiction that $(x_i, 0)$ is not $(\eps'', 1)$-static for infinitely many $i$.
Fix some small $\theta \in (0, \frac12 \eps'' )$.
By Proposition~\ref{Prop_improved_L2} we have
\[ \int_{-2\theta}^{-\theta} \int_{M_i} |{\Ric}|^2 d\nu^i_t dt \leq C(Y, \theta) < \infty. \]
So after passing to a subsequence we can find a $t^* \in [-2\theta, -\theta]$ such that $\nu^\infty_{t^*} (\SS_{t^*}) = 0$ and
\[ \int_{M_i} |{\Ric}|^2 d\nu^i_{t^*}    \leq C(Y, \theta) < \infty. \]
Since $R \equiv 0$ on $\RR$ and by the smooth convergence on $\RR$, we have $\sup_{U_i \cap M_i \times \{ t^* \}} |R| \to 0$ and $\nu^i_{t^*} (M_i \setminus U_i) = 1 - \nu^i_{t^*} ( U_i) \to 0$.
It follows that
\begin{equation*}
 \int_{M_i} R \, d\nu^i_{t^*} 
\leq \int_{U_i} R \, d\nu^i_{t^*}    +  \int_{M_i \setminus U_i} R \, d\nu^i_{t^*}  
\leq  \int_{U_i} R \, d\nu^i_{t^*}   +  \big( \nu^i_{ t^*} (M_i \setminus U_i) \big)^{1/2} \bigg( \int_{M_i \setminus U_i} R^2 d\nu^i_{t^*}  \bigg)^{1/2} \to 0. 
\end{equation*}
Therefore, since $\liminf_{i \to \infty} \inf_{M_i} R (\cdot, t^{**}) \geq 0$  for any $t^{**} < t^*$ we obtain that
\[ 2 \int_{t^{**}}^{t^*} \int_{M_i} |{\Ric}|^2 d\nu^i_{t}  dt =  \int_{M_i} R \, d\nu^i_{t}  \bigg|_{t =t^{**}}^{t =t^*} \to 0 \]
and
\[  \int_{M_i} R \, d\nu^i_{x_i, 0}  \bigg|_{t =t^{**}} \to 0 \]
This finishes the proof of the claim.
\end{proof}

Due to Claim~\ref{Cl_eps_static}, we are now in a position to apply Theorem~\ref{Thm_limit_from_almost_cone} for $k = n-1$, $\Delta = 1$; note that $k = n- \Delta$.
So define for $i = 1, 2, \ldots, \infty$
\begin{equation} \label{eq_def_q_i_codim_2}
 q_i := 4\tau (f_i - W_i) - \sum_{j=1}^{n-1} (y^i_j)^2. 
\end{equation}
Then $q_i \to q_\infty$ locally smoothly.

\begin{Claim}
$\RR$ corresponds to the constant flow on the open Euclidean upper half-plane $\IR^{n-1} \times \IR_+$ and $( y^\infty_1, \ldots, y^\infty_{n-1}, \sqrt{q_\infty})$ are Euclidean coordinates on every time-slice
\end{Claim}

\begin{proof}
Theorem~\ref{Thm_limit_from_almost_cone} implies that $\RR = (M'_\infty \times \IR^{n-1}) \times \IR_-$ for some $1$-dimensional Riemannian manifold $(M'_\infty, g'_\infty)$ with the property that $\{ q_\infty > 0 \} \subset M'_\infty$ is the Riemannian cone over a compact $0$-dimensional link $(N, g^N)$.
Moreover, by (\ref{eq_vol_link})
\[ \# N = \vol (N, g^N) = e^{W_\infty} \vol (S^0, g^{S^0}) < 2. \]
It follows that $\# N = 1$, which implies the claim.
\end{proof}

Consider the maps $\phi_i := ( y^i_1, \ldots, y_{n-1}^i,  q_i) : M_i \times (-2,-1) \to \IR^n_+ := \IR^{n-1} \times \IR_+$, which locally smoothly converge to $(y_1^\infty, \ldots, y^\infty_{n-1},  q_\infty)$ on $\RR_{(-2,-1)}$.
By Propositions~\ref{Prop_improved_L2}, \ref{Prop_properties_splitting_map} we may pass to a subsequence and choose a time $t \in (-2,-1)$ at which for all $j = 1, \ldots, n-1$
\[ \int_{M_i} |\nabla y^i_j|^{4n} d\nu^i_{t} , \; \int_{M_i} |\nabla q_i|^2 d\nu^i_{t} \leq C_1 \]
for some uniform constant $C_1 < \infty$ and such that $\nu^\infty_t (\SS_t) = 0$.
By the local smooth convergence, we can choose an increasing sequence of open subsets $U'_1 \subset U'_2 \subset \ldots \subset \RR_{t}$ with $\bigcup_{i=1}^\infty U'_i = \RR_t$ such that for $V'_i := \psi_{i,t} (U'_i)$ and $\phi'_i := \phi_{i,t} : V'_i \to \IR^n_+$ the following is true:
\begin{enumerate}[label=(\arabic*)]
\item We have
\begin{equation} \label{eq_M_minus_V_i_0}
 (4\pi |t|)^{-n/2} \int_{M_i \setminus V'_i} e^{-f_i} dg_{i,t} = \nu^i_{ t} (M_i \setminus V'_i) \to 0.
\end{equation}
\item $\phi'_i|_{V'_i}$ is an diffeomorphism onto its image.
\item We have
\[ \phi'_1(V'_1) \subset \phi'_2 (V'_2) \subset \ldots, \qquad \bigcup_{i=1}^\infty \phi'_i (V'_i) = \IR^{n}_+. \]
\end{enumerate}
Due to degree reasons, we can find a sequence of open subsets $V''_i \subset M_i$ such that:
\begin{enumerate}[label=(\arabic*), start=4]
\item $V'_i \cap V''_i = \emptyset$.
\item $\phi'_i (V''_i) \subset B( \vec 0, C_2) \subset \IR^n$ for some uniform $C_2 < \infty$.
\item $|\phi'_i (V''_i) | \geq c > 0$ for large $i$ and some uniform $c > 0$.
\end{enumerate}

By (\ref{eq_def_q_i_codim_2}) we must have $|f_i| \leq C_3 < \infty$ on $V''_i$ for some uniform constant $C_3 < \infty$.
Therefore, by (\ref{eq_M_minus_V_i_0}) we have
\[ | V''_i |_{g_{i,t}} \to 0. \]
This implies that for some uniform constant $C_4 < \infty$
\begin{align*}
 0 &< c \leq | \phi'_i (V''_i)|_{g_{i,t}} = \int_{V''_i} \sqrt{ \det \left(\begin{array}{c|c}\nabla y^i_j \cdot \nabla y^i_l & \nabla y^i_l \cdot \nabla q_i \\\hline \nabla q_i \cdot \nabla y^i_j  & |\nabla q_i|^2 \end{array}\right) } dg_{i,t} \\
 &\leq \int_{V''_i} |\nabla y^i_1| \cdots |\nabla y^i_{n-1}| \, |\nabla q_i | dg_{i,t} \\
&\leq   \bigg( \int_{V''_i} |\nabla y^i_1|^2 \cdots |\nabla y^i_{n-1}|^2 dg_{i,t} \bigg)^{1/2} \bigg( \int_{V''_i}  |\nabla q_i |^2 dg_{i,t} \bigg)^{1/2} \\
&\leq | V''_i |_{g_{i,t}}^{1/2n} \bigg( \int_{V''_i} |\nabla y^i_1|^{2n} dg_{i,t} \bigg)^{1/2n} \cdots \bigg( \int_{V''_i}  |\nabla y^i_{n-1}|^{2n} dg_{i,t} \bigg)^{1/2n} \bigg( \int_{V''_i}  |\nabla q_i |^2 dg_{i,t} \bigg)^{1/2} \\
&\leq C_4 | V''_i |_{g_{i,t}}^{1/2n} \bigg( \int_{V''_i} |\nabla y^i_1|^{2n} e^{-f_i} dg_{i,t} \bigg)^{1/2n}  \\
&\qquad\qquad\qquad \cdots \bigg( \int_{V''_i}  |\nabla y^i_{n-1}|^{2n} e^{-f_i} dg_{i,t} \bigg)^{1/2n} \bigg( \int_{V''_i}  |\nabla q_i |^2 e^{-f_i} dg_{i,t} \bigg)^{1/2}  \to 0.
\end{align*}
This gives the desired contradiction.
\end{proof}
\bigskip

\section{\texorpdfstring{An improved $\eps$-regularity theorem and codimension 4 of the singular set}{An improved {\textbackslash}eps-regularity theorem and codimension 4 of the singular set}} \label{sec_eps_reg_codim_4}
\subsection{Statement of the main result and reduction to weaker statement} \label{subsec_codim4_intro}
The goal of this section will be to show the following final improvement of Proposition~\ref{Prop_eps_reg_codim_2}, which can be seen as the parabolic analog of the Codimension 4 Conjecture for spaces with two-sided Ricci curvature bound \cite{Cheeger-Naber-Codim4}.

\begin{Proposition} \label{Prop_reg_codim_4}
Let $Y< \infty$ and suppose that $0 <\eps \leq \ov\eps (Y)$. Then the following holds.

Let $(M, (g_t)_{t \in I})$ be a Ricci flow on a compact manifold, $r > 0$ and $(x_0,t_0) \in M \times I$ such that $\NN_{x_0,t_0} (r^2) \geq - Y$.
Suppose that $(x_0, t_0)$ is strongly $(n-1, \eps, r)$-split or strongly $(n-3, \eps, r)$-split and $(\eps, r)$-static.
Then $\rrm (x_0, t_0) \geq \eps r$.
\end{Proposition}

Proposition~\ref{Prop_reg_codim_4} implies that the singular set of any non-collapsed limit of a sequence of Ricci flows has a singular set of codimension $\geq 4$, which is optimal.
More specifically:

\begin{Corollary} \label{Cor_working_ass_holds}
Assumption~\ref{Aspt_working_ass} holds for $\Delta = 4$.
\end{Corollary}

Proposition~\ref{Prop_reg_codim_4} is a consequence of the following weaker statement, whose proof will occupy the remainder of this section.

\begin{Proposition} \label{Prop_eps_reg_codim_4_weaker}
Let $Y< \infty$ and suppose that $0 <\eps \leq \ov\eps (Y)$. Then the following holds.

Let $(M, (g_t)_{t \in I})$ be a Ricci flow on a compact manifold, $r > 0$ and $(x_0,t_0) \in M \times I$ such that $\NN_{x_0,t_0} (r^2) \geq - Y$.
Suppose that $(x_0, t_0)$ is strongly $(n-2, \eps, r)$-split, $(\eps,r)$-static and $(\eps, r)$-selfsimilar.
Then $\rrm (x_0, t_0) \geq \eps r$.
\end{Proposition}

\begin{proof}[Proof that Proposition~\ref{Prop_eps_reg_codim_4_weaker} implies Proposition~\ref{Prop_reg_codim_4}.]
The case in which $(x_0,t_0)$ is strongly $(n-1, \eps, r)$-split is already addressed in Proposition~\ref{Prop_eps_reg_codim_2}.
It remains to consider the case in which $(x_0,t_0)$ is strongly $(n-3, \eps, r)$-split and $(\eps, r)$-static.

Next, we argue that we may assume $M$ to be orientable.
To see this, consider the orientation double cover $\td M$ of $M$ and let $\td x_0, \td x'_0$ be the two lifts of $x_0$.
Then the average of the conjugate heat kernels on $\td M$ based at $(\td x_0, t_0)$ and $(\td x'_0, t_0)$ agree with the lift of the conjugate heat kernel on $M$ at $(x_0,t_0)$.
By Proposition~\ref{Prop_L_infty_HK_bound} the latter satisfies an $L^\infty$-bound.
So the conjugate heat kernels on $\td M$ based at $(\td x_0, t_0)$ also satisfy an $L^\infty$-bound, which implies a lower bound on the pointed Nash entropy at $(\td x_0, t_0)$.
Next, we may lift the splitting map on $M$ to $\td M$ and conclude that $(\td x_0, t_0)$ is $(n-3, \Psi(\eps), r)$-split and $(\Psi(\eps),r)$-static.
So after replacing $M$ with $\td M$, $x_0$ with $\td x_0$ and adjusting $Y, \eps$, we may assume that $M$ is orientable.

By Proposition~\ref{Prop_NN_almost_constant_selfsimilar} and after adjusting $\eps$ and reducing $r$ by a bounded factor, we may additionally assume that $(x_0,t_0)$ is also $(\eps, r)$-selfsimilar.
So it remains to show that if $\eps \leq \ov\eps(Y)$, $\NN_{x_0,t_0} (r^2) \geq - Y$ and if $(x_0,t_0)$ is $(n-3,  \eps, r)$-split, $(\eps, r)$-static and $(\eps, r)$-selfsimilar, then $\rrm (x_0,t_0) \lb \geq \eps r$.

Without loss of generality, we may assume $r = 1$ and $t_0 = 0$.
Assume that the above statement was false and choose a sequence of counterexamples $(M_i, (g_{i,t})_{t \in I_i})$, $x_i \in M_i$ with orientable $M_i$, such that $(x_i, 0)$ is $(n-3, \eps_i, 1)$-split, $(\eps_i, 1)$-static and $(\eps_i, 1)$-selfsimilar, but $\rrm(x_i,0) \to 0$.
After passing to a subsequence, we may assume that we have convergence $W_i := \NN_{x_i, 0} (1) \to W_\infty$.
Since $\rrm(x_i,0) \to 0$, we must have $W_\infty < 0$ by \cite[\HKThmEpsRegularity]{Bamler_HK_entropy_estimates}.

By the discussion in Section~\ref{Sec_basic_limits} we may pass to an $\mathbb{F}$-convergent subsequence
\[ (M_i, (g_{i,t})_{t \in I_i}, \nu_{x_i,0}) \xrightarrow[i \to \infty]{\quad \IF, \CF \quad} (\mathcal{X}, \nu_{x_\infty}), \]
where $\XX$ is an $H_n$-concentrated metric flow over $(-\infty,0]$ with full support and the $\IF$-convergence is understood to hold on compact time-intervals.
By Proposition~\ref{Prop_eps_reg_codim_2} and since we assumed Proposition~\ref{Prop_eps_reg_codim_4_weaker}, we may apply Theorems~\ref{Thm_SS_dimension_bound_limit} and \ref{Thm_limit_from_almost_cone} for $\Delta = 3$.
After passing to another subsequence, we obtain a regular-singular decomposition $\XX = \RR {\,\, \dotcup \,\,} \SS$ such that for all $t < 0$\begin{equation} \label{eq_nu_SS_0_all_t}
 d\nu_{x_\infty ; t} (\SS_t) = 0 .
\end{equation}
By Theorem~\ref{Thm_limit_from_almost_cone} we find that $\RR$ corresponds to the constant flow on $M'_\infty \times \IR^{n-3}$, where $(M'_\infty, g'_\infty)$ is Ricci flat and isometric to a Riemannian cone over some compact link $(N^2, g^N)$, possibly minus its vertex.
Since $(M'_\infty, g'_\infty)$ is 3-dimensional and Ricci flat, it must be flat and therefore, $(N^2, g^N)$ must have $\sec \equiv 1$.
Moreover,
\[ \vol ( N, g^N ) = e^{W_\infty} \vol ( S^2, g^{S^2} ) < \vol ( S^2, g^{S^2} ). \]
Since we have assumed the manifolds $M_i$ to be orientable, the manifold $M'_\infty$ and therefore $N$ must be orientable, which is impossible.
\end{proof}

The following subsections will be devoted to the proof of Proposition~\ref{Prop_eps_reg_codim_4_weaker}.
The general strategy of this proof is similar to that in \cite{Cheeger-Naber-Codim4}.
However, we will encounter several issues, which stem from the lack of pointwise gradient bounds for strong splitting maps, lower bounds on conjugate heat kernels and distance distortion issues.

\subsection{The heat operator applied to absolute values of forms}
In this section, we will frequently apply the heat operator $\square$ to functions that are only weakly twice differentiable (see Definition~\ref{Def_weakly_2_diff}).
Let us recall some of the important aspects of this theory.
For a more detailed discussion on this topic see Subsection~\ref{subsec_H_operator_identities}. 

Denote by $(M, (g_t)_{t \in I})$ a Ricci flow on a compact manifold and consider a weakly twice differentiable function $u  \in C^0 (M \times I')$, where $I' \subset I$ is a non-trivial sub-interval.
By Proposition~\ref{Prop_weak_square} the heat operator applied to $u$ is represented by a signed measure $\mu_{\square u}$ on $M \times I'$ of locally finite total variation such that for any compact sub-interval $[t_1, t_2] \subset I'$ and any compactly supported $\phi \in C^2_c (M \times [t_1, t_2])$ we have
\begin{equation} \label{eq_weak_square_identity}
 \int_{M \times [t_1, t_2]} \phi \, d\mu_{\square u} = \int_M u \, \phi \, dg_t \bigg|_{t=t_1}^{t=t_2} + \int_{t_1}^{t_2} \int_M u \square^* \phi \, dg_t dt. 
\end{equation}
We recall that $\mu_{\square u}$ is a measure on $M \times I'$, and its integral over a single time-slice always vanishes.
By abuse of notation, we will often write $\square u \, dg$ instead of $\mu_{\square u}$ and $|\square u| dg$ instead of $|\mu_{\square u}| = (\mu_{\square u})_+ + (\mu_{\square u})_-$ and set
\[ \int_{t_1}^{t_2} \int_M \phi \, \square u \, dg_t dt :=  \int_{M \times [t_1, t_2]} \phi \, d\mu_{\square u} , \qquad \int_{t_1}^{t_2} \int_M \phi \, |\square u| \, dg_t dt :=  \int_{M \times [t_1, t_2]} \phi \, d|\mu_{\square u}|. \]
Note again that the double integral on the left-hand side is necessary.
In addition, in the following subsections, $\phi$ will always denote a conjugate heat kernel $d\nu = \phi \, dg = (4\pi \tau)^{-n/2} e^{-f} dg$ and we will often write
\[  \int_{t_1}^{t_2} \int_M  \square u \, d\nu_t dt :=  \int_{M \times [t_1, t_2]} (4\pi \tau)^{-n/2} e^{-f} \, d\mu_{\square u} , \]
similarly for $|\square u|$.
Then the identity (\ref{eq_weak_square_identity}) becomes the usual identity
\[ \int_{t_1}^{t_2} \int_M  \square u \, d\nu_t dt = \int_M u \, d\nu_t \bigg|_{t=t_1}^{t=t_2} . \]

Lastly, we describe one class of weakly twice differentiable functions that we will encounter most frequently.
Consider smooth functions $y_1, \ldots, y_k  \in C^\infty (M \times I')$.
In our case of interest, these functions will always come from a strong splitting map and therefore satisfy the heat equation $\square y_l = 0$, $l = 1, \ldots, k$.
Throughout this section we will use the following $l$-forms for $1 \leq l \leq k$:
\begin{equation} \label{eq_omega_l_def}
 \omega_l := dy_1 \wedge \ldots \wedge dy_l. 
\end{equation}
By Proposition~\ref{Prop_what_is_weakly_2_diff} the absolute value $|\omega_l| \in C^0(M \times I')$ is weakly twice differentiable.
It follows from Propositions~\ref{Prop_weak_square}, \ref{Prop_mu_square_rules} that $d\mu_{\square |\omega_l|} = \square |\omega_l| dg$ on $\{ \omega_l \neq 0 \}$ and that in general
\[ d\mu_{\square |\omega_l|} \leq |\square \omega_l| dg. \]
So the positive part $(d\mu_{\square |\omega_l|})_+$ is absolutely continuous.
We also record the following consequence of the heat equation $\square y_l = 0$, which we will use in the following.
For a local orthonormal frame $\{ e_m \}_{m=1}^{n}$ we have
\[ \square \omega_1 = \square d y_1 = 0 \]
and for $l \geq 2$
\begin{multline*}
 |\square \omega_l|
= \bigg| \sum_{i=1}^l d y_1 \wedge \ldots \wedge \square d y_i \wedge \ldots \wedge d y_l
-2 \sum_{m=1}^n \sum_{1 \leq i < j \leq l} d y_1 \wedge \ldots \wedge  \nabla^2_{e_m, \cdot} y_i \wedge \ldots \wedge  \nabla^2_{e_m, \cdot} y_j \wedge \ldots \wedge d y_l \bigg| \\
\leq C \sum_{1 \leq i < j \leq l} |\nabla y_1| \cdots |\nabla^2 y_i| \cdots |\nabla^2 y_j| \cdots |\nabla y_l|
\leq \bigg( \sum_{i=1}^l |\nabla^2 y_i|^2 \bigg) \bigg(  \sum_{i=1}^l |\nabla y_i|^{l-2} \bigg).
\end{multline*}
Therefore,
\begin{equation} \label{eq_square_omega_Kato_Lp_bound}
d\mu_{\square |\omega_l|}
\leq |\square \omega_l| dg
 \leq 
 \begin{cases}
 0 &\text{if $l =1$} \\ 
 \Big( \sum_{i=1}^l |\nabla^2 y_i|^2 \Big) \Big(  \sum_{i=1}^l |\nabla y_i|^{l-2} \Big) dg &\text{if $l \geq 2$}
 \end{cases}.
\end{equation}

\subsection{The Transformation Theorem}
The following proposition can be viewed as the parabolic analog of \cite[Theorem~1.32]{Cheeger-Naber-Codim4}.

\begin{Proposition}[Transformation Theorem] \label{Prop_Transformation_Theorem}
If $Y < \infty$, $\eps \in (0, 10^{-1})$ and $\delta \leq \ov\delta (Y, \eps)$, then the following holds.
Let $(M, (g_t)_{t \in I})$ be Ricci flow on a compact manifold, $(x_0,t_0) \in M \times I$ and $r_0 > 0$ such that $\NN_{x_0,t_0} (r_0^2) \geq - Y$.
Denote by $d\nu = d\nu_{x_0, t_0} = (4\pi\tau)^{-n/2} e^{-f} dg$ the conjugate heat kernel based at $(x_0,t_0)$.
Let $y_1, \ldots, y_k \in C^\infty (M \times [t_0 - \delta^{-1} r_0^2, t_0])$ be solution to the heat equation $\square y_l = 0$ and set $\vec y := (y_1, \ldots, y_k)$.

Assume that $\vec y$ is a strong $(k, \delta, r_0)$-splitting map at $(x_0, t_0)$.
Moreover, assume that there is some $r_1 \in (0, r_0]$ such that for all $r \in [r_1, r_0]$ and $l = 1, \ldots, k$ the $l$-form from (\ref{eq_omega_l_def}) satisfies (in the sense of weak derivatives)
\[  \int_{t_0- r^2}^{t_0-\delta r^2} \int_M  |\square |\omega_l| | d\nu_{t} dt
\leq \delta r^{-2} \int_{t_0- r^2}^{t_0-  r^2/2} \int_M  |\omega_l|  d\nu_{t} dt. \]
Then there is a lower triangular $k \times k$-matrix $A$ with diagonal entries $\geq 1- \eps$ and $|A| \leq 10 (r_1/r_0)^{-\eps}$ such that $\vec y^{\,\prime} := A \vec y$ is a strong $(k, \eps, r_1)$-splitting map at $(x_0, t_0)$.
\end{Proposition}

\begin{proof}
Without loss of generality, we may assume that $t_0 = 0$.
We will show the proposition with the following additional assertion: If $\vec y^{\,\prime} := A \vec y$, then for $r = r_1$
\begin{equation} \label{eq_ONB_y}
 \frac1{(\eps^{-1} - \eps) r^2} \int_{-\eps^{-1} r^2}^{-\eps r^2} \int_M \nabla y'_i \cdot \nabla y'_j \, d\nu_{t} dt = \delta_{ij}. 
\end{equation}
Note that $A$ is uniquely determined by (\ref{eq_ONB_y}) under the assumption that it is lower triangular and has positive diagonal entries.

By induction, we may assume that the proposition with this additional assertion already holds for $k-1$ if $k \geq 2$.
We will now prove it for $k$.
Fix $Y, \eps$, let $\delta > 0$ be a constant whose value we will determine later and assume by contradiction that the assumptions of the proposition hold, but the stronger assertion is false.
Throughout the proof we will impose assumptions of the form $\eps \leq \ov\eps (Y)$, which does not create any issues since the proposition becomes stronger for smaller $\eps$.

By Gram-Schmidt orthogonalization we obtain the following for any $r \in [r_1, r_0]$: either there is a unique lower triangular $k \times k$ matrix $A_r$ with positive diagonal entries such that (\ref{eq_ONB_y}) holds for $\vec y^{\,\prime} := A_r \vec y$ or the gradient vector fields $\nabla y_1, \ldots, \nabla y_k$ restricted to $M \times [-\eps^{-1}r^2, -\eps r^2]$ are linearly dependent.
Moreover $A_r$ depends continuously on $r$, whenever it exists.
If $\eps \leq \ov\eps$, then whenever $A_r$ exists and $A_r \vec y$ is a strong $(k,\eps, r)$-splitting map, then we can exclude the case of linear dependence and conclude that $A_{r'}$ exists for all nearby $r' \in [r_1, r_0] \cap [\frac12 r, 2 r]$.

By a continuity argument there is a minimal $r'_1 \in [r_1, r_0]$ such that if $r'_1 < r_0$, then for all $r \in [r'_1, r_0]$ the following is true: the matrix $A_r$ exists, $\vec y^{\,\prime} := A_r \vec y$ is a strong $(k, \eps, r)$-splitting map at $(x_0, 0)$, $A_r$ has diagonal entries $\geq 1-\eps$ and $|A_r| \leq 10 (r/r_0)^{-\eps}$.
If $r'_1 = r_1$, then we are done.
So assume by contradiction that $r'_1 > r_0$.
After parabolic rescaling, we may assume that $r'_1 = 1$.

Set $\vec y^{\,\prime} := A_1 \vec y$ and $\omega'_l := dy'_1 \wedge \ldots \wedge dy'_l$, $l = 1, \ldots, k$, for the remainder of this proof.
Since $\int_M y_i \, d\nu_t =0$ and due to a continuity argument, it remains to show:
\begin{enumerate}[label=(\arabic*)]
\item \label{propty_y_p_1} For all $i,j = 1, \ldots, k$
\[   \int_{ - \eps^{-1} }^{ - \eps } \int_M |\nabla y'_i \cdot \nabla y'_j - \delta_{ij} | d\nu_{t} dt < \eps. \]
\item \label{propty_y_p_2} The diagonal entries of $A_1$ are $> 1- \eps$.
\item \label{propty_y_p_3} $|A_1| < 10 (1/r_0)^{-\eps}$
\end{enumerate}

\begin{Claim} \label{Cl_r_0_large}
$r_0 \geq \Psi^{-1} (\delta | Y, \eps)$.
\end{Claim}

\begin{proof}
This also follows from Gram-Schmidt orthogonalization.
\end{proof}

\begin{Claim} \label{Cl_k_m_1_splitting}
$(y'_1, \ldots, y'_{k-1})$ is a strong $(k-1, \Psi (\delta |Y),1)$-splitting map at $(x_0,t_0)$.
\end{Claim}

\begin{proof}
This follows from the induction hypothesis.
\end{proof}

\begin{Claim}
If $\delta \leq \ov\delta (\eps)$, then Property~\ref{propty_y_p_2} holds.
\end{Claim}

\begin{proof}
We have
\[ \frac{d}{dt} \int_M |\nabla y_i|^2 d\nu_t 
=  \int_M \square |\nabla y_i |^2 d\nu_t 
= -2  \int_M  |\nabla^2 y_i|^2 d\nu_t  \leq 0. \]
By (\ref{eq_ONB_y}) we have
\[ \frac1{ \eps^{-1} - \eps }\int_{-\eps^{-1}}^{-\eps} \int_M |\nabla y_i|^2 d\nu_t dt 
=  |(A_1^T)^{-1} \vec e_i|^2. \]
Since $\vec y$ is a strong $(k, \delta, r_0)$-splitting map, there is a $t_1 \in [-2r_0^2, -r_0^2]$ with
\[ \int_M |\nabla y_i|^2 d\nu_{t_1} \leq 1+\Psi (\delta ). \]
Combining these bounds implies $|(A_1^T)^{-1} \vec e_i|^2 \leq 1 + \Psi (\delta)$.
Since $(A_1^T)^{-1}$ is lower triangular, this implies that its diagonal entries are $\leq 1 + \Psi(\delta)$.
Therefore the diagonal entries of $A_1$ are $\geq 1 - \Psi(\delta)$.
\end{proof}

\begin{Claim} \label{Cl_slow_poly}
\begin{enumerate}[label=(\alph*)]
\item \label{Cl_slow_poly_a} For any $r, r' \in [1, r_0]$ with $\frac12 r \leq r' \leq 2 r$ we have $|A_{r} A^{-1}_{r'} - I_k| , |A_{r'} A^{-1}_{r} - I_k| \leq \Psi (\eps | Y)$, where $I_k$ denotes the identity matrix.
\item \label{Cl_slow_poly_b} If $\eps \leq \ov\eps (Y)$, $\delta \leq \ov\delta (Y, \eps)$, then Property~\ref{propty_y_p_3} holds.
\item \label{Cl_slow_poly_c} If $\gamma > 0$ and $\eps \leq \ov\eps (Y, \gamma)$, then for all $t \in [-r_0^2, -\frac12]$
\begin{equation} \label{eq_omega_p_nab_y_p_pol}
 \int_M |\omega'_k| d\nu_t, \, \int_M |\nabla y'_k|^2 d\nu_t  \leq 10 |t|^\gamma. 
\end{equation}
\end{enumerate}
\end{Claim}

\begin{proof}
Assertion~\ref{Cl_slow_poly_a} follows, because $A_{r'} \vec y$ and $A_{r} \vec y$ are strong splitting maps with overlapping domains.
Assertion~\ref{Cl_slow_poly_b} follows from Assertion~\ref{Cl_slow_poly_a} combined with the fact that $|A_2| \leq 10 (2/r_0)^{-\eps}$, which holds due to the choice of $r_1$.
For Assertion~\ref{Cl_slow_poly_c}, consider some $r \in [1, r_0]$ with $t \in [-r^2, -\eps r^2]$ and set $\vec y^{\,\prime\prime} := A_{r} \vec y$, $\omega''_k := dy''_1 \wedge \ldots \wedge dy''_k$.
Then $\vec y^{\, \prime} = A_1 \vec y = A_1 A_r^{-1} \vec y^{\,\prime\prime}$.
Iterating Assertion~\ref{Cl_slow_poly_a} implies that $|A_1 A_r^{-1} - I | \leq r^{\Psi (\eps |Y)}$.
So (\ref{eq_omega_p_nab_y_p_pol}) follows from
\[ \int_M |\omega''_k| d\nu_t, \, \int_M |\nabla y''_k|^2 d\nu_t  \leq 10, \]
which follows from Proposition~\ref{Prop_properties_splitting_map} and Property~\ref{propty_y_p_2} for $\eps \leq \ov\eps (Y)$.
\end{proof}

Lastly, note that by assumption and Claim~\ref{Cl_slow_poly}\ref{Cl_slow_poly_c} for any $r \in [1, r_0]$
\[  \int_{-r^2}^{-\delta r^2} \int_{M} |\square |\omega'_k|| d\nu_t dt \leq \delta r^{-2} 
 \int_{-r^2}^{-r^2/2} \int_{M}  |\omega'_k| d\nu_t dt \leq 10 \delta r^{\gamma}  . \]
Therefore, given some $\delta' \geq \delta$, we have $r_0^2 \geq \delta^{\prime -1}$ for $\delta \leq \ov\delta (Y, \eps, \delta')$ by Claim~\ref{Cl_r_0_large} and hence if $\gamma \leq \ov\gamma$, then
\begin{multline} \label{eq_square_omega_prime}
 \int_{-\delta^{\prime -1}}^{-\delta'} \int_{M} |\square |\omega'_k|| d\nu_t dt
 \leq
  \int_{-1}^{-\delta'} \int_{M} |\square |\omega'_k|| d\nu_t dt
+  \int_{-\delta^{\prime -1}}^{-\delta \delta^{\prime -1}} \int_{M} |\square |\omega'_k|| d\nu_t dt \\
  \leq 10 \delta + 10 \delta (\delta^{\prime -1})^\gamma
  \leq \Psi (\delta  | \delta').
\end{multline}
Combining (\ref{eq_square_omega_prime}), (\ref{eq_ONB_y}) and Claims~\ref{Cl_r_0_large}, \ref{Cl_k_m_1_splitting}, \ref{Cl_slow_poly}\ref{Cl_slow_poly_c} and, using Lemma~\ref{Lem_transformation_estimates} below, implies Property~\ref{propty_y_p_1} if $\gamma \leq \ov\gamma$, $\delta' \leq \ov\delta' (Y, \eps)$ and $\delta \leq \ov\delta (Y, \eps, \delta')$.
This finishes the proof.
\end{proof}

\begin{Lemma} \label{Lem_transformation_estimates}
For any $Y < \infty$, $\gamma \leq \ov\gamma$, $\eps > 0$ and $\delta \leq \ov\delta(Y, \eps)$ the following holds.
Let $(M, (g_{t})_{t \in [-\delta^{-1}, 0]})$ be a Ricci flow on a compact manifold and $x_0 \in M$ a point such that $\NN_{x_0, 0} (1) \geq -Y$.
Denote by $d\nu = (4\pi \tau)^{-n/2} e^{-f} dg$ the conjugate heat kernel based at $(x_0, 0)$.
Let $y_1, \ldots, y_k \in C^\infty (M \times [-\delta^{-1}, -\delta])$ be solutions to the heat equation $\square y_i = 0$ such that
\begin{enumerate}[label=(\roman*)]
\item \label{Lem_transformation_estimates_i} $(y_1, \ldots, y_{k-1})$ is a strong $(k-1, \delta, 1)$-splitting map at $(x_0, 0)$.
\item \label{Lem_transformation_estimates_ii} $\int_M y_k \, d\nu_t = 0$ for all $t \in [-\delta^{-1}, -\eps]$.
\item \label{Lem_transformation_estimates_iii} For $i = 1, \ldots, k$
\[ \frac1{(\eps^{-1} - \eps) } \int_{-\eps^{-1} }^{-\eps } \int_M \nabla y_i \cdot \nabla y_k \, d\nu_{t} dt = \delta_{ik}. \]
\item \label{Lem_transformation_estimates_iv} For $\omega_k := dy_1 \wedge \ldots \wedge dy_k$ we have (in the sense of weak derivatives)
\[ \int_{-\delta^{-1}}^{-\delta} \int_{M} |\square |\omega_k|| d\nu_t dt \leq \delta.  \]
\item \label{Lem_transformation_estimates_v} For all $t \in [-\delta^{-1}, -1]$
\[  \int_{M} |\nabla y_k|^2 d\nu_t  \leq 10 |t|^\gamma .\]
\end{enumerate}
Then for all $i = 1, \ldots, k$
\[   \int_{ - \eps^{-1} }^{ - \eps } \int_M |\nabla y_i \cdot \nabla y_k - \delta_{ik} | d\nu_{t} dt < \eps. \]
\end{Lemma}

\begin{proof}
Choose $a \geq 0$ such that
\begin{equation} \label{eq_omega_k_a_def}
 \int_{-\eps^{-1}}^{-\eps} \int_M (|\omega_k| - a) d\nu_t dt = 0. 
\end{equation}
We first show

\begin{Claim} If $\gamma \leq 10^{-1}$, then for $i = 1, \ldots, k-1$
\[ \int_{-\eps^{-1}}^{-\eps} \int_M \bigg( \big| |\omega_k| - a \big| + \sum_{i=1}^{k-1} |\nabla y_i \cdot \nabla y_k | \bigg) d\nu_t dt \leq \Psi (\delta | Y, \eps). \]
\end{Claim}

\begin{proof}
Set $h := |\omega_k| - a$ or $h := \nabla y_i \cdot \nabla y_k$ for some $i = 1, \ldots, k-1$.
By (\ref{eq_omega_k_a_def}) and Assumption~\ref{Lem_transformation_estimates_iii} we have in both cases
\begin{equation} \label{eq_h_average_0}
 \int_{-\eps^{-1}}^{-\eps} \int_M h \, d\nu_t dt = 0. 
\end{equation}
Assume that $\delta \leq \eps/2$ and let $T \in (2\eps^{-1}, \delta^{-1})$ be a constant whose value we will determine later.
We claim that (in the sense of weak derivatives)
\begin{equation} \label{eq_square_h_Psi}
 \int_{-T}^{-\eps} \int_M |\square h|  d\nu_t dt  \leq \Psi (\delta | Y,  \eps, T). 
\end{equation}
If $h := |\omega_k| - a$, then this is a consequence of Assumption~\ref{Lem_transformation_estimates_iv} and if $h := \nabla y_i \cdot \nabla y_k$, then this follows using Assumptions~\ref{Lem_transformation_estimates_i}, \ref{Lem_transformation_estimates_v} and Proposition~\ref{Prop_properties_splitting_map} from 
\begin{align*}
 \int_{-T}^{-\eps} \int_{M} & | \square ( \nabla y_i \cdot \nabla y_k) | d\nu_t dt 
\leq  2 \int_{-T}^{-\eps} \int_{M} |  \nabla^2 y_i \cdot \nabla^2 y_k|  d\nu_t dt  \\
&\leq \bigg( 2\int_{-T}^{-\eps} \int_{M} | \nabla^2 y_i |^2  d\nu_t dt \bigg)^{1/2} \bigg(  2 \int_{-T}^{-\eps} \int_{M}   | \nabla^2 y_k|^2  d\nu_t dt\bigg)^{1/2}
\displaybreak[1] \\
&\leq \bigg( 2\int_{-T}^{-\eps} \int_{M} | \nabla^2 y_i |^2  d\nu_t dt \bigg)^{1/2} \bigg(  \int_{M}  | \nabla y_k|^2  d\nu_{-T} \bigg)^{1/2}
 \leq \Psi (\delta | Y,  \eps, T).
\end{align*}
Next, we claim that
\begin{equation} \label{eq_h_3_2_T}
  \int_M |h|^{3/2} d\nu_{-T}  \leq CT^{3\gamma/4}. 
\end{equation}
To see this, observe that if $\delta \leq \ov\delta (Y, \eps)$, then by Proposition~\ref{Prop_properties_splitting_map} and Assumptions~\ref{Lem_transformation_estimates_i}, \ref{Lem_transformation_estimates_v}
\[  \int_M |\nabla y_i \cdot \nabla y_k |^{3/2} d\nu_{-T}  
\leq \bigg( \int_M |\nabla y_i |^{6} d\nu_{-T} \bigg)^{1/4} \bigg( \int_M |\nabla y_k |^{2} d\nu_{-T} \bigg)^{3/4}
\leq C T^{3\gamma/4}, \]
\[  \int_M |\omega_k |^{3/2} d\nu_{-T}  
\leq C \bigg( \sum_{i=1}^{k-1} \int_M  |\nabla y_i |^{6(k-1)} d\nu_{-T} \bigg)^{1/4} \bigg( \int_M |\nabla y_k |^{2} d\nu_{-T} \bigg)^{3/4}
\leq C T^{3\gamma/4} \]
and
\begin{align*}
 a 
&= \frac1{\eps^{-1} -\eps} \int_{-\eps^{-1}}^{-\eps} \int_M |\omega_k| d\nu_t dt \\
&\leq C \bigg( \sum_{i=1}^{k-1} \frac1{\eps^{-1} -\eps} \int_{-\eps^{-1}}^{-\eps} \int_M |\nabla y_i |^{2(k-1)} d\nu_t dt \bigg)^{1/2}
 \bigg( \frac1{\eps^{-1} -\eps} \int_{-\eps^{-1}}^{-\eps} \int_M |\nabla y_k|^2 d\nu_t dt \bigg)^{1/2} \\
 &\leq C \bigg( \frac1{\eps^{-1} -\eps} \int_{-\eps^{-1}}^{-\eps} 10 |t|^\gamma dt \bigg)^{1/2} \leq C T^{\gamma/2}.
\end{align*}

Let $Z < \infty$ be a constant whose value we will determine later and let $u = u' + u'' \in C^\infty ( M \times [-T, -\eps])$ be solutions to the heat equation $\square u = \square u' = \square u'' = 0$ with initial conditions
\[ u(\cdot, -T) = h(\cdot, -T), \qquad
u' (x, -T) := 
\begin{cases}
- Z & \text{if $h(x,-T) \leq - Z$} \\
h(x,-T) & \text{if $- Z \leq h(x,-T) \leq Z$} \\
Z & \text{if $Z \leq h(x,-T)$} 
\end{cases}
\]
Then by the maximum principle $|u' | \leq Z$ and by (\ref{eq_h_3_2_T})
\[ \int_M |u'' | d\nu_{-T} 
\leq Z^{-1/2} \int_M |h|^{3/2} d\nu_{-T}  
\leq C Z^{-1/2} T^{3\gamma/4}. \]

For any $t^* \in [-\eps^{-1}, -\eps]$ we have by a gradient estimate, see for example \cite[\HKThmGradientPhiEstimate]{Bamler_HK_entropy_estimates},
\begin{equation} \label{eq_nab_u_p}
 |\nabla u'| (\cdot, t^*) \leq C Z (t^*+T)^{-1/2} 
\leq C Z (T - \eps^{-1})^{-1/2} 
\leq C Z T^{-1/2} . 
\end{equation}
Moreover, using (\ref{eq_square_h_Psi}) and Proposition~\ref{Prop_mu_square_rules}
\begin{align}
 \int_M | h - u | d\nu_{t^*}
&= \int_{-T}^{t^*} \int_M \square | h - u| d\nu_t dt
\leq \int_{-T}^{t^*} \int_M | \square  h | d\nu_t dt
\leq \Psi (\delta | Y, \eps, T), \label{eq_h_u_Psi} \\
 \int_M | u'' | d\nu_{t^*}
&\leq \int_M | u'' | d\nu_{ -T} +  \int_{-T}^{t^*} \int_M \square | u''| d\nu_t dt
\leq \int_M | u'' | d\nu_{ -T} 
\leq C Z^{-1/2} T^{3\gamma/4} . \label{eq_u_pp_CT_gamma}
\end{align}
Next, set
\[ a' := \int_M u' \, d\nu_{-T} \]
and note that for all $t^* \in [-T, -\eps]$ we have
\[ a' = \int_M u' \, d\nu_{t^*}. \]
Therefore, by the $L^1$-Poincar\'e inequality (see Proposition~\ref{Prop_Poincare}) and (\ref{eq_nab_u_p})
\[ \int_M |u' - a' | d\nu_{t^*} 
\leq C |t^*|^{1/2} \int_M |\nabla u'| d\nu_{t^*}
\leq C |t^*|^{1/2} Z T^{-1/2}. \]
Combining this with (\ref{eq_h_u_Psi}), (\ref{eq_u_pp_CT_gamma}) and integrating over $t^*$ implies
\[ \int_{-\eps^{-1}}^{-\eps} \int_M |h- a'| d\nu_t dt 
\leq \Psi (\delta | Y, \eps, T) + C(\eps) Z^{-1/2} T^{3\gamma/4} + C(\eps) Z T^{-1/2}. \]
Setting $Z := T^{\gamma/2+1/3}$ gives
\[ \int_{-\eps^{-1}}^{-\eps} \int_M |h- a'| d\nu_t dt 
\leq \Psi (\delta | Y, \eps, T) + C(\eps) T^{\gamma/2 - 1/6} . \]
So if $\gamma \leq 10^{-1}$, then since the choice of $T$ was arbitrary, we have
\[ \int_{-\eps^{-1}}^{-\eps} \int_M |h- a'| d\nu_t dt 
\leq \Psi (\delta | Y, \eps)  . \]
Combining this with (\ref{eq_h_average_0}) implies the claim.
\end{proof}

By the Claim we have for $\delta \leq \ov\delta (Y, \eps)$ and $i = 1, \ldots, k-1$
\[  \int_{-\eps^{-1} }^{-\eps } \int_M |\nabla y_i \cdot \nabla y_k|  d\nu_{t} dt  \leq \eps. \]
So due to Assumption~\ref{Lem_transformation_estimates_i}, it remains to verify that for $\delta \leq \ov\delta (Y, \eps)$
\begin{equation} \label{eq_trafo_lemma_goal}
  \int_{-\eps^{-1} }^{-\eps } \int_M  \big| |\nabla y_k|^2 - 1 \big|  d\nu_{t} dt  \leq \eps. 
\end{equation}
This will be accomplished using the bound on $\omega_k$ from the Claim.

Since $|\omega_k|^2 = \det ( \nabla y_i \cdot \nabla y_j )$, we have
\[ \big| |\omega_k|^2 - |\nabla y_1|^2 \cdots |\nabla y_k|^2 \big|
\leq C \bigg( \sum_{\substack{i,j=1 \\ i \neq j}}^k |\nabla y_i \cdot \nabla y_j | \bigg) \bigg( \sum_{i=1}^{k} |\nabla y_i|^{2(k-1)} \bigg). \]
On the other hand
\begin{multline*}
 \big| |\nabla y_k |^2 -  |\nabla y_1|^2 \cdots |\nabla y_k|^2  \big|
\leq \sum_{i=1}^{k-1} \big|  |\nabla y_{i+1} |^2 \cdots |\nabla y_k|^2 -  |\nabla y_i |^2 \cdots |\nabla y_k|^2 \big| \\
\leq C \bigg( \sum_{i=1}^{k-1} \big| |\nabla y_i|^2 - 1 \big| \bigg) \bigg( 1 + \sum_{i=1}^{k} |\nabla y_i|^{2(k-1)} \bigg) .
\end{multline*}
Therefore
\begin{equation} \label{eq_omega_nab_y}
 \big| |\omega_k|^2 -  |\nabla y_k|^2 \big|
\leq C \bigg( \sum_{i,j=1}^{k-1} |\nabla y_i \cdot \nabla y_j - \delta_{ij} | + \sum_{i=1}^{k-1} |\nabla y_i \cdot \nabla y_k  | \bigg) \bigg( 1+ \sum_{i=1}^{k} |\nabla y_i|^{2(k-1)} \bigg). 
\end{equation}

By hypercontractivity \cite[\HKThmHypercontractivity]{Bamler_HK_entropy_estimates} combined with Assumptions~\ref{Lem_transformation_estimates_i}, \ref{Lem_transformation_estimates_v} and Proposition~\ref{Prop_properties_splitting_map}, we obtain that for $\delta \leq \ov\delta (Y,\eps)$
\begin{equation} \label{eq_hyp_cont_nab_yi}
 \sum_{i=1}^{k} \int_{-\eps^{-1}}^{-\eps} \int_M  |\nabla y_i|^{4k} d\nu_t dt \leq C (   \eps). 
\end{equation}
So by (\ref{eq_omega_nab_y}) we have
\begin{align}
 \int_{-\eps^{-1}}^{-\eps} & \int_M  \big| |\omega_k|^2 -  |\nabla y_k|^2 \big|  d\nu_t dt 
\leq C\bigg( \int_{-\eps^{-1}}^{-\eps} \int_M \bigg( \sum_{i,j=1}^{k-1} |\nabla y_i \cdot \nabla y_j - \delta_{ij} | + \sum_{i=1}^{k-1} |\nabla y_i \cdot \nabla y_k  |  \bigg) d\nu_t dt\bigg)^{1/2} \notag \\
&\qquad \cdot \bigg( \int_{-\eps^{-1}}^{-\eps} \int_M \bigg( \sum_{i,j=1}^{k-1} |\nabla y_i \cdot \nabla y_j - \delta_{ij} | + \sum_{i=1}^{k-1} |\nabla y_i \cdot \nabla y_k  |  \bigg) \bigg( 1 + \sum_{i=1}^{k} |\nabla y_i |^{2(k- 1)}  \bigg)^2 d\nu_t dt \bigg)^{1/2} \notag  \\
&\leq
\Psi (\delta | Y, \eps).  \label{eq_diff_omega_2}
\end{align}
By (\ref{eq_omega_k_a_def}), (\ref{eq_hyp_cont_nab_yi}) and the Claim, we also obtain $|a| \leq C(\eps)$ and
\begin{align}
 \int_{-\eps^{-1}}^{-\eps} & \int_M  \big| |\omega_k|^2 - a^2 \big|  d\nu_t dt \notag \\
&\leq  \bigg( \int_{-\eps^{-1}}^{-\eps} \int_M  \big| |\omega_k| - a \big|  d\nu_t dt \bigg)^{1/2}
\bigg( \int_{-\eps^{-1}}^{-\eps} \int_M  \big| |\omega_k| - a \big| \, \big| |\omega_k| + a \big|^{2}  d\nu_t dt \bigg)^{1/2} \notag \\
&\leq  \Psi (\delta | Y, \eps)
\bigg( \sum_{i=1}^k \int_{-\eps^{-1}}^{-\eps} \int_M  \big( |\nabla y_i|^{3k} + |a|^3 \big)  d\nu_t dt \bigg)^{1/2}
\leq
\Psi (\delta | Y, \eps). \label{eq_diff_omega_1}
\end{align}
Combining (\ref{eq_diff_omega_2}), (\ref{eq_diff_omega_1}) yields
\[  \int_{-\eps^{-1}}^{-\eps}  \int_M  \big| |\nabla y_k|^2 - a^2 \big|  d\nu_t dt 
\leq \Psi (\delta | Y, \eps). \]
Combining this with Assumption~\ref{Lem_transformation_estimates_iii} implies $|a^2-1| \leq \Psi (\delta | Y, \eps)$.
The bound (\ref{eq_trafo_lemma_goal}) now follows for $\delta \leq \ov\delta (Y, \eps)$, which finishes the proof.
\end{proof}
\bigskip

\subsection{The Slicing Theorem}
The goal of this subsection is to show the following Slicing Theorem, which can be seen as the parabolic analog of \cite[Theorem~1.23]{Cheeger-Naber-Codim4}.
It states a sufficiently precise strong $(n-2, \delta, r)$-splitting map possesses a level set $\Sigma$ at a certain time with the following property:
For all scales $r' \in (0, 10^{-1}r)$ the flow is strongly $(n-2, \eps, r')$-split at all points on $\Sigma$ where $e^{-f/4}$ is bounded from below.
In addition, we will obtain a bound on the local oscillation of $e^{-f/4}$ near such points.

\begin{Proposition}[Slicing Theorem] \label{Prop_Slicing}
If $Y < \infty$, $\eps > 0$ and $\delta \leq \ov\delta (Y, \eps)$, then the following holds.

Let $(M, (g_t)_{t \in I})$ be Ricci flow on a compact manifold, $(x_0,t_0) \in M \times I$ and $r > 0$ such that $\NN_{x_0,t_0} ( r^2) \geq - Y$.
Let $\vec y$ be a strong $(n-2, \delta, r)$-splitting map at $(x_0, t_0)$ and assume that $(x_0, t_0)$ is $(\delta, r)$-static.

Then there are $\vec a \in [-r,r]^{n-2}$, $t' \in [t_0 - 2r^2, t_0 - r^2]$ such that $\Sigma := (\vec y (\cdot, t'))^{-1} (\vec a) \subset M$ is a $2$-dimensional submanifold with the property that for any $x' \in \Sigma$ and $r' \in (0, 10^{-1} r]$ with
\begin{equation} \label{eq_Slicing_condition}
 \NN_{x',t'} (10^{-1} r^2) \geq - Y, \qquad
e^{-f(x',t')/4} \geq \eps  
\end{equation}
the following holds:
\begin{enumerate}[label=(\alph*)]
\item \label{Prop_Slicing_a} There is a lower triangular $(n-2) \times (n-2)$-matrix $A$ with diagonal entries $\geq 1-\eps$ such that $A (\vec y - \vec y(x',t'))$ is a strong $(n-2, \eps, r')$-splitting map at $(x',t')$.
\item \label{Prop_Slicing_b} $(x',t')$ is $(\eps, r')$-static.
\item \label{Prop_Slicing_c} If $r' \leq \rrm (x',t')$, then for any $x'' \in  B^\Sigma (x',t',10^{-1} r') \subset \Sigma$, the time-$t'$-distance ball within $\Sigma$ equipped with the induced length metric, we have
\begin{equation} \label{eq_ef4_ef4_diff}
 \big| e^{-f(x'',t')/4} - e^{-f(x',t')/4} \big| \leq C_0(Y) \bigg( \frac{r'}{r} \bigg)^{ \frac12 - \eps}. 
\end{equation}
\end{enumerate}
\end{Proposition}

Before focusing on the proof of Proposition~\ref{Prop_Slicing}, we need to establish a number of lemmas.

The first lemma will be important in the proof of Proposition~\ref{Prop_Slicing}, as well as in the proof of Lemma~\ref{Lem_improve_rrm_10}.
It provides control near the basepoint $(x_0, t_0)$ over a map $\vec y : M \times [t_0 - \eps^{-1} r^2, t_0 +  r^2] \to \IR^k$ that restricts to a strong splitting map on $M \times [t_0 - \eps^{-1} r^2, t_0 -\eps  r^2]$.
Note that we even obtain control outside $M \times [t_0 - \eps^{-1} r^2, t_0 -\eps  r^2]$, where no bounds are imposed on $\vec y$.

\begin{Lemma} \label{Lem_splitting_map_HE_control_beyond}
If $Y < \infty$, $\zeta > 0$ and $\eps \leq \ov\eps (\zeta)$, then the following holds.
Let $(M, (g_t)_{t \in I})$ be Ricci flow on a compact manifold, $r > 0$ and let $(x_0, t_0) \in M \times I$ be a point with $[t_0 - \eps^{-1} r^2, t_0 + r^2] \subset I$ and $\NN_{x_0,t_0} (\eps^{-1} r^2) \geq - Y$.
Consider a smooth vector valued function $\vec y : M \times [ t_0 - \eps^{-1} r^2, t_0 + r^2] \to \IR^k$ whose components $y_1, \ldots, y_k$ are solutions to the heat equation $\square y_j = 0$.
Assume that $\vec y$ restricted to $M \times [ t_0 - \eps^{-1} r^2, t_0  - \eps r^2]$ is a strong $(k, \eps, r)$-splitting map at $(x_0, t_0)$.
Then for any $r' \in [\zeta r, \zeta^{-1} r]$ we have
\[ \vec y ( P^* (x_0, t_0;  r') ) \subset B^{\IR^k} ( \vec 0, C (Y) r' )  . \]
\end{Lemma}

\begin{proof}
Since $\vec y$ is a strong $(k, \Psi(\eps | \zeta), r')$-splitting map for any $r' \in [\zeta r, \zeta^{-1} r]$, we may reduce the lemma to the case $\zeta = 1$ and $r' = r$.
We may furthermore assume without loss of generality that $r =1$ and $t_0 = 0$.
It is also not hard to see that the lemma can be reduced to the case $k = 1$; we will write $y := y_1$ in the sequel.

Let $(x_1, t_1) \in P^* (x_0, 0; 1)$.
Denote by $d\nu^i = (4\pi \tau_i)^{-n/2} e^{-f_i} dg$, $i=0,1$, the conjugate heat kernels based at $(x_i, t_i)$.
Let $\alpha > 0$ be a constant whose value we will determine later.
By Proposition~\ref{Prop_inheriting_bounds} we find that if $s \leq - C ( \alpha)$, $\eps \leq \ov\eps ( \alpha)$, then
\[ d\nu^1_s \leq C (Y,  \alpha) e^{\alpha f_0} d\nu^0_s. \]
So if $\alpha \leq \ov\alpha$, then Propositions~\ref{Prop_properties_splitting_map}, \ref{Prop_int_ealphaf} imply
\begin{multline*}
 |y|(x_1, t_1) 
= \bigg| \int_M y \, d\nu^1_s \bigg|
\leq \int_M |y| d\nu^1_s
\leq C (Y,  \alpha) \int_M |y| e^{\alpha f_0} d\nu^0_s  \\
\leq C (Y,  \alpha) \bigg( \int_M y^2  d\nu^0_s \bigg)^{1/2} \bigg( \int_M  e^{2\alpha f_0} d\nu^0_s \bigg)^{1/2}
\leq C (Y, \alpha). 
\end{multline*}
This finishes the proof.
\end{proof}

The next lemma is an application of Lemma~\ref{Lem_splitting_map_HE_control_beyond}.
It establishes regularity of the level set $\Sigma$ of $\vec y$ passing through $(x_0, t_0)$ near $(x_0, t_0)$ under the presence of local curvature bounds.

\begin{Lemma} \label{Lem_almost_eucl_slice_loc}
If $Y < \infty$, $\beta > 0$ and $\eps \leq \ov\eps (Y, \beta)$, then the following holds.
Let $(M, (g_t)_{t \in I})$ be Ricci flow on a compact manifold, $r > 0$ and let $(x_0, t_0) \in M \times I$ be a point with $[t_0 - \eps^{-1} r^2, t_0 + r^2] \subset I$ and $\NN_{x_0,t_0} (\eps^{-1} r^2) \geq - Y$.
Consider a smooth vector valued function $\vec y : M \times [ t_0 - \eps^{-1} r^2, t_0 + r^2] \to \IR^{n-2}$ whose components $y_1, \ldots, y_{n-2}$ are solutions to the heat equation $\square y_j = 0$.
Assume that $\vec y$ restricted to $M \times [ t_0 - \eps^{-1} r^2, t_0  - \eps r^2]$ is a strong $(n-2, \eps, r)$-splitting map at $(x_0, t_0)$.
If $r \leq \rrm (x_0, t_0)$, then the following holds:
\begin{enumerate}[label=(\alph*)]
\item $\Sigma := \vec y^{\, -1} ( \vec 0 ) \cap B(x_0, t_0, r) \times \{ t_0 \}$ is a smooth 2-dimensional submanifold.
\item $| \nabla y_i \cdot \nabla y_j - \delta_{ij}| \leq \beta$ along $\Sigma' := B^\Sigma (x_0, t_0, .9 r)$.
\item All principal curvatures at time $t_0$ along $\Sigma'$ are $\leq \beta r^{-1}$.
\item The injectivity radius of the metric on $\Sigma$ induced by $g_{t_0}$ at any point $x' \in \Sigma'$ is $\geq c(Y) r$.
\end{enumerate}
\end{Lemma}

\begin{proof}
This follows using Lemma~\ref{Lem_splitting_map_HE_control_beyond} via a limit argument.
More specifically, assume without loss of generality that $r = 1$, $t_0 = 0$ and consider a sequence of counterexamples $(M_i, (g_{i,t})_{t \in I_i})$, $(x_i, 0) \in M_i \times I_i$, $\vec y_i$, for some fixed $Y, \beta$ and $\eps_i \to 0$.
Due to \cite[\HKThmNLC]{Bamler_HK_entropy_estimates}, we may pass to a subsequence and assume that the flows restricted to the pointed, two-sided parabolic neighborhoods $(P(x_i, 0; 1), (x_i,0))$ converge smoothly to some pointed flow $(M_\infty, (g_{\infty, t})_{t \in (-1,1)}, (x_\infty, 0))$, where $B(x_\infty, 0, r) \subset M_\infty$ is relatively compact for all $r \in (0,1)$ and $B(x_\infty, 0, 1) = M_\infty$.
After passing to another subsequence and using \cite[\HKCenterConstantRmBound]{Bamler_HK_entropy_estimates}, we may moreover assume that the conjugate heat kernels $K(x_i, 0; \cdot, \cdot)$ converge to a positive solution to the conjugate heat equation on $M_\infty \times (-1,0)$.
By \cite[\HKCorPnbhdinPstar]{Bamler_HK_entropy_estimates} and Lemma~\ref{Lem_splitting_map_HE_control_beyond} we obtain that $|\vec y_i|\leq C < \infty$ on $P(x_i, 0; 1)$, which allows us to conclude that, after passing to a subsequence, we have local smooth convergence $\vec y_i \to \vec y_\infty$ on $M_\infty \times (-1,1)$, where $\vec y_\infty$ restriced to $M_\infty \times (-1,0)$ induces a precise local splitting and $\vec y_\infty (x_\infty) = \vec 0$.
It follows that $\vec y_\infty (\cdot, 0)$ induces a precise local splitting and that the its level sets are totally geodesic.
So we obtain a contradiction for sufficiently large $i$.
\end{proof}

We will also need a Vitali covering lemma for $P^*$-parabolic balls. 

\begin{Lemma} \label{Lem_Vitali}
Let $(M, (g_t)_{t \in I})$ be Ricci flow on a compact manifold, $r > 0$, $S \subset M \times I$ a subset and $\rho : S \to (0,r]$ a function.
Suppose that $t - 5 \rho^2(x,t) \in I$ for all $(x,t) \in S$.
Then there is a countable subset $S' \subset S$ such that the $P^*$-parabolic balls $P^* (x, t; \rho(x,t))$, $(x,t) \in S'$ are pairwise disjoint and such that
\begin{equation} \label{eq_S_U_PD}
 S \subset \bigcup_{(x,t) \in S'} P^* (x, t; 5 \rho(x,t)). 
\end{equation}
\end{Lemma}

\begin{proof}
Without loss of generality, we may assume that $r = 1$.
We may inductively choose countable sets $S'_0, S'_1, \ldots \subset S$ such that:
\begin{enumerate}[label=(\arabic*)]
\item \label{propty_S_p_1} $\rho (x,t) \in [2^{-k-1}, 2^{-k}]$ for all $(x,t) \in S'_k$
\item \label{propty_S_p_2} The balls $P^* (x, t; \rho(x,t))$, $(x,t) \in S'_0 \cup \ldots \cup S'_k$, are pairwise disjoint.
\item \label{propty_S_p_3} $S'_k$ is maximal in the sense that for any $(x,t) \in S \setminus (S'_0 \cup \ldots \cup S'_k)$ with $\rho (x,t) \in [2^{-k-1}, 2^{-k}]$ the ball $P^* (x, t; \rho(x,t))$ intersects one of the balls $P^* (x', t'; \rho(x',t'))$, $(x',t') \in S'_0 \cup \ldots \cup S'_k$.
\end{enumerate}
Set $S' := S'_0 \cup S'_1 \cup \ldots$.
It remains to prove (\ref{eq_S_U_PD}).
Assume that $(x,t) \in S \setminus S'$ and choose $k \geq 0$ such that $\rho (x,t) \in [2^{-k-1}, 2^{-k}]$.
Choose $(x',t') \in S'$ according to Property~\ref{propty_S_p_3} and note that $\rho(x',t') \geq \frac12 \rho(x,t)$.
By Proposition~\ref{Prop_basic_parab_nbhd} we have
\begin{equation*}
 (x,t) \in P^* (x,t; \rho(x,t))
 \subset P^* ( x', t'; 2 \rho(x,t) + \rho(x',t')) 
 \subset P^* (x',t'; 5 \rho(x',t')) .
\end{equation*}
This finishes the proof.
\end{proof}

The next lemma provides a global integral bound on a quantity that we will analyze further in the proof of Proposition~\ref{Prop_Slicing}.

\begin{Lemma} \label{Lem_R_omega_n2_omega}
Let $(M, (g_t)_{t \in I})$ be Ricci flow on a compact manifold, $r, \delta, \eps > 0$ and let $(x_0, t_0) \in M \times I$ be a point with $[t_0 - \delta^{-1} r^2, t_0 - \delta r^2] \subset I$ and $\NN_{x_0,t_0} ( r^2) \geq - Y > - \infty$.
Let $d\nu = (4\pi \tau)^{-n/2} e^{-f} dg$ be the conjugate heat kernel based at $(x_0, t_0)$.
Consider a strong $(n-2, \delta, r)$-splitting map $\vec y$ at $(x_0, t_0)$ and set $\omega_l := dy_1 \wedge \ldots \wedge dy_l$ for $l = 1, \ldots, n-2$.
Then (in the sense of weak derivatives)
\[ \int_{t_0 - \eps^{-1} r^2}^{t_0 -\eps r^2} \int_M \bigg( |R| \, |\omega_{n-2}| + \sum_{l=1}^{n-2}    |\square |\omega_l| |  \bigg) d\nu_{t} dt \leq \Psi (\delta | Y, \eps). \]
\end{Lemma}

\begin{proof}
Without  loss of generality, we may assume that $r = 1$ and $t_0 = 0$.
We have using Propositions~\ref{Prop_improved_L2}, \ref{Prop_properties_splitting_map} and the fact that $(x_0, t_0)$ is $(\delta, 1)$-static
\begin{multline*}
  \int_{ - \eps^{-1}}^{-\eps} \int_M |R| \, |\omega_{n-2}| d\nu_{t} dt
= \int_{ - \eps^{-1}}^{-\eps} \int_M |R| d\nu_{t} dt
+ \int_{ - \eps^{-1}}^{-\eps} \int_M |R| \, ( |\omega_{n-2}| - 1 ) d\nu_{t} dt \\
\leq \Psi (\delta |\eps) +
 \bigg(  \int_{ - \eps^{-1}}^{-\eps} \int_M R^2 d\nu_{t} dt \bigg)^{1/2}
\bigg(  \int_{ - \eps^{-1}}^{-\eps} \int_M  (|\omega_{n-2}|-1)^2 d\nu_{t} dt \bigg)^{1/2} 
\leq \Psi (\delta |Y,\eps) .
\end{multline*}

For the bound on $\sum_{l=1}^{n-2} |\square |\omega_l||$ we recall from (\ref{eq_square_omega_Kato_Lp_bound}) and Proposition~\ref{Prop_mu_square_rules} that in the weak sense $\square |\omega_1| \leq 0$ and for $l = 2, \ldots, n-2$
\begin{equation*}
 \square |\omega_l| 
\leq |\square \omega_l|
\leq \bigg( \sum_{i=1}^l |\nabla^2 y_i|^2 \bigg) \bigg(  \sum_{i=1}^l |\nabla y_i|^{l-2} \bigg).
\end{equation*}
So by Proposition~\ref{Prop_properties_splitting_map} we have
\[ \int_{ - \eps^{-1}}^{-\eps} \int_M     |\square \omega_l |   d\nu_{t} dt \leq \Psi (\delta | Y, \eps). \]
It follows that by Propositions~\ref{Prop_weak_square}, \ref{Prop_properties_splitting_map}, see also (\ref{eq_weak_heat_op}),
\begin{align*}
 \int_{ - \eps^{-1}}^{-\eps} \int_M     |\square |\omega_l| |   d\nu_{t} dt
&\leq \int_{ - \eps^{-1}}^{-\eps} \int_M   \big( -  \square |\omega_l| + 2 |\square \omega_l | \big)   d\nu_{t} dt \\
&\leq - \int_{ - \eps^{-1}}^{-\eps} \int_M     \square (|\omega_l|-1)   d\nu_{t} dt + \Psi (\delta | Y, \eps) \\
&=  -\int_M      (|\omega_l|-1)   d\nu_{t} \bigg|_{ t=- \eps^{-1}}^{t=-\eps} + \Psi (\delta | Y, \eps)
\leq \Psi (\delta | Y, \eps).  \qedhere
\end{align*}
\end{proof}
\bigskip

Next we show that the integrals as they appear in Lemma~\ref{Lem_R_omega_n2_omega} don't change much if we vary the basepoint of the background conjugate heat kernel measure slightly.

\begin{Lemma} \label{Lem_second_order_integrals_cont_basepoint}
If $Y < \infty$, $0 < \eps' \leq \eps \leq \ov\eps (Y)$, $\beta > 0$ and $\zeta \leq \ov\zeta (Y, \eps' , \beta)$, then the following holds.
Let $(M, (g_t)_{t \in I})$ be Ricci flow on a compact manifold, $r > 0$ and let $(x_0, t_0) \in M \times I$ be a point with $[t_0 - \eps^{\prime -1} r^2, t_0- \eps' r^2] \subset I$ and $\NN_{x_0,t_0} ( r^2) \geq - Y$.
Consider a smooth map $\vec y : M \times [t_0 - \eps^{\prime -1} r^2, t_0 - \eps' r^2] \to \IR^{n-2}$ whose component functions $y_1, \ldots, y_{n-2}$ satisfy the heat equation $\square y_j = 0$.
Assume that $\vec y$ restricted to $M \times [t_0 - \eps^{ -1} r^2, t_0 - \eps r^2]$ is a strong $(n-2, \eps, r)$-splitting map at $(x_0, t_0)$.

Let $(x_1, t_1) \in P^* (x_0, t_0; \zeta r)$ and denote by $d\nu^i = (4\pi \tau_i)^{-n/2} e^{-f_i} dg$ the conjugate heat kernel based at $(x_i, t_i)$, $i = 0,1$.
Then for $l = 1, \ldots, n-2$
\begin{align}
  \bigg| \int_{t_0 -   r^2}^{t_0 - \eps' r^2} \int_M  |R| \, |\omega_{n-2}|  d\nu^0_t dt -  \int_{t_0 -   r^2}^{t_0 - \eps' r^2} \int_M  |R| \, |\omega_{n-2}|  d\nu^1_t dt \bigg| &\leq \beta, \label{eq_second_order_cont_1} \\
 \bigg| \int_{t_0 -   r^2}^{t_0 - \eps' r^2} \int_M  |\square |\omega_l| |  d\nu^0_t dt -  \int_{t_0 -   r^2}^{t_0 - \eps' r^2} \int_M   |\square |\omega_l| |  d\nu^1_t dt \bigg| &\leq \beta . \label{eq_second_order_cont_2}
\end{align}
\end{Lemma}

\begin{Remark}
We are purposefully not assuming that $\vec y$ is an $(n-2, \eps', r)$-splitting map.
\end{Remark}

\begin{proof}
Without loss of generality, we may assume that $r=1$ and $t_0 = 0$.
Fix $Y,  \eps', \eps, \beta$.
We will determine $\zeta$ and impose conditions of the form $\eps \leq \ov\eps (Y)$ in the course of the proof.
Let $(x_1, t_1) \in P^* (x_0, 0; \zeta)$ and let $\alpha > 0$ be a constant whose value we will determine later.

\begin{Claim} \label{Cl_diff_nu01}
For all $t \in [-1, -\eps']$
\[ d |\nu^0_t - \nu^1_t| \leq \Psi (\zeta |Y, \eps', \alpha) e^{\alpha f_0} d\nu^0_t. \]
\end{Claim}

\begin{proof}
Fix some $(x',t') \in M \times [-1, -\eps']$.
Our goal is to show that if $\delta > 0$ and $\zeta \leq  \ov\zeta ( Y,  \eps', \alpha, \delta)$, then
\begin{equation} \label{eq_eq_Kx0x1_diff_goal}
 \big| K(x_0, 0; x',t') - K(x_1, t_1; x',t') \big| \leq \delta K^{1-\alpha} (x_0, 0; x',t'). 
\end{equation}
Fix $\delta$ for the remainder of the proof.

By Proposition~\ref{Prop_inheriting_bounds}, we obtain that if $\zeta \leq  \ov\zeta (   \eps', \alpha)$, then
\begin{equation} \label{eq_osc_K_C}
K(x_1, t_1; x',t') \leq C (Y, \eps', \alpha) K^{1-\alpha/2} (x_0, 0; x',t'). 
\end{equation}
So if $C(Y,  \eps', \alpha) K^{\alpha/2} (x_0, 0; x',t') + K^{\alpha} (x_0, 0; x',t') \leq \delta$, then (\ref{eq_eq_Kx0x1_diff_goal}) follows from (\ref{eq_osc_K_C}).
Now assume that the reverse inequality holds, which implies
\begin{equation} \label{eq_Kx00xptp_lower_bound_delta}
   K (x_0, 0; x',t') \geq c(Y, \eps', \alpha, \delta) . 
\end{equation}
Using \cite[\HKThmHKboundGauss]{Bamler_HK_entropy_estimates}, we obtain that
\[ d^{g_{t'}}_{W_1} (\delta_{x'}, \nu^0_{t'} ) \leq C (Y, \eps', \alpha, \delta). \]
Thus, by Proposition~\ref{Prop_NN_variation_bound} we have
\[ \NN_{x',t'} ( 1) \geq - C (Y, \eps', \alpha, \delta). \]
Again by  \cite[\HKThmHKboundGauss]{Bamler_HK_entropy_estimates}, this implies that for any $t \in (t', 0]$
\[ K(\cdot ,t; x', t') \leq C (Y, \eps', \alpha, \delta) (t - t')^{-n/2}. \]
Thus, by a gradient estimate, such as \cite[\HKThmGradientPhiEstimate]{Bamler_HK_entropy_estimates}, we obtian
\[ |\nabla K| (\cdot, - \zeta^2; x', t') \leq C (Y, \eps', \alpha, \delta). \]
Therefore, by the reproduction formula
\begin{multline*}
 \big| K(x_0, 0; x',t') - K(x_1, t_1; x',t') \big|
= \bigg| \int_{M}  K (\cdot, - \zeta^2; x', t') d\nu^0_{-\zeta^2} - \int_{M}  K (\cdot, - \zeta^2; x', t') d\nu^1_{-\zeta^2} \bigg| \\
\leq C (Y, \eps', \alpha, \delta) \zeta. 
\end{multline*}
Combining this with (\ref{eq_Kx00xptp_lower_bound_delta}) implies (\ref{eq_eq_Kx0x1_diff_goal}) for $\zeta \leq  \ov\zeta ( Y,  \eps', \alpha, \delta)$.
\end{proof}

\begin{Claim} \label{Cl_y_i_bounds}
If $\alpha \leq \ov\alpha$, $Q < \infty$, $\eps \leq \ov\eps (Y,  Q)$, then for all $t \in [-1, -\eps']$, $p \in [0, Q]$
\begin{align} \label{eq_naby_p_all_t}
 \sum_{j=1}^{n-2} \int_M |\nabla y_j|^p   d\nu^0_t  &\leq C,  \\
\int_{-1}^{-\eps'} \int_M \bigg( \sum_{i=1}^{n-2} |\nabla^2 y_i|^2 \bigg) \bigg( \sum_{j=1}^{n-2} |\nabla y_j|^p \bigg) e^{\alpha f_0}  d\nu^0_t dt &\leq C(Y, \eps', Q). \label{eq_nab2yinabyj_beyond}
\end{align}
\end{Claim}

\begin{proof}
To see (\ref{eq_naby_p_all_t}) we may assume that $p \geq 1$ due to H\"older's inequality.
If $\alpha \leq \ov\alpha$, $\eps \leq \ov\eps (Y,  Q)$, then using Proposition~\ref{Prop_properties_splitting_map}, we have for all $p \in [1, Q]$
\[ \int_M \bigg( \sum_{j=1}^{n-2} |\nabla y_j|^p \bigg)  d\nu_{-1} \leq C  . \]
Since by Kato's inequality $\square |\nabla y_j|^p \leq p |\nabla y_j|^{p-1} \square |\nabla y_j| \leq p |\nabla y_j|^{p-1} |\square \nabla y_j| = 0$, we obtain that for any $t^* \in [-1, - \eps']$
\[ \int_M \bigg( \sum_{j=1}^{n-2} |\nabla y_j|^p \bigg) d\nu_{t^*} =  \int_M \bigg( \sum_{j=1}^{n-2} |\nabla y_j|^p \bigg)  d\nu_{-1} + \int_{-1}^{t^*} \int_M \bigg( \sum_{j=1}^{n-2} \square |\nabla y_j|^p \bigg)  d\nu_{t} dt \leq C. \]
This proves (\ref{eq_naby_p_all_t}).

To see (\ref{eq_nab2yinabyj_beyond}) it suffices to show that for $p = 0$ or $p \geq 2$ and for any $i,j = 1, \ldots, n-2$
\[ \int_{ - 1 }^{ - \eps' } \int_M  |\nabla^2 y_i |^2   |\nabla y_j |^{p}     e^{\alpha f_0} d\nu^0_{t} dt \leq C(Y,\eps', p). \]
This bound is similar to (\ref{eq_nab2yypealphC}) in the proof of Proposition~\ref{Prop_properties_splitting_map}.
However, the fact that  the upper bound of the first integral is $\eps'$ creates some difficulties, which we will overcome by modifying the arguments from the proof of Proposition~\ref{Prop_properties_splitting_map}.
By the same computation as in (\ref{eq_ealphf_square_nab2_nab_p}) we find that for $\delta^*_{ij} = 1 - \delta_{ij}$
\begin{multline*}
 e^{-\alpha f_0} \square \big(  |\nabla y_i|^{2}  |\nabla y_j|^{p} e^{\alpha f_0} \big) 
\leq - \tfrac12 c ( p)  |\nabla^2 y_i|^2  |\nabla y_j|^{p} + C( p ) \delta^*_{ij} \big( |\nabla^2 y_i|^2 |\nabla y_i|^{2} +  |\nabla^2 y_j|^2 |\nabla y_j|^{2p-2} \big)  \\
 + C(\eps', p) \big(  |\nabla y_i|^{8} + |\nabla y_j|^{4p} +  |\nabla^2 f_0|^2 +  |\nabla f_0|^4 + R^2 + 1 \big)   . 
\end{multline*}
Therefore, we obtain using Propositions~\ref{Prop_improved_L2}, \ref{Prop_properties_splitting_map} and (\ref{eq_naby_p_all_t}), similarly as in (\ref{eq_int_nab2nab_p})
\begin{align}
  \int_{ - 1}^{ - \eps' } \int_M   & |\nabla^2 y_i|^2  |\nabla y_j|^{p}  e^{\alpha f_0} d\nu^0_t dt
\leq  C(\eps', p) \int_M   |\nabla y_i|^2  |\nabla y_j|^{p}  e^{\alpha f_0} d\nu^0_{-1}  + C(Y, \eps', p) \notag \\
&\qquad  + 
C(\eps',  p) \delta^*_{ij} \int_{ -1 }^{ - \eps' } \int_M  \big( |\nabla^2 y_i|^2 |\nabla y_i|^{2} +  |\nabla^2 y_j|^2 |\nabla y_j|^{2p-2} \big) e^{\alpha f_0} d\nu^0_t dt \notag \displaybreak[1] \\
&\leq C(Y, \eps', p)+ 
C(\eps',  p) \delta^*_{ij} \int_{ -1 }^{ - \eps' } \int_M  \big( |\nabla^2 y_i|^2 |\nabla y_i|^{2} +  |\nabla^2 y_j|^2 |\nabla y_j|^{2p-2} \big) e^{\alpha f_0} d\nu^0_t dt.  \label{eq_int_m1_epsp_nab2nab}
\end{align}
This proves (\ref{eq_nab2yinabyj_beyond}) for $i =j$ and for $p=0$.
Given (\ref{eq_nab2yinabyj_beyond}) for $i=j$, we can also bound the last integral in (\ref{eq_int_m1_epsp_nab2nab}) and deduce (\ref{eq_nab2yinabyj_beyond}) for $i \neq j$.
\end{proof}

Using the bounds from Claim~\ref{Cl_y_i_bounds} and Proposition~\ref{Prop_improved_L2}, we obtain
\begin{multline*}
  \int_{ - 1}^{ - \eps'} \int_M   |R| \, |\omega_{n-2}| e^{\alpha f_0} d\nu^0_t dt \leq \bigg( \int_{ -  1}^{ - \eps'} \int_M   R^2  e^{2\alpha f_0} d\nu^0_t dt \bigg)^{1/2} \bigg( \int_{ -  1}^{ - \eps'} \int_M    |\omega_{n-2}|^2 d\nu^0_t dt \bigg)^{1/2} \\
\leq C(Y, \eps') \bigg( \sum_{i=1}^{n-2} \int_{ -  1}^{ - \eps'} \int_M    |\nabla y_i|^{2(n-2)} d\nu^0_t dt \bigg)^{1/2}
\leq C(Y, \eps') .
\end{multline*}
Combining this with Claim~\ref{Cl_diff_nu01} implies (\ref{eq_second_order_cont_1}).

Next, we have using Claim~\ref{Cl_y_i_bounds} and (\ref{eq_square_omega_Kato_Lp_bound}) $\square \omega_1 = 0$ and for $l \geq 2$
\begin{equation*}
\int_{ -  1}^{ -  \eps'} \int_M   |\square  \omega_{l}| e^{\alpha f_0} d\nu^0_t dt
\leq \int_{ -1}^{ -  \eps'} \int_M  \bigg( \sum_{i=1}^l |\nabla^2 y_i|^2 \bigg) \bigg(  \sum_{i=1}^l |\nabla y_i|^{l-2} \bigg) e^{\alpha f_0} d\nu^0_t dt \leq C(Y, \eps'). 
\end{equation*}
Therefore, since $0 \leq |\square |\omega_l|| + \square |\omega_l| \leq 2 (\square | \omega_l|)_+ \leq 2 |\square \omega_l|$, and using Claims~\ref{Cl_diff_nu01}, \ref{Cl_y_i_bounds}
\begin{align*}
  \bigg| \int_{ - 1 }^{ - \eps'} \int_M &  |\square |\omega_l| |  d\nu^0_t dt -  \int_{   -1 }^{ - \eps' } \int_M   |\square |\omega_l| |  d\nu^1_t dt \bigg| \\
&\leq \bigg| \int_{ -1}^{ - \eps'} \int_M  \big( |\square |\omega_l|| +  \square | \omega_l|     \big) d(\nu^0_t - \nu^1_t) dt \bigg|
+ \bigg| \int_{ - 1 }^{ - \eps'} \int_M    \square | \omega_l|  d(\nu^0_t - \nu^1_t) dt \bigg|  \\
&\leq  \int_{ - 1 }^{ - \eps'} \int_M  2 |\square \omega_l| d|\nu^0_t - \nu^1_t| dt
+  \bigg| \int_M    |\omega_l|   d(\nu^0_t - \nu^1_t)   \bigg|_{t= - 1 }^{ t=-  \eps'} \bigg| \\
&\leq \Psi (\zeta | Y, \eps', \alpha) \bigg( \int_{ - 1}^{ - \eps'} \int_M  2 |\square \omega_l|  e^{\alpha f_0} d\nu^0_t dt
+  \int_M    |\omega_l|   e^{\alpha f_0} d\nu^0_{-1}  +  \int_M    |\omega_l|   e^{\alpha f_0} d\nu^0_{-\eps'} \bigg) \\
&\leq \Psi (\zeta | Y, \eps', \alpha) .
\end{align*}
This implies (\ref{eq_second_order_cont_2}).
\end{proof}

We can now prove the Slicing Theorem, Proposition~\ref{Prop_Slicing}.

\begin{proof}[Proof of Proposition~\ref{Prop_Slicing}.]
Without loss of generality, we may assume that $r = 1$ and $t_0 = 0$.
Fix $Y, \eps$.
Throughout the proof we will assume that the constant $\eps$ is chosen smaller than certain constants that only depend on $Y$.
This does not create any issues, since the proposition is stronger for smaller $\eps$ and the only dependences in the proposition are of the form $\delta \leq \ov\delta (Y, \eps)$ and $C_0 (Y)$.
For simplicity, we will also show Assertion~\ref{Prop_Slicing_a} for $\eps$ replaced with $\frac{n}4 \eps$, which does not create any issues either, because we can adjust $\eps$ by a dimensional factor.

Denote by $d\nu = (4\pi \tau)^{-n/2} e^{-f}dg$ the conjugate heat kernel based at $(x_0, 0)$ and set $\omega_l := dy_1 \wedge \ldots \wedge dy_l$, $l = 1, \ldots, n-2$.

\begin{Claim} \label{Cl_splitting_macr_scale}
If $(x',t') \in M \times [-2,-1]$ satisfies (\ref{eq_Slicing_condition}), then $\vec y - \vec y(x',t')$ is a strong $(n-2, \Psi (\delta | Y, \eps), 1)$-splitting map at $(x',t')$.
\end{Claim}

\begin{proof}
By \cite[\HKThmHKboundGauss]{Bamler_HK_entropy_estimates} and (\ref{eq_Slicing_condition}) we have
\[ d_{W_1}^{g_{t'}} ( \delta_{x'} , \nu_{t'} ) \leq C(Y, \eps). \]
Thus by Proposition~\ref{Prop_inheriting_bounds} if $\delta' > 0$ and $\delta \leq \ov\delta(\delta')$, then for any $\alpha > 0$ and $t \in [t'-\delta^{\prime -1}, t' - \delta']$
\[ d\nu_{x',t'; t} \leq C(Y, \eps, \alpha, \delta') e^{\alpha f} d\nu_t. \]
It follows using Proposition~\ref{Prop_properties_splitting_map} that if $\alpha \leq \ov\alpha$, then for $i, j = 1, \ldots, n-2$
\[ \int_{t' - \delta^{\prime -1}}^{t' - \delta'} \int_M | \nabla y_i \cdot \nabla y_j - \delta_{ij} | d\nu_{x',t'; t} dt
\leq  C(Y, \eps, \alpha) \int_{t' - \delta^{\prime -1}}^{t' - \delta'} \int_M | \nabla y_i \cdot \nabla y_j - \delta_{ij} | e^{\alpha f} d\nu_t dt
\leq \Psi (\delta | Y, \eps, \alpha, \delta' ). \]
This proves the claim.
\end{proof}

By Lemma~\ref{Lem_R_omega_n2_omega} we have
\begin{equation} \label{eq_R_sq_omega_small}
 \int_{ - 10}^{-1} \int_M \bigg( |R| \, |\omega_{n-2}| + \sum_{l=1}^{n-2}    |\square |\omega_l| |  \bigg) d\nu_{t} dt \leq \Psi (\delta | Y). 
\end{equation}

Let $\eps' \in (0, \eps)$ be a constant whose value we will determine later.
For any $(x',t') \in M \times [-2,-1]$ choose $r_{x',t'} \in [0, 10^{-1}]$ minimal with the property that for all $r' \in [r_{x',t'}, 10^{-1})$ and $l = 1, \ldots, n-2$ we have 
\begin{align*}
  \int_{t'-   r^{\prime 2}}^{t'- \frac12 \eps' r^{\prime 2}} \int_M   |R| \, |\omega_{n-2}| d\nu_{x',t'; t} dt
&\leq \eps'  r^{\prime -2} \int_{t'-  r^{\prime 2}}^{t'-  r^{\prime 2}/2} \int_M  |\omega_{n-2}|  d\nu_{x',t'; t} dt, \\
  \int_{t'-  r^{\prime 2}}^{t'- \frac12 \eps' r^{\prime 2}} \int_M     |\square |\omega_l| |   d\nu_{x',t'; t} dt
&\leq \eps' r^{\prime - 2}  \int_{t'-  r^{\prime 2}}^{t'-  r^{\prime 2}/2} \int_M  |\omega_l|  d\nu_{x',t'; t} dt.
\end{align*}
So, by continuity, if $r_{x',t'} > 0$, then equality holds in one of these inequalities for $r' = r_{x',t'}$.

\begin{Claim} \label{Cl_rprime_A}
If $\eps \leq \ov\eps (Y)$, $\eps' \leq \ov\eps' (Y, \eps)$, $\delta \leq \ov\delta (Y, \eps, \eps')$, then for any $(x',t') \in M \times [-2,-1]$ that satisfies (\ref{eq_Slicing_condition}) and any $r' \in [r_{x',t'}, 10^{-1}] \setminus \{ 0\}$ the following holds:
\begin{enumerate}[label=(\alph*)]
\item \label{Cl_rprime_A_a} There is a lower triangular $(n-2) \times (n-2)$-matrix $A$ with diagonal entries $\geq 1- \eps$ and $|A| \leq 10 r^{\prime - \eps}$ such that $\vec y^{\, \prime} := A (\vec y - \vec y (x',t'))$ is a strong $(n-2,  \eps,  r')$-splitting map at $(x',t')$.
\item \label{Cl_rprime_A_b} $(x',t')$ is $( \eps,  r')$-static.
\end{enumerate}
\end{Claim}

\begin{proof}
By Claim~\ref{Cl_splitting_macr_scale} and the Transformation Theorem, Proposition~\ref{Prop_Transformation_Theorem}, if $\delta \leq \ov\delta (Y, \eps, \eps')$, then there is a lower triangular matrix $A$ with diagonal entries $\geq 1- \Psi(\eps' | Y)$ such that $\vec y^{\,\prime} := A ( \vec y - \vec y ( x', t'))$ is a strong $(n-2, \lb \Psi(\eps' | Y), r')$-splitting map at $(x',t')$ and such that $|A| \leq 10 r^{\prime - \eps}$.
This proves Assertion~\ref{Cl_rprime_A_a} for $\eps' \leq \ov\eps' (Y, \eps)$.

To see Assertion~\ref{Cl_rprime_A_b}, set $\omega'_{n-2} := dy'_1 \wedge \ldots \wedge d y'_{n-2}$.
Then by Proposition~\ref{Prop_properties_splitting_map}
\[ \int_{t'-  r^{\prime 2}}^{t'-  \frac12\eps' r^{\prime 2}} \int_M   |R| \, |\omega'_{n-2}|  d\nu_{x',t'; t} dt
\leq \eps' r^{\prime -2}  \int_{t'-  r^{\prime 2}}^{t'-  r^{\prime 2}/2} \int_M  |\omega'_{n-2}|  d\nu_{x',t'; t} dt
\leq \Psi(\eps' | Y) . \]
It follows, using Propositions~\ref{Prop_improved_L2}, \ref{Prop_properties_splitting_map}, that if $\eps' \leq \ov\eps' (\eps)$
\begin{align*}
 & \int_{t'- \eps r^{\prime 2}}^{t'- \frac12 \eps r^{\prime 2}}   \int_M   R \,   d\nu_{x',t'; t} dt
\leq  \Psi(\eps' | Y) +  \int_{t'- \eps r^{\prime 2}}^{t'- \frac12 \eps r^{\prime 2}} \int_M   |R| \, ( 1 - |\omega'_{n-2}| )  d\nu_{x',t'; t} dt \\
&\leq \Psi(\eps' | Y) + \bigg( r^{\prime 2} \int_{t'- \eps r^{\prime 2}}^{t'- \frac12 \eps r^{\prime 2}} \int_M   R^2  \, d\nu_{x',t'; t} dt \bigg)^{1/2}
\bigg(r^{\prime -2} \int_{t'- \eps r^{\prime 2}}^{t'- \frac12 \eps r^{\prime 2}} \int_M  ( 1 - |\omega'_{n-2}| )^2  d\nu_{x',t'; t} dt \bigg)^{1/2} \\
&\leq \Psi (\eps' | Y, \eps). 
\end{align*}
Pick $t^* \in [t'-\eps r^{\prime 2},t' -\frac12 \eps r^{\prime 2}]$ with
\[  r^{\prime 2} \int_M   R  \, d\nu_{x',t'; t}  \leq \Psi (\eps' | Y, \eps). \]
Then since $R \geq - \Psi (\delta | \eps)$ at time $t' - \eps^{-1} r^{\prime 2} \geq -2 - \eps^{-1}$
\begin{multline*}
 2 r^{\prime 2} \int_{t' - \eps^{-1} r^{\prime 2}}^{t^*} \int_M |{\Ric}|^2 d\nu_{x',t';t} dt 
= r^{\prime 2} \int_{t' - \eps^{-1} r^{\prime 2}}^{t^*} \int_M \square R \, d\nu_{x',t';t} dt \\
= r^{\prime 2} \int_M  R \, d\nu_{x',t';t} \bigg|_{t' - \eps^{-1} r^{\prime 2}}^{t^*}
\leq \Psi (\eps' | Y, \eps) + \Psi (\delta | \eps)
. 
\end{multline*}
This shows Assertion~\ref{Cl_rprime_A_b} if $\eps' \leq \ov\eps' (Y, \eps)$, $\delta \leq \ov\delta (Y, \eps, \eps')$.
\end{proof}
\medskip

Let $S \subset M \times [-2,-1]$ be the set of all points that satisfy (\ref{eq_Slicing_condition}) but violate Assertion~\ref{Prop_Slicing_a} or \ref{Prop_Slicing_b} of the proposition for some $r' \in (0, 10^{-1} ]$.
Note that by Claim~\ref{Cl_rprime_A} we must have $r_{x',t'} > 0$ for all $(x', t') \in S$.

\begin{Claim} \label{Cl_splitting_x_pp}
If $\eps \leq \ov\eps (Y)$, $\zeta \leq \ov\zeta (Y, \eps, \eps')$, $\zeta' \leq \ov\zeta' (Y, \eps, \eps', \zeta)$ and $\omega \leq \ov\omega (Y, \lb  \eps,  \lb \zeta')$, then the following holds.

If $(x',t') \in M \times [-2,-1]$ satisfies (\ref{eq_Slicing_condition}) and $r' \in [r_{x',t'}, 10^{-1}] \setminus \{ 0\}$, then there is a lower triangular $(n-2) \times (n-2)$-matrix $A$ such that for any $(x'', t'') \in P^* (x', t'; 10 \zeta r')$ we have:
\begin{enumerate}[label=(\alph*)]
\item \label{Cl_splitting_x_pp_a} $A$ has diagonal entries $\geq 1-\eps$ and $|A| \leq 10 r^{\prime - \eps}$.
\item \label{Cl_splitting_x_pp_b} $A (\vec y - \vec y (x',t'))$ is a strong $(n-2,  \eps, r')$-splitting map at $(x',t')$.
\item \label{Cl_splitting_x_pp_c} $(x',t')$ is $( \eps,  r')$-static.
\item \label{Cl_splitting_x_pp_d} If $r_{x', t'} > 0$, then
\[  \int_{t''- 2 r^2_{x',t'}}^{t''-  \frac14 \eps' r^2_{x',t'}} \int_M  \bigg( |R| \, |\omega_{n-2}| + \sum_{l=1}^{n-2}    |\square |\omega_l| |  \bigg) d\nu_{x'',t''; t} dt
\geq  \frac{\eps'}{10} (\det A)^{-1}. \]
\item \label{Cl_splitting_x_pp_e} For all $t \in [t' - (\zeta' r')^2 , t' - \frac12 (\zeta' r')^2]$ we have
\[ \nu_t ( P^* (x',t';  \zeta r') \cap M \times \{ t \} ) \geq \omega  r^{\prime n}. \]
\end{enumerate}
\end{Claim}

\begin{proof}
Let $A$ be the matrix from Claim~\ref{Cl_rprime_A} and set $\vec y^{\, \prime} := A (\vec y - \vec y (x',t'))$ and $\omega'_l := dy'_1 \wedge \ldots \wedge dy'_l$, $l = 1, \ldots, n-2$.
Then Assertions~\ref{Cl_splitting_x_pp_a}--\ref{Cl_splitting_x_pp_c} are clear.
If  $\eps \leq \ov\eps (Y)$, $\zeta \leq \ov\zeta (Y, \eps, \eps')$, then we  obtain from Lemma~\ref{Lem_second_order_integrals_cont_basepoint} and Proposition~\ref{Prop_properties_splitting_map} that one of the following two strings of inequalities holds for some $l \in \{ 1, \ldots, n-2 \}$, where $A_l$ denotes the upper left $l\times l$-submatrix of $A$:
\begin{align*}
  \int_{t''- 2 r^2_{x',t'}}^{t''-  \frac14 \eps' r^2_{x',t'}} & \int_M   |R| \, |\omega_{n-2}| d\nu_{x'',t''; t} dt \\
&= (\det A)^{-1}    \int_{t''- 2 r^2_{x',t'}}^{t''-  \frac14 \eps' r^2_{x',t'}}  \int_M   |R| \, |\omega'_{n-2}| d\nu_{x'',t''; t} dt \displaybreak[1] \\
&\geq (\det A)^{-1}  \bigg(  \int_{t'-  r^2_{x',t'}}^{t'-  \frac12 \eps' r^2_{x',t'}}  \int_M   |R| \, |\omega'_{n-2}| d\nu_{x',t'; t} dt -  \frac{\eps'}{10} \bigg) \displaybreak[1]  \\
&=    \int_{t'-  r^2_{x',t'}}^{t'-  \frac12 \eps' r^2_{x',t'}}  \int_M   |R| \, |\omega_{n-2}| d\nu_{x',t'; t} dt -  \frac{\eps'}{10} (\det A )^{-1} \displaybreak[1]  \\
&=  \eps' r^{-2}_{x',t'}  \int_{t'-  r^2_{x',t'}}^{t'-  r^2_{x',t'}/2}  \int_M    |\omega_{n-2}| d\nu_{x',t'; t} dt -  \frac{\eps'}{10} (\det A )^{-1} \displaybreak[1]  \\
&= (\det A)^{-1}  \eps' r^{-2}_{x',t'}  \int_{t'-  r^2_{x',t'}}^{t'-  r^2_{x',t'}/2}  \int_M    |\omega'_{n-2}| d\nu_{x',t'; t} dt -  \frac{\eps'}{10} (\det A )^{-1}  \displaybreak[1]  \\
&\geq   (\det A)^{-1} \frac{\eps'}{4}  -  \frac{\eps'}{10} (\det A )^{-1}
\geq  \frac{\eps'}{10} (\det A)^{-1}.
\end{align*}
\begin{align*}
  \int_{t''- 2 r^2_{x',t'}}^{t''-  \frac14 \eps' r^2_{x',t'}}  \int_M    |\square |\omega_l| |   d\nu_{x'',t''; t} dt 
&= (\det A_l)^{-1}  \int_{t''- 2 r^2_{x',t'}}^{t''-  \frac14 \eps' r^2_{x',t'}}  \int_M    |\square |\omega'_l| |   d\nu_{x'',t''; t} dt \\
&\geq (\det A_l)^{-1} \bigg( \int_{t'-  r^2_{x',t'}}^{t'-  \frac12 \eps' r^2_{x',t'}}  \int_M    |\square |\omega'_l| |   d\nu_{x',t'; t} dt -  \frac{\eps'}{10} \bigg) \displaybreak[1]  \\
&=  \int_{t'-  r^2_{x',t'}}^{t'-  \frac12 \eps r^2_{x',t'}}  \int_M    |\square |\omega_l| |   d\nu_{x',t'; t} dt -  \frac{\eps'}{10} (\det A_l)^{-1} \displaybreak[1]  \\
&=  \eps'  r^{-2}_{x',t'} \int_{t'-  r^2_{x',t'}}^{t'-   r^2_{x',t'}/2}  \int_M   |\omega_l|    d\nu_{x',t'; t} dt - \frac{\eps'}{10} (\det A_l)^{-1} \displaybreak[1]  \\
&= (\det A_l)^{-1}  \eps'  r^{-2}_{x',t'} \int_{t'-  r^2_{x',t'}}^{t'-   r^2_{x',t'}/2}  \int_M   |\omega'_l|    d\nu_{x',t'; t} dt - \frac{\eps'}{10} (\det A_l)^{-1} \displaybreak[1]  \\
&\geq  (\det A_l)^{-1} \frac{\eps'}{4}  -  \frac{\eps'}{10} (\det A_l)^{-1}   \\
&\geq \frac{ \eps'}{5} (\det A_l)^{-1}  
\geq \frac{\eps'}{5} (1-\eps)^{n-2} (\det A)^{-1}  
\geq  \frac{\eps'}{10} (\det A)^{-1}  .
\end{align*}
This implies Assertion~\ref{Cl_splitting_x_pp_d}.

For Assertion~\ref{Cl_splitting_x_pp_e} let $0 < \zeta' \leq \zeta$, $ Z > 0$ be constants whose values we will determine later and assume that $t \in [t' - (\zeta' r')^2 , t' - \frac12 (\zeta' r')^2]$.
Choose $x \in M$ such that $(x,t)$ is an $H_n$-center of $(x',t')$.
By Lemma~\ref{Lem_mass_ball_Var} we have for $B := B(x,t,\sqrt{2H_n} \zeta' r')$
\[ \nu_{x',t';t}(B) \geq  \frac12. \]
Next, recall from \cite[Sec.~9]{Perelman1} (see also (\ref{eq_potential_evolution_equation})) that
\[ \square (\tau f) = - f + \tau \partial_t f - \tau \triangle f
=
- f - 2 \tau \triangle f + \tau |\nabla f|^2 - \tau R + \frac{n}2
\geq - \frac{n}2. \]
So since $f(x',t') \leq - 4 \log \eps$ and $f (\cdot, t) \geq - C(Y)$ on $M$ by Proposition~\ref{Prop_L_infty_HK_bound}, we obtain that for $\zeta' \leq \ov\zeta' (\eps)$
\[ \int_M  f_+ \, d\nu_{x',t';t}
\leq  \int_M (\tau  f)_+ \, d\nu_{x',t';t} 
\leq \int_M \tau f \, d\nu_{x',t';t} + C(Y)
\leq \tau f(x',t') + \frac{n}2 (t' - t) + C(Y) \leq C(Y, \eps). \]
It follows that
\[ \nu_{x',t';t} (\{ f \geq Z \}) \leq C(Y, \eps) Z^{-1}. \]
So if $Z \geq \underline{Z}(Y, \eps)$, then
\[ \nu_{x',t';t} \big( \{ f \leq Z \} \cap B \big) \geq \frac14. \]
Thus, using \cite[\HKThmLinftyKbound]{Bamler_HK_entropy_estimates}, we obtain
\[ \big| \{ f \leq Z \} \cap B  \big|_{t} \geq 
c(Y) (\zeta' r')^n \nu_{x',t';t} \big( \{ f \leq Z \} \cap B \big)
\geq c(Y) (\zeta' r')^n. \]
This implies that
\[ \nu_{t} (B) \geq \int_{ \{ f \leq Z \} \cap B} (4\pi \tau)^{-n/2} e^{-f} dg_t
\geq c(Y,Z)  (\zeta' r')^n. \]
Lastly, we claim that for $\zeta' \leq \ov\zeta' ( \zeta )$ we have $B \times \{ t \} \subset P^* (x',t'; \zeta r')$.
To see this, observe that for any $y \in B$ we have for $\zeta' \leq \ov\zeta' ( \zeta )$
\begin{multline*}
 d_{W_1}^{g_{t' - (\zeta r')^2}} (\nu_{x',t'; t' - (\zeta r')^2}, \nu_{y,t; t' - (\zeta r')^2} )
\leq d_{W_1}^{g_{t}} (\nu_{x',t';t},\delta_y)
\leq d_{W_1}^{g_{t}} (\nu_{x',t';t},\delta_x) + \sqrt{2H_n} \zeta' r' \\
\leq \sqrt{\Var  (\nu_{x',t';t},\delta_x)} +  \sqrt{2H_n} \zeta' r'
\leq  10 \sqrt{H_n} \zeta' r'
\leq \zeta r' .
\end{multline*}
This finishes the proof for $\omega \leq \ov\omega ( Y, \eps, \zeta')$.
\end{proof}

Assume in the following that $\zeta' \leq \zeta \leq 10^{-3}$ and $(\zeta')^2 \leq 10^{-2} \eps$.
By Lemma~\ref{Lem_Vitali} we can find a subset $S' \subset S$ such that the $P^*$-parabolic balls $P^* (x', t';   \zeta r_{x',t'})$, $(x', t') \in S'$, are pairwise disjoint and such that 
\begin{equation} \label{eq_S_covered}
 S \subset \bigcup_{(x',t') \in S'} P^* (x', t'; 10 \zeta r_{x',t'}). 
\end{equation}
Consider the smooth map 
\[ \phi : M \times [-2, -1] \to \IR^{n-2} \times [-2, -1], \qquad (x,t) \mapsto (\vec y (x), t). \]
For any $(x', t') \in S'$ we have by Lemma~\ref{Lem_splitting_map_HE_control_beyond} and Claim~\ref{Cl_splitting_x_pp} for $\eps \leq \ov\eps (Y)$
\begin{align}
 |\phi ( P^* (x', t';  10 \zeta r_{x',t'}) )|
&\leq C(Y) r_{x',t'}^{n+2} (\det A)^{-1} \notag \\
&\leq C(Y, \omega, \eps') \int_{t' - (\zeta' r_{x',t'})^2}^{t' -  \frac12 (\zeta' r_{x',t'})^2} 
\int_{P^* (x',t'; \zeta r') \cap M \times \{ t'' \}} \notag \\
&\qquad\int_{t'- 2 r^2_{x',t'}}^{t'- \frac14 \eps' r^2_{x',t'}} \int_M  \bigg( | R| \, |\omega_{n-2}| + \sum_{l=1}^{n-2}    |\square |\omega_l| |  \bigg)   d\nu_{x'',t''; t} \, dt \, d\nu_{t''} (x'')  \, dt'' \notag \\
&\leq C(Y,  \omega, \eps') \int_{ - 10}^{-1} \int_M \bigg( |R| \, |\omega_{n-2}| + \sum_{l=1}^{n-2}    |\square |\omega_l| |  \bigg) d\nu^*_{x',t'; t} dt , \label{eq_sum_R_squ_omega}
\end{align}
where
\[ \nu^*_{x',t'; t} := \int_{\max \{ t' - (\zeta' r_{x',t'})^2, t \}}^{\max \{ t' - \frac12 (\zeta' r_{x',t'})^2, t \}} \int_{P^* (x',t'; \zeta r') \cap M \times \{ t'' \}} \nu_{x'', t''; t} \, d\nu_{t''} (x'') dt''. \]
Since the balls $P^* (x', t';  \zeta r_{x',t'})$, $(x',t') \in S'$ are pairwise disjoint, we obtain that for $t \leq -1$
\[ \sum_{(x', t') \in S'} \nu^*_{x', t'; t} 
\leq  \int_{ t }^{- 1} \int_{M} \nu_{x'', t''; t} \, d\nu_{t''} (x'') dt'' 
= \int_{ t }^{- 1}  \nu_{ t}\,   dt''  
= \big(-1 - t \big)  \nu_{t}. \]
Combining this with (\ref{eq_S_covered}), (\ref{eq_sum_R_squ_omega}) and (\ref{eq_R_sq_omega_small}) implies that
\begin{multline*}
 |\phi (S)| 
\leq \sum_{(x', t') \in S'}  |\phi ( P^* (x', t';  \zeta r_{x',t'}) )|
\leq C(Y,  \omega, \eps') \int_{ -10}^{-\frac12} \int_M \bigg( |R| \, |\omega_{n-2}| + \sum_{l=1}^{n-2}    |\square |\omega_l| |  \bigg) d\nu_{t} dt \\
\leq \Psi ( \delta | Y,  \omega, \eps'). 
\end{multline*}
The restrictions that we have imposed on $\eps', \zeta, \zeta', \omega$ were of the form $\eps' \leq \ov\eps' (Y, \eps)$, $\zeta \leq \ov\zeta (Y, \eps, \eps')$, $\zeta' \leq \ov\zeta' (Y, \eps, \eps', \zeta)$, $\omega \leq \ov\omega (Y, \eps, \zeta')$.
Therefore, we can discard these constants and conclude that
\begin{equation} \label{eq_phi_S_small}
 |\phi (S)| \leq \Psi ( \delta | Y, \eps).  
\end{equation}

Next, define $\Sigma_{\vec a, t} := (\vec y (\cdot, t))^{-1} (\vec a) \subset M$ for any $(\vec a, t) \in \IR^{n-2} \times [-2,-1]$.
Then by the co-area formula and Proposition~\ref{Prop_improved_L2} we have
\begin{multline} \label{eq_int_nabf_4_bound}
 \int_{-2}^{-1} \int_{\IR^{n-2}} \int_{ \Sigma_{\vec a, t} } |\nabla f|^{4}  e^{-f} ( \det \nabla y_i \cdot \nabla y_j )^{1/2} dg^{\Sigma_{\vec a, t} }_t d\vec a \, dt
\leq \int_{-2}^{-1} \int_M |\nabla f|^{4}  e^{-f}  dg_t dt \\
\leq C \int_{-2}^{-1} \int_M |\nabla f|^{4} d\nu_t dt
\leq C(Y).
\end{multline}
Here $g^{\Sigma_{\vec a, t} }_t$ denotes the metric on $\Sigma_{\vec a, t}$ induced by $g_t$, whenever $\Sigma_{\vec a, t}$ is a submanifold.
Combining (\ref{eq_phi_S_small}) and (\ref{eq_int_nabf_4_bound}) implies that if $\delta \leq \ov\delta (Y, \eps)$, then there is a regular value $(\vec a, t') \in [-1, 1]^{n-2} \times [-2, -1]$ of $\phi$ such that $\Sigma := \Sigma_{\vec a, t'}$ satisfies
\begin{equation} \label{eq_properties_Sigma}
 S \cap \Sigma = \emptyset, \qquad
 \int_{ \Sigma } |\nabla f|^{4}  e^{-f} ( \det \nabla y_i \cdot \nabla y_j )^{1/2} dg^\Sigma_{t'} \leq C(Y).
\end{equation}
Assertions~\ref{Prop_Slicing_a}, \ref{Prop_Slicing_b} of the proposition hold automatically.
For Assertion~\ref{Prop_Slicing_c} consider a point $x' \in \Sigma$ satisfying (\ref{eq_Slicing_condition}) and let $r' \leq \rrm (x', t')$.
By Claim~\ref{Cl_rprime_A} the point $(x', t')$ is $(\eps, r')$-static and there is a matrix $A$ with $|A| \leq 10 r^{\prime -\eps} (x', t')$ such that $\vec y^{\,\prime} := A ( \vec y - \vec y (x',t'))$ is a strong $(n-2, \eps, r')$-splitting map at $(x',t')$.
By Lemma~\ref{Lem_almost_eucl_slice_loc} if $\eps \leq \ov\eps (Y)$, then $\Sigma' := B^\Sigma (x', t', .9 r')$ has the following properties:
\begin{enumerate}
\item $ \det ( \nabla y'_i \cdot \nabla y'_j)  \geq \frac12$ on $\Sigma'$.
\item All principal curvatures of $\Sigma'$ are $\leq 10^{-3} r^{\prime -1}$,
\item The injectivity radius of $(\Sigma, g^{\Sigma'}_{t'})$ at any point $x'' \in \Sigma'$ is $\geq c(Y) r$.
\end{enumerate}
Therefore we have on $\Sigma'$
\[  \det (\nabla y_i \cdot \nabla y_j )
= (\det A )^{-2} \det (\nabla y'_i \cdot \nabla y'_j )
\geq \tfrac12 (\det A )^{-2}
\geq c r^{\prime 2n\eps}. \]
Combining this with (\ref{eq_properties_Sigma}) implies
\[ \int_{ \Sigma' } |\nabla e^{-f/4}|^4  dg^{\Sigma'}_{t'} \leq C(Y) r^{\prime -n\eps}. \]
So by the Sobolev inequality on $\Sigma'$ we obtain that
\[ \osc_{B^\Sigma (x', t', 10^{-1} r')} e^{-f/4} \leq C(Y) (r')^{\frac12 - \frac{n}4 \eps}. \]
This finishes the proof.
\end{proof}
\bigskip

\subsection{Curvature control near certain slices of splitting maps}
In this subsection we analyze the level set $\Sigma$ from the Slicing Theorem (Proposition~\ref{Prop_Slicing}) and we show that $\rrm( \cdot, t')$ is bounded from below along $\Sigma$ where $f$ is bounded from above.

\begin{Lemma} \label{Lem_curv_bound_slice}
If $Y, F < \infty$, $\sigma \leq \ov\sigma (Y, F)$, $\delta \leq \ov\delta (Y, F)$ then the following holds.

Let $(M, (g_t)_{t \in I})$ be Ricci flow on a compact manifold, $r > 0$ and $(x_0,t_0) \in M \times I$ with $\NN_{x_0,t_0} (r^2) \geq - Y$.
Denote by $d\nu = (4\pi \tau)^{-n/2} e^{-f} dg$ the conjugate heat kernel based at $(x_0, t_0)$.
Let $\vec y$ be a strong $(n-2, \delta, r)$-splitting map at $(x_0, t_0)$ and assume that $(x_0, t_0)$ is $(\delta, r)$-static.

Then there are $\vec a \in [-r, r]^{n-2}$, $t' \in [t_0 - 2r^2, t_0 - r^2]$ such that
$\Sigma := (\vec y (\cdot, t'))^{-1} (\vec a) \subset M$
is a smooth 2-dimensional submanifold and
\[ \rrm (\cdot, t') \geq \sigma r \qquad \text{on} \quad \Sigma \cap \{ f (\cdot, t') \leq F \}. \]
\end{Lemma}

The strategy will be to analyze the geometry of the flow at a point $(x', t') \in \Sigma \times \{ t' \}$ for which $r' := \rrm(x',t')$ is locally almost minimal, under some additional restrictions of $f(x', t')$.
We will then use the fact that the flow almost splits off an $\IR^{n-2}$-factor and is almost static at scale $r'$ to deduce a lower bound on $\rrm$.

\begin{proof}
Fix $Y, F < \infty$ and let $Y' \geq Y$, $0 < \sigma \leq 10^{-3} \eps \leq 10^{-4}$ be constants whose values we will determine later.
Without loss of generality we may assume that $r =1$ and $t_0 = 0$.
Assuming $\delta \leq \ov\delta (Y', \eps)$, we may choose $\vec a$, $t'$, $\Sigma \subset M$ according to Proposition~\ref{Prop_Slicing} with the parameters $Y', \eps$.

\begin{Claim} \label{Cl_Y_prime_F}
If $Y' \geq \underline{Y}' (Y, F)$, then for any $x' \in \Sigma$ with $e^{-f(x', t') /4}\geq \frac12 e^{-F}$ we have $\NN_{x', t'} (10^{-1}) \geq - Y'$.
So if $\delta \leq \ov\delta (Y', \eps)$, then the assertions of Proposition~\ref{Prop_Slicing} hold for all $x' \in \Sigma$ with $e^{-f(x', t') /4}\geq \frac12 e^{-F}$ and all $r' \in (0, 10^{-1}]$.
\end{Claim}

\begin{proof}
This is a consequence of \cite[\HKThmHKboundGauss]{Bamler_HK_entropy_estimates} and Proposition~\ref{Prop_NN_variation_bound}.
\end{proof}

Now let $x' \in \Sigma$ be a point with $f(x', t') \leq F$.
Assume by contradiction that 
\[ r' := \rrm (x', t') < \sigma. \]
Let $D < \infty$ be a constant whose value we will determine later.

\begin{Claim} \label{Cl_r_pp_point_picking}
If $\sigma \leq \ov\sigma (Y', F, D)$, then there is a point $x'' \in \Sigma$   such that:
\begin{enumerate}[label=(\alph*)]
\item \label{Cl_r_pp_point_picking_a} $e^{-f(x'', t')/4} \geq \frac12 e^{-F/4}$.
\item \label{Cl_r_pp_point_picking_b} $r'' := \rrm (x'', t') < \sigma$.
\item \label{Cl_r_pp_point_picking_c} $\rrm (x''', t') \geq r'' / 2$ for all $x''' \in B^\Sigma (x', t', D r'')$, where the latter denotes the time-$t'$-distance ball within $\Sigma$ with respect to the induced length metric.
\item \label{Cl_r_pp_point_picking_d} There is an $(n-2) \times (n-2)$-matrix $A''$ such that $A'' (\vec y - \vec y (x'', t'))$ is a strong $(n-2, \eps, r'')$-splitting map at $(x'',t')$.
\item \label{Cl_r_pp_point_picking_e} $(x'', t')$ is $(\eps, r'')$-static.
\end{enumerate}
\end{Claim}

\begin{proof}
Let $C_0(Y')$ be the constant from Proposition~\ref{Prop_Slicing} and assume that $\eps \leq \min \{ \frac14, \frac12 e^{-F/4} \}$.
We will apply a point picking argument.
Set $x'_0 := x'$ and assume inductively that $x'_0, \ldots, x'_i \in \Sigma$ have already been chosen.
We will describe how to choose $x'_{i+1} \in \Sigma$.
If Assertion~\ref{Cl_r_pp_point_picking_c} holds for $x'' = x'_i$ and $r'' = r'_i := \rrm (x'_i, t')$, then we stop.
Otherwise choose $x'_{i+1} \in B^\Sigma (x'_i, t', D r'_i)$ with minimal intrinsic time-$t'$-distance $d_i := d_{t'}^\Sigma (x'_i, x'_{i+1})$ to $x'_i$ such that $r'_{i+1} := \rrm (x'_{i+1}, t') = r'_i / 2$.

Since $\inf_\Sigma \rrm(\cdot, t') > 0$, our process has to terminate after a finite number of steps, producing a sequence $x'_0, \ldots, x'_m$ with $r'_i < 2^{-i} \sigma$.
Choose $m' \in \{ 1, \ldots, m \}$ maximal such that for all $i = 0, \ldots, m'-2$ we have
\[ e^{-f(\cdot, t')/4} \geq \tfrac12 e^{-F} \qquad \text{on} \quad B^\Sigma (x'_i, t', d_i). \]
If $m' < m$, then choose $d' \in [0, d_{m'-1}]$ maximal such that
\[ e^{-f(\cdot, t')/4} \geq \tfrac12 e^{-F} \qquad \text{on} \quad B^\Sigma (x'_{m'-1}, t', d'). \]
Recall that by the choice of the points $x'_i$ we have $\rrm (\cdot, t') \geq r'_i / 2$ on $B^\Sigma(x'_i, t', d_i)$.
For each $i = 1, \ldots, m'-1$ apply Assertion~\ref{Prop_Slicing_c} of Proposition~\ref{Prop_Slicing} along geodesics of length $d_i < D r'_i$ (if $i < m'-1$) or $d'$ (if $i = m'-1$) emanating from $x'_i$.
We obtain that for $i = 0, \ldots, m'-2$
\begin{equation} \label{eq_ef4_iip1}
 \big| e^{-f(x'_{i+1}, t')/ 4} - e^{-f(x'_i, t')/4} \big| 
\leq C D C_0(Y') (r'_i)^{1/2}
\leq C D C_0 (Y') \sigma^{1/2} 2^{-i/2} 
\end{equation}
and, similarly, that
\begin{equation} \label{eq_ef4_mpm1}
 \big| e^{-f(\cdot , t')/ 4} - e^{-f(x'_{m'-1}, t')/2} \big| 
\leq C D C_0(Y') \sigma^{1/2} 2^{-(m'-1)/2} \qquad \text{on} \quad B^\Sigma (x'_{m'-1}, t', d'). 
\end{equation}
Combining (\ref{eq_ef4_iip1}) for $i = 0, \ldots, m'-2$ with (\ref{eq_ef4_mpm1})  and assuming $\sigma \leq \ov\sigma (Y', F, D)$ implies that on $B^\Sigma (x'_{m'-1}, t', d')$ (if $m' < m$) or at $x'_m$ (if $m' = m$)
\[ e^{-f(\cdot , t')/ 4} 
\geq e^{-f(x'_0 , t')/ 4} - \sum_{i=0}^{m'-1} C D C_0(Y') \sigma^{1/2} 2^{-i/2}
\geq e^{-F/4} -  C D C_0(Y') \sigma^{1/2}
>  \tfrac12 e^{-F}. \]
So due to the maximal choice of $d'$, we must have $d' = d_{m'-1}$ if $m' < m$ and therefore, by the maximal choice of $m'$ we must have $m' = m$.
It follows that $x'' := x'_m$ satisfies Assertions~\ref{Cl_r_pp_point_picking_a}--\ref{Cl_r_pp_point_picking_c}.
Assertions~\ref{Cl_r_pp_point_picking_d}, \ref{Cl_r_pp_point_picking_e} are a consequence of Proposition~\ref{Prop_Slicing}.
\end{proof}

We will now apply Lemma~\ref{Lem_improve_rrm_10} below, for $(x_0, t_0) \leftarrow (x'', t')$, $r \leftarrow r''$ and $\vec y \leftarrow A'' (\vec y -  \vec y (x'',t'))$.
The assumptions of this Lemma hold due to Claims~\ref{Cl_Y_prime_F}, \ref{Cl_r_pp_point_picking}.
Therefore, Lemma~\ref{Lem_improve_rrm_10} implies that if $\eps \leq \ov\eps (Y')$ and $D \geq \underline{D} (Y')$, then $\rrm (x'', t') \geq 10 r''$, which contradicts the choice of $r''$.
This finishes the proof of the lemma.

Lastly, let us review the choice of constants.
Given $Y, F$, we may first choose $Y' \geq \underline{Y'} (Y,F)$ according to Claim~\ref{Cl_Y_prime_F}.
Then we can choose $D \geq \underline{D} (Y')$, $\eps \leq \ov\eps (Y')$ according to Lemma~\ref{Lem_improve_rrm_10}.
Next, we choose $\sigma \leq \ov\sigma (Y', F, D)$ according to Claim~\ref{Cl_r_pp_point_picking} and such that $\sigma \leq 10^{-3} \eps$.
Eventually, we choose $\delta \leq \ov\delta (Y', \eps)$ such that Proposition~\ref{Prop_Slicing} can be applied.
\end{proof}
\bigskip

\begin{Lemma} \label{Lem_improve_rrm_10}
If $Y < \infty$ and $D \geq \underline{D} (Y)$, $\eps \leq \ov\eps (Y)$, then the following holds.
Let $(M, (g_t)_{t \in I})$ be Ricci flow on a compact manifold, $r > 0$ and let $(x_0, t_0) \in M \times I$ be a point with $[t_0 - \eps^{-1} r^2, t_0 + r^2] \subset I$ and $\NN_{x_0,t_0} (\eps^{-1} r^2) \geq - Y$.
Consider a vector valued function $\vec y : M \times [ t_0 - \eps^{-1} r^2, t_0 + r^2] \to \IR^{n-2}$ whose component functions $y_1, \ldots, y_{n-2}$ are solutions to the heat equation $\square y_j = 0$.
Let 
\[ \Sigma := (\vec y (\cdot, t_0))^{-1} ( \vec y (x_0, t_0))  \subset M. \]
Assume that:
\begin{enumerate}[label=(\roman*)]
\item \label{Lem_improve_rrm_10_i} $\Sigma$ is a smooth $2$-dimensional submanifold.
\item \label{Lem_improve_rrm_10_ii} $\rrm (x_0, t_0) \geq r$.
\item \label{Lem_improve_rrm_10_iii} $\vec y$ restricted to $M \times [ t_0 - \eps^{-1} r^2, t_0 - \eps r^2]$ is a strong $(n-2, \eps, r)$-splitting map at $(x_0, t_0)$.
\item \label{Lem_improve_rrm_10_iv} $(x_0, t_0)$ is $(\eps, r)$-static.
\item \label{Lem_improve_rrm_10_v} $\rrm(x,t_0) \geq r/2$ for all $x \in B^\Sigma (x_0, t_0, D r)$, where the latter denotes the time-$t_0$-distance ball within $\Sigma$ with respect to the induced length metric.
\end{enumerate}
Then $\rrm (x_0,t_0) \geq 10 r$.
\end{Lemma}

\begin{proof}
By parabolic rescaling and application of a time-shift, we may assume without loss of generality that $r^2= 2$ and $t_0 = - 1$.
Note that this means that $\vec y$ is defined on $M \times [-1 -  \eps^{-1}, 1]$.
Also note that the fact that $\vec y$ is a strong splitting map implies $\vec y (x_0, -1) = 0$, so $\Sigma = (\vec y (\cdot, -1))^{-1} (\vec 0)$.

Assume that the lemma was false for some fixed $Y$.
Consider a sequence of counterexamples $(M_i, (g_{i,t})_{t \in I_i})$, $r_i$, $(x_i, -1) \in M_i \times I_i$, $\vec y_i$, $\Sigma_i$ for sequences $\eps_i \to 0$, $D_i \to \infty$.
As discussed in Section~\ref{Sec_basic_limits}, we may pass to a subsequence and assume that we have $\IF$-convergence of the pointed flows
\begin{equation} \label{eq_Fconv_10rrm}
 (M_i, (g_{i,t})_{t \in I_i}, (\nu_{x_i,0;t})_{t \in I_i}) \xrightarrow[i\to \infty]{\quad \IF, \CF \quad} (\mathcal{X}, ( \nu_{x^*_\infty;t})_{t \leq 0}) 
\end{equation}
within some correspondence $\CF$, where $\XX$ is an $H_n$-concentrated, future continuous metric flow over $(-\infty, 0]$ with full support and the $\IF$-convergence is understood to hold on compact time-intervals.
Note that we have purposefully pointed the flows at $(x_i, 0)$ and not at $(x_i, -1)$.
Consider the regular-singular decomposition $\XX = \RR {\, \dot\cup \,} \SS$ of the limit and recall from Corollary~\ref{Cor_RRstar_RR} that the convergence (\ref{eq_Fconv_10rrm}) is smooth on $\RR$.

\begin{Claim}
After passing to a subsequence, we have $(x_i, -1) \to x_\infty$ within $\CF$ for some $x_\infty \in \RR_{-1}$.
Moreover, the two-sided parabolic neighborhood $P(x_\infty; \sqrt{2}, -T^-, T^+) \subset \RR$ is unscathed for any $T^- < 2$ and $T^+ < 1$.
\end{Claim}

\begin{proof}
By \cite[\HKCenterConstantRmBound]{Bamler_HK_entropy_estimates}, \cite[\SYNConvParabNbhd]{Bamler_RF_compactness} and Assumption~\ref{Lem_improve_rrm_10_ii} we obtain that for small $\theta > 0$ and any $H_n$-center $z \in \XX_{-\theta}$ of $x^*_\infty$ the parabolic neighborhood $P:= P(z; \sqrt{1.9} , -1.9)$ is unscathed and $|{\Rm}| \leq 2$ on $P$.
Moreover, there is a sequence $(z_i, -\theta) \in M_i \times I_i$ such that $(z_i, -\theta) \to z$ within $\CF$ and such that $d_{-1} (z_i, x_i) \leq C \sqrt{\theta}$ for some uniform $C < \infty$.
This implies that there is a point $x_\infty \in \RR_{-1}$ such that for any small $\theta$ and any $H_n$-center $z \in \XX_{-\theta}$ of $x^*_\infty$ we have $d_{-1} (z(-1), x_\infty) \leq C' \sqrt{\theta}$.
Letting $\theta \to 0$ implies the first part of the claim.
The second part follows by applying \cite[\SYNConvParabNbhd]{Bamler_RF_compactness} again.
\end{proof}

It follows, using \cite[\SYNChangeBaseConv]{Bamler_RF_compactness}, that we have $\IF$-convergence
\[ (M_i, (g_{i,t})_{t \in I_i},( \nu_{x_i,-1;t})_{t \in I_i, t < -1}) \xrightarrow[i \to \infty]{\quad \IF, \CF \quad} (\mathcal{X}_{(-\infty, -1)}, (\nu_{x_\infty;t})_{t < -1}), \]
on compact time-intervals, where $\mathcal{X}_{(-\infty, -1)}$ denotes the restriction of $\XX$ to $(-\infty, -1)$.

By Lemma~\ref{Lem_splitting_map_HE_control_beyond} the maps $\vec y_i$ are locally uniformly bounded on $P^*$-parabolic neighborhoods around $(x_i, 0)$.
So, due to Lemma~\ref{Lem_conv_P_star}, we may pass to a subsequence such that we have local smooth convergence $\vec y_i \to \vec y_\infty$ on $\RR$, where $\vec y_\infty : \RR \to \IR^{n-2}$ is smooth and its components are solutions to the heat equation.
By Theorem~\ref{Thm_limit_from_strong_splitting} and Proposition~\ref{Prop_eps_reg_codim_2} the map $\vec y_\infty$ restricted to $\RR_{(-\infty, -1)}$ induces a splitting of the form 
\begin{equation} \label{eq_RR_RR_star_R_n_2}
\RR_{(-\infty, -1)} \cong \RR' \times \IR^{n-2}
\end{equation}
for some Ricci flow spacetime $\RR'$ with 2-dimensional time-slices.
Moreover, by Theorem~\ref{Thm_limit_from_static} we have $\Ric \equiv 0$ on $\RR_{(-\infty, -1)}$, which implies that $\Rm \equiv 0$ on $\RR_{(-\infty, -1)}$.

Let $\Sigma_\infty := \vec y_\infty^{-1} (\vec 0) \cap \RR_{-1}$.
Then $x_\infty \in \Sigma_\infty$.
Also, since the derivatives of $\vec y$ are continuous on $\RR$, we obtain that $\Sigma_\infty$ is a 2-dimensional, totally geodesic submanifold; so its induced metric is flat.
By the local convergence of $\vec y_i \to \vec y_\infty$ and the implicit function theorem, any point $x'_\infty \in \Sigma_\infty$ is a limit of a sequence of points in $x'_i \in \Sigma_i$ within $\CF$.
So by passing Assumption~\ref{Lem_improve_rrm_10_v} to the limit and using \cite[\SYNChangeBaseConv]{Bamler_RF_compactness}, we obtain that the component $\Sigma^0_\infty$ of  $\Sigma_\infty$ that contains $x_\infty$ is complete and that every point of $\Sigma^0_\infty$ survives within $\RR$ until any time of the time-interval $(-\frac32, -\frac12)$.
Consider the image $\Sigma^0_\infty (t) \subset \RR_t$ of $\Sigma_\infty^0$ under the flow of the time vector field on $\RR$, for $t \in (-\frac32,-1)$.
Using the identification (\ref{eq_RR_RR_star_R_n_2}), this image corresponds to $\RR^{\prime 0}_t \times \{ \vec 0 \}$, where $\RR^{\prime 0}_t$ is a component of $\RR'_t$.
The union $\RR^{\prime 0} := \bigcup_{t \in (-\frac32,-1)} \RR^{\prime 0}_t$ is a component of $\RR'_{(-\frac32, -1)}$ and corresponds to a constant flow on an isometric quotient of $\IR^2$.
Therefore, the component of $\RR_{(-\frac32, -1)}$ that intersects the worldline of $x_\infty$ corresponds to a constant flow on an isometric quotient $N = \IR^n / \Gamma$.
Using \cite[\HKThmNLC]{Bamler_HK_entropy_estimates} we conclude that $N$ must have positive asymptotic volume ratio, which implies $N = \IR^n$.
So by pseudolocality \cite[Sec.~10]{Perelman1} and \cite[\SYNChangeBaseConv]{Bamler_RF_compactness} we obtain that the component of $\RR_{(-\frac32 , 0)}$ that contains $x_\infty$ corresponds to the constant flow on $\IR^{n}$.

Since, at this point, Assumption~\ref{Aspt_working_ass} has only been verified for $\Delta = 1,2$, we cannot use Theorem~\ref{Thm_limit_from_static} to show that the flow on the component of $\RR$ that contains $x_\infty$ corresponds to constant flow on $\IR^n$.
Instead, we will take a blow-down limit as described in the following.

By Proposition~\ref{Prop_NN_almost_constant_selfsimilar} we can find sequences $\eps'_i \to 0$, $r'_i \to \infty$ with the property that $(x_i, -1)$ is $(\eps'_i, r'_i)$-selfsimilar, $(\eps'_i, r'_i)$-static, $\vec y_{i}$ is a strong $(n-2,\eps'_i, r'_i)$-splitting map at $(x_i, -1)$ and such that we have smooth Cheeger-Gromov convergence
\[ (M_i, r^{\prime -2}_i g_{i,-1}, x_i) \longrightarrow (\IR^n, \vec 0), \]
with respect to which $\vec y_{i}$ locally smoothly converges to a coordinate function of the last $n-2$ factors.
Let $g'_{i,t} := r^{\prime -2}_i g_{i, r^{\prime 2}_i t -1}$ be the parabolically rescaled flows.
Then by Lemma~\ref{Lem_weak_bckwds_pseudo} below we must have for the original flows: $\lim_{i \to \infty} \NN_{x_i, -1} (r^{\prime 2}_i) = 0$.
By \cite[\HKThmEpsRegularity]{Bamler_HK_entropy_estimates} this implies that $\rrm (x_i, -1) \geq 100$ for large $i$, which finishes the proof.
\end{proof}
\bigskip

\begin{Lemma} \label{Lem_weak_bckwds_pseudo}
Let $(M_i, (g_{i,t})_{t \in [-T_i, 0]}, x_i)$ be a pointed sequence of Ricci flows on compact manifolds with $T_i \to \infty$ and let $\vec y_i : M_i \times [-T_i, 0] \to \IR^{n-2}$ be smooth vector valued functions  whose component functions $y_{i,j}$ satisfy the heat equation $\square y_{i,j} = 0$.
Assume that there is some $Y < \infty$ and a sequence $\eps_i \to 0$ such that:
\begin{enumerate}[label=(\roman*)]
\item $W_i := \NN_{x_i, 0} (1) \geq - Y$.
\item $(x_i, 0)$ is $(\eps_i, 1)$-selfsimilar.
\item $(x_i,0)$ is $(\eps_i, 1)$-static.
\item $\vec y_{i}$ is a strong $(n-2, \eps_i,1)$-splitting map at $(x_i, 0)$.
\item In the pointed Cheeger-Gromov sense
\begin{equation} \label{eq_CG_to_Eucl}
(M_i, g_{i,0}, x_i) \longrightarrow (\IR^n, \vec 0)
\end{equation}
and $\vec y_{i}(\cdot, 0)$ converges to the projection onto the last $n-2$ factors locally smoothly.
\end{enumerate}
Then $\lim_{i \to \infty} W_i = 0$.
\end{Lemma}

\begin{proof}
Denote by $d\nu_{x_i, 0} = (4\pi \tau)^{-n/2} e^{-f_i} dg_i$ the conjugate heat kernels based at $(x_i, 0)$.

After passing to a subsequence, we may assume that $W_\infty := \lim_{i \to \infty} W_i$ exists and that we have $\IF$-convergence
\begin{equation} \label{eq_backwards_proof_F_conv}
 (M_i, (g_{i,t})_{t \in I_i}, \nu_{x_i,0}) \xrightarrow[i \to \infty]{\quad \IF, \CF \quad} (\mathcal{X}, \nu_{x_\infty}), 
\end{equation}
within some correspondence, where $\XX$ is an $H_n$-concentrated, future continuous metric flow over $(-\infty, 0]$ with full support and the $\IF$-convergence is understood to hold on compact time-intervals (see Section~\ref{Sec_basic_limits}).
Moreover, by Theorem~\ref{Thm_SS_dimension_bound_limit} we may decompose the limit $\XX = \RR {\, \dot\cup \,} \SS$, where $\RR$ is a Ricci flow spacetime and $d\nu_{x_\infty;t} (\SS_t) = 0$ for almost all $t < 0$.
We may also assume that we have local smooth convergence $f_i \to f_\infty \in C^\infty (\RR)$, where $d\nu_{x_\infty} = (4\pi \tau)^{-n/2} e^{-f_\infty} dg$.
Lastly, by Theorem~\ref{Thm_limit_from_almost_cone} we may assume  that we have local smooth convergence $\vec y_{ i} \to \vec y_{ \infty} \in C^\infty (\RR)$, where the latter are coordinates on the second factor of a splitting of the form
\begin{equation} \label{eq_RR_splitting_weak_pseudo}
 \RR = (M'_\infty \times \IR^{n-2} ) \times \IR_- .
\end{equation}
Here $(M'_\infty, g'_\infty)$ is a flat, 2-dimensional, Riemannian cone over a compact link $(N, g^N)$, possibly minus its vertex, and
\[ \vol (N, g^N) = e^{W_\infty} \vol (S^1, g^{S^1}) = 2\pi e^{W_\infty} . \]
So it suffices to show that $\vol (N, g^N) \geq 2\pi$.

After passing to a subsequence, we may fix some time $t^* \in [-2, -1]$ with $d\nu_{x_\infty;t^*} (\SS_{t^*}) = 0$ such that the convergence (\ref{eq_backwards_proof_F_conv}) is time-wise at time $t^*$.
\[ (M_i, d_{g_{i,t^*}}, \nu_{x_i,0; t^*}) \xrightarrow[i \to \infty]{GW_1} (\mathcal{X}_{t^*}, \nu_{x_\infty; t^*}) \]
as metric measure spaces.

Consider the diffeomorphisms $\psi_i : U_i \to V_i$, $U_i \subset \RR$, $V_i \subset M_i \times [-T_i, 0]$ from Subsection~\ref{subsec_setup_limit}.
In the following, we will be mainly interested in the restrictions $\psi_{i, t^*} : U_{i, t^*} \to V_{i, t^*}$, where $U_{i, t^*} := U_i \cap \RR_{t^*}$, $V_{i, t^*} := M_i \times \{ t^* \}$.
Recall that $U_{1, t^*} \subset U_{2, t^*} \subset \ldots$, $\bigcup_{i=1}^\infty U_{i, t^*} = \RR_{t^*}$ and $\psi_{i,t^*}^* g_{i, t^*} \to g_{t^*}$, $\psi_{i,t^*}^* f_{i,t^*} \to f_{\infty, t^*}$ in $C^\infty_{\loc}$.
The goal of the proof will be to show that maps $\psi_{i, t^*}$ viewed as mapping into the time-$0$-slics $M_i \times \{ 0 \}$ converge to an isometric embedding of a component of $M'_\infty \times \IR^{n-2}$ into $\IR^n$.
The following two claims ensure, that the images of these maps remain at bounded distance to $x_i$ with respect to the time-$0$ metric, in a certain sense.

Let $A < \infty$ be a constant whose value we will choose later.

\begin{Claim} \label{Cl_x_p_i_existence}
If $D \geq \underline{D} (A)$, then for large $i$ there are points $x'_i \in M_i$ with:
\begin{enumerate}[label=(\alph*)]
\item $d_0 (x_i, x'_i) \leq D$.
\item $ d_{W_1}^{g_{t^*}} (\nu_{x_i,0; t^*}, \nu_{x'_i,0; t^*} ) \geq A$.
\item $\vec y_i (x'_i) = \vec 0$.
\end{enumerate}
\end{Claim}

\begin{proof}
Set
\[ S_i := \{ x' \in M_i \;\; : \;\; d_{W_1}^{g_{t^*}} (\nu_{x_i,0; t^*}, \nu_{x',0; t^*} ) \leq A  + 1 \}. \]
By \cite[\HKPstarVolBound]{Bamler_HK_entropy_estimates} we have 
\begin{equation} \label{eq_S_i_vol_bound}
 |S_i|_0 \leq C(A) 
\end{equation}
and for every $x' \in M_i$ with $d_{W_1}^{g_{t^*}} (\nu_{x_i,0; t^*}, \nu_{x',0; t^*} ) \leq A$ we have $B(x', 0, 1) \subset S_i$.

Assume that the claim was wrong for some fixed $A < \infty$ and some $D < \infty$ whose value we will determine later.
Then we may pass to a subsequence such there is no such point $x'_i \in M_i$ for any $i$.
Fix a point $x'_\infty = (\td x'_\infty, \vec 0 ) \in \IR^{n} = \IR^2 \times \IR^{n-2}$ with $d^{\IR^n} (\vec 0, x'_\infty) < D$ and choose a sequence $x'_i \in M_i$ such that $(x'_i,0) \to x'_\infty$ in (\ref{eq_CG_to_Eucl}) and $\vec y_i (x'_i, 0) = \vec 0$.
Then $B(x'_i, 0,1) \subset S_i$ for large $i$.
If $D \geq \underline{D}(A)$, then repeating this process for different choices of $x'_\infty$ allows us to find larger and larger sets of pairwise disjoint $1$-balls within $S_i$, which contradicts (\ref{eq_S_i_vol_bound}).
\end{proof}

Consider the sequence $x'_i$ from Claim~\ref{Cl_x_p_i_existence}.

\begin{Claim} \label{Cl_x_pp_no_expansion}
If $A \geq \underline{A}$ then, after passing to a subsequence, we have $|{\Rm}| \leq \eps'_i \to 0$ on $P(x'_i, 0; 10)$ and the points $(x'_i, t^*)$ converge to a point $x''_\infty \in \RR_{t^*}$ within $\CF$.
\end{Claim}

\begin{proof}
By Claim~\ref{Cl_x_p_i_existence} we have
\begin{equation} \label{eq_dW1_AD}
 A \leq d_{W_1}^{g_{t^*}} (\nu_{x_i; t^*}, \nu_{x'_i; t^*} ) \leq d_0 (x_i, x'_i) \leq D . 
\end{equation}
By \cite[\SYNThmConvSubsequwithinCF]{Bamler_RF_compactness} we can find a conjugate heat flow $(\mu_t)_{t < 0}$ on $\XX$ such that $(\nu_{x'_i;t})_{t< 0} \to (\mu_t)_{t < 0}$ within $\CF$ on compact time-intervals.
By \cite[\SYNThmConvImpliesStrict]{Bamler_RF_compactness} we even have strict convergence $\nu_{x'_i;t^*} \to \mu_{t^*}$ and
\begin{equation} \label{eq_Var_mu_in_limit}
 \Var (\mu_{t^*}) \leq H_n |t^*| \leq 2H_n. 
\end{equation}
By (\ref{eq_dW1_AD}) we also have
\begin{equation} \label{eq_dW1_AD_limit}
 A \leq d^{\XX_{t^*}}_{W_1} (\nu_{x_\infty; t^*}, \mu_{t^*}) \leq D. 
\end{equation}
Next we derive a bound on $\int_{\RR_{t^*}} y^2_{\infty, j} d\mu_{t^*}$ for any $j =1, \ldots, n-2$.
For this purpose, observe that for all $i \geq 1$ we have
\[  \int_{M_i} y_{i,j} d\nu_{x'_i, 0; t^*} = y_{i,j} (x'_i, 0)  = 0. \]
Using Propositions~\ref{Prop_inheriting_bounds}, \ref{Prop_properties_splitting_map}, we moreover obtain for some $\alpha \leq \ov\alpha$
\[ \int_{M_i} \big| |\nabla y_{i,j}|^2 - 1 \big| d\nu_{x'_i, 0; t^*} 
\leq \int_{M_i} \big| |\nabla y_{i,j}|^2 - 1 \big| e^{\alpha f_i} d\nu_{x_i, 0; t^*} 
\to 0 . \]
So by the $L^2$-Poincar\'e inequality (Proposition~\ref{Prop_Poincare}) we obtain
\begin{equation} \label{eq_y2infty_bound_limit}
 \int_{\RR_{t^*}} y^2_{\infty, j} d\mu_{t^*}
\leq \limsup_{i \to \infty} \int_{M_i}   y_{i,j}^2  d\nu_{x'_i, 0; t^*}
\leq \limsup_{i \to \infty} 2|t^*| \int_{M_i}  | \nabla y_{i,j}|^2 d\nu_{x'_i, 0; t^*}
= 2 |t^*| \leq 4. 
\end{equation}

Combining the bounds (\ref{eq_dW1_AD_limit}), (\ref{eq_Var_mu_in_limit}), (\ref{eq_y2infty_bound_limit}) implies that there is a point $z'_\infty \in \RR_{t^*}$ such that for some universal constant $C < \infty$
\begin{equation} \label{eq_bounds_limit_mu}
 \tfrac12 A \leq d^{\XX_{t^*}}_{W_1} (\nu_{x_\infty; t^*}, \delta_{z'_\infty}) , \qquad
\Var (\mu_{t^*}, \delta_{z'_\infty} ) \leq C, \qquad |\vec y_\infty (z'_\infty)|^2  \leq C. 
\end{equation}
Let $A' < \infty$ be a constant whose value we will determine later.
If $A \geq \underline{A} (A')$, then the splitting (\ref{eq_RR_splitting_weak_pseudo}) and the bounds (\ref{eq_bounds_limit_mu}) imply that the ball $B (z'_\infty, A') \subset \RR_{t^*}$ is relatively compact and isometric to a ball in Euclidean space.
If $A' \geq \underline{A}'$, then
\[ \mu_{t^*} \big( B(z'_\infty, \tfrac18 A') \big) \geq .9 \]
Choose $z'_i \in M_i$ such that $(z'_i, t^*) \to z'_\infty$.
If $A' \geq \underline{A}'$, then by pseudolocality \cite{Perelman1} we have $|{\Rm}| \leq \eps'_i \to 0$ on $P(z'_i, t^*, \tfrac12 A', |t^*|)$ for large $i$ and $\nu_{x'_i, 0; t^*} (B(z'_i, t^*, \tfrac14 A' )) \geq .8$.
So if $A'$ is chosen sufficiently large, then for large $i$ we must have $(x'_i, 0) \in P(z'_i, t^*; \tfrac12 A', |t^*|)$ due to \cite[\HKCorsPnbhdPstarboth]{Bamler_HK_entropy_estimates}, see also \cite[\SYNPropLocConc]{Bamler_RF_compactness}.
The claim now follows after possibly increasing $A'$.
\end{proof}

Let $U^* \subset \RR_{t^*}$ be a connected, open and relatively compact subset with $x''_\infty \in U^*$.
Then for all $t < 0$ the map $U^* \to U^*(t)$ given by the flow of the time-vector field $\partial_{\mathfrak{t}}$ on $\RR$ is a Riemannian isometry.
Thus, using pseudolocality \cite{Perelman1} and Claim~\ref{Cl_x_pp_no_expansion} we find that for large $i$ the maps $U^* \to M_i \times \{ 0 \}$, $w \mapsto (\psi_{i, t^*} (w), 0)$ have image contained in a time-$0$-ball around $x_i$ of uniformly bounded diameter (depending on $U^*$) and subsequentially converge in (\ref{eq_CG_to_Eucl}) to an isometric embedding of the form $U^* \to \IR^n$.
It follows that every open, connected and relatively compact subset $U^* \subset \RR_{t^*}$ that contains $x''_\infty$ isometrically embeds into $\IR^n$ as a Riemannian manifold.
Due to the uniqueness of isometric embeddings into $\IR^n$ up to Euclidean motions, we obtain that the component of $\RR_{t^*} \cong M'_\infty \times \IR^{n-2}$ containing $x''_\infty$ isometrically embeds into $\IR^n$ as a Riemannian manifold.
This implies that $(N, g^N)$ contains a circle of radius $2\pi$, which proves the lemma.
\end{proof}
\bigskip

\subsection{Proof of the main proposition}
Recall that in Subsection~\ref{subsec_codim4_intro} we have reduced the main Proposition~\ref{Prop_reg_codim_4} to Proposition~\ref{Prop_eps_reg_codim_4_weaker}.

\begin{proof}[Proof of Proposition~\ref{Prop_eps_reg_codim_4_weaker}.]
Without loss of generality, we may assume that $r =1$ and $t_0 = 0$.
Assume by contradiction that the proposition was false for some fixed $Y < \infty$.
Then we can choose counterexamples $(M_i, (g_{i,t})_{i \in I_i})$, $(x_i, 0) \in M_i \times I_i$ for which $(x_i, 0)$ is $(n-2, \eps_i, 1)$-split, $(\eps_i, 1)$-static and $(\eps_i, 1)$-selfsimilar for some $\eps_i \to 0$, but $\rrm (x_i , 0) \to 0$.
After passing to a subsequence, we may assume that we have convergence $W_i := \NN_{x_i, 0} (1) \to W_\infty$.
Since $\rrm(x_i,0) \to 0$, we must have $W_\infty < 0$ by \cite[\HKThmEpsRegularity]{Bamler_HK_entropy_estimates}.
Set $d\nu_{x_i, 0} =: (4\pi \tau_i)^{-n/2} e^{-f_i} dg_i$ and choose $(n-2,  \eps_i, 1)$-splitting maps $\vec y^{\, i}$ at $(x_i, 0)$.

As discussed in the proof of Proposition~\ref{Prop_reg_codim_4}  we may pass to an $\IF$-convergent subsequence
\[ (M_i, (g_{i,t})_{t \in I_i}, (\nu_{x_i,0;t})_{t \in I_i}) \xrightarrow[i \to \infty]{\quad \IF , \CF \quad} (\mathcal{X}, (\nu_{x_\infty,0;t})_{t \leq 0}), \]
within some correspondence $\CF$, where $\XX$ is an $H_n$-concentrated, future continuous metric flow on $(-\infty,0]$ with full support and the $\IF$-convergence is understood to hold on compact time-intervals.
By Theorem~\ref{Thm_SS_dimension_bound_limit} we obtain a regular-singular decomposition $\XX = \RR {\,\, \dotcup \,\,} \SS$ with $d\nu_{x_\infty ; t} (\SS_t) = 0$ for almost every $t < 0$.
For $i = 1,2, \ldots$ set
\begin{equation} \label{eq_q_i_f_i_y_i_j_codim_3}
 q_i := 4\tau (f_i - W_i) - \sum_{j=1}^{n-2} (y^i_j)^2  
\end{equation}
By Theorem~\ref{Thm_limit_from_almost_cone} we have $\Ric \equiv 0$ and local smooth convergence $f_i \to f_\infty, \vec y^{\, i} \to \vec y^{\,\infty}, q_i \to q_\infty \in C^\infty (\RR)$ and (\ref{eq_q_i_f_i_y_i_j_codim_3}) holds for $i = \infty$.
Moreover, there is a 2-dimensional Riemannian manifold $(M'_\infty, g'_\infty)$ such that $\RR$ is of the constant form
\begin{equation} \label{eq_RR_RRp_R_n_2}
\RR \cong (M'_\infty  \times \IR^{n-2}) \times \IR_-, 
\end{equation}
where $q_\infty$ only depends on the first factor and equals the square of the radial coordinate of a Riemannian cone minus its vertex over a compact, 1-dimensional link $(N, g^N)$ and
\begin{equation} \label{eq_volN_codim_4}
 \vol (N, g^N) = e^{W_\infty} \vol (S^1, g^{S^1} ) = 2\pi e^{W_\infty} < 2\pi.
\end{equation}

Let us now apply Lemma~\ref{Lem_curv_bound_slice}.
We obtain vectors $\vec a_i \in [-1, 1]^{n-2}$ and times $t'_i \in [-2, -1]$ such that for every $F < \infty$ if $i \geq \underline{i} (F)$ and $x' \in (\vec y_i (\cdot, t'_i))^{-1} (\vec a_i) \cap \{ f_i (\cdot, t'_i) \leq F \}$, then $\rrm (x',t'_i) \geq \sigma (Y,F)$.
After passing to a subsequence, we may assume that $\vec a_i \to \vec a_\infty \in [-1,1]^{n-2}$ and $t'_i \to t'_\infty \in [-2,-1]$.

Set $\Sigma_\infty := \vec y_\infty^{-1} ( \vec a_\infty ) \cap \RR_{t'_\infty}$.
Then $\Sigma_\infty$ is a 2-dimensional submanifold whose induced metric is isometric to $(M'_\infty, g'_\infty)$ due to (\ref{eq_RR_RRp_R_n_2}).
Moreover, by the smooth convergence $\vec y^{\, i} \to \vec y^{\,\infty}$ and the implicit function theorem, we obtain that for every $x' \in \Sigma_\infty$ with $f_\infty (x') < F$ we have $\tdrrm (x') \geq \sigma (Y, F) > 0$.
It follows that $\Sigma_\infty \cap \{ f_\infty \leq F \}$ is complete for all $F < \infty$.
So by (\ref{eq_q_i_f_i_y_i_j_codim_3}), we obtain that $\{ q_\infty \leq Q \} \cap M'_\infty$ is complete for all $Q < \infty$.
So since $q_\infty$ is the square of the radial coordinate of a cone, this implies that $(M'_\infty, g'_\infty)$ is complete and thus it must be isometric to a disjoint union of copies of $\IR^2$.
It follows that $(N, g^N)$ is a disjoint union of circles of circumference $2\pi$.
Therefore, by (\ref{eq_volN_codim_4}) we have $W_\infty \geq 0$,
which yields the desired contradiction.
\end{proof}

\part{Proofs of the main theorems}
\label{part_proofs}
In the last part of this paper we explain how the main theorems, which were stated in the introduction, follow from the theory developed so far.
As most of these theorems have already been established in some form, we will often only refer to the corresponding propositions and theorems in the main part of the paper. 

\section{Proofs of the theorems from Subsections~\ref{subsec_part_reg_limit_intro}, \ref{subsec_Nash_in_limit_intro}, except for Theorem~\ref{Thm_stratification_limit_main} and Addendum~\ref{Add_limit_stratification_addendum_main}}
Let us first recall the setting from Subsections~\ref{subsec_part_reg_limit_intro}, \ref{subsec_Nash_in_limit_intro}.
Consider a sequence $(M_i, \lb (g_{i, t})_{t \in (-T_i, 0]}, \lb  (x_i, 0))$ of pointed Ricci flows on compact manifolds, let $T_\infty := \lim_{i \to \infty} T_i \in (0, \infty]$ and suppose that on compact time-intervals
\begin{equation} \label{eq_F_conv_proofs}
 (M_i, (g_{i,t})_{t \in (-T_i,0]}, (\nu_{x_i,0;t})_{t \in (-T_i, 0]}) \xrightarrow[i \to \infty]{\quad \IF, \CF \quad} (\mathcal{X}, (\nu_{x_\infty;t})_{t \in (-T_\infty,0]}), 
\end{equation}
within some correspondence $\CF$, where $\XX$ is assumed to be defined over $(-T_\infty,0]$, future continuous, $H_n$-concentrated and of full support.
This is the same setting as in Section~\ref{Sec_basic_limits}.
Note that Assumption~\ref{Aspt_working_ass} holds for $\Delta = 4$ due to Corollary~\ref{Cor_working_ass_holds}, so the results of Section~\ref{Sec_basic_limits} hold in the best possible form.

Theorem~\ref{Thm_XX_reg_sing_dec_main} is a restatement of Theorem~\ref{Thm_SS_dimension_bound_limit}\ref{Thm_SS_dimension_bound_limit_a}.

Theorem~\ref{Thm_XX_reg_determines_XX_intro} is a restatement of Theorem~\ref{Thm_SS_dimension_bound_limit}\ref{Thm_SS_dimension_bound_limit_b}--\ref{Thm_SS_dimension_bound_limit_d}.

Theorem~\ref{Thm_R_RR_star_intro} is a restatement of Corollary~\ref{Cor_RRstar_RR}.

Before establishing the remaining results of Subsection~\ref{subsec_part_reg_limit_intro}, we first prove the results from Subsection~\ref{subsec_Nash_in_limit_intro}.

Theorem~\ref{Thm_NN_pass_to_limit} is a direct consequence of Theorem~\ref{Thm_NN_in_limit}.
Theorem~\ref{Thm_Nash_on_tangent_cones} follows from Theorem~\ref{Thm_NN_pass_to_limit}, using \cite[\SYNThmTangenFlowasLimit]{Bamler_RF_compactness}, \cite[\HKThmEpsRegularity]{Bamler_HK_entropy_estimates} and Lemma~\ref{Lem_tdrrm}. 

We can now return to the results of Subsection~\ref{subsec_part_reg_limit_intro}.

Theorem~\ref{Thm_tangent_cone_metric_cone_main} and Addendum~\ref{Add_tangent_flow} follow from Theorems~\ref{Thm_limit_from_selfsimilar}, \ref{Thm_Nash_on_tangent_cones}, using \cite[\SYNThmTangenFlowasLimit]{Bamler_RF_compactness}.

\section{Proof of the theorem from Subsection~\ref{subsec_more_general_F_limit}} \label{sec_proofs_limit}
Note that we have the following lemma, which holds by the continuity of a smooth Ricci flow.

\begin{Lemma} \label{Lem_slight_timeshift}
Consider a Ricci flow $(M, (g_t)_{t \in (-T,0)})$  and a conjugate heat flow $(\mu_t)_{t \in (-T,0)}$ with $\lim_{t \nearrow 0} \Var(\mu_t) = 0$.
Let $t_j \nearrow 0$, $t_j < 0$, and consider points $x'_j \in M$ with $d^{g_{t_j}}_{W_1} (\mu_{t_j}, \delta_{x'_j}) \to 0$; such points always exist.
Then for every compact subinterval $I' \subset (-T,0)$ we have
\[ \sup_{t \in I'} d^{g_t}_{W_1}(\nu_{x'_j,t_j;t}, \mu_t) \to 0, \qquad \sup_{-\tau \in I'} |\NN_{x'_j,t_j} (\tau+t_j)-\NN_{(\mu_t)}(\tau)| \to 0. \] 
Moreover, within $\IF_{I'}$
\[ d_{\IF} \big( (M, (g_{t + t_j})_{t \in I'}, (\nu_{x'_j, t_j;t+t_j})_{t \in I'}), (M, (g_t)_{t \in I'}, (\mu_{t})_{t \in I'}) \big) \to 0. \]
\end{Lemma}

So, using \cite[\SYNFconvimplieswithin]{Bamler_RF_compactness}, we may represent any $\IF$-limit of flow pairs $(M_i, \lb (g_{i,t})_{t \in (-T_i, 0)}, \lb (\mu_{i,t})_{t \in (-T_i, 0)})$ as an $\IF$-limits of flows of the form $(M_i, (g_{i,t-\theta_i})_{t \in (-T_i+\theta_i, 0]}, (\nu_{x_i,-\theta_i;t-\theta_i})_{t \in (-T_i+\theta_i, 0]})$ for some suitable sequences $\theta_i \to 0$ and $x_i \in M_i$.
This proves Theorem~\ref{Thm_more_general_F_limit} (except for the claim concerning Theorem~\ref{Thm_stratification_limit_main} and Addendum~\ref{Add_limit_stratification_addendum_main}) and allows us to assume in the following that all $\IF$-limits are obtained from Ricci flows that are defined at time $0$.

\section{Proof of the theorems from Subsection~\ref{subsec_quant_strat_intro}}

\begin{proof}[Proof of Theorem~\ref{Thm_loc_reg_RF_intro}.]
Consider the metric flow pair $(\XX', (\nu_{x'; t})_{t \leq 0})$ from Definition~\ref{Def_symm_points}.
We claim that $\XX'$ is isometric to the constant flow on $\IR^n$.
To see this observe that in the corresponding cases of Theorem~\ref{Thm_stratification_limit_main} and Addendum~\ref{Add_limit_stratification_addendum_main} we have:
\begin{enumerate}[label=(b\arabic*)]
\item The singular part $\SS'$ of $\XX'$ satisfies $\SS'_{<0} = \SS''_{<0} \times \IR^{n-1}$.
So by Theorem~\ref{Thm_XX_reg_sing_dec_main} we must have $\SS'_{<0} = \emptyset$.
Since $\RR' = \RR'' \times \IR^{n-1}$, this implies that all time-slices of $\RR'$ are isometric to $\IR^n$.
\item The singular part $\SS'_{<0}$ of $\XX'_{<0}$ satisfies $\SS'_{<0} = \SS''_{<0} \times \IR^{n-3}$ and $\XX'_{<0} = \XX''_{<0} \times \IR^{n-3}$, where all time-slices of $\XX''_{<0}$ are isometric to the same metric cones.
So again by Theorem~\ref{Thm_XX_reg_sing_dec_main} we must have $\SS''_{<0} = \emptyset$.
Since $\Ric = 0$ on $\RR'$ and $\RR' = \RR'' \times \IR^{n-3}$, this implies that $\Rm = 0$ on $\RR'$ and since all time-slices of $\XX''_{<0} = \RR''$ are metric cones, this implies that they are isometric to $\IR^3$.
So, again all time-slices of $\RR'$ are isometric to $\IR^n$.
\end{enumerate}
The theorem now follows from a basic limit argument, Theorem~\ref{Thm_R_RR_star_intro} and Perelman's Pseudolocality Theorem \cite[10.1]{Perelman1}.
\end{proof}

Consider the following modification of Definition~\ref{Def_symm_points}:
\begin{Definition}[Weakly $(k,\eps,r)$-symmetric points] \label{Def_symm_points_weak}
Let $\XX$ be a metric flow over some interval $I$, $x_0 \in \XX_{t_0}$ a point and $r, \eps > 0$, $k \geq 0$.
We say that $x_0$ is {\bf weakly  $(k, \eps, r)$-symmetric} if $[t_0 - \eps^{-1} r^2, t_0] \subset I$ and if there is a Ricci flow $(M', (g'_t)_{t \in [-\eps^{-1},0]},x')$ on a compact manifold such that $(x',0)$ is $(\eps, 1)$-selfsimilar and
\begin{enumerate}
\item strongly $(k,\eps,1)$-split or
\item strongly $(k-2,\eps,1)$-split and $(\eps, 1)$-static,
\end{enumerate}
 such that after application of a time-shift and parabolic rescaling the metric flow pair
\[ \big(\XX_{[t_0 - \eps^{-1} r^2, t_0]}, (\nu_{x_0;t})_{t \in [t_0 - \eps^{-1} r^2, t_0]} \big) \] 
has $d_{\IF}$-distance $< \eps$ to $(M', \lb (g'_t)_{t \in [-\eps^{-1},0]}, \lb (\nu_{x';t})_{t \in [-\eps^{-1},0]})$ and such that $|\NN_{x',0}(1) - \NN_{x_0} (1)| \leq \eps$.
\end{Definition}

Analogously to Definition~\ref{Def_SS_quantitative} we define:

\begin{Definition}[Weak quantitative strata] \label{Def_SS_quantitative_weak}
Let $\XX$ be a metric flow over some interval $I$.
For $\eps > 0$ and $0 \leq r_1 < r_2 \leq \infty$ the {\bf quantitative strata},
\[ \widehat\SS^{\eps, 0}_{r_1, r_2} \subset \widehat\SS^{\eps, 1}_{r_1, r_2} \subset \widehat\SS^{\eps,2}_{r_1, r_2} \subset \ldots \subset \widehat\SS^{\eps,n-2}_{r_1, r_2} \subset  \XX \]
are defined as follows:
$x \in \widehat\SS^{\eps, k}_{r_1, r_2}$ if and only if $[t'- \eps^{-1} r_2^{2}, t'] \subset I$ and $x$ is \emph{not} weakly $(k+1, \eps, r')$-symmetric for any $r' \in (r_1, r_2)$.
\end{Definition}

By a standard compactness argument, involving the results from Subsection~\ref{subsec_part_reg_limit_intro}, we obtain:

\begin{Lemma} \label{Lem_weak_symmetric_symmetric}
Let $\XX$ be a metric flow that is obtained as an $\IF$-limit of Ricci flows as described in Subsection~\ref{subsec_more_general_F_limit}.
Let $\eps, r > 0$, $Y < \infty$ and suppose that some point $x_0 \in \XX$ satisfies $\NN_{x_0} (r^2) \geq - Y$ and is weakly $(k, \eps'(Y, \eps), r)$-symmetric.
Then it is also $(k, \eps, r)$-symmetric.
\end{Lemma}

\begin{proof}
This follows from \cite[\SYNCorCompactness]{Bamler_RF_compactness}, the results form Subsection~\ref{subsec_part_reg_limit_intro} and Theorems~\ref{Thm_limit_from_strong_splitting}, \ref{Thm_limit_from_static}, \ref{Thm_limit_from_selfsimilar}, \ref{Thm_limit_from_almost_cone}.
\end{proof}

Due to Lemma~\ref{Lem_weak_symmetric_symmetric}, Proposition~\ref{Prop_NN_variation_bound} and Lemma~\ref{Lem_conv_P_star} it suffices to prove Theorems~\ref{Thm_Quantitative_Strat}, \ref{Thm_Quantitative_Strat_XX} for the quantitative strata $\widehat\SS^{\eps, k}_{r_1, r_2}$ instead of $\SS^{\eps, k}_{r_1, r_2}$.

Theorem~\ref{Thm_Quantitative_Strat} now follows directly from Propositions~\ref{Prop_Quantitative_Strat_preliminary}, \ref{Prop_weak_splitting_map_to_splitting_map}.
The last part follows using Theorem~\ref{Thm_loc_reg_RF_intro}.

Theorem~\ref{Thm_int_curv_bounds_RF_intro} follows from Theorem~\ref{Thm_Quantitative_Strat} and the volume bound in \cite[\HKPstarVolBound]{Bamler_HK_entropy_estimates}.

Theorem~\ref{Thm_loc_reg_XX_intro} follows from Theorem~\ref{Thm_loc_reg_RF_intro} and Lemma~\ref{Lem_tdrrm} via a limit argument.

Theorem~\ref{Thm_Quantitative_Strat_XX} follows similarly as Theorem~\ref{Thm_Quantitative_Strat}, but using Proposition~\ref{Prop_prelim_quanti_strat_XX} instead of Proposition~\ref{Prop_Quantitative_Strat_preliminary} and using \cite[\SYNChangeBaseConv]{Bamler_RF_compactness}.
The last part follows using Theorem~\ref{Thm_loc_reg_XX_intro}.

Theorem~\ref{Thm_integral_bounds_F_limit} follows similarly as Theorem~\ref{Thm_int_curv_bounds_RF_intro}, but using Theorem~\ref{Thm_Quantitative_Strat_XX} and the volume bound from Lemma~\ref{Lem_limit_HK_bound}.

\section{Proof of Theorem~\ref{Thm_stratification_limit_main} and Addendum~\ref{Add_limit_stratification_addendum_main}}
Recall that Theorem~\ref{Thm_Quantitative_Strat_XX} also holds for the quantitative strata $\widehat\SS^{\eps, k}_{r_1, r_2}$.
Since
\[ \widehat\SS^{\eps, k}_{0, r_2} = \bigcap_{r_1 \in (0,r_2)} \widehat\SS^{\eps, k}_{r_1, r_2}, \]
we obtain that $\dim_{\MM^*} \widehat\SS^{\eps, k}_{0, r_2} \leq k$.
So if we set
\[ \SS^{k} := \bigcup_{\eps \in (0,1)} \widehat\SS^{\eps, k}_{0, \eps}, \]
then $\dim_{\mathcal{H}^*} \SS^k \leq k$.
The remaining properties concerning tangent flows at points in $\XX_{<0} \setminus \SS^k$ follow as in the proof of Lemma~\ref{Lem_weak_symmetric_symmetric}.

\section{Proofs of the theorems from Subsection~\ref{subsec_further_struc_static_soltion}}
By Lemma~\ref{Lem_slight_timeshift} we may assume that the metric flow $\XX$ in Theorems~\ref{Thm_static_limit_main}, \ref{Thm_metric_soliton_limit_main} is obtained as a limit of Ricci flows that are defined at time $0$, so we may assume in the following that we are in the same situation as described in Section~\ref{sec_proofs_limit}.

We will need the following definition.

\begin{Definition}
Consider a Ricci flow spacetime $(\MM, \tf, \partial_{\tf}, g)$ over some interval $I$.
A curve $\gamma : I' \to \MM$, $I' \subset I$ is called {\bf spacetime curve} if $\tf(\gamma(t)) = t$ for all $t \in I'$.
If $I' = [T_1, T_2)$, $\gamma$ is $C^1$, then for any $T\geq T_2$ we define its {\bf $\mathcal{L}$-length based at time $T$} as follows
\[ \mathcal{L}_T (\gamma) := \int_{T-T_2}^{T-T_1} \sqrt{\tau} \big( |\gamma'(T-\tau)|^2 + R(\gamma(T-\tau) \big) d\tau. \]
Here $|\gamma'(\tau)|$ denotes the norm of the projection of $\gamma'$ to $\ker d\tf$.
\end{Definition}

Note that if $\MM$ corresponds to a conventional Ricci flow on some manifold $M$, then any spacetime curve $\gamma$ corresponds to a curve of the form $(\td\gamma(t), t)$ for some $\td\gamma : I' \to M$ and $\mathcal{L}_T(\gamma)$ corresponds to the standard $\LL$-length of $\td\gamma$ based at time $T$.

\begin{Lemma} \label{Lem_LL_in_XX}
For any $\eps > 0$ there are constants $c(\eps) > 0$, $C(\eps) < \infty$ such that the following holds.
Let $\XX$ be an $\IF$-limit of Ricci flows, as described in Section~\ref{sec_proofs_limit}.
Consider a $C^1$ spacetime curve $\gamma : [s, t] \to \RR$ between points $x := \gamma(s)$ and $y := \gamma(t)$ and assume that $s + T_\infty > \eps$.
Then
\begin{equation} \label{eq_K_lower_lem}
 K(x;y) \geq (4\pi(t-s))^{-n/2} \exp \bigg( {- \frac{\LL_t (\gamma)}{2 \sqrt{t-s}} } \bigg), 
\end{equation}
\begin{equation} \label{eq_d_XX_nu_xy_LL}
 d^{\XX_s}_{W_1} (\delta_y, \nu_{x;s})
\leq C(\eps)\bigg( 1+ \frac{\LL_t (\gamma)}{2 \sqrt{t-s}} - \NN_x (t-s) \bigg)^{1/2} \sqrt{t-s} =: r ,
\end{equation}
\begin{equation} \label{eq_lower_vol_LL}
 |B(y, r \big) \cap \RR_t |_{g_t} \geq c(\eps) \exp (\NN_x(t-s)) r^n. 
\end{equation}
\end{Lemma}

\begin{proof}
Consider the maps $\psi_i : U_i \to V_i$ describing the smooth convergence on $\RR = \RR^* \subset \XX$ (compare with the discussion in Subsection~\ref{subsec_setup_limit}).
Then for large $i$, we have $\gamma([s,t]) \subset U_i$ and if we set $\gamma_i := \psi_i \circ \gamma$, then 
\[ \lim_{i \to \infty} \LL_t (\gamma_i) = \LL_t(\gamma). \]
Therefore, by \cite[9.5]{Perelman1} we have for large $i$
\begin{equation*} 
K(x;y) = \lim_{i \to \infty} K_i (\psi_i(x); \psi_i(y) ) \geq  (4\pi(t-s))^{-n/2} \exp \bigg({ - \frac{\LL_t(\gamma)}{2 \sqrt{t-s}} }\bigg), 
\end{equation*}
which implies (\ref{eq_K_lower_lem}).
On the other hand, we obtain from Lemma~\ref{Lem_limit_HK_bound}\ref{Lem_limit_HK_bound_a} that for any $y' \in \RR_s$
\begin{equation} \label{eq_K_i_upper_HK_bound}
 K(x;y') \leq \frac{C(\eps)}{(t-s)^{n/2}}  
\exp \bigg( { - \NN_x(t-s) - \frac{(d^{\XX_s}_{W_1}(\nu_{x;s}, \delta_{y'}))^2}{10(t-s)} }\bigg). 
\end{equation}
Combining (\ref{eq_K_lower_lem}), (\ref{eq_K_i_upper_HK_bound}) implies (\ref{eq_d_XX_nu_xy_LL}).

For (\ref{eq_lower_vol_LL}) we may use Lemma~\ref{Lem_mass_ball_Var} and possibly adjust $r$ to conclude
\[ \int_{B(y,r)} K(x;\cdot) dg_s = \nu_{x;s} (B(y,r)) \geq \tfrac12 \]
Combining this bound with (\ref{eq_K_i_upper_HK_bound}) implies (\ref{eq_lower_vol_LL}).
\end{proof}
\bigskip

\begin{proof}[Proof of Theorem~\ref{Thm_static_limit_main}.]
We first prove the following claim.

\begin{Claim}
For any compact interval $[T_1, T_2] \subset (-T_\infty,0)$ we have
\[ \int_{T_1}^{T_2}\int_{M_i} |{\Ric}|^2 d\nu_{x_i,0;t}dt \to 0. \]
\end{Claim}

\begin{proof}
Fix $[T_1, T_2] \subset (-T_\infty,0)$ and choose $T'_1, T'_2$ such that $[T_1, T_2] \subset (T'_1, T'_2) \subset (-T_\infty,0)$.
By Proposition~\ref{Prop_improved_L2} we have
\begin{equation} \label{eq_RsquTpp1Tpp2}
 \int_{T'_1}^{T'_2} \int_{M_i} R^2 d\nu_{x_i,0;t} dt \leq C < \infty. 
\end{equation}
On the other hand, since $\Ric = 0$ on $\RR$, we can find subsets $Y_i \subset M_i \times [T'_1, T'_2]$ with
\[ \sup_{Y_i} |R| \to 0, \qquad \int_{T'_1}^{T'_2} \int_{M_i \times \{t \} \setminus Y_i}  d\nu_{x_i,0;t} dt \to 0. \]
Combining this with (\ref{eq_RsquTpp1Tpp2}) implies 
\[  \int_{T'_1}^{T'_2} \int_{M_i} |R| \, d\nu_{x_i,0;t} dt \to 0. \]
Let now $\eta : (T'_1, T'_2) \to [0,1]$ be a compactly supported cutoff function with $\eta \equiv 1$ on $[T_1, T_2]$.
Then
\begin{multline*}
  \int_{T_1}^{T_2} \int_{M_i} |{\Ric}|^2 d\nu_{x_i,0;t} dt
\leq 2 \int_{T'_1}^{T'_2}\eta(t)  \int_{M_i}  |{\Ric}|^2 d\nu_{x_i,0;t} dt \\
= \int_{T'_1}^{T'_2} \eta(t)  \int_{M_i} \square R \, d\nu_{x_i,0;t} dt
= -  \int_{T'_1}^{T'_2} \eta'(t)\int_{M_i}  R \, d\nu_{x_i,0;t} dt \to 0 \qedhere
\end{multline*}
\end{proof}
\medskip

We may now apply Theorem~\ref{Thm_limit_from_static} and obtain that $\XX_{<0}$ is static.
Moreover, we obtain an identification $\XX_{<0} = X \times (-T_\infty,0)$ for a static model
\[ (X,d, (\nu'_{x;t})_{x \in X; (-T_\infty,0) \cap (t + (-T_\infty,0) \neq \emptyset} ) \]
such that Assertion~\ref{Thm_static_limit_main_a} of the theorem holds and there is an open subset $\RR^*_X \subset \RR_X$ such that Assertions~\ref{Thm_static_limit_main_b}, \ref{Thm_static_limit_main_c} hold if we replace $\RR_X$ with $\RR^*_X$.
In addition, $\RR^*_X \subset X$ is dense and $d|_{\RR^*_X}$ is equal to the length metric of $(\RR^*_X, g_X|_{\RR^*_X})$.
It follows from Theorem~\ref{Thm_SS_dimension_bound_limit}, see also Theorem~\ref{Thm_XX_uniquely_det_RR}, that the families of conjugate heat kernels $(\nu'_{x;t})_{x \in X; (-T_\infty,0) \cap (t + (-T_\infty,0) \neq \emptyset}$ are uniquely determined by $(X, d, \RR^*_X)$.
By \cite[\HKThmUpperVolBound]{Bamler_HK_entropy_estimates} we also obtain that for any compact subset $K \subset X$ and any $D < \infty$ there is a constant $\kappa_2(K,D) < \infty$ such that
\[ |B(x,r) \cap \RR^*_X|_{g_X} < \kappa_2 r^n. \]

Next we apply Lemma~\ref{Lem_LL_in_XX} and obtain:

\begin{Claim} \label{Cl_HK_no_drift}
For any $\eps > 0$ there are constants $c(\eps) > 0$, $C(\eps) < \infty$ such that the following holds:
\begin{enumerate}[label=(\alph*)]
\item \label{Cl_HK_no_drift_a} For any $x \in \RR_t$, $t < 0$, and $s < t$, $s \in (-T_\infty,0)$ with $T_\infty + s \geq \eps$ we have
\[  d^{\XX_s}_{W_1} (\delta_{x(s)}, \nu_{x;s}) \leq C(\eps) (1-\NN_{x}(t-s))^{1/2} \sqrt{t-s} =: r \]
and
\[ |B(x(s), r) \cap \RR^*_X| \geq c(\eps) \exp(\NN_x(t-s)) r^n. \]
The same bounds also hold if $x \in \XX_t$ and $x(s)$ is replaced with the point in $\XX_s$ with the same first coordinate with respect to the identification $\XX_{<0} = X \times (-T_\infty,0)$.
\item \label{Cl_HK_no_drift_b} For any compact subset $K \subset X$ and any $D < \infty$ there is a constant $\kappa_1(K,D) > 0$ such that
\[ |B(x,r) \cap \RR^*_X|_{g_X} > \kappa_1 r^n. \]
\item \label{Cl_HK_no_drift_c} \label{Cl_HK_properties_c} If we consider the identification $\XX_{<0} = X \times (-T_\infty,0)$, then we have the following estimate for $P^*$-parabolic neighborhoods around any point $x \in \XX_t$, whenever  $0 < r \leq \eps^{-1}$, $[t - T^-, t + T^+] \subset (-T_\infty + 2\eps, -\eps)$, $\eps > 0$, $0 \leq T^-, T^+ \leq \eps^{-1}$ and $\NN_x (\eps) \geq -\eps^{-1}$
\begin{multline*}
 \qquad\qquad B\big(x, r - C(\eps) \sqrt{T^-+T^+} \big) \times [t - T^-, t + T^+] \subset P^* (x, t; r, -T^-, T^+) \\ \subset B\big(x,  r + C(\eps) \sqrt{T^-+T^+}\big) \times [t - T^-, t + T^+]. 
\end{multline*}
In particular, if $t - r^2 > -T_{\infty} +2\eps$, then
\begin{equation*}
\qquad\qquad B( x, c(\eps) r) \times [t - c(\eps)r^2, t+c(\eps)r^2] \subset P^* (x, t; r) 
 \subset  B( x, C(\eps) r) \times [t- r^2,t+r^2]. 
\end{equation*}
\item \label{Cl_HK_no_drift_d} The natural topology on $\XX_{<0}$ agrees with the product topology on $X \times (-T_\infty, 0) = \XX_{<0}$.
\end{enumerate}
\end{Claim}

\begin{proof}
The first part of Assertion~\ref{Cl_HK_no_drift_a} follows by applying Lemma~\ref{Lem_LL_in_XX} to spacetime curves of the form $t \mapsto x(t)$, which have zero $\LL$-length.
The second part follows from the first part via a limit argument.
Assertion~\ref{Cl_HK_no_drift_b} is a direct consequence of Assertion~\ref{Cl_HK_no_drift_a}; note that it suffices to prove this assertion for small $r$.
Assertion~\ref{Cl_HK_no_drift_c} is a consequence of Assertion~\ref{Cl_HK_no_drift_a}; observe that by passing Proposition~\ref{Prop_NN_variation_bound} to the limit, using Lemma~\ref{Lem_conv_P_star}, we obtain a lower bound on the pointed Nash entropy on $P^* (x, t; r, -T^-, T^+)$ in terms of $\eps$.
Assertion~\ref{Cl_HK_no_drift_d} is a direct consequence of Assertion~\ref{Cl_HK_no_drift_c}.
\end{proof}

Claim~\ref{Cl_HK_no_drift}\ref{Cl_HK_no_drift_c} and Theorem~\ref{Thm_SS_dimension_bound_limit} imply that for $\SS^*_X := X \setminus \RR^*_X$ we have
\begin{equation} \label{eq_dim_SS_star}
 \dim_{\mathcal{M}} \SS^*_X \leq n-4. 
\end{equation}
It follows that $\RR_X \setminus \RR^*_X = \SS^*_X \setminus \SS_X$ is a set of measure zero and thus $(X,d)$ is a singular space.

\begin{Claim} \label{Cl_RRXRRXstar}
$\RR^*_X = \RR_X$.
\end{Claim}

\begin{proof}
Consider the identification $\XX_{<0} = X \times (-T_\infty,0)$ and recall that $\RR = \RR^*_X \times (-T_\infty,0)$.
By Claim~\ref{Cl_HK_no_drift}\ref{Cl_HK_no_drift_d} we can extend the Ricci flow spacetime structure on $\RR$ to an open subset of the form $\RR' := \RR_X \times (-T_\infty,0) \subset \XX_{<0}$.
We claim that $\RR' = \RR$.
To see this, we need to show that bounded heat flows and conjugate heat flows of $\XX$, over open time-intervals, are smooth on $\RR'$.
Note that these flows are smooth on $\RR$.
In addition, in order to prove smoothness of conjugate heat flows on $\RR'$ it also suffices to prove smoothness of conjugate heat kernels, due to the reproduction formula.
By Lemma~\ref{Lem_limit_HK_bound}\ref{Lem_limit_HK_bound_a} these flows are of the form $v \,  dg_t$ on $\RR$ for some locally bounded scalar function $v$.

So in summary, it suffices to show that any locally defined and bounded solution to the heat equation or conjugate heat equation on $\RR$ can be extended smoothly to $\RR'$.
Due to (\ref{eq_dim_SS_star}) such solutions are local weak solutions to the corresponding equation, so therefore they are smooth.
\end{proof}

Claim~\ref{Cl_RRXRRXstar} and (\ref{eq_dim_SS_star}) imply Assertion~\ref{Thm_static_limit_main_d} of this theorem.
Assertions~\ref{Thm_static_limit_main_e}, \ref{Thm_static_limit_main_f} follow from what we have shown so far, from Theorems~\ref{Thm_tangent_cone_metric_cone_main}, \ref{Thm_stratification_limit_main} and from the following claim, which allows to deduce Gromov-Hausdorff convergence from Gromov-$W_1$-convergence.

\begin{Claim} \label{Cl_lower_ptwise_K}
For any $x \in \XX_t$, $t < 0$, and $D <\infty$ the following is true for $r \leq \ov r (x,D)$.
Let $x' \in \XX_{t - r^2}$ be the point whose projection to the first factor of $\XX_{<0} = X \times (-T_\infty,0)$ agrees with that of $x$.
Then for all $y \in \RR_{t-r^2}$ with $d_{t - r^2} (x',y) \leq D r$ we have
\[ K(x;y) \geq \frac{c}{(t-s)^{n/2} } \exp\bigg({ - \frac{d^2_{t-r^2} (x',y) }{10 r^2} }\bigg). \]
\end{Claim}

\begin{proof}
This is a direct consequence of Lemma~\ref{Lem_LL_in_XX} and the fact that the restriction of $d$ to $\RR_X$ is equal to the length metric of $g_X$.
\end{proof}

This finishes the proof.
\end{proof}
\bigskip

\begin{proof}[Proof of Theorem~\ref{Thm_metric_soliton_limit_main}.]
The proof is similar to that of Theorem~\ref{Thm_static_limit_main}.
In the following we only explain the modifications that need to be made.
Theorem~\ref{Thm_limit_from_selfsimilar} implies the first two identities in (\ref{eq_sol_id_met_sol_limit_thm}) and the existence of a metric measure space $(X,d, \mu)$ and an identification $\XX_{<0} = X \times (-T_\infty,0)$ such that Assertions~\ref{Thm_metric_soliton_limit_main_a}--\ref{Thm_metric_soliton_limit_main_d} hold, where we need to replace $\RR_X$ with some dense open subset $\RR^*_X \subset \RR_X$.
The last identity of (\ref{eq_sol_id_met_sol_limit_thm}) follows from the first one and the evolution equation (\ref{eq_potential_evolution_equation}) for $f$.
The fact that $(X,d, \mu)$ characterizes $(\XX_{<0}, (\nu_{x_\infty;t})_{t \in (T_\infty,0)})$ up to flow-isometry follows as in the proof of Theorem~\ref{Thm_static_limit_main}.

Due to selfsimilary of $\RR$ and the fact that $R$ is bounded from below, we must have $R \geq 0$ on $\RR$.
So the second identity in (\ref{eq_sol_id_met_sol_limit_thm}) implies that $\tau |\nabla f|^2 \leq f - W$, which implies a bound on $f$ on bounded subsets of time-slices.
So, again by the second identity in (\ref{eq_sol_id_met_sol_limit_thm}), we obtain uniform bounds on $|\nabla f|$ and $R$ on bounded subsets of time-slices.
Therefore, we may apply Lemma~\ref{Lem_LL_in_XX} to trajectories of $\partial_{\tf} - \nabla f$ and obtain Claim~\ref{Cl_HK_no_drift} from the proof of Theorem~\ref{Thm_static_limit_main} under the additional assumption that $f(x) \leq \eps^{-1}$.
In a similar fashion, we can account for the vector field $\nabla f$ in the proof of Claim~\ref{Cl_lower_ptwise_K}.
The fact that tangent cones of $(X,d)$ are Ricci flat follows from the fact that the scalar curvature on $\RR_X$ is locally bounded; so blowups of $\XX$ have $R = 0$ on the regular part and therefore, by the evolution equation for $R$, we obtain that $\Ric = 0$ on the regular part of blowups.
\end{proof}
\bigskip

\begin{proof}[Proof of Theorem~\ref{Thm_almost_cone_rigidity_intro}.]
This follows from Theorem~\ref{Thm_metric_soliton_limit_main} via a compactness argument.
\end{proof}
\bigskip

\section{Proofs of Theorems from Subsection~\ref{subsec_first_sing_intro}}
Consider the definition of $(M_T, d_{M_T})$ from Definition~\ref{Def_MT}.
We first establish the claim below this definition.
Let $(\mu_t)_{t \in [0,T)} \in M_T$ and consider a sequence $t_i \nearrow T$.
Choose $y_i \in M$ such that $\Var (\delta_{y_i}, \mu_{t_i}) \leq H_n (T - t_i)$.
Then for any $t \in [0,T)$ we have for sufficiently large $i$
\[ d^{g_t}_{W_1} (\nu_{y_i,t_i;t}, \mu_{t})
\leq d^{g_{t_i}}_{W_1} (\delta_{y_i}, \mu_{t_i})
\leq \sqrt{\Var (\delta_{y_i}, \mu_{t_i})}
\leq \sqrt{H_n (T - t_i)} \to 0. \]
\bigskip

\begin{proof}[Proof of Lemma~\ref{Lem_MT}.]
To see the completeness, consider a Cauchy sequence $(\mu_{i,t})_{t \in [0,T)}$ in $(M_T, \lb d_{M_T})$.
Since by monotonicity of the $W_1$-distance we have for any $t \in [0,T)$
\[ d^{g_t}_{W_1} (\mu_{i,t}, \mu_{j,t})
\leq d_{M_T} \big( (\mu_{i,t})_{t \in [0,T)}, (\mu_{j,t})_{t \in [0,T)} \big) \xrightarrow[i,j \to \infty]{} 0, \]
we obtain that for any $t \in [0,T)$
\[ \mu_{i,t} \xrightarrow[i \to \infty]{} \mu_{\infty,t}, \]
where $(\mu_{\infty,t})_{t \in [0,T)} \in M_T$.
To see that $(\mu_{i,t})_{t \in [0,T)} \to (\mu_{\infty,t})_{t \in [0,T)}$ in $(M_T, d_{M_T})$, fix some $\eps > 0$ and choose $i$ large enough such that for all $j \geq i$ and $t \in [0,T)$
\[ d^{g_t}_{W_1} (\mu_{i,t}, \mu_{j,t}) \leq d_{M_T} \big( (\mu_{i,t})_{t \in [0,T)},(\mu_{i,t})_{t \in [0,T)} \big) \leq \eps . \]
Letting $j \to \infty$ implies that for all $t \in [0,T)$
\[ d^{g_t}_{W_1} (\mu_{i,t}, \mu_{\infty,t}) \leq \eps, \]
which implies
\[ d_{M_T} \big( (\mu_{i,t})_{t \in [0,T)}, (\mu_{\infty,t})_{t \in [0,T)} \big) \leq \eps. \]
The second assertion can be shown using \cite[\HKCenterConstantRmBound]{Bamler_HK_entropy_estimates}.
\end{proof}
\bigskip

\begin{proof}[Proof of Theorem~\ref{Thm_first_sing_time_intro}.]
Fix some $(\mu_t)_{t \in [0,T)} \in M_T$ and consider a sequence $(y_j,t_j) \in M \times [0,T)$ such that for any fixed $t \in [0,T)$ we have $\nu_{y_j,t_j;t} \to \mu_t$ as $j \to \infty$.
By local smooth convergence of the conjugate heat kernel we also have $\NN_{y_j,t_j} (t_j - t) \to \NN_{(\mu_t)} (T-t)$ as $j \to \infty$.
So as in the proof of Lemma~\ref{Lem_slight_timeshift} the $\IF$-convergence in (\ref{eq_IF_conv_first_sing_time}) can be replaced by 
\[ \big( M, (\tau_i^{-1} g_{\tau_i t+ t_{j_i}} )_{t \in [-\tau_i^{-1}t_{j_i}, 0)}, (\nu_{y_{j_i},t_{j_i};\tau_i t+ t_{j_i}})_{t \in [-\tau_i^{-1} t_{j_i}, 0)} \big) \xrightarrow[i \to \infty]{\quad \IF, \CF \quad} \big( \XX, (\mu^\infty_t)_{t < 0} \big) \]
The theorem follows now using \cite[\SYNCorCompactness]{Bamler_RF_compactness}, the results from Subsections~\ref{subsec_part_reg_limit_intro}, \ref{subsec_Nash_in_limit_intro}, Theorem~\ref{Thm_metric_soliton_limit_main} and \cite[\HKThmEpsRegularity]{Bamler_HK_entropy_estimates}.
\end{proof}
\bigskip

\section{Proof of the Theorem from Subsection~\ref{subsec_asympt_soliton}}
The theorem follows again using \cite[\SYNCorCompactness]{Bamler_RF_compactness}, the results from Subsections~\ref{subsec_part_reg_limit_intro}, \ref{subsec_Nash_in_limit_intro} and Theorem~\ref{Thm_metric_soliton_limit_main}.

\section{Proof of the Theorem from Subsection~\ref{subsec_thick_thin_intro}}
Consider an immortal Ricci flow $(M, (g_t)_{t \geq 0})$ on a compact, $n$-dimensional manifold.
By Lemma~\ref{Lem_lower_scal} we have
\[ R + \frac{n}{2t} \geq 0. \]
So if $V(t) := |M|_t$ denotes the volume of $(M,g_t)$, then
\[ \frac{dV}{dt} = - \int_M R \, dg_t = - \int_M \Big( R + \frac{n}{2t} \Big) dg_t + \frac{n}{2t} V, \]
which implies that
\[ \frac{d}{dt} \big( t^{-n/2} V(t) \big)
= - t^{-n/2} \int_M \Big( R + \frac{n}{2t} \Big) dg_t \leq 0. \]
So
\[ \int_1^\infty \int_M \Big( R + \frac{n}{2t} \Big) dg_t dt < \infty. \]
It follows that for any $\eps > 0$ we have
\[ \int_{\eps t_0}^{t_0} \int_M \Big| R + \frac{n}{2t} \Big| dg_t dt \xrightarrow[t_0 \to \infty]{} 0. \]
So Theorem~\ref{Thm_longtime_limit_intro} is a consequence of the following proposition.

\begin{Proposition}
Consider a sequence of pointed Ricci flows $(M_i, (g_{i,t})_{t \in (0,1]}, x_i)$ on compact, $n$-dimensional manifolds such that we have $\IF$-convergence
\[ (M_i, (g_{i,t})_{t \in (0,1]}, (\nu_{x_i,1;t})_{t \in (0, 1]}) \xrightarrow[i \to \infty]{\quad \IF, \CF \quad} (\mathcal{X}, (\nu_{x_\infty;t})_{t \in (0,1]}), \]
within some correspondence $\CF$, where we  assume that $\XX$ is a future continuous and $H_n$-concentrated metric flow of full support over $(0,1]$ (this is the same setting as assumed in Subsection~\ref{subsec_part_reg_limit_intro} modulo a time-shift).

Suppose that for any $\eps > 0$ we have
\[ \int_{\eps}^{1-\eps} \int_{M_i} \Big| R + \frac{n}{2t} \Big| d\nu_{x_i,1;t} dt \xrightarrow[i \to \infty]{} 0. \]
Then $\XX$ satisfies all assertions of Theorem~\ref{Thm_longtime_limit_intro}.
\end{Proposition}

\begin{proof}
Fix some small $\eps > 0$ and let $\eta : (\eps,1-\eps) \to [0,1]$ be a compactly supported cutoff function with $\eta \equiv 1$ on $[2\eps, 1-2\eps]$ and $|\eta'| \leq 10 \eps^{-1}$.
Then
\begin{align*}
 \int_{2\eps}^{1-2\eps} \int_{M_i} & \Big| \Ric + \frac{1}{2t} g_t \Big|^2d\nu_{x_i,1;t} dt
\leq \int_{\eps}^{1-\eps} \int_{M_i} \eta(t) \Big| \Ric + \frac{1}{2t} g_t \Big|^2d\nu_{x_i,1;t} dt \\
&= \int_{\eps}^{1-\eps} \int_{M_i} \eta(t) \Big( |{\Ric}|^2 + \frac{1}{t} R + \frac{n}{4t^2} \Big) d\nu_{x_i,1;t} dt \displaybreak[1] \\
&\leq \int_{\eps}^{1-\eps} \int_{M_i} \eta(t) \Big( \frac12 \square R - \frac{n}{4t^2}  \Big) d\nu_{x_i,1;t} dt  + \eps^{-1} \int_{\eps}^{1-\eps} \int_{M_i} \Big| R + \frac{n}{2t} \Big| d\nu_{x_i,1;t} dt \displaybreak[1] \\
&= - \frac12 \int_{\eps}^{1-\eps} \int_{M_i}  \eta' (t) R \,   d\nu_{x_i,1;t} dt - \int_{\eps}^{1-\eps} \eta(t) \frac{n}{4t^2} \, dt  + \eps^{-1} \int_{\eps}^{1-\eps} \int_{M_i} \Big| R + \frac{n}{2t} \Big| d\nu_{x_i,1;t} dt \displaybreak[1] \\
&\leq  \frac12 \int_{\eps}^{1-\eps} \eta' (t) \frac{n}{2t}\, dt  - \int_{\eps}^{1-\eps} \eta(t) \frac{n}{4t^2} \, dt + C (\eps) \int_{\eps}^{1-\eps} \int_{M_i} \Big| R + \frac{n}{2t} \Big| d\nu_{x_i,1;t} dt \\
&= C (\eps)\int_{\eps}^{1-\eps} \int_{M_i} \Big| R + \frac{n}{2t} \Big| d\nu_{x_i,1;t} dt \xrightarrow[i \to \infty]{} 0.
\end{align*}
The proof is now analogous to the proof of Theorem~\ref{Thm_static_limit_main} with suitable modifications in Theorem~\ref{Thm_limit_from_static} and its proof.
\end{proof}
\bigskip

\section{Proofs of the Theorems from Subsection~\ref{subsec_bckw_pseudo}}
We first show the following lemma:

\begin{Lemma} \label{Lem_bckwd_pseudo_preliminary}
For $\alpha > 0$ there is an $\eps ( \alpha) > 0$ such that the following holds.

Let $(M, (g_t)_{t \in I})$ be a Ricci flow on a compact, $n$-dimensional manifold and let $(x_0, t_0) \in M \times I$ a point and $r > 0$ a scale such that $[t_0 -   r^2, t_0] \subset I$ and:
\begin{enumerate}[label=(\roman*)]
\item $|B(x_0,t_0,r)| \geq \alpha r^n$.
\item \label{Lem_bckwd_pseudo_preliminary_ii} $|{\Rm}| \leq \alpha^{-1} r^{-2}$ on $P(x_0, t_0; r, -(\alpha r)^2)$.
\item $R(x_0,t_0) \leq \eps r^{-2}$.
\item $R (\cdot, t_0 - r^2) \geq - \eps r^{-2}$.
\end{enumerate}
Then $|{\Rm}| \leq \alpha r^{-2} + \sup_{B(x_0, t_0,r)} |{\Rm}|$ and $|{\Ric}| \leq \alpha r^{-2}$ on $P(x_0,t_0; (1-\alpha) r, - (1-\alpha) r^2)$. 
\end{Lemma}

\begin{proof}
Without loss of generality, we may assume that $r = 1$, $t_0 = 0$ and $I = [-1, 0]$.
Assume by contradiction that the lemma is false for some fixed $\alpha > 0$.
Choose a sequence of counterexamples $(M_i, (g_{i, t})_{t \in [-1,0]})$, $x_i \in M_i$ for some sequence $\eps_i \to 0$.
By \cite[\HKThmUpperVolBound]{Bamler_HK_entropy_estimates} we have a uniform bound of the form $\NN_{x_i,0}(\frac12) \geq - C(\alpha)$.
So after passing to a subsequence, we may assume that we have $\IF$-convergence within some correspondence $\CF$
\begin{equation} \label{eq_F_conv_pseudo}
 (M_i, (g_{i, t})_{t \in (-1,0]}), (\nu_{x_i,0;t})_{t \in (-1,0] }) \xrightarrow[i \to \infty]{\quad \IF, \CF \quad } (\XX, (\nu_{x_\infty;t})_{t \in (-1,0]}), 
\end{equation}
where $\XX$ is an $H_n$-concentrated, future continuous metric flow of full support over $(-1,0]$.
By Theorem~\ref{Thm_limit_from_static} the flow $\XX_{<0}$ is static.

Fix some $t^* \in (-\alpha^2, 0)$ close to $0$ such that the convergence in (\ref{eq_F_conv_pseudo}) is time-wise at time $t^*$.
Let $z_\infty \in \XX_{t^*}$ be an $H_n$-center of $x_\infty$ and choose points $z_i \in M_i$ such that we have strict convergence $z_i \to z_\infty$ within $\CF$.
By Lemma~\ref{Lem_mass_ball_Var} we have $\nu_{x_i,0;t^*} (B(z_i, t^*,2\sqrt{2H_n |t^*|})) \geq \frac12$ for large $i$.
So if $(z'_i, t^*) \in M_i \times \{ t^* \}$ is an $H_n$-center of $(x_i,0)$, then again by Lemma~\ref{Lem_mass_ball_Var}
\[ d_{W_1}^{g_{t^*}} ( \delta_{z_i}, \nu_{x_i,0;t^*} )
\leq d_{t^*} (z_i, z'_i) + d_{W_1}^{g_{t^*}} ( \delta_{z'_i}, \nu_{x_i,0;t^*} )
\leq 3 \sqrt{2H_n |t^*|} + \sqrt{H_n |t^*|} 
= 4\sqrt{2H_n |t^*|}. \]
It follows, using \cite[\HKCenterConstantRmBound]{Bamler_HK_entropy_estimates}, that
\begin{equation} \label{eq_dtstarzixi}
  d_{t^*} (z_i, x_i)
\leq d_{W_1}^{g_{t^*}} ( \delta_{z_i}, \nu_{x_i,0;t^*} )
+ d_{W_1}^{g_{t^*}} ( \nu_{x_i,0;t^*},  \delta_{x_i} )
\leq C(\alpha) \sqrt{|t^*|}. 
\end{equation}

So for every $\delta > 0$ we can choose some $t^* \in (-\delta, 0)$ such that for large $i$ we have $|{\Rm}| \leq \alpha^{-1}$ on $P(z_i, t^*; 1-\delta -(\alpha+ t^*), -t^*)$.
Therefore, by \cite[\SYNConvParabNbhd]{Bamler_RF_compactness} and the fact that $\XX_{<0}$ is static we obtain that $B_{t^*} := B(z_\infty, 1-\delta) \subset \RR_{t^*}$ and
$|{\Rm}| \leq \alpha^{-1}$ on $B_{t^*}$.
Moreover, the open parabolic neighborhood $P_{t^*} := B_{t^*} ( (-1,0)) \subset \RR$ is unscathed and $|{\Rm}| \leq \alpha^{-1}$ and $\Ric = 0$ on $P_{t^*}$.
Due to the smooth convergence on $\RR$, this implies that for large $i$ we have $|{\Rm}| \leq 2\alpha^{-1}$ and $|{\Ric}| \leq \delta$ on $P(z_i, t^*; 1-2\delta, - (1-\delta + t^*), - t^* - \delta)$.
Combining this with (\ref{eq_dtstarzixi}) implies that for large $i$ we have $|{\Rm}| \leq 2\alpha^{-1}$ and $|{\Ric}| \leq \delta$ on \[ P(x_i, t^*; 1- 2\delta - C(\alpha) \sqrt{\delta},  - (1-\delta + t^*), - t^* - \delta). \]
Using Assumption~\ref{Lem_bckwd_pseudo_preliminary_ii}, this yields the desired contradiction for small enough $\delta$.
\end{proof}

By an induction argument over scales, Lemma~\ref{Lem_bckwd_pseudo_preliminary} implies the following lemma:

\begin{Lemma} \label{Lem_almost_bckwds_pseudo}
For $\alpha > 0$ there is an $\eps ( \alpha) > 0$ such that the following holds.

Let $(M, (g_t)_{t \in I})$ be a Ricci flow on a compact, $n$-dimensional manifold and let $(x_0, t_0) \in M \times I$ a point and $r > 0$ a scale such that $[t_0 -   r^2, t_0] \subset I$ and:
\begin{enumerate}[label=(\roman*)]
\item $|B(x'_0,t_0,r')| \geq \alpha r^{\prime n}$ for any ball $B(x'_0, t_0, r') \subset B(x_0, t_0, r)$.
\item $|{\Rm}| \leq \alpha^{-1} r^{-2}$ and $R \leq \eps r^{-2}$ on $B(x_0, t_0, r)$.
\item $R  \geq - \eps r^{-2}$ on $M \times [t_0 - r^2, t_0]$.
\end{enumerate}
Then $|{\Rm}| \leq 2\alpha^{-1} r^{-2}$ on $P(x_0,t_0; \frac12 r, - \frac12 r^2)$. 
\end{Lemma}

\begin{proof}
Fix $\alpha \in (0,1)$ and choose $\eps (\frac18 (\frac34)^2 \alpha)$ as in Lemma~\ref{Lem_bckwd_pseudo_preliminary}.
Assume that the lemma was false.
After possibly replacing $x_0, r$ with some $x'_0, r'$ such that $B(x'_0, t_0, r') \subset B(x_0, t_0,r)$, we may assume that the lemma holds for any such point and scale if $r' \leq \frac12 r$.
Since the assumptions of the lemma continue to hold after such a replacement, we obtain that for any $B(x'_0, t_0, r') \subset B(x_0, t_0,r)$ with $r' \leq \frac12 r$ we have $|{\Rm}| \leq 2 \alpha^{-1} r^{\prime -2}$ on $P(x'_0,t_0; \frac12 r', - \frac12 r^{\prime 2} )$.
This implies that
\[ |{\Rm}| \leq 8 \alpha^{-1} r^{-2} = (\tfrac18 (\tfrac34)^2 \alpha)^{-1} ( \tfrac34 r)^{-2} \qquad \text{on} \quad P(x_0,t_0; \tfrac34 r, - \tfrac18 r^{2} ). \]
Applying Lemma~\ref{Lem_bckwd_pseudo_preliminary} to the ball $B(x_0,t_0, \frac34r)$ implies that
\[ |{\Rm}| \leq \tfrac18 (\tfrac34)^2 \alpha (\tfrac34)^{-2} +  \alpha^{-1} r^{-2} \leq 2 \alpha^{-2} r^{-2} \qquad \text{on} \quad P(x_0,t_0; \tfrac12 r, - \tfrac12 r^{2} ),  \]
in contradiction to our assumption.
\end{proof}
\bigskip

\begin{proof}[Proof of Theorem~\ref{Thm_bckwds_pseudo}.]
The theorem follows by applying Lemma~\ref{Lem_almost_bckwds_pseudo} to sufficiently small subballs of $B(x_0, t_0, r)$.
\end{proof}
\bigskip

\begin{proof}[Proof of Corollary~\ref{Cor_combination_fwd_bckwd_pseudo}.]
After application of a time-shift, parabolic rescaling and possibly replacing $B(x_0, t_0, r)$ with a smaller subball, we may assume that $r = 1$, $t_0 = 0$, $[-1, 1] \subset I$ and $R \geq -\eps$ and it suffices to show that 
\begin{equation} \label{eq_needtoshowRmepsm2}
|{\Rm}|\leq \eps^{-2} \qquad \text{on} \quad B(x_0, t_0, \eps ).
\end{equation}
Let $\delta > 0$ be a constant whose value we will determine later.
By applying Perelman's Pseudolocality Theorem \cite[10.1]{Perelman1} to subballs of the form $B(x'_0, 0, \frac12 ) \subset B(x_0, 0, 1)$ and a distance distortion estimate, we obtain that if $\eps \leq \ov\eps (\delta)$, then $|{\Rm}| \leq \delta t^{-1}$ on $B(x_0, 0, \frac12 ) \times (0,1)$.
Moreover, if $\eps \leq \ov\eps$, then by \cite[10.2]{Perelman1} we obtain a volume bound of the form $|B(x_0, t, t^{1/2})| \geq c t^{n/2}$, $t \in (0,1)$, for some constant $c > 0$ that is independent of $\delta$.
So if $t$ and $\delta$ are sufficiently small, then we obtain from Theorem~\ref{Thm_bckwds_pseudo} that $|{\Rm}| \leq C$ on $P(x_0, t; \frac12 t^{1/2}, - t)$, which implies (\ref{eq_needtoshowRmepsm2}) for $\eps \leq \ov\eps$.
\end{proof}
\bigskip

\begin{proof}[Proof of Theorem~\ref{Thm_bckwds_pseudo_optimal}.]
The theorem follows via a limit argument and successive application of Theorem~\ref{Thm_bckwds_pseudo} and Lemma~\ref{Lem_bckwd_pseudo_preliminary} at various scales.
\end{proof}

\section{Proofs of the theorems from Subsection~\ref{subsec_dimensions23_intro}} \label{subsec_dimensions23_proof}
Assume that $\XX$ is a non-collapsed $\IF$-limit of a sequence of $(n \leq 3)$-dimensional Ricci flows, as assumed in Subsection~\ref{subsec_dimensions23_intro}.

\begin{proof}[Proof of Theorem~\ref{Thm_SS_empty}.]
Suppose by contradiction that $x \in \SS_{t} \neq \emptyset$, $t < 0$.
By the discussion in Subsection~\ref{subsec_dimensions23_intro}, the tangent flows at $x$ are isometric to the round shrinking cylinder or its $\IZ_2$-quotient.
This implies that there is a sequence of points $x'_j \in \RR_{t'_j}$ with $t'_j \nearrow t$ and
\[ d^{\XX_{t'_j}}_{W_1}(\nu_{x;t'_j}, \delta_{x'_j}) \to 0 \]
such that $(\RR_{t'_j}, R(x'_j) g_{t'_j}, x'_j)$ smoothly converges to a pointed Riemannian manifold $(M_\infty, g_\infty, x_\infty)$ that is isometric to the standard round cylinder or its $\IZ_2$-quotient.
By adjusting the times $t'_j$ slightly, we may also assume that $\SS_{t'_j} = \emptyset$ for all $j$.
Pick a point $y \in \RR_{t}$ and choose $y'_j \in \RR_{t'_j}$ such that $y'_j \to y$.
Then
\[ \limsup_{j \to \infty} d_{t'_j} (x'_j, y'_j) 
= \limsup_{j \to \infty} d^{\XX_{t'_j}} (\nu_{x;t'_j}, \nu_{y;t'_j})
\leq d_t (x, y) =: D. \]
Choose $0 < r < 1$ such that $B(y, r) \subset \RR_{t}$.
Since $(\RR_{t'_j}, g_{t'_j})$ is a complete Riemannian manifold with non-negative sectional curvature we have by volume comparison for any $A < \infty$  and for large $j$
\begin{multline*}
 |B(y'_j, r)|_{g_{t'_j}} \leq |B(x'_j, D+2)|_{g_{t'_j}}
\leq \Big(\frac{D+2}{A R^{-1/2}(x'_j)}\Big)^3 \big|B(x'_j, A R^{-1/2}(x'_j))\big|_{g_{t'_j}} \\ 
\xrightarrow[j \to \infty]{} (D+2)^3 \frac{\big|B^{M_\infty}(x_\infty, A)\big|_{g_\infty}}{A^3}. 
\end{multline*}
So
\[ 0 < |B(y, r)|_{g_t} \leq (D+2)^3 \frac{\big|B^{M_\infty}(x_\infty, A)\big|_{g_\infty}}{A^3}. \]
Letting $A \to \infty$ produces the desired contradiction.

So $\SS_{<0} = \emptyset$, which implies that $(\RR_t, g_t)$ is a complete Riemannian manifold with non-negative sectional curvature for all $t <0$.
Therefore distances don't expand, which implies that the trajectories of $\partial_{\tf}$ are defined up to time $0$.
\end{proof}
\bigskip

To prove Theorem~\ref{Thm_lim_dim23_global_NC}, we first establish the following lemma.
Let $\XX$ be again a limit of $(n \leq 3)$-dimensional Ricci flows that is not equivalent to the constant flow on $\IR^n$.

\begin{Lemma} \label{Lem_XX_3d_if_loc_split_off}
Suppose that for some uniform $Y_0 <\infty$ we have $\NN_{x} (\tau) \geq - Y_0$ for all $x \in \XX$ and $\tau > 0$.
Suppose moreover, that there is some $t_0 \leq 0$ such that for all $t \leq t_0$ the time-slice $\RR_t$ locally splits off a line.
Then $n = 3$ and the flow on $\RR_{\leq t_0}$ is homothetic to the round shrinking cylinder  or its $\IZ_2$-quotient.
\end{Lemma}

\begin{proof}
If $\RR_t$ was flat for some $t < 0$, then by the strong maximum principle, the same would be true for all earlier time-slices.
Due to \cite[\HKThmNLC]{Bamler_HK_entropy_estimates} these time-slices must have positive asymptotic volume ratio and therefore be isometric to $\IR^3$.
By Perelman's Pseudolocality Theorem \cite[Sec.~10]{Perelman1} this would imply that all time-slices of $\RR$ were flat, which contradicts the assumptions from the beginning of Subsection~\ref{subsec_dimensions23_intro}.
So no time-slice $\RR_t$ is flat, which also implies that $n =3$.

By assumption, the universal cover of every time-slice $\RR_t$, $t \leq t_0$, is isometric to a cartesian product $M'_t \times \IR$ for some 2-dimensional, complete Riemannian manifold  $(M'_t, g'_t)$.
By Theorem~\ref{Thm_tangent_flow_at_infty_intro} the flow $\XX$ has tangent flow at infinity, which by the discussion in Subsection~\ref{subsec_dimensions23_intro} must be a gradient shrinking soliton that is isometric to the round shrinking cylinder or its $\IZ_2$-quotient.
This implies that $M'_t$ is compact for small $t$.
A simple open-closed argument implies that $M'_t$ must be compact for all $t \leq t_0$.
So $\RR_{\leq t_0}$ has bounded curvature on compact time-slices and thus the trajectories of $\partial_{\tf}$ are complete on $(-\infty,t_0]$.
So $\RR_{\leq t_0}$ is given by a $\kappa$-solution that locally splits off a line.
The lemma now follows from the classification of 2-dimensional $\kappa$-solutions; see for example \cite[Corollary~40.1]{Kleiner_Lott_notes}.
\end{proof}
\bigskip

\begin{proof}[Proof of Theorem~\ref{Thm_lim_dim23_global_NC}.]
Due to Perelman's Pseudolocality Theorem \cite[Sec.~10]{Perelman1} and the Backwards Pseudolocality Theorem~\ref{Thm_bckwds_pseudo}, it suffices to show that every time-slice of $\RR$ has bounded curvature.
So fix some $t < 0$ and assume by contradiction that $\RR_t$ has unbounded curvature.
By the same arguments as in the proof of \cite[Proposition~41.13]{Kleiner_Lott_notes}, we can find a sequence of points $x_j \in \RR_t$ and scales $\lambda_j \to 0$ such that the pointed Riemannian manifolds $(\RR_t, \lambda_j^{-2} g_t, x_j)$ smoothly converge to a non-flat, complete Riemannian manifold $(M'_\infty, g'_\infty,x'_\infty)$ that splits off an $\IR$-factor.
Then for large $j$ the points $x'_j := x_j (t+ \lambda_j^2)$ are defined and due to Perelman's Pseudolocality Theorem and \cite[\HKCenterConstantRmBound]{Bamler_HK_entropy_estimates} we have $d^{\XX_t}_{W_1} (\delta_{x_j}, \nu_{x'_j;t}) \leq C \lambda_j$ for some uniform $C < \infty$.

After passing to a subsequence, the metric flow pairs $(\XX_{< t+\lambda_j^2}, (\nu_{x'_j;t'})_{t'<t+\lambda_j^2})$, time-shifted by $t+\lambda_j^2$ and parabolically rescaled by $\lambda_j^{-1}$, converge to a metric flow pair $(\XX', (\nu'_{x'_\infty;t'})_{t' <0})$, which also occurs as a limit of $(n \leq 3)$-dimensional Ricci flows as assumed in Subsection~\ref{subsec_dimensions23_intro}.
So it has full regular part $\RR' = \XX'$ and by our choice of points $x'_j$ we know that $(\RR'_{-1}, g'_{-1})$ is isometric to $(M'_\infty,g'_\infty)$.
So by the strong maximum principle, $\RR'_{t'}$ locally splits off an $\IR$-factor for all $t' \leq -1$.
It follows from Lemma~\ref{Lem_XX_3d_if_loc_split_off} that $(M'_\infty,g'_\infty)$ is homothetic to the round cylinder.
The fact that $(\RR_t, g_t)$ has bounded curvature now follows as in the proof of \cite[Theorem~46.1]{Kleiner_Lott_notes}.
\end{proof}
\bigskip

\begin{proof}[Proof of Theorem~\ref{Thm_limit_3d_to_time_0_intro}.]
By Theorem~\ref{Thm_lim_dim23_global_NC}, we already know that the flow on $\RR$ is given by a smooth $\kappa$-solution $(M_\infty, (g_{\infty,t})_{t < 0})$.
The condition $\limsup_{i \to \infty} R(x_i, 0) < \infty$ implies that $\int_{M_\infty} R(\cdot,t)\, d\nu_{x_\infty;t} < C < \infty$ for $t < 0$.
So there are points $(x_i, t_i) \in M_\infty \times \IR_- \cong \RR$, with $t_i \nearrow 0$ and $d^{\XX_t}_{W_1} (\delta_{x_i}, \nu_{x_\infty;t}) \leq C \sqrt{-t_i}$, $R(x_i, t_i) < 0$.
By \cite[10.3, 11.7]{Perelman1} and \cite[\HKCenterConstantRmBound]{Bamler_HK_entropy_estimates} this implies that $(M_\infty, (g_{\infty,t})_{t < 0})$ can be extended smoothly to time 0.
\end{proof}
\bigskip

\section{Proof of the Theorem from Subsection~\ref{subsec_dim4_intro}}

\begin{proof}[Proof of Theorem~\ref{Thm_dim4_intro}.]
Assume that $n=4$ and consider the singular metric space $(X,d)$ in Theorem~\ref{Thm_static_limit_main} or \ref{Thm_metric_soliton_limit_main}.
By Assertion~\ref{Thm_static_limit_main_e} of Theorem~\ref{Thm_static_limit_main}, the tangent cones at every singular point of $X$ must be metric cones and by the dimensional bound in Assertion~\ref{Thm_static_limit_main_d}, these metric cones must be smooth and Ricci flat and therefore flat away from the origin.
This implies that all tangent cones are of the form $\IR^4/\Gamma$.
So by the smooth convergence statement in Assertion~\ref{Thm_static_limit_main_e} of Theorem~\ref{Thm_static_limit_main}, we obtain that the singular set $\SS_X \subset X$ consists of isolated points.
The fact that the Riemannian metric on $\RR_X$ can be extended smoothly to an orbifold metric over the singular set follows as in \cite{Haslhofer_Mueller_compactness_soliton, Cao_Sesum_compactness_KSol,X_Zhang_compactness_soliton, Tian_1990,Uhlenbeck_1982}.

The last statement of the theorem is a direct consequence of the last statement of Theorem~\ref{Thm_metric_soliton_limit_main}.
\end{proof}
\bigskip

\bibliography{bibliography}{}
\bibliographystyle{amsalpha}

\end{document}